\documentclass[12pt,pdftex,a4paper]{amsart}

\selectfont

\usepackage{qtree}
\usepackage{forest}

\usepackage{caption}

\usepackage{etoolbox}
\usepackage{epic,eepic,ecltree}
\usepackage{graphicx}
\usepackage{tikz}
\usetikzlibrary{decorations.text}
\usepackage{amssymb,color, euscript, enumerate}
\usepackage{amsthm}
\usepackage{amsmath}
\usepackage{braket}
\usepackage{amscd}
\usepackage{txfonts}
\usepackage{comment}
\usepackage{bm}
\usepackage{wrapfig}
\usepackage{amscd}
\numberwithin{equation}{section}

\makeatletter
\let\old@tocline\@tocline
\let\section@tocline\@tocline
\newcommand{\subsection@dotsep}{4.5}
\newcommand{\subsubsection@dotsep}{4.5}
\patchcmd{\@tocline}
  {\hfil}
  {\nobreak
     \leaders\hbox{$\m@th
        \mkern \subsection@dotsep mu\hbox{.}\mkern \subsection@dotsep mu$}\hfill
     \nobreak}{}{}
\let\subsection@tocline\@tocline
\let\@tocline\old@tocline

\patchcmd{\@tocline}
  {\hfil}
  {\nobreak
     \leaders\hbox{$\m@th
        \mkern \subsubsection@dotsep mu\hbox{.}\mkern \subsubsection@dotsep mu$}\hfill
     \nobreak}{}{}
\let\subsubsection@tocline\@tocline
\let\@tocline\old@tocline

\let\old@l@subsection\l@subsection
\let\old@l@subsubsection\l@subsubsection

\def\@tocwriteb#1#2#3{%
  \begingroup
    \@xp\def\csname #2@tocline\endcsname##1##2##3##4##5##6{%
      \ifnum##1>\c@tocdepth
      \else \sbox\z@{##5\let\indentlabel\@tochangmeasure##6}\fi}%
    \csname l@#2\endcsname{#1{\csname#2name\endcsname}{\@secnumber}{}}%
  \endgroup
  \addcontentsline{toc}{#2}%
    {\protect#1{\csname#2name\endcsname}{\@secnumber}{#3}}}%

\newlength{\@tocsectionindent}
\newlength{\@tocsubsectionindent}
\newlength{\@tocsubsubsectionindent}
\newlength{\@tocsectionnumwidth}
\newlength{\@tocsubsectionnumwidth}
\newlength{\@tocsubsubsectionnumwidth}
\newcommand{\settocsectionnumwidth}[1]{\setlength{\@tocsectionnumwidth}{#1}}
\newcommand{\settocsubsectionnumwidth}[1]{\setlength{\@tocsubsectionnumwidth}{#1}}
\newcommand{\settocsubsubsectionnumwidth}[1]{\setlength{\@tocsubsubsectionnumwidth}{#1}}
\newcommand{\settocsectionindent}[1]{\setlength{\@tocsectionindent}{#1}}
\newcommand{\settocsubsectionindent}[1]{\setlength{\@tocsubsectionindent}{#1}}
\newcommand{\settocsubsubsectionindent}[1]{\setlength{\@tocsubsubsectionindent}{#1}}

\renewcommand{\l@section}{\section@tocline{1}{\@tocsectionvskip}{\@tocsectionindent}{}{\@tocsectionformat}}%
\renewcommand{\l@subsection}{\subsection@tocline{2}{\@tocsubsectionvskip}{\@tocsubsectionindent}{}{\@tocsubsectionformat}}%
\renewcommand{\l@subsubsection}{\subsubsection@tocline{3}{\@tocsubsubsectionvskip}{\@tocsubsubsectionindent}{}{\@tocsubsubsectionformat}}%
\newcommand{\@tocsectionformat}{}
\newcommand{\@tocsubsectionformat}{}
\newcommand{\@tocsubsubsectionformat}{}
\expandafter\def\csname toc@1format\endcsname{\@tocsectionformat}
\expandafter\def\csname toc@2format\endcsname{\@tocsubsectionformat}
\expandafter\def\csname toc@3format\endcsname{\@tocsubsubsectionformat}
\newcommand{\settocsectionformat}[1]{\renewcommand{\@tocsectionformat}{#1}}
\newcommand{\settocsubsectionformat}[1]{\renewcommand{\@tocsubsectionformat}{#1}}
\newcommand{\settocsubsubsectionformat}[1]{\renewcommand{\@tocsubsubsectionformat}{#1}}
\newlength{\@tocsectionvskip}
\newcommand{\settocsectionvskip}[1]{\setlength{\@tocsectionvskip}{#1}}
\newlength{\@tocsubsectionvskip}
\newcommand{\settocsubsectionvskip}[1]{\setlength{\@tocsubsectionvskip}{#1}}
\newlength{\@tocsubsubsectionvskip}
\newcommand{\settocsubsubsectionvskip}[1]{\setlength{\@tocsubsubsectionvskip}{#1}}

\patchcmd{\tocsection}{\indentlabel}{\makebox[\@tocsectionnumwidth][l]}{}{}
\patchcmd{\tocsubsection}{\indentlabel}{\makebox[\@tocsubsectionnumwidth][l]}{}{}
\patchcmd{\tocsubsubsection}{\indentlabel}{\makebox[\@tocsubsubsectionnumwidth][l]}{}{}

\newcommand{\@sectypepnumformat}{}
\renewcommand{\contentsline}[1]{%
  \expandafter\let\expandafter\@sectypepnumformat\csname @toc#1pnumformat\endcsname%
  \csname l@#1\endcsname}
\newcommand{\@tocsectionpnumformat}{}
\newcommand{\@tocsubsectionpnumformat}{}
\newcommand{\@tocsubsubsectionpnumformat}{}
\newcommand{\setsectionpnumformat}[1]{\renewcommand{\@tocsectionpnumformat}{#1}}
\newcommand{\setsubsectionpnumformat}[1]{\renewcommand{\@tocsubsectionpnumformat}{#1}}
\newcommand{\setsubsubsectionpnumformat}[1]{\renewcommand{\@tocsubsubsectionpnumformat}{#1}}
\renewcommand{\@tocpagenum}[1]{%
  \hfill {\mdseries\@sectypepnumformat #1}}

\let\oldappendix\appendix
\renewcommand{\appendix}{%
  \leavevmode\oldappendix%
  \addtocontents{toc}{%
    \protect\settowidth{\protect\@tocsectionnumwidth}{\protect\@tocsectionformat\sectionname\space}%
    \protect\addtolength{\protect\@tocsectionnumwidth}{2em}}%
}
\makeatother

\makeatletter
\settocsectionnumwidth{2em}
\settocsubsectionnumwidth{2.5em}
\settocsubsubsectionnumwidth{3em}
\settocsectionindent{1pc}%
\settocsubsectionindent{\dimexpr\@tocsectionindent+\@tocsectionnumwidth}%
\settocsubsubsectionindent{\dimexpr\@tocsubsectionindent+\@tocsubsectionnumwidth}%
\makeatother

\settocsectionvskip{10pt}
\settocsubsectionvskip{0pt}
\settocsubsubsectionvskip{0pt}
    
\settocsectionformat{\bfseries}
\settocsubsectionformat{\mdseries}
\settocsubsubsectionformat{\mdseries}
\setsectionpnumformat{\bfseries}
\setsubsectionpnumformat{\mdseries}
\setsubsubsectionpnumformat{\mdseries}

\let\oldtableofcontents\tableofcontents
\renewcommand{\tableofcontents}{%
  \vspace*{-\linespacing}
  \oldtableofcontents}

\newcommand{\PB}{\mathbb{PB}}

\newcommand{\prk}{{\mathrm{pr}_{h}}}
\newcommand{\Mrr}{{M_{[r]}}}
\newcommand{\mrr}{{m_{[r]}}}
\newcommand{\comp}{{\mathrm{comp}}}
\newcommand{\Cset}{{\underline{\text{Set}}}}
\newcommand{\magma}{{\underline{\text{Magma}}}}
\newcommand{\magmas}{{\underline{\text{Magma}_*}}}

\newcommand{\Hom}{{\mathrm{Hom}}}

\newcommand{\an}{{\mathrm{an}}}

\newcommand{\cC}{{\mathcal C}}
\newcommand{\cD}{{\mathcal D}}

\newcommand{\Irr}{{\text{Irr}(V)}}

\newcommand{\Vmodu}{{\underline{V{\text{-mod}}}}}
\newcommand{\Vmodo}{{\underline{V{\text{-mod}}}^\op}}

\newcommand{\AVm}{{\underline{A(V){\text{-mod}}}}}
\newcommand{\AVmodu}{{\underline{A(V)\text{-mod}}_{\text{ fin}}}}

\newcommand{\Vmodf}{{\underline{V\text{-mod}}_{f.g.}}}
\newcommand{\Vmodfr}{{\underline{V\text{-mod}}_{f.g.}^r}}
\newcommand{\Vmodfo}{{\underline{V\text{-mod}}_{f.g.}^\op}}

\newcommand{\CPaB}{{\underline{\text{PaB}}}}
\newcommand{\PaB}{{\underline{\text{PaB}}}}

\newcommand{\N}{\mathbb{N}}
\newcommand{\Z}{\mathbb{Z}}
\newcommand{\R}{\mathbb{R}}
\newcommand{\C}{\mathbb{C}}
\newcommand{\Q}{\mathbb{Q}}

\newcommand{\bF}{{\overline{F}}}
\newcommand{\pr}{{\mathrm{pr}_{\Delta'}}}

\newcommand{\Y}{\mathcal{Y}}
\newcommand{\mO}{\mathrm{O}}
\newcommand{\Xr}{{X_r}}

\newcommand{\Or}{\mathrm{O}_{\Xr}}
\newcommand{\Om}{\Omega}
\newcommand{\Ort}{\mathrm{O}_{\Xr/\C}}
\newcommand{\Mr}{{M_{[0;r]}}}
\newcommand{\Mrh}{{M_{[0;\hat{r}]}}}
\newcommand{\Mrt}{{M_{[0;r]}^t}}
\newcommand{\Mra}{{M_{[0;r]}^{r_A}}}

\newcommand{\mr}{{m_{[0;r]}}}
\newcommand{\mrh}{{m_{[0;\hat{r}]}}}
\newcommand{\Ora}{\mathrm{O}^{\text{an}}(X_r)}
\newcommand{\Dr}{{\mathrm{D}_{\Xr}}}

\newcommand{\Leaf}{\text{Leaf}}
\newcommand{\reg}{\text{-reg}}
\newcommand{\Ae}{{A\circ_p \emptyset}}

\newcommand{\alg}{\text{alg}}

\newcommand{\conv}{\text{conv}}

\newcommand{\PI}{\mathcal{I}}

\newcommand{\std}{{\text{std}}}
\newcommand{\tstd}{{\overline{\text{std}}}}
\newcommand{\sstd}{{1(2(\dots(r-1 r)\dots)}}
\newcommand{\ee}{{e_{0}}}

\newcommand{\Drt}{{\mathrm{D}_{\Xr/\C}}}

\newcommand{\va}{\bm{1}}
\newcommand{\1}{\bm{1}}

\newcommand{\id}{{\mathrm{id}}}
\newcommand{\tF}{{\tilde{F}}}
\newcommand{\z}{{\bar{z}}}

\newcommand{\pa}{{\partial}}

\newcommand{\Vect}{{\underline{\text{Vect}}_\C}}
\newcommand{\cM}{{\mathcal{M}}}

\newcommand{\Ob}{{\text{Ob}}}

\newcommand{\al}{\alpha}
\newcommand{\ep}{\epsilon}

\newcommand{\ga}{\gamma}
\newcommand{\ze}{\zeta}

\newcommand{\om}{\omega}
\newcommand{\la}{\lambda}

\newcommand{\si}{\sigma}

\newcommand{\Log}{\mathrm{Log}}
\newcommand{\Arg}{\mathrm{Arg}}

\newcommand{\g}{{\mathfrak{g}}}

\newcommand{\cO}{\mathcal O}

\newcommand{\CB}{{\mathcal{CB}}}

\newcommand{\D}{{\mathbb{D}}}

\newcommand{\End}{\mathrm{End}}
\newcommand{\Func}{\mathrm{Func}}
\newcommand{\Endo}{{\mathcal{E}\mathrm{nd}}}
\newcommand{\Endp}{{\mathcal{PE}\mathrm{nd}}}
\newcommand{\cat}{{\underline{\mathrm{Cat}}_\mathbb{C}}}

\newcommand{\Tr}{{\mathcal{T}}}
\newcommand{\PW}{{\mathcal{PW}}}

\newcommand{\mfp}{{\mathfrak{p}}}
\newcommand{\cut}{\mathrm{cut}}

\newcommand{\op}{{\mathrm{op}}}

\newtheorem{thm}{Theorem}[section]
\newtheorem{dfn}[thm]{Definition}
\newtheorem{lem}[thm]{Lemma}
\newtheorem{prop}[thm]{Proposition}
\newtheorem{cor}[thm]{Corollary}
\newtheorem{rem}[thm]{Remark}

\newtheorem{mainthm}{Main Theorem}

\begin{document}

\begin{center}
{{\LARGE \bf Vertex operator algebra and parenthesized braid operad}
} \par \bigskip

\renewcommand*{\thefootnote}{\fnsymbol{footnote}}
{\normalsize
Yuto Moriwaki \footnote{email: \texttt{moriwaki.yuto (at) gmail.com}}
}
\par \bigskip
{\footnotesize RIKEN Center for Interdisciplinary Theoretical and
Mathematical Sciences,\\ Wako 351-0198, Japan}

\par \bigskip
\end{center}

\noindent

\vspace{10mm}

\begin{center}
\textbf{Abstract}
\end{center}

We study conformal blocks of vertex operator algebras on configuration spaces from the viewpoint of the parenthesized braid operad,
a combinatorial model of the fundamental groupoid of the little 2-disk operad. For each binary tree we introduce coordinates and a simply connected domain in the configuration space, and show that conformal blocks admit convergent expansions on these domains.
 Inserting one binary tree into a leaf of another
gives gluing maps for the corresponding conformal blocks, while analytic continuation along paths in configuration spaces gives isomorphisms between conformal blocks associated with different trees. We prove that these operations are compatible with the operadic composition in the parenthesized braid operad.

As a consequence, the category of $C_1$-cofinite modules whose contragredient
modules are finitely generated carries a canonical unital pseudo-braided
category structure, without assuming rationality or $C_2$-cofiniteness of the
vertex operator algebra. In the rational $C_2$-cofinite case, this structure
is represented by tensor products and recovers the balanced braided tensor
category structure with twist $\exp(2\pi iL(0))$.

\clearpage

\tableofcontents

\begin{center}
{\large \bf Introduction}
\end{center}
\vspace{5mm}

{Braided tensor categories} arise naturally from the geometry of configuration
spaces.  For \(n\geq 1\), let
\[
  X_n(\C)
  =
  \{(z_1,\ldots,z_n)\in \C^n \mid z_i\neq z_j
  \text{ for } i\neq j\}
\]
be the configuration space of \(n\) ordered points in the complex plane.
The {Knizhnik--Zamolodchikov (KZ) connections} on \(X_n(\C)\) illustrate this relation: their flat sections form local systems whose monodromy gives representations of pure braid groups. 
Moreover, the connections \(\mathrm{KZ}_n\) are compatible with insertion
of configurations: inserting a cluster of \(m\) points into one point of an
\(n\)-point configuration relates 
solutions on \(X_n(\C)\) and \(X_m(\C)\) to
those on \(X_{n+m-1}(\C)\).
Passing to monodromy, this compatibility
gives the relations between the associator and the braiding that underlie the
pentagon and hexagon identities of a braided tensor category
\cite{KZ,Dr2,Kohno,EFK}.

This paper studies an analogous structure for conformal blocks of vertex operator algebras. For a sequence of modules over a vertex operator algebra, conformal blocks form a sheaf of solutions of differential equations on configuration spaces \cite{FB,NT}.
For affine vertex operator algebras, this system specializes to the KZ equations \cite{TK}.
Thus conformal blocks provide a natural setting for studying the interplay
between analytic continuation and gluing operations arising from insertion of
configurations.

Factorization of conformal blocks and coinvariants has been studied extensively in the algebro-geometric setting, especially for sheaves on moduli spaces of stable pointed curves and under strong finiteness assumptions such as rationality and \(C_2\)-cofiniteness \cite{TUY,DGT1,DGT2}.
Analytic sewing and factorization for conformal blocks of
\(C_2\)-cofinite vertex operator algebras have also been developed \cite{Gui,GZ1,GZ2,GZ3}.

In contrast, we formulate the gluing of conformal blocks directly on
configuration spaces and describe it in terms of the combinatorics of binary
trees. A binary tree encodes a parenthesization of an \(n\)-fold product. In the
setting of vertex operator algebras, 
this combinatorial datum
specifies the shape
of an iterated composition of vertex operators and determines a convergence
domain in the configuration space, together with a natural system of coordinates
on that domain. On these domains, conformal blocks admit convergent expansions.
Moreover, inserting one tree into another gives gluing maps for conformal blocks, defined by absolutely convergent sums over intermediate
states. Analytic continuation along paths in configuration spaces gives
isomorphisms between conformal blocks associated with different trees.

%

The main result, proved as Theorem \ref{thm_action}, is that these gluing maps
and analytic continuations satisfy the compatibility governed by the
\emph{parenthesized braid operad} \(\PaB\), a combinatorial model of the
fundamental groupoid of the little \(2\)-disks operad \cite{Bar,Ta1,Fr}. 
The operad \(\PaB\) encodes both insertion of parenthesized configurations
and analytic continuation between the associated convergence domains.

As a consequence, the category of \(C_1\)-cofinite modules whose contragredient
modules are finitely generated carries a canonical \emph{unital pseudo-braided
category structure}, without assuming rationality or \(C_2\)-cofiniteness of the
vertex operator algebra. In the rational \(C_2\)-cofinite case, Theorem
\ref{thm_BTC} shows that this pseudo-braided structure is represented by tensor
products and recovers the balanced braided tensor category structure constructed
by Huang--Lepowsky and Huang--Lepowsky--Zhang.

We now make this construction more concrete by explaining the relation between
binary trees, parenthesized products of vertex operators and their convergence domains.

%
 A vertex operator algebra \(V\) has
a product depending on a formal variable,
\begin{align*}
  V\otimes V \longrightarrow V((z)),
  \qquad
  a\otimes b \longmapsto Y(a,z)b
  =
  \sum_{n\in\mathbb Z}a(n)b\,z^{-n-1}.
\end{align*}
Iterating this product requires a choice of parenthesization, hence a binary
tree, e.g., 
\begin{align}
  Y(Y(a_1,z_{12})a_2,z_{23})a_3,
  \qquad
  Y(a_1,z_{13})Y(a_2,z_{23})a_3
  \label{intro_iterated}
\end{align}
correspond to the binary trees
\begin{tikzpicture}[baseline=-.55ex,scale=.45]
  \coordinate (r) at (0,2);
  \coordinate (v) at (-.7,1);
  \coordinate (l1) at (-1.4,0);
  \coordinate (l2) at (-.25,0);
  \coordinate (l3) at (.8,0);
  \draw (r)--(v);
  \draw (r)--(l3);
  \draw (v)--(l1);
  \draw (v)--(l2);
  \node[below=1pt] at (l1) {$1$};
  \node[below=1pt] at (l2) {$2$};
  \node[below=1pt] at (l3) {$3$};
\end{tikzpicture}
\quad\text{and}\quad
\begin{tikzpicture}[baseline=-.55ex,scale=.45]
  \coordinate (r) at (0,2);
  \coordinate (v) at (.7,1);
  \coordinate (l1) at (-.8,0);
  \coordinate (l2) at (.25,0);
  \coordinate (l3) at (1.4,0);
  \draw (r)--(l1);
  \draw (r)--(v);
  \draw (v)--(l2);
  \draw (v)--(l3);
  \node[below=1pt] at (l1) {$1$};
  \node[below=1pt] at (l2) {$2$};
  \node[below=1pt] at (l3) {$3$};
\end{tikzpicture},
respectively.  Writing \(z_{ij}=z_i-z_j\), these formal series converge
absolutely on the regions in $X_3(\C)$
\begin{align}
\{|z_{1}-z_2|<|z_{2}-z_3|\},\qquad \{|z_{2}-z_3|<|z_{1}-z_3|\},\label{intro_domain_ex3}
\end{align}
%
and their analytic continuations coincide as holomorphic functions on
\(X_3(\C)\).
The two domains \eqref{intro_domain_ex3} are described by $|\zeta|<1$, $|\zeta'|<1$, where
\begin{align}
  \zeta=\frac{z_1-z_2}{z_2-z_3},
  \qquad
  \zeta'=\frac{z_2-z_3}{z_1-z_3}.
  \label{intro_zeta}
\end{align}
In general, for each binary tree
\(A\in\Tr_n\), we introduce coordinates on \(X_n(\C)\) adapted to \(A\) and a simply connected open subset
\[
  U_A\subset X_n(\C),
\]
which is the convergence domain of the \(A\)-shaped iterated vertex operators.
Conformal blocks are expanded on these domains, and the gluing maps are defined on the domains corresponding to tree insertions.

Let \(V\) be a vertex
operator algebra of CFT type,
\[
  V=\bigoplus_{n\geq 0}V_n,\qquad
  V_0=\C\mathbf 1,\qquad
  \dim V_n<\infty.
\]
A sequence of \(V\)-modules $M_{[0;n]}=(M_0,M_1,\ldots,M_n)$ determines a \(\mathcal D\)-module \(\mathcal{D}_{M_{[0;n]}}\) on \(X_n(\C)\), which encodes a
system of differential equations. The corresponding conformal block is the sheaf of
holomorphic solutions
\[
  \CB_{M_{[0;n]}}(U)
  =
  \Hom_{\mathcal{D}_{X_n(\C)}}
  \bigl(\mathcal{D}_{M_{[0;n]}},\mathcal O_{X_n(\C)}(U)\bigr).
\]
Here \(M_1,\ldots,M_n\) are the modules inserted at \(z_1,\ldots,z_n\), while \(M_0\) is the module inserted at
\(\infty\).
If \(M_1,\ldots,M_n\) are \(C_1\)-cofinite and \(M_0^\vee\) is finitely
generated, then this sheaf is locally constant.  Hence paths in \(X_n(\C)\)
act on conformal blocks by analytic continuation.


\begin{wrapfigure}{r}{0.4\textwidth}
  \centering
  \vspace{-2mm}
  \includegraphics[width=\linewidth]{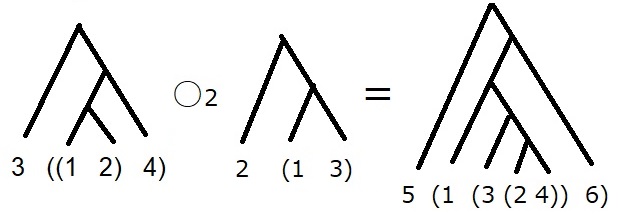}
\end{wrapfigure}

For \(A\in\Tr_r\), \(B\in\Tr_s\), and \(1\leq p\leq r\), let
\[
  A\circ_p B\in\Tr_{r+s-1}\qquad\qquad\qquad\qquad\qquad\qquad\qquad\qquad
\]
be the tree obtained by inserting \(B\) into the\\ \(p\)-th leaf of \(A\), as
shown on the right.

On conformal blocks, the corresponding operation is a gluing map
\begin{align}
\begin{split}
\comp_p:
  \CB_{M_0;M_1,\ldots,M_{p-1},L,M_{p+1},\ldots,M_r}(U_A)
  \otimes
  \CB_{L;N_1,\ldots,N_s}(U_B)\\
  \longrightarrow
  \CB_{M_0;M_1,\ldots,M_{p-1},N_1,\ldots,N_s,M_{p+1},\ldots,M_r}
  (U_{A\circ_p B}).
\end{split}
\label{intro_glue2}
\end{align}
The target is the conformal block on \(U_{A\circ_p B} \subset X_{r+s-1}(\C)\).  The map is defined by an absolutely convergent sum over
intermediate states, and it agrees with the corresponding iterated
composition of logarithmic intertwining operators \eqref{intro_iterated}.

For \(A,B\in\Tr_n\), analytic continuation along a path \(\gamma\) in
\(X_n(\C)\) from \(U_A\) to \(U_B\) gives an isomorphism
\begin{align}
  \rho(\gamma):
  \CB_{M_{[0;n]}}(U_A)
  \longrightarrow
  \CB_{M_{[0;n]}}(U_B).
  \label{intro_mon2}
\end{align}

We prove that the gluing maps in \eqref{intro_glue2} are compatible with analytic continuations \eqref{intro_mon2}.
This compatibility is encoded by the {parenthesized braid operad} \(\PaB\). 

\begin{wrapfigure}{r}{0.4\textwidth}
  \centering
  \includegraphics[scale=0.13]{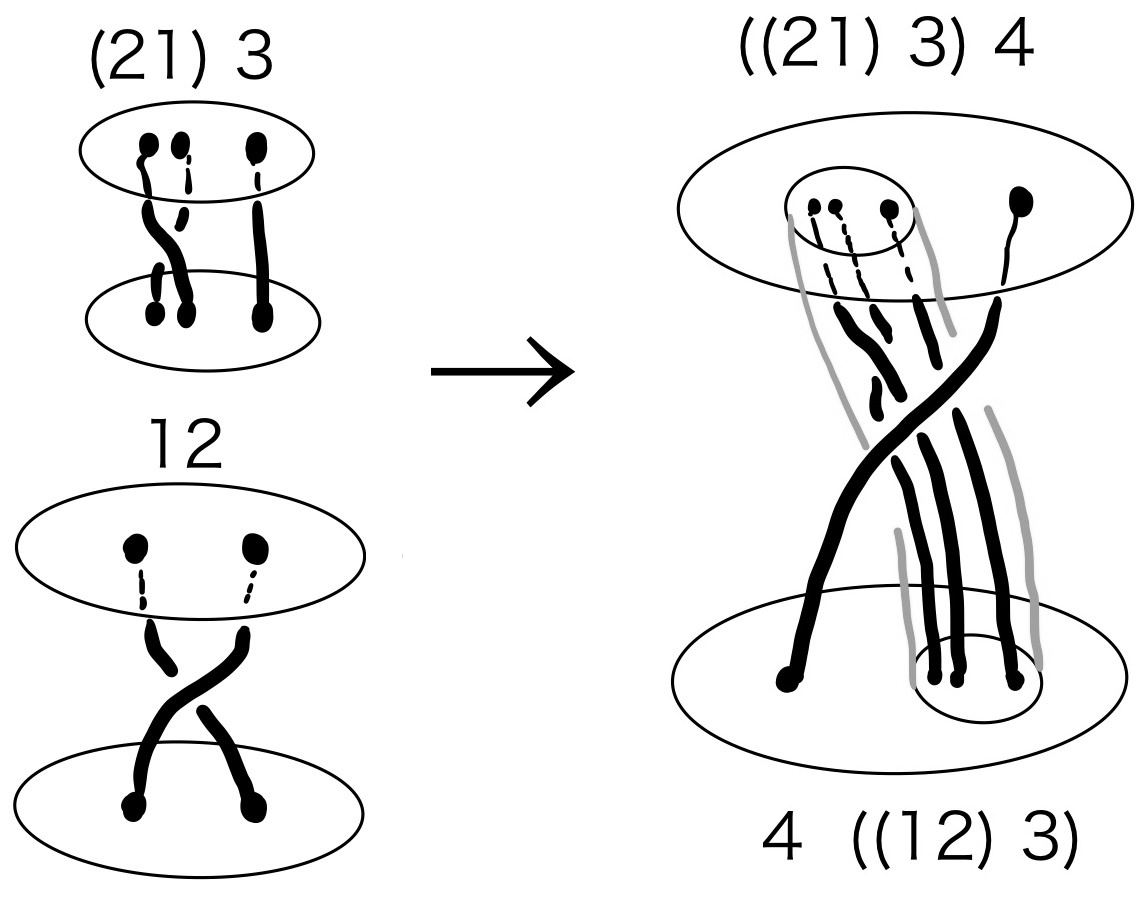}
\end{wrapfigure}

The objects of $\PaB(r)$ are the binary trees \(\Tr_r\), and its
morphisms are parenthesized braids; geometrically, they may be represented by homotopy classes of paths in $X_r(\mathbb C)$ whose endpoints are specified by the corresponding parenthesizations.
The operadic composition inserts one parenthesized braid into a strand of another, as shown on the right. Thus \(\PaB\) encodes both the
insertion of trees and analytic continuation between the associated
convergence domains.

%

Let \(\Vmodf\) denote the category of \(C_1\)-cofinite \(V\)-modules whose
contragredient modules are finitely generated (for the precise definition, see Section \ref{sec_pre_vertex}).  For \(A\in\Tr_r\), set
\[
  \Hom_A(M_1,\ldots,M_r;M_0)
  =
  \CB_{M_{[0;r]}}(U_A).
\]
As a functor, this is contravariant in \(M_1,\ldots,M_r\) and covariant in
\(M_0\).  We regard it as the multi-hom space with inputs
\(M_1,\ldots,M_r\) and output \(M_0\).
The gluing maps give the composition of these multi-hom spaces, while
morphisms in \(\PaB\) act by analytic continuation between the spaces
associated with different trees.
We prove that these operations are compatible with the operadic composition in $\PaB$, and hence define a {unital pseudo-braided category} in the sense of Soibelman \cite{So}.

When \(V\) is rational and \(C_2\)-cofinite, this pseudo-braided category is
represented by tensor products.
Choose a complete set \(\{N_\lambda\}_{\lambda\in\Irr}\) of simple
\(V\)-modules. We realize the tensor product by
\[
  M_1\boxtimes M_2
  =
  \bigoplus_{\lambda\in\Irr}
  N_\lambda\otimes_{\C}
  I\binom{N_\lambda}{M_1\,M_2}^{*}.
\]
The dual appears because the space of intertwining operators $I\binom{N_\lambda}{M_1\,M_2}$ is
contravariant in $M_1$ and $M_2$, whereas the tensor product is
covariant.  For a binary tree \(A\), let
\(\boxtimes_A(M_1,\ldots,M_r)\) be the iterated tensor product of shape
\(A\); for example,
$\boxtimes_{(31)2}(M_1,M_2,M_3)
  =
  (M_3\boxtimes M_1)\boxtimes M_2$.
We prove natural isomorphisms
\begin{align}
  \Hom_{V\text{-mod}_{f.g.}}
  \bigl(\boxtimes_A(M_1,\ldots,M_r),M_0\bigr)
  \cong
  \CB_{M_{[0;r]}}(U_A)\qquad \text{ for all \(A\in\Tr_r\).}
  \label{intro_yoneda}
\end{align}
By the Yoneda lemma, analytic continuation of conformal blocks gives the
associator and braiding.  Together with the  twist
\[
  \theta_M=\exp(2\pi iL(0)),
\]
this yields a balanced braided tensor category structure on the module
category of \(V\). 

\vspace{3mm}
\noindent
\begin{center}
{0.1. \bf 
Trees, coordinates and convergence domains
}
\end{center}
We now describe more systematically the geometric data attached to a binary tree.



Notice that the variables $z_{ij}$ of vertex operators appearing in \eqref{intro_iterated} are different. The expression \(Y(a,z)b\) corresponds to inserting \(a\) at \(z\) and \(b\) at \(0\).  Since this operation is not symmetric in \(a\) and \(b\), 
the natural variables appearing in an iterated product depend on the shape
of the binary tree.  They are attached to the vertices of the tree, according
to the rule described in Section \ref{sec_config}.  By contrast, the ratio
functions in \eqref{intro_zeta} should be regarded as coordinates attached
to the internal edges.

%
%
\begin{minipage}[c]{.6\textwidth}
For example, for the tree on the right, the variables are assigned to the
internal vertices.  The corresponding iterated vertex operator can be
written as
\begin{align*}
Y\left(Y\left(a_2,z_{23}\right)a_3,z_{34}\right)Y\left(Y\left(a_1,z_{15}\right)a_5,z_{54}\right)a_4
\end{align*}
The internal edges, labeled by \(a,b,c\), give the ratios
\begin{align*}
  \ze_a=\frac{z_2-z_3}{z_3-z_4},\qquad
  \ze_b=\frac{z_1-z_5}{z_5-z_4},\qquad
  \ze_c=\frac{z_5-z_4}{z_3-z_4}.
\end{align*}
Then
$\Psi_{(23)((15)4)}=
(z_4,z_3-z_4,\ze_a,\ze_b,\ze_c):X_5(\C) \rightarrow \C^5$
is a local coordinate system on $X_5(\C)$.
\end{minipage}
\begin{minipage}[c]{.4\textwidth}
\centering
\begin{forest}
for tree={
  l sep=20pt,
  parent anchor=south,
  align=center
}
[$z_3-z_4$
[$z_2-z_3$,edge label={node[midway,left]{a}}[2][3]]
[$z_5-z_4$,edge label={node[midway,right]{c}}[$z_1-z_5$,edge label={node[midway,left]{b}}[1][5]]
[4]]
]
\end{forest}
\label{intro_fig_tree_example2}
\end{minipage}

\vspace{2mm}

In general, let \(E(A)\) be the set of internal edges of
\(A\in \Tr_n\).  To each \(e\in E(A)\), we assign the ratio of the variables attached
to the upper and lower vertices of \(e\).  This gives a local coordinate
system
\[
\Psi_A=(z_A,x_A,(\zeta_e)_{e\in E(A)}):X_n(\C) \rightarrow \C^n.
\]
By imposing suitable inequalities on the variables
$\{\zeta_e\}_{e\in E(A)}$ together with a
choice of branches for the multivalued functions, we define a simply connected domain
$U_A\subset X_n(\mathbb C)$.

The same coordinates also define spaces of formal expansions. Let \(T_A\)
be the space of formal series in the \(A\)-coordinates. A typical element
has the form
\[
  x_A^\alpha (\log x_A)^{k}
  \prod_{e\in E(A)} \zeta_e^{\alpha_e} (\log\ze_e)^{k_e} \cdot F(\zeta)
\]
with $F(\zeta) \in  \C[[\ze_e \mid e\in E(A)]]$, $k,k_e \in \Z_{\geq 0}$ and $\al,\al_e \in \C$. 
We denote by \(T_A^{\mathrm{conv}}\) the subspace of \(T_A\) consisting of
series that converge absolutely on the domain \(U_A\).

\vspace{3mm}
\noindent
\begin{center}
{0.2. \bf 
Conformal blocks and gluing
}
\end{center}

We next describe conformal blocks, their gluing, and analytic continuation.
Let
$M_{[0;r]}=(M_0,M_1,\ldots,M_r)$ be a sequence of \(V\)-modules. One associates to \(M_{[0;r]}\) a \(D_{X_r(\mathbb C)}\)-module
\(D_{M_{[0;r]}}\) and its sheaf of holomorphic solutions
\[
  \CB_{M_{[0;r]}}
  =
  \Hom_{D_{X_r(\mathbb C)}}
  \bigl(D_{M_{[0;r]}},\mathcal O_{X_r(\mathbb C)}\bigr).
\]
By refining the argument in \cite{H2}, if $M_i \in \Vmodf$, then
\(D_{M_{[0;r]}}\) is a finitely generated \(\mathcal O_{X_r(\mathbb C)}\)-module, and the
solution sheaf
\(\CB_{M_{[0;r]}}\) is locally constant of finite rank.
For each tree \(A\in\Tr_r\), we analyze the singularities of
\(D_{M_{[0;r]}}\) along the divisors \(\zeta_e=0\) (${e\in E(A)}$) in the \(A\)-coordinates and show that expansion induces an isomorphism
(Theorem \ref{thm_expansion})
\[
  s_A:
  \Hom_{D_{X_r(\mathbb C)}}(D_{M_{[0;r]}},T_A^{\mathrm{conv}})
  \xrightarrow{\;\cong\;}
  \CB_{M_{[0;r]}}(U_A).
\]

Thus a conformal block on \(U_A\) can
be regarded as a convergent formal series in the $A$-coordinates. In the case \(r=2\), this gives the identification
between the spaces of logarithmic intertwining operators and conformal blocks on \(U_{12} \subset X_2(\mathbb C)\):
\begin{align}
  I_{\log}\binom{M_0}{M_1\,M_2}
  \cong
  \CB_{(M_0;M_1,M_2)}(U_{12}).
\end{align}

We next define the gluing operation. Suppose \(A\in\Tr_n\), \(B\in\Tr_m\), and
\(1\leq p\leq n\). Let
\[
  C_A\in
  \CB_{(M_0;N_1,\ldots,N_{p-1},L,N_{p+1},\ldots,N_n)}(U_A)
\]
and
\[
  C_B\in
  \CB_{(L;M_1,\ldots,M_m)}(U_B).
\]
The gluing operation inserts \(C_B\) into the \(p\)-th input of \(C_A\):
choosing a homogeneous basis \(\{e_i\}\) of \(L\) with dual basis
\(\{e^i\}\subset L^\vee\), it is formally given by
\[
  C_A\circ_p C_B
  =
  \sum_i
  C_A(\ldots,e_i,\ldots)\,
  C_B(e^i,\ldots).
\]
We prove that, in the coordinates associated with \(A\circ_p B\), this
series converges absolutely on \(U_{A\circ_p B}\) and defines a conformal
block (Theorem \ref{thm_composition}). 
This gives a composition map
\[
\begin{aligned}
\comp_p:&\CB_{(M_0;N_1,\ldots,N_{p-1},L,N_{p+1},\ldots,N_n)}(U_A)
  \otimes
  \CB_{(L;M_1,\ldots,M_m)}(U_B)
  \\
  &\qquad\longrightarrow
  \CB_{(M_0;N_1,\ldots,N_{p-1},M_1,\ldots,M_m,N_{p+1},\ldots,N_n)}
  (U_{A\circ_p B}).
\end{aligned}
\]
Moreover, if \(\ga:A\to A'\) is a morphism in
the parenthesized braid operad, choose a geometric representative of \(\ga\) by
a path from \(U_A\) to \(U_{A'}\) in \(X_n(\mathbb C)\). Since the conformal
block is locally constant, analytic continuation along this path gives
an isomorphism
\[
  \rho(\ga):
  \CB_{M_{[0;n]}}(U_A)
  \rightarrow
  \CB_{M_{[0;n]}}(U_{A'}).
\]
The gluing maps are compatible with these analytic continuations.
Namely,
for a path $\ga:A \rightarrow A'$ in $X_n(\C)$ and a path
$\mu: B \rightarrow B'$ in $X_m(\C)$, the analytic continuations associated
with these paths make the following diagram commute
(Proposition \ref{prop_natural_comp}):
\begin{align}
\begin{array}{ccc}
\CB(U_A)\otimes \CB(U_B)  & \overset{\comp_p}{\rightarrow} &\CB(U_{A\circ_p B})\\
{}_{\rho(\ga)\otimes \rho(\mu)}\downarrow && \downarrow_{\rho(\ga \circ_p \mu)} \\
\CB(U_{A'})\otimes \CB(U_{B'}) & \overset{\comp_p}{\rightarrow} &\CB(U_{A'\circ_p B'}).
\end{array}
\label{eq_intro_com}
\end{align}
Consequently, $\Vmodf$ has the structure of a unital pseudo-braided category with multi-hom spaces:
$\Hom_A(M_1,\ldots,M_n;M_0)
  =
  \CB_{M_{[0;n]}}(U_A)$ (see Theorem \ref{thm_action}).

\vspace{3mm}
\noindent
\begin{center}
{0.3. \bf 
From pseudo-braided categories to braided tensor categories
}
\end{center}

The pseudo-braided category constructed above is not, by itself, a tensor
category.  A braided tensor category structure is obtained only after the
multi-morphism functors are represented by tensor products.  In Theorem
\ref{thm_add_BTC}, we give a necessary and sufficient criterion for this
representability.
We apply this criterion in the rational \(C_2\)-cofinite case. 

Assume that
\(V\) is rational and \(C_2\)-cofinite.
Let
\(\{N_\lambda\}_{\lambda\in\Irr}\) be a complete set of representatives of
simple \(V\)-modules. For \(V\)-modules \(M_1,M_2\), define
\begin{align}
  M_1\boxtimes M_2
  =
  \bigoplus_{\lambda\in\Irr}
  N_\lambda\otimes_{\mathbb C}
  I\binom{N_\lambda}{M_1\,M_2}^{*}.
  \label{intro_dual}
\end{align}
Here \(I\binom{N_\lambda}{M_1\,M_2}\) is the space of intertwining operators
of type \(\binom{N_\lambda}{M_1\,M_2}\).
For a tree \(A\in\Tr_r\), let
$\boxtimes_A(M_1,\ldots,M_r)$
be the iterated tensor product of shape \(A\). We prove natural isomorphisms (Theorem \ref{cor_factorize})
\begin{align}
\mu_A:
\Hom_V\bigl(\boxtimes_A(M_1,\ldots,M_r),M_0\bigr)
  \cong
  \CB_{M_{[0;r]}}(U_A).\label{intro_rep}
\end{align}
Thus the multi-hom spaces of the pseudo-braided category are represented by the
tensor products \(\boxtimes_A\).
A morphism \(\ga:A\to B\) in the parenthesized braid operad gives
an analytic continuation $\rho(\ga):\CB_{M_{[0;r]}}(U_A)
  \longrightarrow
  \CB_{M_{[0;r]}}(U_B)$.
Through \eqref{intro_rep}, by the Yoneda lemma, this becomes a natural isomorphism in the opposite direction
\begin{align*}
\boxtimes_B (M_1,\ldots,M_r) \rightarrow \boxtimes_A (M_1,\ldots,M_r).
\end{align*}
Hence, we define $\rho_\boxtimes(\ga):\boxtimes_A (M_1,\ldots,M_r) \rightarrow \boxtimes_B (M_1,\ldots,M_r)$
by the following diagram:
\begin{align*}
\begin{split}
\begin{array}{ccc}
\Hom(\boxtimes_B (M_1,\ldots,M_r),L) &\overset{\mu_{B}}{\longrightarrow}& \CB_{\Mr}(U_B)\\
\downarrow{\rho_{\boxtimes}(\ga)^*}
      && 
    \downarrow{\rho(\ga)^{-1}}
    \\
\Hom(\boxtimes_A (M_1,\ldots,M_r),L) &\overset{\mu_{A}}{\longrightarrow}& \CB_{\Mr}(U_A)
\end{array}
\end{split}
\end{align*}

This contravariance is the reason inverse maps appear in the explicit
formulas for the braiding and associator, as well as the dual appearing in
\eqref{intro_dual}.  Together with \eqref{eq_intro_com}, these formulas
define a braided tensor category structure on $\Vmodf$.

\begin{wrapfigure}{r}{0.2\textwidth}
  \centering
  \vspace{-5mm}
    \includegraphics[width=2cm]{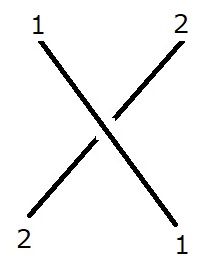}
    \caption{$\si$}\label{intro_fig_sigma}
\end{wrapfigure}

We conclude by describing this braided tensor category structure explicitly.
Let $\sigma:(12)\longrightarrow(21)$
be the morphism in $\PaB(2)$ shown 
in Fig. \ref{intro_fig_sigma}.
We choose its geometric representative in \(X_2(\mathbb C)\)
to be the clockwise half-turn
\footnote{The operad $\CPaB$ is defined abstractly in terms of braid
groups (see Section \ref{sec_def_cpab}).  When one realizes $\sigma$ as a path in $X_2(\C)$, there are two
choices, clockwise and counterclockwise.  Choosing the opposite direction
gives the reverse braiding.  The clockwise convention used here is the one
compatible with the standard twist \(\theta_M=\exp(2\pi iL(0))\).}.
 For \(M_1,M_2\), analytic
continuation along this path gives a linear isomorphism
\[
  R_{\lambda;M_1,M_2}:
  I\binom{N_\lambda}{M_1\,M_2}
  \longrightarrow
  I\binom{N_\lambda}{M_2\,M_1}.
\]
More precisely, \(R_{\lambda;M_1,M_2}\) is obtained by identifying
intertwining operators with conformal blocks on \(U_{12}\), analytically
continuing along \(\sigma\), and then expanding on \(U_{21}\). Under the
decomposition of \(M_1\boxtimes M_2\), the categorical braiding is given by
\[
  B_{M_1,M_2}
  =
  \bigoplus_{\lambda\in\Irr}
  \id_{N_\lambda}\otimes
  \bigl(R_{\lambda;M_1,M_2}^{-1}\bigr)^*.
\]

\begin{wrapfigure}{r}{0.2\textwidth}
  \centering
  \vspace{-12mm}
    \includegraphics[width=2.5cm]{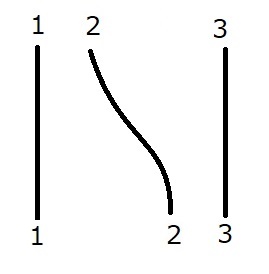}
    \caption{$\al$}\label{intro_fig_alpha}
\end{wrapfigure}
The associator is obtained in the same way from the morphism
$\alpha:(12)3\longrightarrow 1(23)$ of \(\CPaB(3)\) shown in Fig. \ref{intro_fig_alpha}.
 For \(M_1,M_2,M_3\), set
\[
  \mathcal H^{(12)3}_\lambda
  =
  \bigoplus_{\mu\in\Irr}
  I\binom{N_\lambda}{N_\mu\,M_3}
  \otimes
  I\binom{N_\mu}{M_1\,M_2},
\]
and
\[
  \mathcal H^{1(23)}_\lambda
  =
  \bigoplus_{\nu\in\Irr}
  I\binom{N_\lambda}{M_1\,N_\nu}
  \otimes
  I\binom{N_\nu}{M_2\,M_3}.
\]
By \eqref{intro_rep}, these spaces are naturally identified with $\CB_{N_\la;M_1,M_2,M_3}(U_{(12)3})$ and
$\CB_{N_\la;M_1,M_2,M_3}(U_{1(23)})$.
Let $F_{\lambda;M_1,M_2,M_3}:
  \mathcal H^{(12)3}_\lambda
  \longrightarrow
  \mathcal H^{1(23)}_\lambda$ be the linear isomorphism obtained by analytic continuation along \(\alpha\).
Then the categorical associator is given by
\[
  \alpha_{M_1,M_2,M_3}
  =
  \bigoplus_{\lambda\in\Irr}
  \id_{N_\lambda}\otimes
  \bigl(F_{\lambda;M_1,M_2,M_3}^{-1}\bigr)^*.
\]

The pentagon identities follow from the operadic relations in
\(\CPaB\).  In fact,
\begin{align*}
\begin{array}{ccc}
& (M_1 \boxtimes M_2) \boxtimes (M_3 \boxtimes M_4)\\
 {}^{{\alpha_{M_1 \boxtimes M_2, M_3, M_4}}}\nearrow
&&
\searrow^{{\alpha_{M_1,M_2,M_3 \boxtimes M_4}}}
\\
((M_1 \boxtimes M_2 ) \boxtimes M_3) \boxtimes M_4
&&
M_1 \boxtimes (M_2 \boxtimes (M_3 \boxtimes M_4))
\\
{}^{{\alpha_{M_1,M_2,M_3}} \boxtimes id_{M_4} }\downarrow 
&& 
\uparrow^{{ id_{M_1} \boxtimes \alpha_{M_2,M_3,M_4} }}
\\
(M_1 \boxtimes (M_2 \boxtimes M_3)) \boxtimes M_4
& \underset{\alpha_{M_1,M_2 \boxtimes M_3, M_4}}{\longrightarrow} &
M_1 \boxtimes ( (M_2 \boxtimes M_3) \boxtimes M_4)
\end{array}
\end{align*}
\begin{figure}[h]
\begin{minipage}[l]{.58\textwidth}
\vspace{-4mm}

follows from \eqref{eq_intro_com}, because the corresponding paths in
$X_4(\C)$ are homotopic, as indicated in the figure on the right, and each
path in the diagram decomposes as an operadic composition of paths in
$X_3(\C)$ and $X_2(\C)$.
The hexagon identities are obtained in the same way from the relations in \(\PaB\).
Together with the twist
\[
  \theta_M=\exp(2\pi iL(0)),
\]
\end{minipage}
\begin{minipage}[r]{.4\textwidth}
\vspace{-5mm}
    \centering
    \includegraphics[scale=0.62]{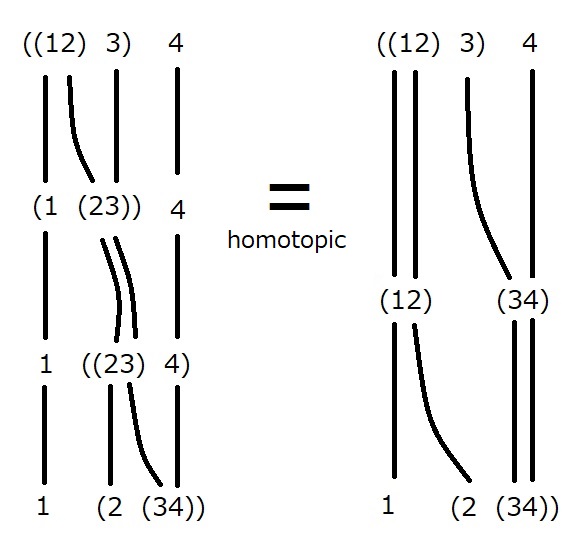}
\end{minipage}
\end{figure}
\vspace{-4mm}

\noindent
the above construction gives a balanced braided tensor category structure on $\Vmodf$ (see Theorem \ref{thm_BTC}).
This conclusion is in line with the tensor category theory of Huang--Lepowsky and
Huang--Lepowsky--Zhang, where the associativity and braiding constraints are
constructed from the convergence and analytic continuation properties of logarithmic intertwining operators \cite{HL1,HL2,HL3,HL4,H1,H2,HLZ1,HLZ2,HLZ3,HLZ4,HLZ5,HLZ6,HLZ7,HLZ8}.

The formulas for the braiding and associator are worked out in concrete
examples in \cite{MPalermo}.
For the simple Virasoro vertex operator algebra \(L(\frac12,0)\), the relevant conformal blocks on
\(X_3(\mathbb C)\) are solutions of the BPZ differential equations and can be
written in terms of hypergeometric functions \cite{BPZ}. Their analytic continuation
gives the braiding and associator of the module category, which is identified
with a Tambara--Yamagami category over \(\mathbb Z/2\mathbb Z\).

Finally, the construction also has a natural interpretation in full
two-dimensional conformal field theory. In a full, non-chiral theory, correlation
functions are real-analytic functions on configuration spaces and are locally
expanded by real-analytic operator product expansions \cite{HK, M1,AMT}. From
this viewpoint, the domains \(U_A\) are the convergence regions for iterated
real-analytic OPEs associated with binary trees \cite{MSwiss}. In rational theories, full
correlation functions are obtained by gluing monodromy-invariant combinations
of products of holomorphic and anti-holomorphic conformal blocks
\cite{MS1, MS2, FRS,HK}.

\vspace{5mm}

This paper is organized as follows.  In Section \ref{sec_pre}, we recall the
basic definitions of vertex operator algebras, the class of modules used in
the paper, contragredient modules, \(C_1\)-cofiniteness, and logarithmic
intertwining operators.  In Section \ref{sec_operad}, we review the operadic
background.  We recall the magma operad of binary trees, the parenthesized
braid operad \(\PaB\), the formulation of unital pseudo-braided categories
in terms of \(\PaB\), 
and the representability criterion under which a unital pseudo-braided
category gives rise to a braided tensor category.
In Section \ref{sec_config}, we introduce the coordinates and convergence
domains attached to binary trees.  For each \(A\in\Tr_r\), we define the
\(A\)-coordinates, the simply connected domain \(U_A\subset X_r(\C)\), and
the space \(T_A^{\mathrm{conv}}\) of convergent parenthesized formal series.
In Section \ref{sec_D}, we define the \(\mathcal D_{X_r(\C)}\)-modules whose
solution sheaves are conformal blocks.  We prove the finiteness of these
\(\mathcal D\)-modules over \(\mathcal O_{X_r(\C)}\), analyze their
singularities in the coordinates associated with a tree, and identify
conformal blocks on \(U_A\) with convergent \(A\)-shaped formal expansions.
Section \ref{sec_parenthesized} constructs the gluing operation for
conformal blocks.  The operation is defined by summing over intermediate
states, and the resulting series is absolutely convergent on \(U_{A\circ_p B}\).  In Section
\ref{sec_lax}, we prove that these gluing maps are compatible with analytic
continuation along morphisms in the parenthesized braid operad.  This gives
a unital pseudo-braided category structure on \(\Vmodf\).
In Section \ref{sec_rep}, we assume that \(V\) is rational and
\(C_2\)-cofinite.  Using the Zhu algebra and Frenkel--Zhu bimodules, we
prove that the multi-hom spaces given by conformal blocks are represented by tensor products. 
Analytic continuation is then transported to the representing objects,
giving explicit natural isomorphisms for the braiding and associator.
This yields a balanced braided tensor category structure on \(\Vmodf\). 
Appendix A contains the analytic result on differential systems arising from the tree coordinates, which is used to establish the expansion and convergence statements for conformal blocks.

\section{Preliminaries}
\label{sec_pre}

\subsection{Notation}
We will use the following notations:
\begin{itemize}
\item[$\Z_{\geq 0}$:] the set of non-negative integers
\item[$\Z_{> 0}$:] the set of positive integers
\item[$\cC$:] a $\C$-linear category
\item[$\Endo_\cC$:] the endomorphism operad associated with $\cC$
\item[$\Endp_\cC$:] a $2$-operad defined by coends associated with $\cC$, \S \ref{sec_appendix_B}
\item[$\circ_p$:] the $p$-th composition of an operad
\item[${[r]}$:] $=\{1,2,\dots,r\}$
\item[$\Tr_r$:] the set of all binary trees with $r$ leaves labeled by $[r]$, \S \ref{sec_operad_def}
\item[$\PW_r$:] the set of all binary trees with $r$ leaves ordered from $1$ to $r$, \S \ref{sec_operad_def}
\item[$\magma$:] the non-unitary magma operad, \S \ref{sec_operad_def}
\item[$\CPaB$:] the parenthesized braid operad, \S \ref{sec_def_cpab}
\item[$\D_p$:] $=\{\ze \in \C\mid |\ze|<p \}$, the unit disk of radius $p \in \R_{>0}$
\item[$\D_p^\times$:] $=\{\ze \in \C\mid 0<|\ze|<p \}$
\item[$\D_p^\cut$:] $=\{\ze \in \D_p^\times \mid -\pi < \Arg(\ze)<\pi \}=\D_p \setminus \R_{\leq 0}$
\item[$X_r$:] $=\{(z_1,\dots,z_r)\in \C^r\mid z_i\neq z_j \text{ for any }i\neq j\}$, the $r$ configuration space
\item[$\Or^\alg$:] $=\C[z_1,\dots,z_r,(z_i-z_j)^\pm\mid 1\leq i\neq j \leq r]$, the ring of regular functions on $X_r$
\item[$\Dr$:] $=\C[z_1,\pa_1,\dots,z_r,\pa_r,(z_i-z_j)^\pm \mid 1\leq i\neq j \leq r]$, the ring of differential operators on $X_r$
\item[$\Drt$:] $=\C[\pa_1,\dots,\pa_r,(z_i-z_j)^\pm \mid 1\leq i\neq j \leq r]$, a subalgebra of $\Dr$
\item[$\Or^\an$:] the sheaf of holomorphic functions on $X_r$
\item[$\Ort^\an$:] the sheaf of holomorphic functions on $X_r/\C$
\item[$A$:] a binary tree in $\Tr_r$
\item[$\Leaf(A)$:] a set of leaves of $A$
\item[$V(A)$:] a set of vertices of $A$ which are not leaves.
\item[$E(A)$:] a set of edges of $A$ which are not connected to leaves
\item[$d(e)$:] the lower vertex of an edge $e \in E(A)$, \S \ref{sec_config_coordinate}
\item[$u(e)$:] the upper vertex of an edge $e\in E(A)$
\item[$\Leaf(e)$:] a set of all leaves which are descendants of $d(e)$, \S \ref{sec_config_D}
\item[alpha-type edge:] a type of an edge (Definition \ref{def_alpha})
\item[$U_A$:] an open subset of $X_r$ associated with $A$, \S \ref{sec_config_coordinate}
\item[$z$-coordinate:] the standard coordinate $(z_1,z_2,\dots,z_r)$ on $\C^r$
\item[$x$-coordinate:] a local coordinate on $U_A$ associated with $A$, \S \ref{sec_config_coordinate}
\item[$A$-coordinate:] a local coordinate on $U_A$ associated with $A$, \S \ref{sec_config_coordinate}
\item[$T_A$:] a space of formal power series on $A$-coordinate, \S \ref{sec_config_conv}
\item[$V$:] a vertex operator algebra of CFT type
\item[$\Delta_a$:] the $L(0)$-degree of $a \in V$
\item[$\Vmodu$:] the category of $V$-modules which satisfies (M1)-(M4), \S \ref{sec_pre_vertex}
\item[$\Vmodf$:] the category of $C_1$-cofinite $V$-modules, \S \ref{sec_pre_vertex}
\item[$\Mr$:] $=M_0^\vee \otimes \bigotimes_{i=1}^r M_i$ for $V$-modules $M_0,\dots,M_r$
\item[$\Mrr$:] $=\bigotimes_{i=1}^r M_i$
\item[$a(n)_p$:] $n$-th product on $p$-th component on $\Mr$ for $a\in V$ and $n\in\Z$
\item[$I\binom{M_0}{M_1M_2}$:] the space of intertwining operators of type $\binom{M_0}{M_1M_2}$
\item[$I_{\log} \binom{M_0}{M_1M_2}$:] the space of logarithmic intertwining operators of type $\binom{M_0}{M_1M_2}$
\item[$D_\Mr$:] a $\Dr$-module associated with $\Mr$, \S \ref{sec_D_def}
\item[$\CB_\Mr$:] $=\Hom_{\Dr}(D_\Mr,\Or^\an)$, the solution sheaf, \S \ref{sec_D_def}
\item[$D_\Mrt$:] a $\Drt$-module associated with $\Mr$, \S \ref{sec_D_translation}
\item[$\CB_\Mrt$:] $=\Hom_{\Drt}(D_\Mrt,\Ort^\an)$, the solution sheaf, \S \ref{sec_D_translation}
\item[$\PI_A\binom{M_0}{\Mrr}$] a space of parenthesized intertwining operator of type $(A,M_0,M_{[r]})$, \S \ref{sec_parenthesized_def}
\end{itemize}

\vspace{5mm}

Let $z$ be a formal variable. We consider $\log z$ a formal variable independent of $z$.
For a vector space $V$,
we denote by $V[[z^\C]][\log z]$
the set of formal sums 
$$\sum_{k=0}^N \sum_{s \in \C} a_{s,k}
z^{s} (\log z)^k $$ such that $N\in \Z_{\geq 0}$ and $a_{s,k} \in V$.
The differential operator $\frac{d}{dz}$ acts on this linear space as 
$\frac{d}{dz}z^s=sz^{s-1}$ and $\frac{d}{dz}\log z=z^{-1}$.
The space $V[[z^\C]][\log z]$ contains various subspaces:
\begin{align*}
V[[z^\C]]&=\{\sum_{s \in \C} a_{s}z^{s}\;|\;a_s \in V \}, \\
V[z^\C]&=\{\sum_{s \in \C} a_{s}z^{s}\;|\;a_s \in V\text{ and }a_s=0 \text{ for all but finitely many }s \in \C \},\\
V[[z^\pm]]&=\{\sum_{n \in \Z} a_{n}
z^{n}\;|\;a_n \in V \},\\
V[z^\pm]&=\{\sum_{n \in \Z} a_{n}z^{n}\;|\;a_n \in V\text{ and }a_n=0 \text{ for all but finitely many }n \in \Z \},\\
V[[z]]&=\{\sum_{n \in \Z_{\geq 0}} a_{n}
z^{n}\;|\;a_n \in V \}.
\end{align*}

We often consider a combination of the above formal series. 
 For example, $V[[z]][z^\C]$ is the linear space spanned by formal series of the form:
 \begin{align*}
\sum_{n \in \Z_{\geq 0}}a_{n} z^{s+n}
\end{align*}
for $s \in \C$ and $a_n\in V$
and $V[[z]][z^{\pm}]$ is the same as the familiar space $V((z))$, the space of formal Laurent series.


For a map $f:X\rightarrow Y$, write its input as $f(\bullet)$.
When there is more than one input, that is, $f:X_1\times X_2 \times \cdots \times X_r\rightarrow Y$, 
write $f(\bullet_1,\bullet_2,\dots,\bullet_r)$ etc.

\subsection{Vertex operator algebra and modules}
\label{sec_pre_vertex}
This section briefly recalls the definitions of a vertex operator algebra and its modules \cite{FLM,FHL,LL}.

{\it A vertex algebra} is a $\C$-vector space $V$ equipped with a linear map 
$$Y(\bullet,z):V \rightarrow \End (V)[[z^\pm]],\; a\mapsto Y(a,z)=\sum_{n \in \Z}a(n)z^{-n-1}$$
and a non-zero element $\1 \in V$ satisfying the following conditions:
\begin{enumerate}
\item[V1)]
For any $a,b \in V$, $Y(a,z)b\in V((z))$, i.e.,
$a(n)b=0$ for all sufficiently large $n \in \Z$;
\item[V2)]
For any $a \in V$, $Y(a,z)\1 \in V[[z]]$ and $\lim_{z \to 0}Y(a,z)\1 = a(-1)\1=a$;
\item[V3)]
$Y(\1,z)=\id_V$, where $\id_V$ is the identity map on $V$.
\item[V4)]
For any $a,b \in V$ and $n\in \Z$,
\begin{align}
\begin{split}
[a(n),Y(b,z)]&= \sum_{k \geq 0} \binom{n}{k}Y(a(k)b,z)z^{n-k},\\
Y(a(n)b,z)&=\sum_{k \geq 0}\binom{n}{k} \left(a(n-k)Y(b,z)(-z)^k - Y(b,z) a(k) (-z)^{n-k}\right).
\label{eq_Borcherds}
\end{split}
\end{align}
\end{enumerate}
The Fourier mode $a(n)$ is called the {\it $n$-th product.}

\begin{rem}
\label{rem_vertex_borcherds}
Equation \eqref{eq_Borcherds} is equivalent to the Borcherds identity, 
which is usually used in the definition of a vertex algebra  (see \cite[Proposition 3.4.3]{LL} or Proposition \ref{prop_int_equiv}).
Writing the definition in this form is more convenient for our purposes.
\end{rem}

{\it A vertex operator algebra (of CFT type)}
is a vertex algebra with a distinguished vector
$\om \in V$ such that:
\begin{enumerate}
\item
There exists a scalar $c \in \C$ such that 
\begin{align}
[L(m),L(n)]=(m-n)L(m+n)+\frac{m^3-m}{12}\delta_{m+n,0}\,c
\label{eq_cc}
\end{align}
holds for any $n,m \in \Z$, where $L(n)=\om(n+1)$;
\item
$[L(-1),Y(a,z)]=\frac{d}{dz}Y(a,z)$ for any $a\in V$;
\item
$L(0)$ is semisimple on $V$, and all its eigenvalues are non-negative integers;
\item
$V_n$ is a finite-dimensional vector space for any $n \in \Z_{\geq 0}$,
where $V_n=\{a\in V\mid L(0)a=na \}$.
\item
$V_0$ is spanned by $\1$.
\end{enumerate}
The scalar $c \in \C$ in \eqref{eq_cc} is called {\it the central charge} of the vertex operator algebra.
We will denote the $L(0)$-degree of $a \in V$ by $\Delta_a$.

Let $V$ be a vertex operator algebra.
A $V$-module is defined similarly by
a linear map $Y_M(-,z):V \rightarrow \End M[[z^\pm]]$ which satisfies 
trivial modifications of (V1), (V3), and (V4)
(see for example \cite{LL}).
\begin{dfn}
\label{def_module_grade}
Throughout this paper, we assume that a $V$-module $M$ satisfies the following conditions:
\begin{enumerate}
\item[M1)]
The action of $L(0)=\om(1)$ on $M$ is locally finite;
\item[M2)]
$M_h$ is a finite-dimensional vector space for any $h\in \C$,
where 
\begin{align*}
M_h=\{m\in M\;|\; (L(0)-h)^k m=0 \text{ for sufficiently large }k \};
\end{align*}
\item[M3)]
There exists $N \in \Z_{\geq 0}$ such that for any $h\in \C$ and $m\in M_h$,
$(L(0)-h)^N m=0$;
\item[M4)]
There exists $n\in \Z_{>0}$ and $\Delta_1,\dots,\Delta_n \in \C$ such that
\begin{align*}
M=\bigoplus_{i=1}^n \bigoplus_{k\geq 0} M_{\Delta_i+k}.
\end{align*}
\end{enumerate}
\end{dfn}
Denote the category of $V$-modules which satisfy (M1)-(M4) by $\Vmodu$.

For a $V$-module $M$,
set 
$$M^\vee=\bigoplus_{h\in \C} M_h^*,
$$
a restricted dual space, where $M_h^*$ is the dual vector space of $M_h$.
Denote the canonical pairing $M^\vee \otimes M\rightarrow \C$
by $\langle \bullet,\bullet \rangle$.
A $V$-module structure on $M^\vee$ is defined by
\begin{align}
\langle Y(a,z)f, \bullet \rangle 
= \langle f, Y\left(\exp(L(1)z)(-z^{-2})^{L(0)} a,z^{-1}\right)\bullet \rangle \label{eq_dual}
\end{align}
for $a\in V$ and $f \in M^\vee$ \cite{FHL}.
It is called {\it a dual module}.

For $a\in V$, $n\in \Z$ and $f \in M^\vee$,
set
\begin{align*}
(a(n)^* f)(\bullet) = f \Bigl(a(n)\bullet \Bigr).
\end{align*}

Then, \eqref{eq_dual} can be written as:
\begin{align}
a(\Delta_a-1 - n)f=\sum_{k \geq 0}\frac{(-1)^{\Delta_a}}{k!} (L(1)^k a)\Bigl((\Delta_a-k) -1+n\Bigr)^*f
\label{eq_dual2}
\end{align}
for $a\in V$ and $n\in \Z$.

\begin{rem}
\label{rem_dual_closed}
Note that if $M\in \Vmodu$, then $M^\vee\in \Vmodu$.
\end{rem}



Let $M$ be a $V$-module.
For any $n \in \Z_{>0}$, set 
\begin{align*}
C_n(M)= \{a(-n)m\mid m\in M \text{ and }a \in V_+=\bigoplus_{k\geq 1}V_k \}.
\end{align*}

\begin{dfn}
A $V$-module $M$ is called {\it $C_n$-cofinite} if $M/C_n(M)$ is a finite-dimensional vector space.
\end{dfn}

Since $(L(-1)a)(-n)=na(-n-1)$ for any $a \in V$ and $n\in \Z_{>0}$,
$C_{n+1}(M) \subset C_{n}(M)$. Hence, if $M$ is $C_{n+1}$-cofinite,
then $M$ is $C_{n}$-cofinite.
Note that any vertex algebra is $C_1$-cofinite by (V2).

The main result is shown for the following category:
\begin{dfn}
\label{def_category}
Let $\Vmodf$ be the full subcategory of $\Vmodu$ consisting of $V$-modules $M$ such that
$M$ is $C_1$-cofinite and $M^\vee$ is a finitely generated $V$-module.
\end{dfn}

\begin{dfn}
A vertex operator algebra $V$ is called {\it $C_2$-cofinite} if $V$ is $C_2$-cofinite as a $V$-module.
\end{dfn}

We end this section by showing the following proposition:
\begin{prop}
\label{prop_C2_C1}
If $V$ is $C_2$-cofinite, then any module in $\Vmodu$ is a finitely generated $V$-module.
In particular, $\Vmodf$ and $\Vmodu$ coincide.
\end{prop}
To prove this proposition, we need the following result by \cite{Bu,ABD}:
\begin{prop}{\cite[Proposition 5.2]{ABD}}
\label{prop_ABD}
If $V$ is $C_2$-cofinite and $M$ is a finitely generated $V$-module,
then $M$ is $C_2$-cofinite.
\end{prop}

\begin{proof}[proof of Proposition \ref{prop_C2_C1}]
Since $V$ is $C_2$-cofinite, there are finitely many irreducible modules $M_1,\dots,M_r$.
Take $\Delta_i \in \C$ such that $M_i=\bigoplus_{k\geq 0} (M_i)_{\Delta_i+k}$ and $(M_i)_{\Delta_i}\neq 0$.
Set $\Delta =\max_{i=1}^r \mathrm{Re}\,\Delta_i$,
where $\mathrm{Re}$ is the real part.

Let $M \in \Vmodu$ and set
\begin{align*}
S=\bigoplus_{\substack{h \in \C\\ \mathrm{Re}(h) \leq  \Delta}}M_h.
\end{align*}
Then, by (M2) and (M4), $S$ is finite-dimensional.
It is easy to show that any non-zero $V$-module which satisfies (M4) contains an irreducible module as its sub-quotient.
Hence, $M$ is generated by $S$ and thus finitely generated.
Two categories $\Vmodf$ and $\Vmodu$ coincide by Remark \ref{rem_dual_closed} and Proposition \ref{prop_ABD}.
\end{proof}

\subsection{Logarithmic intertwining operators}

Now, we recall the definition of a logarithmic intertwining operator among
modules of a vertex operator algebra from \cite{FHL,Mil1,Mil2}.
Let $M_0,M_1,M_2$ be $V$-modules.
{\it A logarithmic intertwining operator} of type $\binom{M_0}{M_1 M_2}$
is a linear map
$$\Y_1(\bullet,z):M_1 \rightarrow \text{Hom} (M_2,M_0)[[z^\C]][\log z],
\; m \mapsto \Y_1(m,z)=\sum_{k \geq 0} \sum_{r \in \C} m(r;k) z^{-r-1}(\log z)^k$$
such that:
\begin{enumerate}
\item[I1)]
For any $m \in M_1$ and $m' \in M_2$,
$\Y_1(m,z)m' \in M_3[[z]][z^\C,\log z]$;
\item[I2)]
$[L(-1),\Y_1(m,z)]=\frac{d}{dz}\Y_1(m,z)$ for any $m \in M_1$;
\item[I3)]
For any $m \in M_1$, $a \in V$ and $n\in \Z$,
\begin{align*}
[a(n), \Y_1(m,z)]&=\sum_{k\geq 0} \binom{n}{k} \Y_1(a(k)m,z)z^{n-k}\\
\Y_1(a(n)m,z)&=\sum_{k\geq 0}\binom{n}{k}
\left( 
a(n-k)\Y_1(m,z)(-z)^k
- \Y_1(m,z)a(k)(-z)^{n-k}
\right).
\end{align*}
\end{enumerate}

\begin{rem}
As in Remark \ref{rem_vertex_borcherds}, this definition of a (logarithmic) 
intertwining operator differs from the usual definition of a (logarithmic) intertwining operator.
The equivalence with the usual definition will be seen later in Proposition 
\ref{prop_int_equiv}.
\end{rem}

The space of all logarithmic intertwining operators of type $\binom{M_0}{M_1M_2}$ forms a vector space,
which is denoted by $I_{\log} \binom{M_0}{M_1M_2}$.
If $\Y_1(\bullet,z) \in I_{\log} \binom{M_0}{M_1M_2}$ does not contain any logarithmic term,
i.e., $\Y_1(m,z) \in \mathrm{Hom}(M_2,M_0)[[z]][z^\C]$ for any $m\in M_1$,
then
$\Y_1(\bullet,z)$ is called {\it an intertwining operator} of type $\binom{M_0}{M_1M_2}$.
Denote by $I\binom{M_0}{M_1M_2}$
the space of all intertwining operators of type $\binom{M_0}{M_1M_2}$.

By substituting $a=\om$ and $n=1$ into (I3), we obtain from (I2) the following equality:
\begin{align*}
[L(0),\Y_1(m,z)]=\Y_1(L(0)m,z) +z\frac{d}{dz}\Y_1(m,z),
\end{align*}
which relates terms in $\log z$ and $z^r$ to the $L(0)$-eigenvalues of $M_0,M_1,M_2$.
This equation has been well examined by \cite{Mil1}.
We quote the following result from \cite[Proposition 1.10]{Mil1}:
\begin{prop}
Let $N_1,N_2$ be the integers in Definition \ref{def_module_grade} (M3) for the modules $M_1,M_2$, respectively.
Then, for any $\Y_1(\bullet,z) \in I_{\log} \binom{M_0}{M_1M_2}$ and $m\in M_1$
\begin{align*}
\Y_1(m, z) \in \bigoplus_{k=0}^{N_1+N_2-2} (\log z)^k \mathrm{Hom}(M_2,M_0)[[z^\C]].
\end{align*}
In particular, if $L(0)$ is semisimple on $M_1$ and $M_2$, then
\begin{align*}
I_{\log} \binom{M_0}{M_1M_2}=I \binom{M_0}{M_1M_2}.
\end{align*}
\end{prop}

%

For any $a \in V$ and $n \geq 0$, set
\begin{align*}
Y^{-}(a,z) &= \sum_{k\geq 0}a(-k-1)z^{k},\\
Y^+(a,z) &= \sum_{k\geq 0}a(k)z^{-k-1},\\
P_n(a,z) &= \sum_{k\geq 0}^n \binom{n}{k} a(k)(-z)^{n-k}.
\end{align*}

The following lemma immediately follows from the definition of a logarithmic intertwining operator:
\begin{lem}
\label{lem_P}
Let $M_1,M_2,M_0$ be $V$-modules and
$Y_1(\bullet,z) \in I_{\log}\binom{M_0}{M_1M_2}$.
For any $a \in V$ and $m \in M_1$ and $n \geq 0$, 
the following equalities hold:
\begin{enumerate}
\item
$\exp(L(-1)z)a(n)\exp(-L(-1)z)=P_n(a,z)$;
\item
$\exp(L(-1)z)a(-n-1)\exp(-L(-1)z)=\pa_z^{(n)}Y^-(a,z)$;
\item
$\pa_z P_n(a,z) = -n P_{n-1}(a,z) =P_n(L(-1)a,z)=[L(-1),P_n(a,z)]$;
\item
\begin{align*}
\Y_1(a(n)m,z)&=P_n(a,z) \Y_1(m,z) - \Y_1(m,z)P_n(a,z),\\
[a(n),\Y_1(m,z)]&= \Y_1(P_n(a,-z)m,z).
\end{align*}
\item
\begin{align*}
Y_1(a(-n-1)m,z)
&=\left(\frac{1}{n!}\pa_z^n Y^-(a,z)\right)\Y_1(m,z)+
\Y_1(m,z)\left(\frac{1}{n!}\pa_z^nY^+(a,z)\right).
\end{align*}
\item
For any $h_i \in \C$, $m_i \in (M_i)_{h_i}$ and $r\in \C$, $k \in \Z_{\geq 0}$,
$m_1(r;k)m_2 \in (M_3)_{h_1+h_2-r-1}$.
\end{enumerate}
\end{lem}

The following identities will be used repeatedly in this paper:
\begin{lem}
\label{lem_binom}
For any $n,k\in \Z$ with $k \geq 0$, 
\begin{align*}
\binom{n}{k}= (-1)^k\binom{-n-1+k}{k}.
\end{align*}
Furthermore, if $n \geq 0$ and $k\geq 0$,
then
\begin{align*}
\binom{-n-1}{k}=(-1)^{n+k}\binom{-k-1}{n}=(-1)^k\binom{k+n}{n}.
\end{align*}
\end{lem}
\begin{lem}
\label{identity}
For any $n, d \in \Z$ with $d \geq 0$,
\begin{align*}
\sum_{k \geq 0} \sum_{l\geq 0}
\binom{n}{k}\binom{k-d}{l} p^k q^l &=\frac{(1+p+pq)^n}{(1+q)^d}|_{|p|,|q|<1} \in \C[[p,q]]
\end{align*}
holds for formal variables $p,q$.
In particular, for any $l \geq 0$,
\begin{align*}
\sum_{k \geq 0}
\binom{n}{k}\binom{k}{l} p^k &=\binom{n}{l} p^l (1+p)^{n-l}|_{|p|<1},
\end{align*}
where $(1+p)^{n-l}|_{|p|<1}$ means the Taylor expansion around $p=0$.
\end{lem}

The proof of the following proposition is essentially the same as \cite{MN}:
\begin{prop}
\label{prop_int_equiv}
Let $\Y_1(\bullet,z) \in  I_{\log}\binom{M_0}{M_1M_2}$.
Then, $\Y_1(\bullet,z)$ satisfies the Borcherds identity, that is, for any $a\in V$, $m\in M$ and $p,r\in\Z$,
\begin{align}
\begin{split}
&\sum_{i\geq 0}\binom{p}{i}x_2^{p-i}\Y_1(a(r+i)m)x_2)\\
&= \sum_{i \geq 0}\binom{r}{i}(-1)^i (a(p+r-i)\Y_1(m,x_2)x_2^i -(-1)^r \Y_1(m,x_2)x_2^{r-i}a(p+i)).
\end{split}
\label{eq_Borcherds_id}
\end{align}

\end{prop}
\begin{proof}
Let $B_{p,r}(x_2)$ denote the formal power series with coefficients in $\mathrm{Hom}(M_2,M_0)$ obtained by subtracting the right-hand side of \eqref{eq_Borcherds_id} from the left-hand side.  
It suffices to show that $B_{p,r}(x_2)=0$ for all $p,r \in \Z$.  
We claim that
\begin{align}
B_{p+1,r}(x_2)-B_{p,r+1}(x_2)-x_2 B_{p,r}(x_2)=0 \label{eq_Borcherds_recur}
\end{align}
holds for any $p,r \in \Z$.
Since $\binom{q+1}{i}=\binom{q}{i}+\binom{q}{i-1}$ holds for any $q\in\Z$ and $i\geq 0$,
\begin{align*}
&\sum_{i\geq 0}\left(\binom{p+1}{i}x_2^{p+1-i}\Y_1(a(r+i)m)x_2)
-\binom{p}{i}x_2^{p+1-i}\Y_1(a(r+i)m)x_2)\right)\\
&=\sum_{i\geq 0}\binom{p}{i-1}x_2^{p+1-i}\Y_1(a(r+i)m)x_2)=\sum_{i\geq 0}\binom{p}{i}x_2^{p-i}\Y_1(a(r+1+i)m)x_2)
\end{align*}
and
\begin{align*}
&\sum_{i \geq 0}\left(\binom{r+1}{i}(-1)^i a(p+(r+1)-i)\Y_1(m,x_2)x_2^i-\binom{r}{i}(-1)^i a((p+1)+r-i)\Y_1(m,x_2)x_2^i\right)\\
&=\sum_{i \geq 0}\binom{r}{i-1}(-1)^i a(p+r+1-i)\Y_1(m,x_2)x_2^i = - x_2\sum_{i \geq 0}\binom{r}{i}(-1)^i a(p+r-i)\Y_1(m,x_2)x_2^i
\end{align*}
and
\begin{align*}
&\sum_{i \geq 0}\left(\binom{r+1}{i}(-1)^{i+r+1}\Y_1(m,x_2)x_2^{r+1-i}a(p+i) + \binom{r}{i}(-1)^{i+r}\Y_1(m,x_2)x_2^{r-i+1}a(p+i)\right)\\
&\sum_{i \geq 0}\binom{r}{i-1}(-1)^{i+r+1}\Y_1(m,x_2)x_2^{r+1-i}a(p+i)=\sum_{i \geq 0}\binom{r}{i}(-1)^{i+r}\Y_1(m,x_2)x_2^{r-i}a(p+1+i),
\end{align*}
which imply \eqref{eq_Borcherds_recur}.
If two among $B_{p+1,r}(x_2)$, $B_{p,r+1}(x_2)$, and $x_2 B_{p,r}(x_2)$ vanish, then the remaining one also vanishes.  
Here, $B_{p,0}(x_2)$ and $B_{0,r}(x_2)$ are nothing but (I3), so that
\begin{align*}
B_{p,0}(x_2)=0=B_{0,r}(x_2)
\end{align*}
for any $p,r\in\Z$.  
Therefore, from \eqref{eq_Borcherds_recur}, if either $p$ or $r$ is a non-negative integer, we obtain $B_{p,r}(x_2)=0$.  
In the case $p,r<0$, the claim follows by induction on $-(p+r)$.
\end{proof}
Proposition \ref{prop_int_equiv} identifies the definition of logarithmic
intertwining operators used here with the standard one.  This allows us to
apply the results on intertwining operators used below, in particular those
of \cite{FZ,Li}.

\section{PaB operad, 2-action and representativity}
\label{sec_operad}

This section reviews the definitions and basics of operads used in this paper.
In Section \ref{sec_operad_def}, we will recall the definition of an operad in a symmetric tensor category
from \cite{Fr,LoV}.
In Section \ref{sec_def_cpab}, 
we recall the definition of the parenthesized braid operad \(\PaB\) and its
relation to braided tensor categories.
In Section \ref{sec_appendix_B}, we give a definition of a unital pseudo-braided  category by using $\CPaB$. Necessary and sufficient conditions for a unital pseudo-braided category to determine a braided tensor category are provided in Section \ref{app_pseudo}.


\subsection{Operad and its action}
\label{sec_operad_def}
In this section, we briefly review the definitions of operads and their algebras.
For more detailed accounts of operad theory, we refer the reader to \cite{LoV, Fr}.

Let $\cC$ be a strict symmetric tensor category with the tensor product $\times$ and the unit object $*_\cC$.
We assume that $\cC$ is closed, i.e., the Hom-functor provides $\mathrm{Hom}_\cC(\bullet,\bullet):\cC^\op\times \cC \rightarrow \cC$, and the tensor product of $\cC$ distributes over colimits (see \cite[Chapter 3]{Fr}).
Throughout this paper, categories are enriched by the category of $\C$-vector spaces $\Vect$
unless otherwise noted.

\begin{dfn}
A {\bf plain operad} $\mathcal{O}$ in a symmetric monoidal category $\cC$ consists of the following data \cite[Section 5.3.4]{LoV}:
\begin{enumerate}
  \item For each $n \geq 0$, an object $\mathcal{O}(n)$ (the set or space of $n$-ary operations).
  \item For each $m,n \geq 0$ and $1 \leq i \leq m$, a \emph{partial composition map}
  \[
     \circ_i : \mathcal{O}(m) \times \mathcal{O}(n) \longrightarrow \mathcal{O}(m+n-1),
  \]
  which inserts an $n$-ary operation into the $i$-th input of an $m$-ary operation.
  \item A distinguished unit element $i_{\mathcal{O}}: *_\cC \rightarrow \mathcal{O}(1)$.
\end{enumerate}
These data satisfy the following axioms:
\begin{itemize}
  \item \textbf{Associativity:} For all $f \in \mathcal{O}(m)$, $g \in \mathcal{O}(n)$, $h \in \mathcal{O}(\ell)$,
  \[
     (f \circ_i g) \circ_{i+j-1} h \;=\; f \circ_i (g \circ_j h),
  \]
  and if $i < k$ then
  \[
     (f \circ_i g) \circ_{k+n-1} h \;=\; (f \circ_k h) \circ_i g.
  \]
  \item \textbf{Unit:} For all $f \in \mathcal{O}(n)$, both
\begin{align*}
\cO(n) \cong \cO(n)\times *_\cC \overset{\id\times i_\cO}{\rightarrow} \cO(n)\times \cO(1) \overset{\circ_i}{\rightarrow} \cO(n),\\
\cO(n) \cong *_\cC \times \cO(n) \overset{i_\cO \times \id}{\rightarrow} \cO(1)\times \cO(n) \overset{\circ_1}{\rightarrow} \cO(n)
\end{align*}
are equal to $f$ for any $i=1,\dots,n$.
\end{itemize}
\end{dfn}

\begin{dfn}\label{def_permutation_operad}
A {\bf permutation operad} \cite[Section 5.3.4]{LoV} is a plain operad $\{\cO(n)\}_{n\geq 0}$ equipped with a right action of the symmetric group $S_n$ on $\mathcal{O}(n)$ for any $n \geq 0$ such that for any $f \in\cO(n)$, $g\in \cO(m)$
\begin{itemize}
\item
For any $\si \in S_n$, $\tau \in S_m$ and $i \in [n]$, the following equalities hold
\begin{align*}
f^\si \circ_i g = (f\circ_{i} g)^{\si'},
\end{align*}
where $\si' \in S_{n+m-1}$ is the permutation which acts by the identity, except on the
 block $\{i,\dots,i-1+m\}$ on which it acts via $\sigma$, and
\begin{align*}
f \circ_i g^\tau = (f\circ_{\tau(i)} g)^{\tau'},
\end{align*}
$\tau' \in S_{n+m-1}$ is acting like $\tau$ on the block 
$\{1,\dots,n+m-1\} \setminus \{i,\dots,i-1+m\}$ with values in
$\{1,\dots, n-1+m\} \setminus \{ \tau(i),\dots, \tau(i)-1+m\}$ and identically on the
 block $\{i,\dots,i-1+m\}$ with values in $\{\tau(i),\dots, \tau(i)-1+m\}$.
\end{itemize}
\end{dfn}
We simply refer to the permutation operad as an operad.

\begin{dfn}
A morphism of operads $f:P\rightarrow Q$ consists of a sequence
of morphisms $f_n:P(n)\rightarrow Q(n)$ in $\cC$ such that for any $n \geq 0$
\begin{itemize}
\item[MO1)]
$f_1\circ i_P=i_Q$.
\item[MO2)]
$f_{n+m-1}(\bullet \circ_i \bullet)=f_n(\bullet)\circ_i f_m(\bullet)$ for any $m \geq 0$ and $i \in 1,\dots,n$.
\item[MO3)]
For any permutation $g \in S_n$,
$f_n(\bullet\cdot g)=f_n(\bullet)\cdot g$.
\end{itemize}
\end{dfn}
For an object $A \in \cC$, set
\begin{align*}
\Endo_A(n) = \mathrm{Hom}_{\cC}(A^n,A)
\end{align*}
for $n \geq 0$. Then, $\Endo_A=\{\Endo_A(n)\}_{n\geq 0}$ is a symmetric operad, which is called an {\it endomorphism operad}.

\begin{dfn}
Let $\mathcal{O}$ be an operad in $\cC$ and $A \in \cC$.
An {\it action} of $\cO$ on $A$ is an operad homomorphism $\cO \rightarrow \Endo_A$.
\end{dfn}

Denote the category of small sets by $\Cset$,
which is a symmetric monoidal category by the direct product.
We first consider an operad in $\Cset$.

Let $\Tr_{r}$ be the set of all binary trees whose leaves are labeled by $[r]=\{1,2,\dots,r\}$.
Each element in $\Tr_r$ can be regarded as a parenthesized word of $\{1,2,\dots,r\}$,
that is non-associative, non-commutative monomials on this set in which every letter appears exactly once.
For example, $(5(23))((17)(64))$ corresponds to the tree in Fig. \ref{fig_tree_example0}.
Note that $\Tr_0$ consists of the empty word
and $\Tr_3$, for example, is a set of 12 elements
\begin{align*}
\Tr_0&=\{\emptyset\},\\
\Tr_3&=\{1(23), (12)3, 1(32), (13)2, 2(13),(21)3,2(31),(23)1,3(12),(31)2,3(21),(32)1\}.
\end{align*}

\begin{minipage}[c]{.35\textwidth}
\centering
\begin{forest}
for tree={
  l sep=20pt,
  parent anchor=south,
  align=center
}
[
[[5][[2][3]]]
[[[1][7]]
[[6][4]]]
]
\end{forest}
\captionof{figure}{}
\label{fig_tree_example0}
\end{minipage}
\begin{minipage}[l]{.7\textwidth}
\centering
For $A\in \Tr_r$, we will use the following notations:
\begin{align*}
\Leaf(A)&=\{\text{the set of all leaves of }A\}\\
V(A)&=\{\text{the set of all vertices of }A  \text{ which are}\\
&\quad\quad\text{ not leaves}\}\\
E(A)&=\{\text{the set of all edges of }A \text{ which are}\\
&\quad\quad\text{ not connected to leaves}\}.
\end{align*}
\end{minipage}

We will see that an operad structure can be introduced on $\{ \Tr_r\}_{r\geq 0}.$
Let $A \in \Tr_n$ and $B\in \Tr_m$ with $n> 0$ and $p \in [n]$.
The partial composition of the operad is then defined as shown in Fig. \ref{fig_partial_comp}.
The figure shows the composition of $3((12)4)\circ_2 2(13)$.
\begin{figure}[b]
    \centering
    \includegraphics[scale=1]{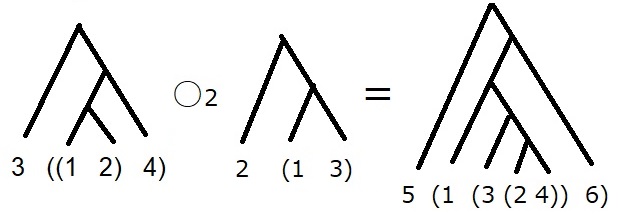}
    \caption{$3((12)4)\circ_2 2(13)$}\label{fig_partial_comp}
\end{figure}
In general, $A\circ_p B$ is defined by inserting the tree $B$ into the leaf labeled with $p$ in $A$,
adding $p-1$ to labels of leaves in $B$, and adding $m-1$ to the labels of leaves after $p+1$ in $A$.
If $B=\emptyset$, the $p$-th leaf in $A$ is simply erased, and the numbers are shifted forward.
For example,
\begin{align*}
3((12)4) \circ_2 \emptyset = 2(13).
\end{align*}
The symmetric group $S_r$ acts on $\Tr_r$ by the permutation of labels,
which satisfies the definition of a symmetric operad.

This operad is called {\it a magma operad} and is denoted by $\magmas$.
Since $\{\Tr_r\}_{r\geq 1}$ is closed under the operadic compositions, this is also an operad, which is denoted by
$\magma$ and called {\it a non-unitary magma operad}.

\begin{rem}
\label{rem_magma}
Giving an action of $\magma$ on a set $S$ is equivalent to giving a binary operation $m:S\times S\rightarrow S$.
The binary operation $m$ is not assumed to be associative nor commutative, just a binary operation.
Hence, $(S,m)$ is a magma.
In addition, there is an equivalence between giving an action of $\magmas$ on a set $S$ 
and giving a binary operation $m:S\times S\rightarrow S$ together with an element $1\in S$ such that
\begin{align*}
m(1,a)=m(a,1)=a,\quad \text{ for any }a\in S,
\end{align*}
that is, $(S,m,1)$ is a unital magma.
\end{rem}
Denote by $\PW_r$ the subset of $\Tr_r$ consisting of all binary trees whose leaves are ordered
from $1$ to $r$. For example, $\PW_4$ is a set of $5$ elements
\[
\PW_4=\{(12)(34), 1(2(34)), 1((23)4), (1(23))4, ((12)3)4\}.
\]
Note that $\{\PW_r\}_{r\geq 0}$ is a plain operad.
Moreover, every tree in $\Tr_r$ is uniquely obtained from an element of
$\PW_r$ by relabelling the leaves. Thus we have a natural bijection
\begin{equation}
\label{eq_Tr_PW_Sr}
\Tr_r \simeq \PW_r\times S_r,
\qquad
A\longmapsto (w_A,g_A).
\end{equation}

\subsection{Parenthesized braid operad and braided tensor category}
\label{sec_def_cpab}
In this section, based on \cite{Fr}, we review the parenthesized braid operad, which plays a central role in this paper, and its relation to braided tensor categories.

\begin{figure}[h]
\begin{minipage}[c]{.8\textwidth}
    \centering
    \includegraphics[scale=0.15]{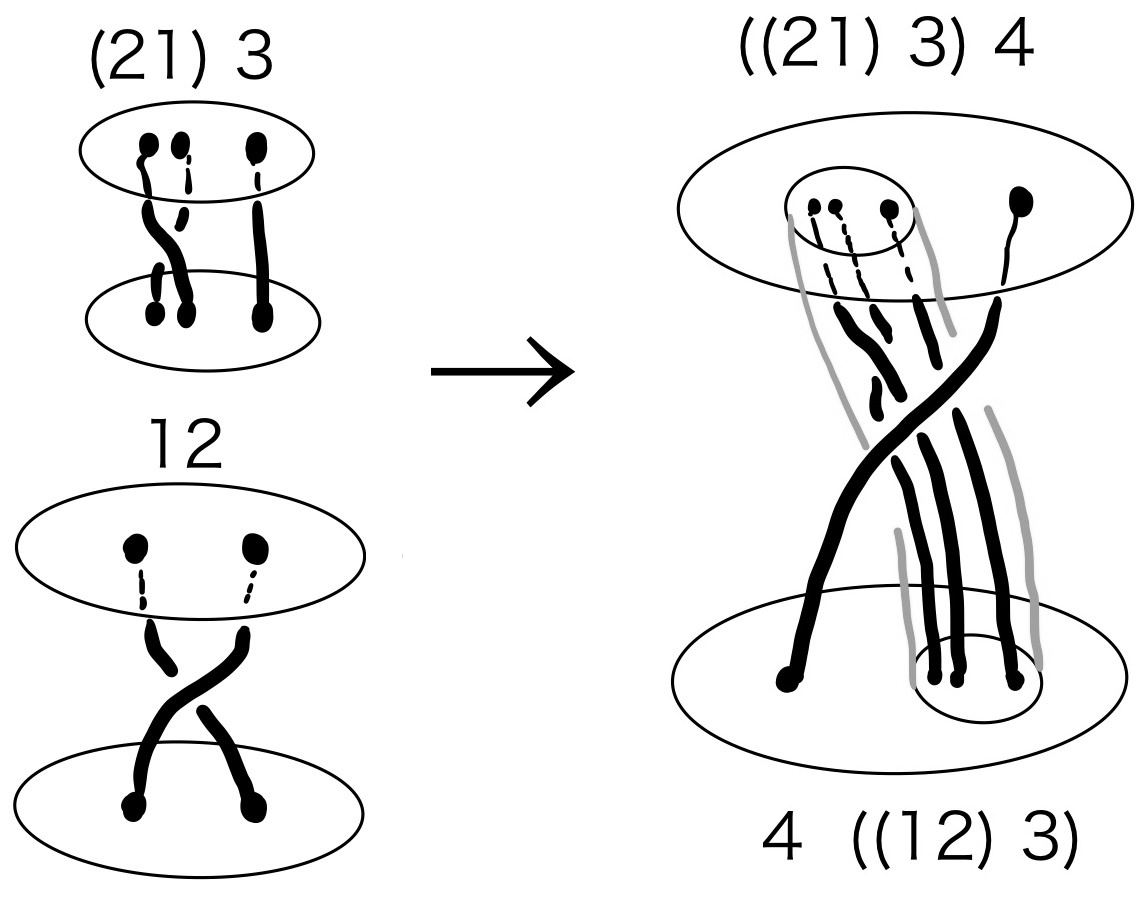}
    \caption{}\label{fig_thin}
\end{minipage}
\end{figure}

The idea of the parenthesized braid operad is as follows. Set
\begin{align*}
X_n(\C) &= \{(z_1,\dots,z_n) \in \C^n \mid z_i \neq z_j \text{ for } i\neq j \},\\
X_n(\R) &= \{(z_1,\dots,z_n) \in \R^n \mid z_i \neq z_j \text{ for } i\neq j \},
\end{align*}
the configuration space of $n$ points.
A path in $X_n(\C)$ can be represented by a braid as in Figure~\ref{fig_thin}.
The two braids on the left of the figure correspond to paths in $X_2(\C)$ and $X_3(\C)$, respectively.
In this situation, by thickening the path of $z_1$ in $X_2(\C)$ and inserting the path from $X_3(\C)$ into it, one can construct a new path in $X_4(\C)$.
The operad defined in this way is called the {\it parenthesized braid operad} \cite{Bar,Ta1}.

Let $\Pi_1(X_n(\C))$ denote the fundamental groupoid of $X_n(\C)$, that is, the category whose objects are points of $X_n(\C)$ and whose morphisms are homotopy classes of paths between points with fixed endpoints. 

Conformal blocks in conformal field theory are multivalued functions on $X_n(\C)$, and their monodromy defines a representation of $\Pi_1(X_n(\C))$.
One of the purposes of this paper is to show that this monodromy representation is furthermore compatible with the operad structure (see Figure~\ref{fig_thin}).
\begin{rem}
The functor $\Pi_1$ (fundamental groupoid) is symmetric monoidal. Therefore, applying $\Pi_1$ to the little 2-disk operad $E_2$ yields an operad in the category of categories.
%
The parenthesized braid operad is essentially equivalent to this operad \cite{Fr}.
\end{rem}

%
%

More precisely, let $\cat$ be the category of small categories, i.e., objects are small categories and morphisms are functors. 
By the direct product of categories, $\cat$ has a structure of a symmetric monoidal category.
The parenthesized braid operad $\CPaB$ is an operad in $\cat$.
For each $r \geq 1$, the category $\CPaB(r)$ can be defined as a full subcategory of $\Pi_1(X_r(\C))$.  
We first give an abstract definition, and then show how it can be realized geometrically.  

For each element of $\Tr_r$, by forgetting the parenthesization and retaining only the ordering of the leaves, one obtains a map $g:\Tr_r \longrightarrow S_r$.
%

For each $r\geq 1$, $\CPaB(r)$ is the category defined as follows:
The set of all objects in $\CPaB(r)$ is the set of binary trees $\Tr_r$,
\begin{align*}
\mathrm{Ob}(\CPaB(r))=\Tr_r.
\end{align*}
Let $p:B_r \rightarrow S_r$ be the canonical projection from the braid group to the symmetric group whose kernel is the pure braid group $PB_r$.
Then, for  $A,B \in \Tr_{r}$, the space of homomorphisms is defined by
\begin{align}
\Hom_{\CPaB(r)}(A,B)=\C p^{-1}({g_A^{-1}g_B}),
\label{eq_perm_braid}
\end{align}
where $\C p^{-1}({g_A^{-1}g_B})$ is a $\C$-linear space with a basis $p^{-1}({g_A^{-1}g_B})$.
The composition law is induced from the one on $B_r$.
The symmetric group $S_r$ acts on $\CPaB(r)$ via renumbering the objects $\Tr_r$ and acts identically on morphisms.

Next, we realize $\CPaB(r)$ as a full subcategory of $\Pi_1(X_r(\C))$.  
Define a map
\begin{align}
q_{[r]} : \Tr_r \;\longrightarrow\; X_r(\R), 
\qquad 
A \;\longmapsto\; q_{[r]}^A = (q_1^A, \dots, q_r^A),\label{eq_q_def}
\end{align}
subject to the following conditions:
\begin{enumerate}
\item[(Q1)]
For any $i,j \in \{1,\dots,r\}$ with $i\neq j$, we have 
\[
q_i^A < q_j^A \quad \text{if and only if the leaf $i$ lies to the right of the leaf $j$ in $A$}.
\]
\item[(Q2)]
If $A\neq B$ for $A,B\in \Tr_r$, then $q^A \neq q^B$.
\end{enumerate}

Later, we will see that $q_{[r]}$ reflects the parenthesization data of $A$.  
For example, when $A=((12)3)(45)$, the configuration $q^A \in X_5(\R)$ is chosen so that the points $1$ and $2$ are placed close to each other, as illustrated in the following figure.  
However, for the definition of $\CPaB$ itself, such geometric information is not required.

%
%

\vspace{4mm}

\begin{tikzpicture}[baseline=(current bounding box.center)] 
\tikzstyle point=[circle, fill=black, inner sep=0.05cm]
 \node[point, label=below:$(5$] at (0.5,0.5) {};
 \node[point, label=below:$4)$] at (1,0.5) {};
 \node[point, label=below:$(3$] at (3,0.5) {};
 \node[point, label=below:$(2$] at (4,0.5) {};
 \node[point, label=below:$1))$] at (4.5,0.5) {};
\end{tikzpicture}
\vspace{4mm}

For each map $q_{[r]}:\Tr_r \to X_r(\R)$, one may consider the full subcategory
\[
\Pi_1\bigl(X_r(\C)\bigr)\big|_{q_{[r]}}
\]
of $\Pi_1(X_r(\C))$ whose objects are precisely the points $\{q_{[r]}^A \in X_r(\R)\}_{A \in \Tr_r}$.
If $q_{[r]}, q'_{[r]}:\Tr_r \to X_r(\R)$ both satisfy (Q1) and (Q2), then the ordering of the points is preserved by (Q1), and hence there is a natural isomorphism of categories $\Pi_1(X_r(\C))\big|_{q_{[r]}} \cong \Pi_1(X_r(\C))\big|_{q'_{[r]}}$.
Moreover, the category $\CPaB(r)$ defined above is canonically equivalent to $\Pi_1(X_r(\C))|_{q_{[r]}}$. Consequently, the morphisms in $\CPaB(r)$ can be realized as braids in the configuration space, as illustrated in Figure~\ref{fig_thin}.

The composition
\begin{align*}
\circ_p: \CPaB(n)\times \CPaB(m)\rightarrow \CPaB(n+m-1)
\end{align*}
is given by replacing the $p$-th strand of the first braid, by the second braid made very thin (see Fig. \ref{fig_thin}). This composition is consistent with the magma operad when restricted to objects.

\begin{figure}[h]
  \begin{minipage}[b]{0.45\linewidth}
    \centering
    \includegraphics[width=2.5cm]{sigma.jpg}
    \caption{morphism $\sigma$}\label{fig_sigma}
  \end{minipage}
  \begin{minipage}[b]{0.45\linewidth}
    \centering
    \includegraphics[width=3cm]{alpha.jpg}
    \caption{morphism $\alpha$}\label{fig_alpha}
  \end{minipage}
      \begin{minipage}[b]{0.45\linewidth}
    \centering
    \includegraphics[width=3.2cm]{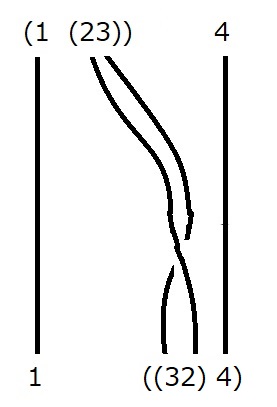}
    \caption{morphism $\alpha\circ_2\sigma$}\label{fig_alphasigma}
  \end{minipage}
    \begin{minipage}[b]{0.45\linewidth}
    \centering
    \includegraphics[width=3.2cm]{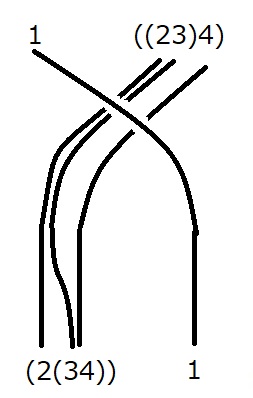}
    \caption{morphism $\si\circ_2\al$}\label{fig_sigmaalpha}
  \end{minipage}
\end{figure}

For $r=0$, $\CPaB(0)$ is a category whose object is the only empty parenthesized word $\emptyset$ and whose 
morphism consists only of the identity map $\Hom(\emptyset,\emptyset)=\C\{\id\}$.
The composition
\begin{align*}
\circ_p: \CPaB(n)\times \CPaB(0)\rightarrow \CPaB(n-1)
\end{align*}
is given by just erasing the $p$-th strand.

An important point here is that the operad $\CPaB$ is generated by $\sigma \in \CPaB(2)$ and $\alpha \in \CPaB(3)$ (see Fig. \ref{fig_sigma}, Fig.~\ref{fig_alpha}, Fig.~\ref{fig_alphasigma}, and Fig.~\ref{fig_sigmaalpha}). These elements correspond to the braiding and the associator in a braided tensor category, respectively, which is explained in detail by Fresse \cite{Fr}.

\begin{rem}
\label{rem_fix_braid}
We fix the following convention for the generator
$\sigma:(12)\longrightarrow(21)$
of \(\mathrm{PaB}(2)\). Choose $a_1, a_2 \in \mathbb{R}$ such that $a_1 > a_2 > 0$,
and set $q^{(12)} = (a_1,a_2)$ and $q^{(21)}=(a_2,a_1)$, which are points of $X_2(\R)$. Define a path
\[
\sigma: [0,1] \rightarrow X_2(\mathbb{C})
\]
by
\begin{align}
\si(t) = \left(\frac{a_1+a_2}{2}+\frac{a_1-a_2}{2}e^{-\pi i t}, \frac{a_1+a_2}{2}-\frac{a_1-a_2}{2}e^{-\pi i t}\right) \in X_2(\C).
\label{eq_add_path_si}
\end{align}
Then, $\si(0)=q^{12}$ and 
$\si(1) = q^{21}$.
In particular, the difference of the two coordinates is analytically continued as 
\begin{align}
z_1-z_2=(a_1-a_2)e^{-\pi i t},\label{eq_add_difference}
\end{align}
so that it makes a clockwise half-turn in \(\mathbb C^\times\). 
This clockwise choice gives the geometric representative of the morphism
\(\sigma\) in \(\mathrm{PaB}(2)\) that we use throughout the paper.

Since the braiding is constructed from the monodromy representation on conformal
blocks, such a choice of representative is essential. There is no canonical
choice between the clockwise and counterclockwise representatives; they are
exchanged by complex conjugation and give rise to opposite braidings.
However, once we choose the balancing isomorphism to be
$\theta_M=\exp(2\pi iL(0))$,
the compatible choice of braiding is fixed (see Section \ref{sec_explicit_braiding}). With the conventions of this
paper, this requires the above clockwise representative of \(\sigma\).
\end{rem}

\begin{dfn}
\label{def_alpha}
Let $r \geq 3$ and $A\in \Tr_r$.
An edge $e \in E(A)$ is called {\it alpha-type} if
the vertex on the right side of $d(e)$ is not a leaf.
\end{dfn}
Let $e\in E(A)$ be alpha-type. Then, $A$ can be written as $A=B\circ_p 1(23) \circ (A_1,A_2,A_3)$ for some trees $A_1,A_2,A_3,B$ and $p\in \{1,2,\dots,n\}$.
Set $A'=B\circ_p (12)3 \circ (A_1,A_2,A_3)$.

\vspace{4mm}

\begin{minipage}[c]{7cm}
\begin{forest}
for tree={
  l sep=20pt,
  parent anchor=south,
  align=center
}
[$\fbox{B}$,edge label={node[midway,right]{p}}
[$e$
[\fbox{$A_1$}]
[[\fbox{$A_2$}][\fbox{$A_3$}]]
]
]
\end{forest}
\end{minipage}
\hfill
\begin{minipage}[c]{7cm}
\begin{forest}
for tree={
  l sep=20pt,
  parent anchor=south,
  align=center
}
[$\fbox{B}$
[$e$
[[\fbox{$A_1$}][\fbox{$A_2$}]]
[\fbox{$A_3$}]
]
]
\end{forest}
\end{minipage}
\vspace{4mm}

Define a $\CPaB(r)$-morphism
$\alpha_e:A\rightarrow A'$ by
\begin{align}
\alpha_e = \id_B \circ_p \alpha \circ (\id_{A_1},\id_{A_2},\id_{A_3}):
B\circ_p 1(23) \circ (A_1,A_2,A_3)\rightarrow 
B\circ_p (12)3 \circ (A_1,A_2,A_3),
\label{eq_alpha_type_def}
\end{align}
which we call {\it an alpha-type map}.
We will also write $\alpha_e A$ for $A'$ using the same symbol as the morphism.
\begin{rem}
\label{rem_alpha_untwist}
Note that $\Hom_{\CPaB(r)}(A,\al_e A)=\C p^{-1}(\id)=\C \PB_r$
and $\al_e:A\rightarrow A'$ corresponds to the identity map in $\PB_r$.
In particular, $\al_e$ is a topologically untwisted braid (see Fig. \ref{fig_alpha}).
\end{rem}

Next, we will recall the notion of a braided tensor category (see \cite{EGNO} for more detail).
A {\it tensor category} is an essentially small, $\C$-linear, category $\mathcal{C}$
equipped with a bifunctor 
$$\boxtimes\colon \cC \times \cC \rightarrow \cC,\quad (M,N)\mapsto M\boxtimes N$$
and a distinguished object $1_\mathcal{B}$ together with natural isomorphisms
\begin{align*}
&l_{\bullet}\colon 1_\mathcal{B} \boxtimes \bullet\xrightarrow{\simeq} \bullet,\quad (l_M\colon 1_\mathcal{B}\boxtimes M\xrightarrow{\simeq} M),\\
&r_{\bullet}\colon \bullet \otimes 1_\mathcal{B}\xrightarrow{\simeq} \bullet,\quad (r_M\colon M\otimes  1_\mathcal{B}\xrightarrow{\simeq} M),\\
&{\al}_{\bullet,\bullet,\bullet}\colon (\bullet \boxtimes\bullet)\otimes \bullet\xrightarrow{\simeq} \bullet\otimes (\bullet\boxtimes \bullet),\quad \left({\al}_{M,N,L}\colon (M\boxtimes N)\boxtimes L\xrightarrow{\simeq} M\boxtimes (N\boxtimes L)\right),
\end{align*}
satisfying the pentagon and triangle identities:
\begin{enumerate}
\item[(triangle identities)]
\begin{align*}
\begin{array}{ccc}
(M\boxtimes 1) \boxtimes N & \overset{\alpha_{M,1,N}}{\longrightarrow} & M\boxtimes (1\boxtimes N)\\
{}_{r_N}\searrow&&\swarrow_{l_N}\\
&M\boxtimes N &
\end{array}
\end{align*}
\item[(pentagon identity)]
\begin{align}
\begin{array}{ccc}
& (M_1 \boxtimes M_2) \boxtimes (M_3 \boxtimes M_4)\\
 {}^{{\alpha_{M_1 \boxtimes M_2, M_3, M_4}}}\nearrow
&&
\searrow^{{\alpha_{M_1,M_2,M_3 \boxtimes M_4}}}
\\
((M_1 \boxtimes M_2 ) \boxtimes M_3) \boxtimes M_4
&&
M_1 \boxtimes (M_2 \boxtimes (M_3 \boxtimes M_4))
\\
{}^{{\alpha_{M_1,M_2,M_3}} \boxtimes id_{M_4} }\downarrow 
&& 
\uparrow^{{ id_{M_1} \boxtimes \alpha_{M_2,M_3,M_4} }}
\\
(M_1 \boxtimes (M_2 \boxtimes M_3)) \boxtimes M_4
& \underset{\alpha_{M_1,M_2 \boxtimes M_3, M_4}}{\longrightarrow} &
M_1 \boxtimes ( (M_2 \boxtimes M_3) \boxtimes M_4)
\end{array}
\label{eq_pentagon_id}
\end{align}
\end{enumerate}

{\it A braided tensor category} is a tensor category $\mathcal{C}$ equipped with a natural isomorphism of functors, called the \emph{braiding},
$${B}_{\bullet,\bullet}\colon(\bullet\otimes\bullet)\xrightarrow{\simeq} ( \bullet\otimes \bullet)\circ \sigma,\quad ({B}_{M,N}\colon M\otimes N\xrightarrow{\simeq}N\otimes M),$$
where $\sigma$ is the functor
$$\sigma\colon \mathcal{C}\times \mathcal{C}\rightarrow \mathcal{C}\times \mathcal{C},\quad (M,N)\mapsto (N,M).$$
The natural isomorphism ${B}_{\bullet,\bullet}$ is required to satisfy {\it the hexagon identity},
$$
\begin{array}{ccccc}
   (M \otimes N) \otimes L 
   &\stackrel{\al_{M,N,L}}{\rightarrow}&
   M \otimes (N \otimes L)
   &\stackrel{B_{M,N \otimes L}}{\rightarrow}&
   (N \otimes L) \otimes M
   \\
   \downarrow^{B_{M,N}\otimes Id}
   &&&&
   \downarrow^{\al_{N,L,M}}
   \\
   (N \otimes M) \otimes L
   &\stackrel{\al_{N,M,L}}{\rightarrow}&
   N \otimes (M \otimes L)
   &\stackrel{Id \otimes B_{M,L}}{\rightarrow}&
   N \otimes (L \otimes M)
\end{array}
$$
and
$$
\begin{array}{ccccc}
   M \otimes (N \otimes L) 
   &\stackrel{\al_{M,N,L}^{-1}}{\to}&
   (M \otimes N) \otimes z
   &\stackrel{B_{M \otimes N, L}}{\to}&
   L \otimes (M \otimes N)
   \\
   \downarrow^{Id \otimes B_{N,L}}
   &&&&
   \downarrow^{\al_{L,M,N}^{-1}}
   \\
   M \otimes (L \otimes N)
   &\stackrel{\al^{-1}_{M,L,N}}{\to}&
   (M \otimes L) \otimes N
   &\stackrel{B_{M,L} \otimes Id}{\to}&
   (L \otimes M) \otimes N
\end{array}.
$$

\begin{dfn}
\label{def_balanced}
A \emph{twist}, or a \emph{balance}, on a braided tensor category $\cC$ is a natural isomorphism $\theta$ from the identity functor on $\cC$ to itself satisfying 
$$
B_{M,N}B_{N,M}=\theta_{M\otimes N}\circ (\theta_M^{-1} \otimes \theta_N^{-1})
$$
for any $M,N\in \cC$.
A {\it balanced braided tensor category} is a braided tensor category equipped with such a balance.
\end{dfn}
We say that a braided tensor category $\cC$ has strict unitors if $l_M =\id_M$ and $r_M= \id_M$ for any $M\in \cC$.

Let $\cC$ be a small $\C$-linear category.
Fresse shows that a braided tensor category structure on $\cC$ with strict unitors is equivalent to an operad homomorphism $\CPaB\rightarrow \Endo_\cC$.
We will briefly recall this fact (for more detail, see \cite[Section 6]{Fr}).

We first consider a braided tensor category with strict unitors $(\cC,\boxtimes,\al,\si)$.
Let $A \in \Tr_{[r]}$ and $M_{[r]}=(M_1,\dots,M_r)$ be a sequence consisting of $r$ objects in $\cC$.
Since $A$ is a parenthesized product of words in $[r]$,
we can define a parenthesized tensor product by using the bifunctor 
$\boxtimes$. We denote it by $\boxtimes_A:\cC^r \rightarrow \cC$.
For example, 
\begin{align}
\boxtimes_{1(23)}\cM_{[3]}&=\cM_1 \boxtimes (\cM_2 \boxtimes \cM_3) \nonumber\\
\boxtimes_{(3(24))(15)}\cM_{[5]}&=(M_3\boxtimes (M_2\boxtimes M_4))\boxtimes (M_1 \boxtimes M_5)\label{eq_iterated_tensor} \\
\boxtimes_{1(2(3(\cdots (r-1 r)\cdots))} \cM_{[r]}&=M_1 \boxtimes(M_2 \boxtimes M_3 \boxtimes (\cdots
(M_{r-1}\boxtimes M_r)\cdots)).\nonumber
\end{align}
Moreover, set 
\begin{align}
\boxtimes_1&=\id : \cC \rightarrow \cC \in \Func(\cC,\cC),\label{eq_iterated_tensor2} \\
\boxtimes_\emptyset &= I \in \cC=\Endo_\cC(0) \nonumber.
\end{align}
Hence, we associate each object $A \in \CPaB(r)$
with a functor $\boxtimes_A \in \Endo_\cC(r) = \Func(\cC^r,\cC)$.
Next, for each morphism $\ga:A\rightarrow B$ in $\CPaB(r)$, we would like to define a natural isomorphism
$\rho_{\mathrm{BTC}}(\ga):\boxtimes_A\rightarrow \boxtimes_B$, which follows from the Mac Lane coherence theorem
(see \cite[Chapter XI]{Mac}). Therefore, we obtain an operad homomorphism
\begin{align}
\rho_{\mathrm{BTC}}:\CPaB(r)\rightarrow \Endo_\cC(r),\qquad r \geq 0.
\label{eq_add_BTC_rho}
\end{align}

Conversely, let $\rho:\CPaB \rightarrow \Endo_\cC$ be an operad homomorphism.
Set 
\begin{align*}
\boxtimes &= \rho(12):\cC \times \cC \rightarrow \cC,\\
I&=\rho(\emptyset) \in \cC=\Endo_\cC(0).
\end{align*}
By using the functor $\boxtimes$, for each $A\in \Tr_{[r]}$,
we can define a functor $\boxtimes_A:\cC^r \rightarrow \cC$ as above \eqref{eq_iterated_tensor} and \eqref{eq_iterated_tensor2}.

Let $A \in \PW_r$.
Then, $A$ can be obtained from iterated operadic compositions of the
2-leaves tree $(12) \in \Tr_{[2]}$.
For example, 
\begin{align}
(12)3=(12)\circ_1 (12) \text{ and } 1(23)=(12)\circ_2 (12). \label{eq_PW}
\end{align}
Combining such a decomposition of $A$ with the definition of $\boxtimes_A$ and $\boxtimes = \rho(12)$, 
we obtain $\boxtimes_A = \rho(A)$ for any $A \in \PW_r$.

\begin{figure}[h]
    \centering
    \includegraphics[scale=1]{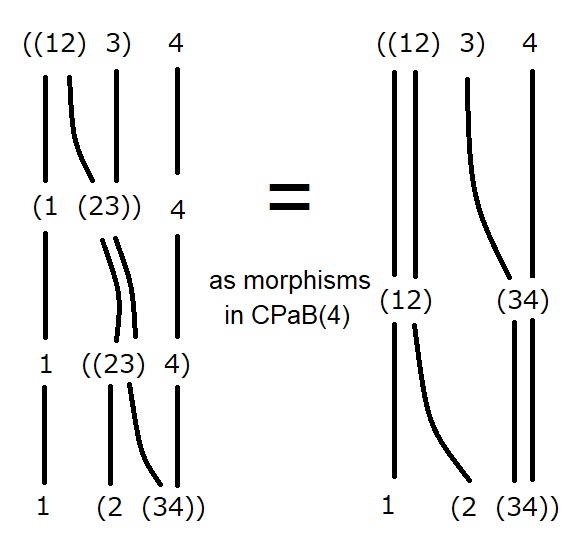}
    \caption{Pentagon type braid}\label{fig_pentagon}
\end{figure}

Consider the map $\al:(12)3\rightarrow 1(23)$ in the category $\CPaB(3)$ (see Fig. \ref{fig_alpha}).
Since $\rho$ is a functor, $\rho(\alpha)$ is a natural isomorphism between $\rho((12)3)\rightarrow \rho(1(23))$.
Hence, we have:
\begin{align*}
\rho(\alpha):
(\bullet \boxtimes \bullet)\boxtimes \bullet \rightarrow 
\bullet \boxtimes (\bullet \boxtimes \bullet).
\end{align*}
We denote it by $\tilde{\alpha}$.
We will show that $\tilde{\al}$ satisfies the pentagon identity.

See Fig. \ref{fig_pentagon}.
All maps between the five elements of $\PW_4$
are alpha type maps (see Remark \ref{rem_alpha_untwist}), and thus can be written of the form \eqref{eq_alpha_type_def}.
Hence, $\tilde{\al}$ satisfies the pentagon identity by Fig. \ref{fig_pentagon}.


Similarly,
the map $\si:(12)\rightarrow (21)$ in $\CPaB(2)$ (see Fig. \ref{fig_sigma})
defines a natural isomorphism
\begin{align*}
\tilde{\si}: \bullet_1 \boxtimes \bullet_2\rightarrow \bullet_2 \boxtimes \bullet_1.
\end{align*}
It can be similarly verified that this satisfies the hexagon identity.
%

\subsection{Pseudo-braided category and proendomorphism operad}
\label{sec_appendix_B}
In this section, we review the definition of a pseudo-braided category introduced by Soibelman. Furthermore, we introduce the proendomorphism operad, a multi-input analogue of the profunctor endomorphism, and we show that considering a pseudo-braided category is equivalent to giving a lax 2-morphism of operads from $\CPaB$ to the proendomorphism operad.

A {\it pseudo-tensor category} $C$ consists of a set of objects $\Ob(C)$ and multi-hom spaces,
\begin{align*}
\mathrm{Hom}(M_1,\dots,M_r;L),
\end{align*}
which are \(\C\)-vector spaces assigned to \(r\) input objects
\(M_1,\ldots,M_r\) and one output object \(L\),
together with composition maps satisfying associativity (for the precise definition, see \cite{Lam,BD}).
%
For a strict tensor category $(C,\boxtimes,1)$, we can define a pseudo-tensor category by
\begin{align*}
\mathrm{Hom}(M_1,\dots,M_r;L) = \mathrm{Hom}_C(M_1 \boxtimes \cdots \boxtimes M_r,L).
\end{align*}
The notion of a pseudo-braided category was introduced by Soibelman as a generalization of pseudo-tensor categories to the non-strict setting equipped with a braiding \cite{So}.
Note that the original definition is not unital. Here, we will give the unital version in terms of the parenthesized braid operad as follows:
\begin{dfn}\label{def_non_strict_pseudo}
A {\it unital pseudo-braided category} consists of the following data:
\begin{description}
\item[object)]
A set of objects $\Ob(C)$;
\item[hom space)]
$\C$-vector spaces
\begin{align*}
\mathrm{Hom}_A(M_1,\dots,M_r;L)
\end{align*}
for each $r \geq 1$, objects $L,M_1,\dots,M_r \in \Ob(C)$ and a tree $A\in\Tr_r$;
\item[unit]
A distinguished object $V \in \Ob(C)$, which is associated with $\emptyset \in \Tr_0$.
\item[identity]
A distinguished element $\id_M \in  \mathrm{Hom}_1(M,M)$ for any $M \in \cC$.
\item[composition]
For each $p \in \{1,\dots,r\}$,
$\C$-linear maps 
\begin{align}
\begin{split}
\circ_p: \mathrm{Hom}_A&(N_1,\dots,N_{p-1},L,N_{p+1},\dots,N_r;R)\otimes \mathrm{Hom}_B(M_1,\dots,M_s;L)\\
&\rightarrow \mathrm{Hom}_{A\circ_p B}(N_1,\dots,N_{p-1},M_1,\dots,M_s,N_{p+1},\dots,N_r;R),\quad (f,g) \mapsto f\circ_p g 
\end{split}
\label{eq_pseudo_p}
\end{align}
for any $r \geq 1$ and $s \geq 0$ and
\begin{align*}
\circ_p \emptyset: \mathrm{Hom}_A(N_1,\dots,N_{p-1},V,N_{p+1},\dots,N_r;R) \rightarrow \mathrm{Hom}_{A\circ_p \emptyset}(N_1,\dots,N_{p-1},N_{p+1},\dots,N_r;R);
\end{align*}
for any $r \geq 1$.
\item[braiding and associator]
$\CPaB(r)$-action on $\mathrm{Hom}_A(M_1,\dots,M_r;L)$, i.e.,
\begin{align*}
\rho(g): \mathrm{Hom}_A(M_1,\dots,M_r;L) \rightarrow \mathrm{Hom}_{A'}(M_1,\dots,M_r;L)
\end{align*}
for any $A,A' \in \Tr_r$ and $\CPaB$-morphisms $g:A\rightarrow A'$, which is associative and unital, that is,
$\rho(g')\circ \rho(g) =\rho(g'\circ g)$ for any $A,A',A'' \in \Tr_r$ and $\CPaB$-morphisms $g:A\rightarrow A'$
and $g':A' \rightarrow A''$ and $\rho(1)=\mathrm{id}$;
\begin{itemize}
\item
Here, $\rho(g)$ depends on $M_1,\dots,M_r;L$.If we want to make this explicit, we write it as $\rho(g)_{M_1,\dots,M_r;L}$.
\end{itemize}
\end{description}
such that they satisfy the following conditions:
\begin{enumerate}
\item[(SB0)]
For any $\si \in S_r$, $A,A'\in \Tr_{r}$ and $g: A \rightarrow A'$ in $\CPaB(r)$,
\begin{align*}
\mathrm{Hom}_{A^\si}(M_1,\dots,M_r;L) = \mathrm{Hom}_{A}(M_{\si(1)},\dots,M_{\si(r)};L),
\end{align*}
 and $\rho(g^\si)_{M_{1},\dots,M_r,L} = \rho(g)_{M_{\si(1)},\dots,M_{\si(r)};L}$, that is, the following diagram commutes
\begin{align*}
\begin{split}
\begin{array}{ccc}
\mathrm{Hom}_{A^\si}(M_1,\dots,M_r;L) &{=}&
\mathrm{Hom}_{A}(M_{\si(1)},\dots,M_{\si(r)};L)\\
\downarrow^{\rho(g^\si)_{M_{1},\dots,M_r,L}}
      && 
    \downarrow^{{\rho(g)_{M_{\si(1)},\dots,M_{\si(r)};L}}}
    \\
\mathrm{Hom}_{(A')^\si}(M_1,\dots,M_r;L) &{=}&
\mathrm{Hom}_{A'}(M_{\si(1)},\dots,M_{\si(r)};L).
\end{array}
\end{split}
\end{align*}
\item[(SB1)]
The compositions of morphisms $\circ_p$ are associative.
\item[(SB2)]
For any $A \in \Tr_r$,
\begin{align*}
\begin{split}
\circ_p: \mathrm{Hom}_A&(N_1,\dots,N_{p-1},L,N_{p+1},\dots,N_r;R)\otimes \mathrm{Hom}_1(L;L)\\
&\rightarrow \mathrm{Hom}_{A}(N_1,\dots,N_{p-1},L,N_{p+1},\dots,N_r;R),\quad (f,g) \mapsto f\circ_p g 
\end{split}
\end{align*}
satisfies $f\circ_p \id_L= f$, and
\begin{align*}
\begin{split}
\circ_1: \mathrm{Hom}_1&(L;L)\otimes \mathrm{Hom}_B(M_1,\dots,M_r;L)\\
&\rightarrow \mathrm{Hom}_{B}(M_1,\dots,M_r;L),\quad (f,g) \mapsto f\circ_1 g 
\end{split}
\end{align*}
satisfies $\id_L \circ_1 g= g$.
\item[(SB3)]
For any $n\geq 1$ and $m \geq 0$ and 
$A,A'\in \Tr_n$, $B,B' \in \Tr_m$, $g_A:A \rightarrow A'$
and $g_B:B\rightarrow B'$ and $p \in \{1,\dots,n\}$,
the following diagram commutes:
\begin{align*}
\begin{split}
\begin{array}{ccc}
\mathrm{Hom}_A(\bullet;\bullet) \otimes \mathrm{Hom}_B(\bullet;\bullet)
      &\overset{\circ_p}{\longrightarrow}&
\mathrm{Hom}_{A\circ_p B}(\bullet;\bullet)\\
    {}^{{\rho(g_A)\otimes \rho(g_B)}}\downarrow 
      && 
    \downarrow^{{\rho(g_A\circ_p g_B)}}
    \\
\mathrm{Hom}_{A'}(\bullet;\bullet) \otimes \mathrm{Hom}_{B'}(\bullet;\bullet)
      &\overset{\circ_p}{\longrightarrow}&
\mathrm{Hom}_{A'\circ_p B'}(\bullet;\bullet).
\end{array}
\end{split}
\end{align*}
\item[(SB4)]
For any $\si \in S_n$, $\tau \in S_m$, $i \in [n]$ and $A \in \CPaB(n)$, $B\in \CPaB(m)$, the following diagram commutes (see Definition \ref{def_permutation_operad}):
\begin{align*}
\begin{split}
\begin{array}{ccc}
\mathrm{Hom}_{A^\tau}(\bullet;\bullet) \otimes \mathrm{Hom}_{B^\si}(\bullet;\bullet)
      &\overset{\circ_p}{\longrightarrow}&
\mathrm{Hom}_{A^\tau\circ_p B^\si}(\bullet,\bullet)\\
    {=}\downarrow 
      && 
    \downarrow {=}
    \\
\left(\mathrm{Hom}_{A}(\bullet^\tau;\bullet) \otimes \mathrm{Hom}_{B}(\bullet^\si;\bullet)\right)
      &\overset{\circ_{\tau(p)}}{\longrightarrow}&
\mathrm{Hom}_{A\circ_{\tau(p)} B}(\bullet^{\tau' \si'};\bullet).
\end{array}
\end{split}
\end{align*}
Here, the vertical arrow (equality) is obtained by (SB0).
\end{enumerate}
\end{dfn}

It is clear that for any unital non-strict braided tensor category $(C,\boxtimes,\al)$ we can define a unital pseudo-braided category by
\begin{align*}
\mathrm{Hom}_A(M_1,\dots,M_r;L) = \mathrm{Hom}_C(\boxtimes_A M_{[r]};L)
\end{align*}
for any tree $A\in \Tr_r$ with the canonical pure braid group action (see Section \ref{sec_def_cpab}).

In the following, we reformulate the definition of a pseudo-braided category using a proendomorphism operad.
Let $\cC$ be a unital pseudo-braided category.  
Then, by (SB1) and (SB2), $\Hom_1(\bullet,\bullet)$ defines a $\C$-linear category structure on $\cC$ with $\id_\bullet$ as the identity morphism.  
We write $\Hom_\cC(\bullet,\bullet)$ for $\Hom_1(\bullet,\bullet)$.  

%
%
%

In this situation, for any $A \in \Tr_r$, it follows from (SB1) and (SB2) that
\begin{align*}
\Hom_A: (\cC^r)^\op \times \cC \rightarrow \Vect
\end{align*}
is a functor with respect to the categorical structure of $\cC$.
When $A = \emptyset$, we have
\begin{align*}
\mathrm{Hom}_{\emptyset}:\cC \rightarrow \Vect,\quad M \mapsto \mathrm{Hom}_\cC(V,M)
\end{align*}
which is the representable functor represented by $V$.
Moreover, we have a sequence of functors:
\begin{align}
\CPaB(r) &\rightarrow \Func\left((\cC^r)^{\op} \times \cC,\Vect\right),\qquad (\text{for $r \geq\ 0$})\label{eq_seq_pab_add}\\
g:A \rightarrow A' &\mapsto \rho(g): \Hom_A \rightarrow \Hom_{A'}.
\nonumber
\end{align}
Here $\Func((\cC^r)^{\op} \times \cC,\Vect)$ denotes the functor category.

The defining data of a pseudo-braided category further include the composition of $\Hom_A$ and $\Hom_B$. Such structures can be described by means of coends.
More precisely, the collection $\{\Func\left((\cC^r)^{\op} \times \cC,\Vect\right)\}_{r \geq 0}$ admits a natural 2-operad structure defined using coends, which we call the \emph{proendomorphism operad}.
In this section, we show that the defining data of a pseudo-braided category on a $\C$-linear category $\cC$ are equivalent to giving an operad homomorphism from $\CPaB$ to the proendomorphism operad.

We begin by recalling the notion of a coend in category theory \cite{Mac}, which will be used to define the proendomorphism operad.
\begin{dfn}\label{def_coend}
Let $\cC$ and $\cD$ be small $\C$-linear categories and
$H:\cC^\op \times \cC \;\longrightarrow\; \cD$ a functor. 
The \emph{coend} of $H$ consists of an object $x_H$ of $\cD$ together with a family of morphisms $\{\iota_a: H(a,a) \rightarrow x_H\}_{a\in \cC}$
such that for any morphism $f:a \rightarrow b$ in $\cC$, the following diagram commutes:
\begin{align}
\begin{split}
\begin{array}{ccc}
H(b,a) & \overset{{H(f,\id)}}{\longrightarrow} & H(a,a) \\
{}_{H(\id,f)}\downarrow & &\downarrow{}_{\iota_a} \\
H(b,b) & \overset{\iota_b}{\longrightarrow} & x_H
\end{array}
\end{split}
\label{eq_coend_com}
\end{align}
and it satisfies the following universal property:
\begin{itemize}
\item
For any object $z \in \cD$ and a family of morphisms $\{\omega_a: H(a,a) \rightarrow z\}_{a\in\cC}$ making the diagram \eqref{eq_coend_com} commute, there exists a unique morphism $t:x_H \rightarrow z$ such that the following diagram commutes for all $a\in \cC$:
\begin{align}
\begin{split}
\begin{array}{cc}
H(a,a) &   \\
  \iota_a\downarrow & \quad\searrow {\omega_a} \\
x_H &\underset{t}{\longrightarrow}  z.
\end{array}
\end{split}
\label{eq_coend_univ}
\end{align}
\end{itemize}
\end{dfn}
If $\cD$ is cocomplete, the coend always exists. We denote it by
\[
\int^{a \in \cC} H(a,a).
\]
When $H$ depends on multiple variables, e.g.,
$H:(\cC^\op \times \cC)^n \longrightarrow \cD$,
there exist natural isomorphisms
\begin{align}
 \int^{(a_1,\dots,a_n) \in \cC^n}H(a_1,a_1,\dots,a_n,a_n) &\cong \int^{a_1\in \cC}\dots \int^{a_n \in \cC} H(a_1,a_1,\dots,a_n,a_n)  \label{eq_add_Fubini}\\
 &\cong \int^{a_1,\dots,a_k \in \cC^k}\dots \int^{a_{k+1},\dots ,a_n \in \cC^{n-k}} H(a_1,a_1,\dots,a_n,a_n), 
 \label{eq_add_Fubini2}
\end{align}
so that partial coends can be performed in any order as in \eqref{eq_add_Fubini2}.
This property is called \emph{Fubini's theorem} for coends, and it holds whenever the target category is cocomplete (see for example \cite{Mac}).

Next, we will defined the proendomorphism operad using coends.
Let $F,G \in \Func(\cC^\op \times \cC,\Vect)$.
Define $F \otimes G: \cC^\op \times \cC \times \cC^\op \times \cC \rightarrow \Vect$ by $F(\bullet_1,\bullet_2) \otimes_\C G(\bullet_1,\bullet_2)$. Then, by taking the coend,
\begin{align*}
F \circ G =\int^{N\in \cC} F(N,\bullet) \otimes_\C G(\bullet,N) \in \Func(\cC^\op \times \cC).
\end{align*}
Hence, one can define a composition on $\Func(\cC^\op \times \cC,\Vect)$, which is associative up to the unique universal isomorphism. Hence, $\Func(\cC^\op \times \cC,\Vect)$ is a 2-monoid in the category of small categories, which is a strict 2-category. An element in $\Func(\cC^\op \times \cC,\Vect)$ is called a {\it profunctor} \cite{Bor}.
Set
\begin{align*}
\Endp_\cC(n) =\Func((\cC^\op)^n \times \cC,\Vect)
\end{align*}
for $n \geq 0$. 
Then, in analogy with profunctors, the collection $\{\Endp_\cC(n)\}_{n \geq 0}$ carries the structure of a 2-operad (up to the universal isomorphism, the composition is associative), which we call the {\it proendomorphism operad} 
(for 2-operads, see for example \cite{Le}).
Here, the unit of $\Endp_\cC$ is given by $\Hom_\cC(\bullet,\bullet) \in \Endp_\cC(1)$, since the coend with $\Hom(\bullet,\bullet)$ reduces simply to substitution, that is,
there are universal isomorphisms
\begin{align}
\begin{split}
F \circ_p \mathrm{Hom}_\cC(-,-) &= 
\int^{N\in \cC} F(\bullet_1,\dots,\bullet_{p-1},N,\bullet_{p+1}, \dots,\bullet_n,\bullet_0) \otimes_\C \mathrm{Hom}_\cC(\bullet_p,N)\\
&\cong F(\bullet_1,\dots,\bullet_{p-1},\bullet_p,\bullet_{p+1},\dots,\bullet_n,\bullet_0)
\end{split}
\label{eq_substi1}
\end{align}
and
\begin{align}
\mathrm{Hom}_\cC(-,-) \circ_1 F = 
\int^{N\in \cC} \mathrm{Hom}_\cC(N,\bullet_0) \otimes_\C F(\bullet_{[n]},N) \cong F(\bullet_{[n]},\bullet_0)
\label{eq_substi2}
\end{align}
for any $F \in \Endp_\cC(n)$.


\begin{dfn}\label{def_2action}
Let $(P,\circ_P,i_P)$ be an 1-operad in $\cat$ (e.g. $\CPaB$) and $\cC \in \cat$.
A lax 2-morphism from $P$ to $\Endp_\cC$ is a sequence of functors
\begin{align*}
f_n:P(n) \rightarrow \Endp(n)
\end{align*}
and natural transformations $\mu_{n,m,p}: f_n(\bullet) \circ_p f_m(\bullet)\rightarrow f_{n+m-1}(\bullet\circ_p \bullet)$, that is,
\begin{align}
\begin{split}
\begin{array}{ccc}
P(n)\times P(m) &\overset{f_n \times f_m}{\longrightarrow}& \Endp_\cC(n)\times \Endp_\cC(m)\\
\downarrow_{\circ_p}
      &\Downarrow{\mu_{n,m,p}}& 
   \downarrow_{\circ_p}
    \\
    P(n+m-1) &\overset{f_{n+m-1}}{\rightarrow}& \Endp_\cC(n+m-1),
\end{array}
\end{split}
\label{eq_mu_diag_add}
\end{align}
such that the following conditions hold:
\begin{itemize}
\item[LM0)]
With respect to the actions of $S_r$ on $P(r)$ and on $\Endp_\cC(r)$ as category isomorphisms (not merely equivalences),  the map $f_r : P(r) \rightarrow \Endp_\cC(r)$ is $S_r$-equivariant.
\item[LM1)]
$\mu$ is associative.
\item[LM2)]
$f_1$ sends the unit to the unit, $f_1\circ i_P=\mathrm{Hom}(\bullet,\bullet)$.
\item[LM3)]
For any $\si \in S_n$, $\tau \in S_m$, $i \in [n]$ and $A \in P(n)$, $B\in P(m)$,
\begin{align*}
\mu_{n,m,p}(A^\tau,B^{\si}) &= (\mu_{n,m,\tau(p)}(A,B))^{\tau'\si'},
\end{align*}
holds (see Definition \ref{def_permutation_operad} and Definition \ref{def_non_strict_pseudo} (SB4)).
\end{itemize}
\end{dfn}

Let $(\cC,\Hom_A(\bullet,\bullet),\rho)$ be a unital pseudo-braided category.
Then, define a sequence of functors 
\begin{align*}
f_n : \CPaB(n) \rightarrow \Endp_\cC(n)
\end{align*}
as in \eqref{eq_seq_pab_add},
and $f_0:\CPaB(0) \rightarrow \Endp_\cC(0)$ by
\begin{align*}
f_0(*)= \Hom(V,\bullet).
\end{align*}
Define the natural transformations $\mu$ \eqref{eq_mu_diag_add} by the coend of
\begin{align*}
\circ_p: \Hom_A(\bullet;\bullet) \otimes \Hom_B(\bullet;\bullet) \rightarrow \Hom_{A \circ_p B}(\bullet;\bullet),
\end{align*}
where the naturality follows from (SB3).
It is clear that (SB0) and (SB4) imply (LM0) and (LM3), respectively.  
Conversely, if a lax 2-morphism $(f,\mu):\CPaB \to \Endp_\cC$ is given and $f_0(*)$ is representable by some $V \in \cC$, then $\cC$ admits a unital pseudo-braided category structure.

\begin{prop}\label{prop_app_operad}
For a \(\mathbb C\)-linear category \(\mathcal C\), giving a unital
pseudo-braided category structure on \(\mathcal C\) is equivalent to giving
a lax \(2\)-morphism \((f,\mu):\PaB\to\Endp_{\mathcal C}\) together with a
choice of an object \(V\in\mathcal C\) representing \(f_0(*)\).
\end{prop}

\subsection{From unital pseudo-braided category to braided tensor category}\label{app_pseudo}
In this section, we investigate when a unital pseudo-braided category gives the structure of a braided tensor category, and explicitly describe the braided tensor category structure.


%
Let $\cC$ be a unital pseudo-braided category.
Assume that there is a bifunctor 
\begin{align*}
\boxtimes: \cC \times \cC \rightarrow \cC
\end{align*}
and a natural isomorphism
\begin{align}
\mu: \Hom_\cC(M_1 \boxtimes M_2, N) \rightarrow \Hom_{(12)}(M_1,M_2;N), \label{eq_mu_def}
\end{align}
where the naturality is taken with respect to functors $\cC^\op \times \cC^\op \times \cC \rightarrow \Vect$.
Then, for any $A \in \Tr_r$, one obtains a (not necessarily isomorphic) natural transformation
\begin{align*}
\mu_A: \Hom(\boxtimes_A M_{[r]},M_0) \rightarrow \Hom_A(M_{[r]}; M_0),
\end{align*}
where $\boxtimes_A$ is defined in \eqref{eq_iterated_tensor}.
We explain the case $A=1(23)$.  
Considering the decomposition $1(23) = (12)\circ_2 (12)$, 
\begin{align}
\circ_2:\Hom_{(12)}(M_1,L;M_0) \otimes_\C \Hom_{(12)}(M_2,M_3;L) \rightarrow \Hom_{1(23)}(M_1,M_2,M_3;M_0)
\label{eq_coend_L}
\end{align}
is obtained by \eqref{eq_pseudo_p}. 
Then,
the universality of the coend yields
\begin{align}
\circ_2: \int^L \Hom_{(12)}(M_1,L;M_0) \otimes_\C \Hom_{(12)}(M_2,M_3;L) \rightarrow \Hom_{1(23)}(M_1,M_2,M_3;M_0).
\label{eq_coend2}
\end{align}
%
%
%
%
Here, by the existence of $\boxtimes$, equation \eqref{eq_coend2} can be rewritten as
\begin{align*}
\circ_2: \int^L \Hom(M_1 \boxtimes L;M_0) \otimes_\C \Hom(M_2 \boxtimes M_3;L) \rightarrow \Hom_{1(23)}(M_1,M_2,M_3;M_0).
\end{align*}
Hence, by \eqref{eq_substi1} and \eqref{eq_substi2},
we obtain a natural isomorphism
\[
\circ_2: \int^L \Hom(M_1 \boxtimes L;M_0) \otimes_\C \Hom(M_2 \boxtimes M_3;L) \cong \Hom(M_1 \boxtimes (M_2 \boxtimes M_3);M_0),
\]
and thus, a natural transformation
\begin{align}
\mu_{1(23)}:\Hom(M_1 \boxtimes (M_2 \boxtimes M_3);M_0) \rightarrow \Hom_{1(23)}(M_1,M_2,M_3;M_0).
\label{eq_coend3}
\end{align}

\begin{rem}
Equation \eqref{eq_coend3} can be described more explicitly as follows: in
\[
\circ_2:\Hom(M_1 \boxtimes L;M_0) \otimes_\C \Hom(M_2 \boxtimes M_3;L) \rightarrow \Hom_{1(23)}(M_1,M_2,M_3;M_0),
\]
taking $L = M_2 \boxtimes M_3$ and considering the image of $\id_{M_2 \boxtimes M_3}$ gives \eqref{eq_coend3}.
\end{rem}

Note that the natural transformation $\mu_{1(23)}$ is determined uniquely from the universality once \eqref{eq_mu_def} is given.  
For any $A \in \Tr_r$, by decomposing $A$ into products of $(12)$ in the magma operad as in \eqref{eq_PW}, one obtains a natural transformation
\begin{align*}
\mu_A: \Hom(\boxtimes_A M_{[r]},M_0) \rightarrow \Hom_A(M_{[r]}; M_0).
\end{align*}

From (SB0) and (SB4), the natural transformation $\mu_A$ is equivariant with respect to the action of $S_r$ on $A$.  
Although $\mu_A$ is obtained by taking coends $r-2$ times, by Fubini's theorem for coends, it coincides with the partially taken coend.  
Hence we have the following:

\begin{lem}\label{lem_add_factor}
For any $A\in \Tr_r$ and $B \in \Tr_s$, with $p \in [r]$ and $r,s \geq 1$, the following diagram commutes:
\begin{align}
\begin{split}
\begin{array}{ccc}
\mathrm{Hom}(\boxtimes_{A \circ_p B} \bullet,\bullet)
&\overset{\mu_{A \circ_p B}}{\longrightarrow}&
\mathrm{Hom}_{A \circ_p B}(\bullet,\bullet)\\
\downarrow^{\cong} 
      && 
    \uparrow{\mu_{A,B}}
    \\
\int (\mathrm{Hom}(\boxtimes_{A} \bullet, \bullet) \otimes_\C \mathrm{Hom}(\boxtimes_{B} \bullet,\bullet))
      &\overset{\int \mu_{A}\otimes \mu_{B}}{\longrightarrow}&
\int (\mathrm{Hom}_{A}(\bullet;\bullet) \otimes_\C \mathrm{Hom}_{B}(\bullet;\bullet)).
\end{array}
\end{split}
\label{add_factor_lem}
\end{align}
Here, the vertical arrows are the morphism determined by the universality of the coend.
\end{lem}

Set
\begin{align}
\std_r =\sstd = (12) \circ_2 \biggl( (12)\circ_2  \Bigl(\cdots\bigl( (12) \circ_2 (12)\bigr)\dots \biggr) \in \Tr_r,
\label{eq_std_fac}
\end{align}
which we call a \emph{standard tree}.
Then,
\begin{align*}
\boxtimes_{\std_r}(M_{[r]}) = M_1\boxtimes (M_2 \boxtimes (\dots (M_{r-1}\boxtimes M_r )\dots ).
\end{align*}


\begin{lem}\label{prop_add_pseudo1}
Let $\mathcal C$ be a unital pseudo-braided category.
Assume that the following conditions hold:
\begin{enumerate}
\item
There exists a bifunctor
\[
  \boxtimes:\mathcal C\times \mathcal C\longrightarrow \mathcal C
\]
together with a natural isomorphism
\[
  \mu_{(12)}:
  \mathrm{Hom}_{\mathcal C}(M_1\boxtimes M_2,N)
  \rightarrow
  \mathrm{Hom}_{(12)}(M_1,M_2;N).
\]
\item
For any $r\geq 2$, the natural transformation 
\begin{align}
\mu_{\std_r}:
\mathrm{Hom}(\boxtimes_{\std_r}(M_{[r]}),N) \rightarrow \mathrm{Hom}_{\std_r}(M_{[r]}; N)\label{eq_prop_std_isom}
\end{align}
is a natural isomorphism.
\end{enumerate}
Then, for any $A \in \Tr_r$, and $M_1,\dots,M_r, N\in \cC$,
\begin{align}
\mu_A: \mathrm{Hom}(\boxtimes_A M_{[r]}, N) \rightarrow \Hom_A(M_1,\dots,M_r;N)
\label{eq_prop_muA_isom}
\end{align}
is a natural isomorphism.
In particular, the functor
\[
  \mathrm{Hom}_A(M_1,\dots,M_r; \bullet):
  \mathcal C \longrightarrow \Vect
\]
is represented by $\boxtimes_A(M_1,\ldots,M_r)$.
\end{lem}
\begin{proof}
We prove that \eqref{eq_prop_muA_isom} is an isomorphism by induction on $r$.  
For $r=2$, the claim follows from the assumption and (SB0), since $\mathrm{Hom}_{21}(M_1,M_2;L)$ is represented by the bifunctor 
\[
\cC^2 \rightarrow \cC, \qquad (M_1,M_2) \mapsto M_2 \boxtimes M_1.
\]

Assume $r \geq 3$.  
By taking coends in (SB3), we obtain the commutative diagram
\begin{align}
\begin{split}
\begin{array}{ccc}
\int (\mathrm{Hom}_A(\bullet;\bullet) \otimes_\C \mathrm{Hom}_B(\bullet;\bullet))
      &\overset{\int \circ_p}{\longrightarrow}&
\mathrm{Hom}_{A\circ_p B}(\bullet;\bullet)\\
    {{\rho(g_A)\otimes \rho(g_B)}}\downarrow 
      && 
    \downarrow{{\rho(g_A\circ_p g_B)}}
    \\
\int (\mathrm{Hom}_{A'}(\bullet;\bullet) \otimes_\C \mathrm{Hom}_{B'}(\bullet;\bullet))
      &\overset{\int \circ_p}{\longrightarrow}&
\mathrm{Hom}_{A'\circ_p B'}(\bullet;\bullet).
\end{array}
\end{split}
\label{add_coend1}
\end{align}

For any $k=2,\dots,r-1$, there is a decomposition 
\[
\std_r = \std_k \circ_k \std_{r-k+1}.
\]
Applying this decomposition to Lemma \ref{lem_add_factor}, and using that $\mu_{\std_i}$ is an isomorphism, we see that
\[
\int (\mathrm{Hom}_{\std_k}(\bullet;\bullet) \otimes_\C \mathrm{Hom}_{\std_{r-k+1}}(\bullet;\bullet)) \rightarrow \Hom_{\std_r}(\bullet,\bullet)
\]
is an isomorphism.  
Therefore, for any $A' \in \Tr_k$ and $B' \in \Tr_{r-k+1}$, choosing $g_A:\std_k \rightarrow A'$ and $g_B:\std_{r-k+1} \rightarrow B'$, it follows from \eqref{add_coend1} that
\[
\mu_{A',B'}:\int (\mathrm{Hom}_{A'}(\bullet;\bullet) \otimes_\C \mathrm{Hom}_{B'}(\bullet;\bullet)) \rightarrow \Hom_{A' \circ_k B'}(\bullet,\bullet)
\]
is also an isomorphism.  
Hence, by Lemma \ref{lem_add_factor} and the induction hypothesis, using that $\mu_{A'}$ and $\mu_{B'}$ are isomorphisms, we deduce that
\[
\mu_{A' \circ_k B'}: \mathrm{Hom}(\boxtimes_{A'\circ_k B'} \bullet;\bullet) \rightarrow \Hom_{A' \circ_k B'}(\bullet,\bullet)
\]
is an isomorphism.  
Since every element of $\Tr_r$ admits such a decomposition, the assertion holds.

\end{proof}

\begin{lem}
\label{prop_add_pseudo2}
Assume that the hypotheses of
Lemma \ref{prop_add_pseudo1}
are satisfied. 
Suppose moreover that, for any $r\geq 1$ and 
$p\in\{1,\ldots,r\}$, the vacuum insertion map for the standard tree
\[
  \circ_p\emptyset:
  \mathrm{Hom}_{\mathrm{std}_r}
  (M_1,\dots,M_{p-1},V,M_{p+1},\dots,M_r;L)
  \rightarrow
  \mathrm{Hom}_{\mathrm{std}_r\circ_p\emptyset}
  (M_1,\dots,M_{p-1},M_{p+1},\dots,M_r;L)
\]
is a natural isomorphism.
Then, for any $A\in T_r$ and $p\in\{1,\ldots,r\}$, 
\[
  \circ_p\emptyset:
  \operatorname{Hom}_{A}
  (M_1,\ldots,M_{p-1},V,M_{p+1},\ldots,M_r;L)
  \longrightarrow
  \operatorname{Hom}_{A\circ_p\emptyset}
  (M_1,\ldots,M_{p-1},M_{p+1},\ldots,M_r;L)
\]
is a natural isomorphism.
%
\end{lem}
\begin{proof}
Take a morphism $g:\std_r \rightarrow A$ in $\CPaB(r)$.
From (SB3), we obtain the following commutative diagram:
\begin{align*}
\begin{split}
\begin{array}{ccc}
\Hom_{\std_r}(M_{[r]};L) &\overset{\circ_p \emptyset}{\longrightarrow}& \Hom_{\std_r \circ_p \emptyset}(M_1,\dots,\hat{M_p},\dots, M_r;L)\\
{}_\cong\downarrow{\rho(g)}
      && 
    {}_\cong\downarrow{\rho(g \circ_p \emptyset)}\\
\Hom_A(M_{[r]};L) &\overset{\circ_p \emptyset}{\longrightarrow}& \Hom_{A\circ_p \emptyset}(M_1,\dots,\hat{M_p},\dots, M_r;L),
\end{array}
\end{split}
\end{align*}
where $M_p=V$.
By the assumption, the top horizontal arrow is an isomorphism, hence the bottom arrow is also an isomorphism.
\end{proof}

Assume that the hypotheses of
Lemma \ref{prop_add_pseudo1}
and Lemma \ref{prop_add_pseudo2} are satisfied.
Let $A\in \Tr_r$ and let $p\in\{1,\ldots,r\}$.
We define a natural morphism
\[
  \nu_A^p:
  \boxtimes_A(M_1,\ldots,M_{p-1},V,M_{p+1},\ldots,M_r)
  \longrightarrow
  \boxtimes_{A\circ_p\varnothing}
  (M_1,\ldots,M_{p-1},M_{p+1},\ldots,M_r)
\]
by the following compositions and the Yoneda lemma:
\[
\begin{aligned}
&
\operatorname{Hom}_{\mathcal C}
\Bigl(
  \boxtimes_{A\circ_p\varnothing}
  (M_1,\ldots,M_{p-1},M_{p+1},\ldots,M_r),
  L
\Bigr)
\\
&\xrightarrow{\ \mu_{A\circ_p\varnothing}\ }
\operatorname{Hom}_{A\circ_p\varnothing}
(M_1,\ldots,M_{p-1},M_{p+1},\ldots,M_r;L)
\\
&\xrightarrow{\ (\circ_p\varnothing)^{-1}\ }
\operatorname{Hom}_{A}
(M_1,\ldots,M_{p-1},V,M_{p+1},\ldots,M_r;L)
\\
&\xrightarrow{\ \mu_A^{-1}\ }
\operatorname{Hom}_{\mathcal C}
\Bigl(
  \boxtimes_A(M_1,\ldots,M_{p-1},V,M_{p+1},\ldots,M_r),
  L
\Bigr).
\end{aligned}
\]

%

We now prove that the representability conditions above allow us to transport
the PaB-action on the pseudo-braided Hom-spaces to the representing objects.
This gives the braided tensor category structure.
%
\begin{thm}\label{thm_add_BTC}
Let $\cC$ be a unital pseudo-braided category.
Assume that the following conditions hold:
\begin{enumerate}
\item
There exists a bifunctor
\[
  \boxtimes:\mathcal C\times \mathcal C\longrightarrow \mathcal C
\]
together with a natural isomorphism
\[
  \mu_{(12)}:
  \mathrm{Hom}_{\mathcal C}(M_1\boxtimes M_2,N)
  \rightarrow
  \mathrm{Hom}_{(12)}(M_1,M_2;N).
\]
\item
For any $r\geq 2$, the natural transformation
\begin{align}
\mu_{\std_r}:
\mathrm{Hom}(\boxtimes_{\std_r}(M_{[r]}),N) \rightarrow \mathrm{Hom}_{\std_r}(M_{[r]}; N)\label{eq_prop_std_isom}
\end{align}
is a natural isomorphism.
\item
For any $r\geq 1$ and 
$p\in\{1,\ldots,r\}$, 
\[
  \circ_p\emptyset:
  \mathrm{Hom}_{\mathrm{std}_r}
  (M_1,\dots,M_{p-1},V,M_{p+1},\dots,M_r;L)
  \rightarrow
  \mathrm{Hom}_{\mathrm{std}_r\circ_p\emptyset}
  (M_1,\dots,M_{p-1},M_{p+1},\dots,M_r;L)
\]
is a natural isomorphism.
\end{enumerate}
Then, $(\cC,\boxtimes)$ inherits a braided tensor category structure, where the braiding $B$ and the associator $\al$ are defined by
the Yoneda lemma from
\begin{align}
\begin{split}
\begin{array}{ccc}
\Hom(M_2\boxtimes M_1,L) &\overset{\mu_{21}}{\rightarrow}& \Hom_{(21)}(M_1,M_2;L)\\
\downarrow{B_{M_1,M_2}^*}
      && 
    \downarrow{\rho(\si)^{-1}}
    \\
\Hom(M_1\boxtimes M_2,L) &\overset{\mu_{12}}{\rightarrow}& \Hom_{(12)}(M_1,M_2;L).
\end{array}
\end{split}
\label{eq_add_braid}
\end{align}
and
\begin{align}
\begin{split}
\begin{array}{ccc}
\Hom(M_1\boxtimes (M_2\boxtimes M_3),L) &\overset{\mu_{1(23)}}{\rightarrow}& \Hom_{1(23)}(M_1,M_2,M_3;L)\\
\downarrow{\al_{M_1,M_2,M_3}^*}
      && 
    \downarrow{\rho(\al)^{-1}}
    \\
\Hom((M_1\boxtimes M_2)\boxtimes M_3,L) &\overset{\mu_{(12)3}}{\rightarrow}& \Hom_{(12)3}(M_1,M_2,M_3;L).
\end{array}
\end{split}
\label{eq_add_ass}
\end{align}
and the right and left unitor are defined by $\nu_{12}^2:M\boxtimes V {\rightarrow} M$ and 
$\nu_{12}^1:V \boxtimes M {\rightarrow} M$. Moreover, the following diagram commutes
\begin{align*}
\begin{split}
\begin{array}{ccc}
\Hom(\boxtimes_B M_{[r]},L) &\overset{\mu_{B}}{\longrightarrow}& \Hom_{B}(M_{[r]};L)\\
\downarrow{\rho_{\mathrm{BTC}}(g)^*}
      && 
    \downarrow{\rho(g)^{-1}}
    \\
\Hom(\boxtimes_A M_{[r]},L) &\overset{\mu_{A}}{\longrightarrow}& \Hom_{A}(M_{[r]};L)
\end{array}
\end{split}
\end{align*}
for any $A \in \Tr_r$ and $g:A \rightarrow B$, where $\rho_{\mathrm{BTC}}(g)$ is a map defined by the associator and the braiding in \eqref{eq_add_BTC_rho}.
Conversely, every braided tensor category
$(\mathcal C,\boxtimes,\mathbf 1,\alpha,l,r,B)$
defines a unital pseudo-braided category satisfying the assumptions in Lemma 
\ref{prop_add_pseudo1} and Lemma \ref{prop_add_pseudo2}.
\end{thm}

\begin{proof}
By Lemma \ref{prop_add_pseudo1}, for any \(A\in T_r\), $\mathrm{Hom}_A(M_{[r]};\bullet)$ is represented by $\boxtimes_A M_{[r]}$.
For \(\gamma:A\to B\) in \(\mathrm{PaB}(r)\), define a natural isomorphism $\rho^{\boxtimes}(\gamma): \boxtimes_A M_{[r]}
  \rightarrow
  \boxtimes_B M_{[r]}$
by 
\[
\begin{CD}
\operatorname{Hom}_{\mathcal C}(\boxtimes_B M_{[r]},L)
@>{\mu_B}>>
\operatorname{Hom}_B(M_{[r]};L)
\\
@V{(\rho^{\boxtimes}(\gamma))^*}VV
@VV{\rho(\gamma)^{-1}}V
\\
\operatorname{Hom}_{\mathcal C}(\boxtimes_A M_{[r]},L)
@>{\mu_A}>>
\operatorname{Hom}_A(M_{[r]};L).
\end{CD}
\]
For \(\gamma:A\to A'\) and \(\mu:B\to B'\) in
\(\mathrm{PaB}\), the following diagram commutes by (SB3)
\begin{align}
\begin{CD}
\boxtimes_{A\circ_p B}
@>{\rho^{\boxtimes}(\gamma\circ_p\mu)}>>
\boxtimes_{A'\circ_p B'}
\\
@V{\simeq}VV
@VV{\simeq}V
\\
\boxtimes_A\circ_p\boxtimes_B
@>{\rho^{\boxtimes}(\gamma)\circ_p\rho^{\boxtimes}(\mu)}>>
\boxtimes_{A'}\circ_p\boxtimes_{B'} .
\end{CD}
\label{eq_add_operadic_comm2}
\end{align}
We define the braiding and associator by
$B_{M_1,M_2}=\rho^{\boxtimes}(\sigma)$ and $\alpha_{M_1,M_2,M_3} =  \rho^{\boxtimes}(\alpha)$.
It is now clear that the pentagon and hexagon identities follow from \eqref{eq_add_operadic_comm2} and arguments similar to
the one illustrated in Figure \ref{fig_pentagon}. Finally, by  Lemma \ref{prop_add_pseudo2}, vacuum insertion gives natural isomorphisms
$\nu^2_{(12)}:M\boxtimes V\longrightarrow M$ and $\nu^1_{(12)}:V\boxtimes M\longrightarrow M$.
We define $r_M=\nu^2_{(12)}$ and $l_M=\nu^1_{(12)}$.
The compatibility of vacuum insertion with the PaB-action and operadic
composition gives the triangle identity in the same way as above. Therefore
$(\mathcal C,\boxtimes,V,\alpha,l,r,B)$ is a braided tensor category.
Moreover, since \(\mathrm{PaB}\) is generated, as an operad, by \(\sigma\) and \(\alpha\), \eqref{eq_add_operadic_comm2}  shows that all morphisms \(\rho^{\boxtimes}(\gamma)\) are determined by \(B\) and \(\alpha\).
Thus $\rho^\boxtimes$ coincides with $\rho_{\mathrm{BTC}}$ \eqref{eq_add_BTC_rho} defined by the braided tensor category structure.
The converse follows from the construction in Section \ref{sec_def_cpab}.
\end{proof}

\begin{rem}
\label{rem_PaB_op}
The opposite-category construction
\[
  (-)^{\op}:\cat\rightarrow \cat,\quad \cC \mapsto \cC^\op
\]
is symmetric monoidal with respect to the cartesian product and the terminal category is preserved. Therefore, $\{\CPaB^{\op}(r)\}_{r \geq 0}$ is an operad object in $\cat$.
Moreover, since each $\CPaB(r)$ is a groupoid, inversion of braids
gives an isomorphism of operads
\[
  \iota:\mathrm{PaB}^{\op}\xrightarrow{\ \sim\ }\mathrm{PaB}.
\]
Here, \(\iota_r:\mathrm{PaB}(r)^{\op}\to \mathrm{PaB}(r)\) is the identity
on objects and sends a morphism \(g:B \to A\), regarded as a morphism
\(A\to B\) in \(\mathrm{PaB}(r)^{\op}\), to its inverse $g^{-1}:A\to B$
in \(\mathrm{PaB}(r)\). This is compatible with operadic compositions since $(g\circ_p h)^{-1}=g^{-1}\circ_p h^{-1}$,
and it is also compatible with relabelling of leaves.
\end{rem}

\section{Configuration space and parenthesized formal power series}
\label{sec_config}
Let $r\in \Z_{>0}$ and set 
\begin{align*}
X_r&=\{(z_1,\dots,z_r) \in \C^r \mid z_i\neq z_j \text{ for any }i\neq j \},
\end{align*}
which is called {\it an $r$ configuration space}.
$r$-point  conformal blocks of a conformal field theory are multivalued holomorphic functions on $X_r$.
These multivalued functions typically have the following singularities along $\{z_i=z_j\}$
\begin{align}
f(z_1,\dots,z_r) \sim (z_i-z_j)^\al(\log (z_i-z_j))^k \quad (k\in \Z_{\geq 0},\al\in \C).
\label{eq_singularity}
\end{align}
For conformal blocks defined from $C_1$-cofinite modules of a vertex operator algebra, 
$f(z_1,\dots,z_r)$ has a series expansion from the theory of differential equations of regular singularities. 
Notably, the series expansions are expressed in terms of a broader class of series than the usual Laurent expansion $\C((z_1,\dots,z_r))$, which allows for \eqref{eq_singularity}.

Since the function is multivalued, it is necessary to specify in what order $(z_1,z_2,\dots,z_r)$ are approached to each other to expand the series for all variables.
Such general singularity behavior can be described using trees $\Tr_r$.
For example, $(12)3$ means to bring $z_3$ closer to $z_2$ after bringing $z_1$ and $z_2$ closer together.
The purpose of this section is to introduce and study the properties of the open set $U_A \subset X_r$, its local coordinates $\{\zeta_e\}_{e\in E(A)}$ and the space of formal power series $T_A$ for each tree $A\in \Tr_r$.

Section \ref{sec_config_coordinate} introduces local coordinate systems, $x$-coordinate and $A$-coordinate on $X_r$ for each $A\in \Tr_r$, and the space of formal power series with the singularity \eqref{eq_singularity}.
Section \ref{sec_config_conv} examines the convergence property of $T_A$ and Section \ref{sec_config_D} examines a structure of a D-module on $T_A$. Section \ref{sec_config_deg} defines a degree map on formal expansions, which will be used in
the analysis of singularities of conformal blocks.




\subsection{Configuration space and local coordinates}
\label{sec_config_coordinate}
The $z$-coordinate of $X_r$ is the standard coordinate $(z_1,\dots,z_r)$ of $\C^r$.
When examining global properties, it is convenient to use $z$-coordinate.
In this section, we define local coordinates $x$-coordinate and $A$-coordinate for $A\in \Tr_r$.

For each edge $e\in E(A)$, let $u(e)$ denote the upper vertex and $d(e)$ denote the lower vertex.
Define maps $L,R: V(A)\rightarrow \Leaf(A)$ as follows:
For each vertex $v\in V(A)$, $R(v)$ is defined by the rightmost leaf that is the descendant of $v$
and $L(v)$ by the rightmost leaf among the leaves that are descendants of the child to the left of $v$. 
Let $t_A$ be the uppermost vertex and $r_A$ be the rightmost leaf among all leaves.
Then, $r_A=R(t_A)$.

\begin{minipage}[c]{.5\textwidth}
\centering
\begin{forest}
for tree={
  l sep=20pt,
  parent anchor=south,
  align=center
}
[$t_{A}$
[$v_1$[5][[2][3]]]
[$v_2$,edge label={node[midway,right]{$e_0$}}[[1][7]]
[[6][4]]]
]
\end{forest}
\end{minipage}
\begin{minipage}[c]{.5\textwidth}
\centering
In the case of the left figure,
\begin{align*}
A&=(5(23))((17)(64))\\
d(e_0)&=v_2\quad u(e_0)=t_{A}\\
L(v_1)&=5 \quad R(v_1)=3\\
L(v_2)&=7 \quad R(v_2)=4\\
r_A&=4.
\end{align*}
\end{minipage}


The functions $\{x_v:\Xr\rightarrow \C\}_{v\in V(A)}$ and $\{\zeta_e:\Xr\rightarrow \C\}_{e\in E(A)}$ are
defined by
\begin{align}
x_v &= z_{L(v)}-z_{R(v)},\\
\zeta_e &= \frac{x_{d(e)}}{x_{u(e)}}
\end{align}
This gives the family of $r-1$ functions $\{x_v:\Xr\rightarrow \C\}_{v\in V(A)}$
and the family of $r-2$ functions $\{\zeta_e:\Xr\rightarrow \C\}_{e\in E(A)}$.


{\it The $x$-coordinate system} is the system of functions
\begin{align*}
(x_v)_{v\in V(A)}\times z_A: \Xr\rightarrow \C^{r-1}\times \C,
\end{align*}
where $z_A:\Xr \rightarrow \C,\quad(z_1,\dots,z_r)\mapsto z_{r_A}$, the projection onto the $r_A$-th component.

\begin{minipage}[c]{.5\textwidth}
\centering
To see $x$-coordinate, it is easier to draw a tree with the function $x_v=z_{L(v)}-z_{R(v)}$ filled in at each vertex $v \in V(A)$.
The right figure is an example for $(23)((15)4)  \in P_5$.
\end{minipage}
\begin{minipage}[c]{.5\textwidth}
\centering
\begin{forest}
for tree={
  l sep=20pt,
  parent anchor=south,
  align=center
}
[$z_3-z_4$
[$z_2-z_3$,edge label={node[midway,left]{a}}[2][3]]
[$z_5-z_4$,edge label={node[midway,right]{c}}[$z_1-z_5$,edge label={node[midway,left]{b}}[1][5]]
[4]]
]
\end{forest}
\captionof{figure}{}
\label{fig_tree_example2}
\end{minipage}

{\it The $A$-coordinate system} is the system of functions
\begin{align}
\Psi_A = 
z_A\times x_A\times (\zeta_e)_{e\in E(A)}: \Xr\rightarrow \C\times \C\times  \C^{r-2},
\label{eq_zeta_coordinate}
\end{align}
where $x_A=x_{t_A}:\Xr \rightarrow \C$.

For $A=(23)((15)4) \in \Tr_5$, $A$-coordinate is given as:
\begin{align}
\Psi_{(23)((15)4)}=
(z_4,z_3-z_4,\ze_a=\frac{z_2-z_3}{z_3-z_4},\ze_b=\frac{z_1-z_5}{z_5-z_4},\ze_c=\frac{z_5-z_4}{z_3-z_4}),
\label{eq_example_psi}
\end{align}
where the labels $\{a,b,c\}$ of edges are given as in Fig. \ref{fig_tree_example2}.

It is easy to see that the inverse function $\Psi_A^{-1}:\C\times \C^{E(A)}\rightarrow \C^r$ is a polynomial
of $\{\ze_e\}_{e\in E(A)}$ and $x_A,z_A$.
For example,
\begin{align*}
\Psi_{(23)((15)4)}^{-1}=(z_1,z_2,z_3,z_4,z_5)
=(x_A\ze_c(1+\ze_b)+z_A, (1+\ze_a)x_A+z_A,x_A+z_A,z_A,\ze_cx_A+z_A).
\end{align*}
Thus, we have:
\begin{prop}
\label{prop_psiA}
For any $A\in \Tr_r$, 
$\Psi_A$ is a bi-holomorphic function from $X_r$ onto the image in $\C^r$.
Furthermore, $\Psi_A^{-1}$ is a polynomial of $\{\ze_e\}_{e\in E(A)}$ and $x_A,z_A$,
and thus can be extended to a holomorphic function $\Psi_A^{-1}:\C^2\times \C^{E(A)}\rightarrow \C^r$.
\end{prop}

\begin{rem}
Another local coordinate of $X_r$ is defined in \cite{Va},
which is called an {\it asymptotic zone},
to study the asymptotic behavior of the Knizhnik-Zamolodchikov
equation.
The definition of the asymptotic zone and \eqref{eq_zeta_coordinate}
are similar but different.
The local coordinate \eqref{eq_zeta_coordinate} is natural from the viewpoint of vertex algebras, since it describes the change of variables that naturally arises in their theory.
\end{rem}

A conformal block has a translation symmetry $(z_1,\dots,z_r) \mapsto (z_1+a,\dots, z_r+a)$ ($a \in\C$)
(see Section \ref{sec_D_translation}).
So the conformal block can be naturally regarded as a function on $\Xr/\C$.
Since $x_v=z_{L(v)}-z_{R(v)}$ is translation-invariant, $x_v$ and $\zeta_e$ can be naturally regarded as functions on the quotient space $\Xr/\C$.
The system of functions on $\Xr/\C$
\begin{align*}
(x_v)_{v\in V(A)}:\Xr/\C\rightarrow \C^{r-1},
\end{align*}
and
\begin{align*}
x_A\times (\zeta_e)_{e\in E(A)}: \Xr/\C\rightarrow \C\times \C^{r-2},
\end{align*}
are also called $x$-coordinate and $A$-coordinate, respectively.


\subsection{Formal power series and convergence}
\label{sec_config_conv}
Set 
\begin{align*}
\Or^\alg = \C[z_1,\dots,z_r,(z_i-z_j)^\pm],
\end{align*}
a ring of regular functions on $X_r$,
and
\begin{align*}
\Ort^\alg=\C[(z_i-z_j)^\pm \mid i \neq j],
\end{align*}
a ring of translation invariant functions in $\Or^\alg$.
In this section, we consider formal power series and their radii of convergence using the $A$-coordinate.
Any function of $\Ort^\alg$ can be expanded as a formal power series in
$\C[[\zeta_e \mid e\in E(A)]][x_A^\pm]$.
Let us first look at an example in the case of $A=(23)((15)4) \in \Tr_5$.
For $(z_2-z_1)^{-1} \in \mO_{X_5}^\alg$, we have:
\begin{align}
\begin{split}
(z_2-z_1)^{-1} &= \left((z_2-z_3)+(z_3-z_4)-(z_5-z_4)-(z_1-z_5)
\right)^{-1}\\
&= (z_3-z_4)^{-1}\left(1 + \frac{(z_2-z_3)}{(z_3-z_4)}-\frac{(z_5-z_4)}{(z_3-z_4)}-\frac{(z_1-z_5)}{(z_3-z_4)}
\right)^{-1}\\
&=x_{(23)((15)4)}^{-1}(1+\zeta_a-\zeta_c-\zeta_b\zeta_c)^{-1}\\
&=x_{(23)((15)4)}^{-1}\sum_{l=0}^\infty (-\zeta_a+\zeta_c+\zeta_b\zeta_c)^l
\in \C[[\zeta_a,\zeta_b,\zeta_c]][x_{(23)((15)4)}^{-1}].\label{eq_example_conv}
\end{split}
\end{align}
As this example shows, the series
 expansion determines an injective ring homomorphism 
\begin{align}
e_A: \Ort^\alg \rightarrow \C[[\zeta_e\mid e \in E(A)]][x_{A}^\pm,\zeta_e^\pm \mid e\in E(A)]
\label{eq_emap}
\end{align}
for $A\in \Tr_r$.
\begin{rem}
To obtain expansions of functions in $\Or^\alg$,
which are not always translation invariant, we have to consider
\begin{align}
\C[[\zeta_e\mid e \in E(A)]][z_A,x_{A}^\pm,\zeta_e^\pm \mid e\in E(A)].
\end{align}
\end{rem}

Consider the following linear space consisting of formal power series which allow the singularity \eqref{eq_singularity}:
\begin{align*}
T_A &= \C[[\zeta_e\mid e \in E(A)]][\log x_A, x_A^\C,\log\zeta_e,\zeta_e^\C \mid e \in E(A)],
\end{align*}
which is a space of formal power series spanned by
the finite sum of formal power series of the form:
\begin{align*}
x_A^r (\log x_A)^k \Pi_{e\in E(A)}\zeta_e^{r_e}(\log \zeta_e)^{k_e}F
\end{align*}
with $F\in\C[[\zeta_e\mid e \in E(A)]]$,
$k,k_e \in \Z_{\geq 0}$ and $r_e,r \in \C$ ($e\in E(A)$).

\begin{rem}
\label{rem_TA_sub}
By $\ze_e=\frac{x_{d(e)}}{x_{u(e)}}$,
$T_A$ is a subspace of 
\begin{align*}
\C[[x_v^\pm\mid v\in V(A)]][x_v^\C,\log x_v \mid v\in V(A)].
\end{align*}
\end{rem}


We remark that for $A=(23)((15)4)$
\begin{align*}
\log(z_3-z_5)&=\log((z_3-z_4)-(z_5-z_4))\\
&=\log(z_3-z_4)+\log(1-\frac{z_5-z_4}{z_3-z_4})\\
&=\log(x_{(23)((15)4)}) - \sum_{k \geq 1}\frac{1}{k}\zeta_c^k \in T_{(23)((15)4)}.
\end{align*}
Thus $\log(z_3-z_5)$ and so on are expressed as an elements in $T_{(23)((15)4)}$.
The series in $T_A$ is called {\it a parenthesized formal power series}.
It is noteworthy that $T_A$ naturally inherits a $\C$-algebra structure, and we have:
\begin{lem}
\label{lem_module_ort}
$e_A:\Ort^\alg \rightarrow T_A$ is a $\C$-algebra homomorphism.
In particular, $T_A$ is an $\Ort^\alg$-module.
\end{lem}

Next, consider the radius of convergence of parenthesized formal power series.
For $p>0$, set
\begin{align*}
\D_p &= \{\ze\in \C\mid |\ze|<p\},\\
\D_p^\times &=\{\ze \in \C\mid 0<|\ze|<p\}.
\end{align*}

Let $\mathfrak{p}=(p_e)_{e\in E(A)}\in \R_{>0}^{E(A)}$.
Let $\C[[\zeta_e\mid e\in E(A)]]_{\mathfrak{p}}^\conv$
be a subspace of $\C[[\zeta_e\mid e\in E(A)]]$
consisting of formal power series which is absolutely convergent in $\Pi_{e\in E(A)}\D_{p_e}$
and set
\begin{align*}
T_A^{\mathfrak{p}}=\C[[\zeta_e\mid e\in E(A)]]_{\mathfrak{p}}^\conv[x_v^\C,\log x_v\mid v\in V(A)],
\end{align*}
a subspace of $T_A$ spanned by
the finite sum of formal power series of the form:
\begin{align*}
\Pi_{v\in V(A)}x_v^{r_v}(\log x_v)^{k_v}F
\end{align*}
with $x_v\in \C$, $k_v\in\Z_{\geq 0}$
and $F\in \C\{\zeta_e\mid e \in E(A)\}_{\mathfrak{p}}$.
It is important to note that the region of absolute convergence of 
$e_{(23)((15)4)}((z_2-z_1)^{-1})$ in the example  \eqref{eq_example_conv}
is not $|\zeta_a|<1,|\zeta_b|<1,|\zeta_c|<1$.
Since 
$e_{(23)((15)4)}((z_2-z_1)^{-1})=x_{{(23)((15)4)}}^{-1}\sum_{l=0}^\infty (-\zeta_a+\zeta_c+\zeta_b)^l$,
if $p_a+p_b+p_c <1$, then $e_{(23)((15)4)}\left((z_2-z_1)^{-1}\right) \in T_{(23)((15)4)}^{\mathfrak{p}}$.


\begin{dfn}
A sequence of positive real numbers $(p_e)_{e\in E(A)} \in \R_{>0}^{E(A)}$ is called {\it $A$-admissible} if $\Psi_A^{-1}(\C\times \C^\times \times \Pi_{e\in E(A)}\D_{p_e}^\times) \subset X_r$,
where $\Psi_A^{-1}$ is the polynomials in Proposition \ref{prop_psiA}.
\end{dfn}
\begin{rem}
\label{rem_pole_p}
For example, for $A=(23)((15)4)$, $\{z_{ij}=z_i-z_j\}_{1\geq i<j\geq 5}$ in $A$-coordinate are
\begin{align*}
(z_{12},z_{13},z_{14},z_{15},z_{23})
&=\Bigl(x_A(\ze_c+\ze_b\ze_c-1-\ze_a),x_A(\ze_c+\ze_b\ze_c-1),x_A\ze_c(1+\ze_b),x_A\ze_b\ze_c,\ze_a x_A\Bigr)\\
(z_{24},z_{25},z_{34},z_{35},z_{45})
&=\Bigl(
x_A(1+\ze_a),x_A(1+\ze_a-\ze_c),x_A,x_A(1-\ze_c),-x_A \ze_c\Bigr).
\end{align*}
Therefore, for the image of $\Psi_A^{-1}$ to be in $\Xr$, the following holds:
\begin{align*}
&\ze_a,\ze_b,\ze_c,x_A \neq 0,\quad 1+\ze_a,1-\ze_c,1+\ze_b\neq 0\\
&1+\ze_a-\ze_c,1-\ze_c-\ze_b\ze_c,1+\ze_a-\ze_c-\ze_b\ze_c \neq 0.
\end{align*}
Thus, $\mathfrak{p}$ is $(23)((15)4)$-admissible if and only if
\begin{align*}
p_a,p_b,p_c<1,\\
p_a+p_c,p_c+p_c p_b,p_a+p_c+p_cp_b<1.
\end{align*}
In particular, $\mathfrak{p}$ is $(23)((15)4)$-admissible if $p_a,p_b,p_b$ are sufficiently small.

More to the point, let $r$ be $0<r<1$
and let one of $p_a,p_b,p_c$ be $r$.
Then, if the remaining two are sufficiently small, then $\mathfrak{p}$ is $(23)((15)4)$-admissible.
\end{rem}

Then, we have:
\begin{lem}
\label{lem_admissible_radius}
Let $A \in \Tr_r$ and $\mathfrak{p}=(p_e)_{e\in E(A)} \in \R_{>0}^{E(A)}$ be $A$-admissible.
Then, $e_A(f) \in T_A^{\mathfrak{p}}$ for any $f \in \Ort^\alg$.
\end{lem}
\begin{proof}
Note that $f$ can have poles only along $\{z_i=z_j\}_{i\neq j}$.
Since $\mathfrak{p}$ is A-admissible, $f \circ \Psi^{-1}$ is a holomorphic function on $\C \times \C^\times \Pi_{e\in E(A)}\D_{p_e}^\times$ with possible poles at $\zeta_e=0$ ($e\in E(A)$). Hence, the assertion holds.
\end{proof} 
The following lemma is easy to show (see the last sentence of Remark \ref{rem_pole_p}):
\begin{lem}
\label{lem_one_fix}
Let $\ee\in E(A)$ and $r \in \R$ with $0<r<1$.
Then, there exists $\{p_e\}_{e\in E(A)} \in \R_{>0}^{E(A)}$ such that:
\begin{enumerate}
\item
$p_\ee=r$;
\item
$\{p_e\}_{e\in E(A)}$ is A-admissible.
\end{enumerate}
\end{lem}

A convergent series $f\in T_A^{\mathfrak{p}}$ is a multi-valued holomorphic function on $\Psi_A^{-1}(\C \times \C^\times \times \Pi_{e\in E(A)}\D_{p_e}^\times)$ 
because it contains $\log(x_v)$ and $x_v^r$.
Below, we will fix the branch.
For $A$-admissible numbers $\mathfrak{p}$,
set
\begin{align*}
U_{A}^{\mathfrak{p}}=\Psi_A^{-1}(\C \times \C^\cut \times \Pi_{e\in E(A)}\D_{p_e}^\cut),
\end{align*}
where 
\begin{align*}
\R_-&= \{r\in \R\mid r\leq 0\},\\
\C^\cut &= \C \setminus \R_-,\\
\D_p^\cut &=\{\ze \in \C \setminus \R_- \mid |\ze|<p \}.
\end{align*}

Define the branch of $\Log:\C^\cut \rightarrow \C$ by
\begin{align}
\Log(\exp(\pi i t))= \pi i t
\label{eq_log_def}
\end{align}
for $t\in (-1,1)$. In particular, $\Arg= \mathrm{Im}\,\Log$ takes the values in $(-\pi,\pi)$.

Then, each formal power series in $T_A^{\mathfrak{p}}$
can be regarded as a single-valued holomorphic function on
$U_{A}^\mathfrak{p}$.
Set
\begin{align*}
U_A  = \cup_{\mathfrak{p}:A\text{-admissible}}U_{A}^{\mathfrak{p}} \subset \Xr.
\end{align*}
The following lemma is clear from the definition:
\begin{lem}
\label{lem_simply_conn}
For A-admissible number $\mathfrak{p}\in \R_{>0}^{E(A)}$,
$U_{A}^{\mathfrak{p}}$ and $U_A$ are connected simply-connected open subsets of $\Xr$.
\end{lem}

Set
\begin{align*}
T_A^\conv = \cap_{\mathfrak{p}:A\text{-admissible}} T_A^\mathfrak{p} \subset T_A,
\end{align*}
which is a linear space of convergent parenthesized formal power series.
Then, by Lemma \ref{lem_module_ort} and Lemma \ref{lem_admissible_radius},
we have a $\C$-algebra homomorphism
\begin{align}
e_A: \Ort^\alg \rightarrow T_A^\conv.
\label{eq_eA_conv}
\end{align}
In the next section, we will define and study D-module structures on $T_A$ and $T_A^\conv$.

\subsection{D-module structure}
\label{sec_config_D}
Set
\begin{align*}
\pa_i=\frac{d}{dz_i},
\end{align*}
the partial differential operator  on $\Xr$ with respect to $z$-coordinate,
and
\begin{align*}
\Dr = \C[\pa_1,\dots,\pa_r, z_1,\dots,z_r,(z_i-z_j)^\pm\mid 1\leq i <j\leq r],
\end{align*}
a ring of differential operators on $\Xr$,
and
\begin{align*}
\Drt=\C[\pa_1,\dots,\pa_r, (z_i-z_j)^\pm\mid 1\leq i <j\leq r],
\end{align*}
which is a subalgebra of $\Dr$.

In this section we will study a $\Drt$-module structure on $T_A$ and $T_A^\conv$.
Note that differential operators $\pa_i$ naturally acts on $\C[[x_v^\pm\mid v\in V(A)]][x_v^\C,\log x_v \mid v\in V(A)]$
and by Remark \ref{rem_TA_sub},
$T_A$ is a subspace of 
$\C[[x_v^\pm\mid v\in V(A)]][x_v^\C,\log x_v \mid v\in V(A)]$.
It is easy to show that $T_A$ and $T_A^\conv$ are closed under the action of $\{\pa_i\}_{i=1,\dots,r}$.
Hence, by \ref{eq_eA_conv}, $T_A$ and $T_A^\conv$ are $\Drt$ modules and 
\begin{prop}
$e_A:\Ort\rightarrow T_A^\conv$ is a $\Drt$-module homomorphism.
\end{prop}

More specifically, for example, for $A=(23)((15)4)$, by \eqref{eq_example_psi},
\begin{align*}
\pa_3 \zeta_a &= \pa_3 (\frac{z_2-z_3}{z_3-z_4})=\frac{z_4-z_2}{(z_3-z_4)^2}=-x_{(23)((15)4)}^{-1}(1+\zeta_a).
\end{align*}

Note that $\{\pa_i\}_{i=1,\dots,r}$ are the partial derivatives in $z$-coordinate on $\Xr$.
We can also define formal $\zeta$-derivations with respect to $A$-coordinate.

For $A=(23)((15)4)$, by \eqref{eq_example_psi},
\begin{align}
\begin{split}
\frac{d}{d\zeta_b} z_1 &= \frac{d}{d\zeta_b} ((z_1-z_5) + (z_5-z_4)+z_4)\\
 &= \frac{d}{d\zeta_b} (z_5-z_4) (\frac{z_1-z_5}{z_5-z_4} +1+\frac{z_3-z_4}{z_5-z_4}\frac{z_4}{z_3-z_4})\\
&=\frac{d}{d\zeta_b} \zeta_c^{-1}x_{A}(1+\zeta_b +\zeta_c \frac{z_A}{x_A})=\zeta_c^{-1}x_A=z_5-z_4.
\label{eq_nandemo_ex}
\end{split}
\end{align}
Let $\Leaf(e)$ be the set of all leaves which are descendants of $d(e)$ for an edge $e\in E(A)$.
The next proposition is significant in showing a coherence of $\Drt$-modules near the singularity:
\begin{lem}
\label{lem_change_of_derivation}
For any $e\in E(A)$,
\begin{align*}
\zeta_e \frac{d}{d\zeta_e} = \sum_{l \in \Leaf(e)}(z_l - z_{R(d(e))}) \pa_l,
\end{align*}
as differential operators on $\Xr$.
\end{lem}
(For example, $\zeta_b \frac{d}{d\zeta_b} = (z_1-z_5)\pa_1$,
which is compatible with \eqref{eq_nandemo_ex})
\begin{proof}
Set $D_e = \sum_{l \in \Leaf(e)}(z_l - z_{R(d(e))}) \pa_l$.
By considering the action of the differential operator $D_e$ on $\{\zeta_{e'}\}_{e'\in E(A)}$ and $z_A,x_A$, we show that the assertion holds.
Set
$$B(e)=\{e'\in E(A)\mid e' \text{is located below }e\} \cup \{e\}.$$
For any $e' \in E(A)\setminus B(e)$, 
\begin{align*}
\zeta_{e'}=\frac{z_{L(d(e'))}-z_{R(d(e'))}}{z_{L(u(e'))}-z_{R(u(e'))}}
\end{align*}
and
$$\{L(u(e')),R(u(e')),L(d(e')),R(d(e'))\} \cap \Leaf(e) \subset \{R(d(e))\}.$$
Since the coefficient of $\pa_{R(u(e))}$ on $D_e$ is $0$,
$D_e \zeta_{e'}=0$ for any $e'\in E(A)\setminus B(e)$.
Similarly, $D_e z_A=D_e x_A=0$.

Next consider the case of $e' \in B(e)$ and $e'\neq e$.

On the variable $\{z_i\}_{i \in \Leaf(e)}$, $\sum_{l \in \Leaf(e)}z_l\pa_l$ is a scale transformation, and
$\sum_{l \in \Leaf(e)}\pa_l$ corresponds to a translation.
Since $\zeta_{e'}$ is scale-invariant and translation-invariant
and 
\begin{align*}
\{L(u(e')),R(u(e')),L(d(e')),R(d(e'))\} \subset \Leaf(e),
\end{align*}
we have $D_e \zeta_{e'}=0$.

%
Finally, we show that $D_e \zeta_e=\zeta_e$.
Assume that the tree $A$ has the shape shown in Fig. \ref{fig_ij_up}.

\begin{minipage}[c]{.5\textwidth}
\centering
\begin{forest}
for tree={
  l sep=20pt,
  parent anchor=south,
  align=center
}
[[
$z_j-z_p$
[$z_i-z_j$,edge label={node[midway,left]{e}}[$\cdots i$][$\cdots j$]]
[$\cdots p$]
]]
\end{forest}
\captionof{figure}{}
\label{fig_ij_up}
\end{minipage}
\begin{minipage}[c]{.5\textwidth}
\centering
\begin{forest}
for tree={
  l sep=20pt,
  parent anchor=south,
  align=center
}
[[
$z_p-z_j$
[$\cdots p$]
[$z_i-z_j$,edge label={node[midway,left]{e}}[$\cdots i$][$\cdots j$]]
]]
\end{forest}
\captionof{figure}{}
\label{fig_ij_down}
\end{minipage}
Then, $\zeta_e=\frac{z_i-z_j}{z_j-z_p}$
and 
\begin{align*}
D_e \frac{z_i-z_j}{z_j-z_p}&=\left(\sum_{l\neq p} (z_l - z_{j}) \pa_l\right)\frac{z_i-z_j}{z_j-z_p}
 \\
&= \left(-z_p\pa_p+z_j\pa_p \right) \frac{z_i-z_j}{z_j-z_p}= \frac{z_i-z_j}{z_j-z_p}=\zeta_e.
\end{align*}
The same is true when the tree $A$ is Fig. \ref{fig_ij_down}. Thus, the assertion holds.
\end{proof}

\subsection{Degree map}
\label{sec_config_deg}
In this section, for each $A\in \Tr_r$, we introduce a degree map
\begin{align*}
\deg_A^\ee:\Ort \rightarrow \Z\cup \{\infty\}
\end{align*}
which measures the order of the poles in $\{z_i=z_j\}$ and studies its properties.

Let $A\in \Tr_r$.
By the $\C$-algebra homomorphism $e_A:\Ort\rightarrow \C((\zeta_e \mid e\in E(A)))[x_A^\pm]$ \eqref{eq_emap}, 
$\Ort$ is embedded in a ring of formal power series.
Let $\ee \in E(A)$ and $f\in \C((\zeta_e \mid e\in E(A)))[x_A^\pm]$.
Define $\deg_A^\ee(f)$ by the smallest integer $n$ for which the coefficient of $\zeta_\ee^n$ is non-zero.
Here, if $f$ is zero, let $\deg_A^\ee(0)=\infty$.
This determines the map $\deg_A^\ee: \C((\zeta_e \mid e\in E(A)))[x_A^\pm]\rightarrow \Z\cup \{\infty\}$.
Denote the composite map
\begin{align*}
\Ort \overset{e_A}{\rightarrow} \C((\zeta_e \mid e\in E(A)))[x_A^\pm] \overset{\deg_A^\ee}{\rightarrow} \Z\cup \{\infty\}
\end{align*}
by the same symbol $\deg_A^\ee$

The following lemma is easy to show:
\begin{lem}
\label{lem_deg_formula}
For $1\leq i <\leq j$,
\begin{align*}
\deg_A^\ee(z_i-z_j)=
\begin{cases}
&1\quad \text{ if }i,j \in \Leaf(\ee),\\
&0 \quad \text{otherwise}.
\end{cases}
\end{align*}
\end{lem}

The following lemma will be used in Section \ref{sec_parenthesized_comp}:
\begin{lem}
\label{lem_well_inverse}
Let $A \in \Tr_r$ and $e_0\in E(A)$. Let $r$ be the rightmost leaf among $\Leaf(e_0)$
and $s \in \Leaf(e_0)$ with $s\neq r$ and $i \in \Leaf(A) \setminus \Leaf(e_0)$.
Then, $e_A\left(\frac{z_s-z_r}{z_r-z_i}\right)\in \ze_{e_0}\C[[\ze_e\mid e\in E(A)]]$
and for any $n\in \Z$ and $l \geq 0$,
\begin{align}
\sum_{k\geq 0}\binom{n}{k}\binom{n-k}{l} e_A\left(\frac{z_s-z_r}{z_r-z_i}\right)^k&=
\binom{n}{l} e_A\left(\left(\frac{z_s-z_i}{z_r-z_i}\right)^{n-l} \right), \label{eq_well_sum_inv}\\
\sum_{k \geq 0}\binom{n}{k}\binom{k}{l}e_A\left(\frac{z_s-z_r}{z_r-z_i}\right)^k &=
\binom{n}{l} e_A\left(\left(\frac{z_s-z_r}{z_r-z_i} \right)^l\right)e_A\left(\left(\frac{z_s-z_i}{z_r-z_i} \right)^{n-l}\right).
\label{eq_well_sum_inv2}
\end{align}
\end{lem}

\begin{rem}
\label{rem_ill_exp}
Before showing the above lemma, a few remarks are in order.
It is easy to show that as formal power series
\begin{align*}
\sum_{k\geq 0}\binom{n}{k}\binom{n-k}{l} x^k = \binom{n}{l}\sum_{k\geq 0}\binom{n-l}{k}x^k \in \C[[x]].
\end{align*}
Assume that $n \geq 0$.
Then, the right-hand-side of \eqref{eq_well_sum_inv} is a finite sum.
Since $e_A$ is a ring homomorphism,
\begin{align}
\begin{split}
\sum_{k\geq 0}\binom{n}{k}\binom{n-k}{l} e_A\left(\frac{z_s-z_r}{z_r-z_i}\right)^k
&=e_A\left(\sum_{k\geq 0}\binom{n}{k}\binom{n-k}{l} \left(\frac{z_s-z_r}{z_r-z_i}\right)^k\right)\\
&=\binom{n}{l}
e_A\left(
\left(1+\frac{z_s-z_r}{z_r-z_i}\right)^{n-l}
\right)\\
&=\binom{n}{l}
e_A\left(
\left(\frac{z_s-z_i}{z_r-z_i}\right)^{n-l}
\right).
\label{eq_well_first}
\end{split}
\end{align}
If $n <0$, then the first equality of \eqref{eq_well_first} is not always true since the sum is infinite.
In fact, that the right-hand-side of \eqref{eq_well_sum_inv} is in $T_A$ follows from $e_A\left(\frac{z_r-z_i}{z_i-z_j}\right)\in \ze_{e_0}\C[[\ze_e\mid e\in E(A)]]$.
This does not hold for general $i,j,r$ as below.

\begin{minipage}[c]{.5\textwidth}
\centering
\begin{forest}
for tree={
  l sep=20pt,
  parent anchor=south,
  align=center
}
[$z_2-z_4$
[$z_1-z_2$,edge label={node[midway,left]{a}}[1][2]]
[$z_3-z_4$,edge label={node[midway,right]{b}}[3][4]]
]
\end{forest}
\captionof{figure}{}
\label{fig_no_ref}
\end{minipage}
\begin{minipage}[c]{.5\textwidth}
Since
\begin{align*}
e_{(12)(34)}\left(\frac{z_1-z_2}{z_3-z_4}\right)
&=\frac{\ze_a}{\ze_b} \notin \C[[\ze_a,\ze_b]]\\
e_{(12)(34)}\left(\frac{z_1-z_4}{z_1-z_2}\right)
&=\frac{1+\ze_a}{\ze_a} \notin \C[[\ze_a,\ze_b]],
\end{align*}
Lemma \ref{lem_well_inverse} does not hold.
\end{minipage}
\end{rem}

\begin{proof}[proof of Lemma \ref{lem_well_inverse})]
Since $s,r \in \Leaf(e_0)$ and $i \notin \Leaf(e_0)$, by Lemma \ref{lem_deg_formula}, $\deg_A^{e_0}(z_s-z_r)=1$ and $\deg_A^{e_0}(z_s-z_r)=0$.
Let $e\in E(A)$ satisfy $\deg_A^e(z_r-z_i)=1$. Then, it is clear that $\deg_A^e(z_s-z_r)=1$.
Hence, $e_A\left(\frac{z_s-z_r}{z_r-z_i}\right)\in \ze_{e_0}\C[[\ze_e\mid e\in E(A)]]$.

Since $\binom{n}{k}\binom{n-k}{l}=\frac{n(n-1)\cdots (n-k-l+1)}{k!l!}=\binom{n}{l}\binom{n-l}{k}$,
\begin{align}
\sum_{k,l\geq 0}\binom{n}{k}\binom{n-k}{l} x^k&=\binom{n}{l}\sum_{k\geq 0}\binom{n-l}{k} x^k.
\label{eq_x_well}
\end{align}
The formal power series \eqref{eq_x_well} converges absolutely and uniformly to $\binom{n}{l}(1+x)^{n-l}$
in the region $\{x\in \C\mid|x|<1\}$.
Hence, the right-hand-side of \eqref{eq_well_sum_inv} converges absolutely and uniformly
to $\binom{n}{l}\left(1+\frac{z_s-z_r}{z_r-z_i}\right)^{n-l}=\binom{n}{l}\left(\frac{z_s-z_i}{z_r-z_i}\right)^{n-l}$
if all $|\ze_e|$ $e\in E(A)$ are sufficiently small.
Hence, it coincides with $\binom{n}{l}
e_A\left(
\left(\frac{z_s-z_i}{z_r-z_i}\right)^{n-l}
\right)$ as formal power series.
Similarly, by Lemma \ref{identity}, \eqref{eq_well_sum_inv2} converges absolutely and uniformly
to 
\begin{align*}
\binom{n}{l}\left(\frac{z_s-z_r}{z_r-z_i} \right)^l\left(1+\frac{z_s-z_r}{z_r-z_i}\right)^{n-l}=
\binom{n}{l}\left(\frac{z_s-z_r}{z_r-z_i} \right)^l\left(\frac{z_s-z_i}{z_r-z_i}\right)^{n-l}
\end{align*}
if if all $|\ze_e|$ $e\in E(A)$ are sufficiently small. Hence, the assertion holds.
\end{proof}

%
%

%

\section{Global and formal conformal blocks}
\label{sec_D}
Let $V$ be a vertex operator algebra, $r\in \Z_{>0}$,
and $\{M_i\}_{i =0,1,\dots,r}$ $V$-modules.
Set 
\begin{align*}
\Mr=M_0^\vee \otimes M_1 \otimes \cdots \otimes M_r.
\end{align*}
In this section we introduce a $\Dr$-module $D_{\Mr}$ for $\{M_i\}_{i =0,1,\dots,r}$ (see also \cite{FB,NT}).
The analytic solution sheaf of this D-module is {\it the conformal block} in physics.
In section \ref{sec_D_coherent}, by refining the idea of Huang \cite{H2},
we will prove that $D_{\Mr}$ is a holonomic D-module on $\Xr$ and the analytic solution sheaf $\mathrm{Hom}_{\Dr}(D_{M_{[r]}},\Or^\an)$ is a locally constant sheaf,
where $\Or^\an$ is a sheaf of holomorphic functions.


Section \ref{sec_D_translation} shows that the conformal block can be modified to be
invariant under translations.
In section \ref{sec_D_local_coherent}, we prove a deeper result, a coherence of $D_{\Mr}$ on the singular sets $\{z_i=z_j\}$.

It follows from this that for any tree $A\in \Tr_r$, analytic solutions $\mathrm{Hom}(D_{\Mr},\Or^\an(U_A))$ on 
the simply-connected domain $U_A\subset \Xr$ has expansions as parenthesized formal power series, i.e., we obtain an isomorphic map (Theorem \ref{thm_expansion}):
\begin{align*}
e_A:\mathrm{Hom}(D_\Mr,\Or^\an(U_A)) \rightarrow \mathrm{Hom}(D_\Mr,T_A^\conv).
\end{align*}
The existence of this series expansion is the main result of this section and will play an important role when defining the operadic compositions in the next section.
In section \ref{sec_D_rotation}, we will discuss a rotation-scale symmetry of the conformal block.


\subsection{Definition of conformal block}
\label{sec_D_def}
For each $a \in V$, $i =1,\dots,r $ and $n\in \Z$,
define a linear map $a(n)_i:\Mr\rightarrow M_{[r]}$ by
the action of $a(n):M_i \rightarrow M_i$ on the $i$-th component.

On the $0$-th component, we consider two actions.
Define $a(n)_0:\Mr\rightarrow \Mr$ by
\eqref{eq_dual}, the dual module structure on $M_0^\vee$,
and
$a(n)_0^*:\Mr\rightarrow \Mr$ by
$a(n)^*:M_0^\vee \rightarrow M_0^\vee$ on the $0$-th component,
where 
\begin{align*}
(a(n)^*u)(\bullet)=u(a(n)\bullet) \text{ for } u\in M_0^\vee.
\end{align*}
\begin{rem}
\label{rem_dual}
Note that for $a,b\in V$, $n,m\in \Z$, $a(n)_0^*,b(m)_0^* \in \End M_0^\vee$ are actions on the dual vector space, thus $\left(b(m)a(n)\right)_0^*=a(n)_0^*b(m)_0^*$, that is,
\begin{align}
(a(n)_0^*b(m)_0^*f)(\bullet)=f(b(m)_0a(n)_0\bullet)\quad\text{ for }f\in M_0^\vee.
\label{eq_rem_dual}
\end{align}
\end{rem}


Let $\Or^\alg \otimes \Mr$ be an $\Or^\alg$-module,
where the $\Or^\alg$-module structure is defined by the multiplication on the left component.
Define a $\Dr$-module structure on $\Or^\alg \otimes \Mr$
by
\begin{align}
\pa_i \cdot (f\otimes \mr) = (\pa_i f) \otimes \mr + f\otimes L(-1)_i \mr
\label{eq_D_def}
\end{align}
for $i=1,\dots,r$, $f \in \Or^\alg$, $\mr \in \Mr$.


Let $N_{\Mr}$ be the $\Dr$-submodule of 
$\Or^\alg \otimes \Mr$ generated by the following elements:
\begin{align}
1 \otimes a(n)_i \mr - &\sum_{k \geq 0}\binom{n}{k} (-z_i)^k \otimes a(-k+n)_0^* \mr
+ \sum_{1\leq s\leq r, s \neq i} \sum_{k \geq 0} \binom{n}{k}(z_s-z_i)^{n-k}\otimes a(k)_s \mr,
\label{eq_ker1}
\end{align}
for all $\mr \in \Mr$, $a \in V$ and $i \in \{1,\dots,r\}$ and $n \in \Z$.
We note that the all above sums are finite
and $z_s^l, (z_p-z_q)^{k}$ is in $\Or^\alg$ for any $l \geq 0$ and $k\in \Z$.
Thus, the above definition is well-defined.

The following lemma is useful:
\begin{lem}
\label{N_contain}
Let $\mr \in \Mr$, $i \in \{1,\dots,r\}$ $a \in V$ and $\Delta \in \Z_{>0}$,
$b \in V_\Delta$, $n\in \Z_{\geq 0}$.
The following elements are contained in $N_{\Mr}$:
\begin{align*}
&1 \otimes a(-1)_i \mr - \sum_{k \geq 0}z_i^k \otimes a(-k-1)_0^* \mr
 - \sum_{s\in [r], s \neq i} \sum_{k \geq 0} (z_i-z_s)^{-k-1}\otimes a(k)_s \mr,\\
&1\otimes a(n)_0^* \mr - \sum_{s \in [r]} \sum_{k \geq 0} \binom{n}{k}z_s^{n-k}\otimes a(k)_s \cdot \mr,\\
&1\otimes b(-1)_0 \mr - \sum_{s\in [r]} \sum_{k,l \geq 0} 
\frac{(-1)^\Delta}{l!}\binom{2\Delta-l-1}{k}
z_s^{2\Delta-l-1-k} \otimes \Bigl(L(1)^l b\Bigr)(k)_s \mr.
\end{align*}
\end{lem}
\begin{proof}
The first equality is obtained by substituting $n=-1$ into \eqref{eq_ker1}.
By \eqref{eq_ker1}, for any $i \in \{1,\dots,r\}$ and $n\geq 0$, we have
\begin{align*}
&\sum_{k \geq 0} \binom{n}{k} 
z_i^{n-k} \otimes a(k)_i \cdot \mr \\
&=\sum_{k,l \geq 0} \binom{n}{k} \binom{k}{l}
(-1)^l z_i^{n-k+l} \otimes a(-l+k)_0^* \mr
- \sum_{s \in [r], s \neq i} \sum_{k, l \geq 0}\binom{n}{k}
\binom{k}{l}z_i^{n-k}(z_s-z_i)^{k-l} \otimes a(l)_s\cdot \mr.
\end{align*}
By setting $l'=k-l$ and by Lemma \ref{identity}, we have
\begin{align*}
\sum_{k,l \geq 0} \binom{n}{k} \binom{k}{l}
(-1)^l z_i^{n-k+l} \otimes a(-l+k)_0^* \mr
&=
\sum_{k,l \geq 0} \binom{n}{k} \binom{k}{k-l}
(-1)^l z_i^{n-k+l} \otimes a(-l+k)_0^* \mr\\
&=
\sum_{k,l' \geq 0} \binom{n}{k} \binom{k}{l'}
(-1)^{k+l'} z_i^{n-l'} \otimes a(l')_0^* \mr\\
&=\sum_{l' \geq 0}(-1)^{l'}z_i^{n-l'} 
 \left(\sum_{k\geq 0} \binom{n}{k} \binom{k}{l'}(-1)^{k}  \right)\otimes a(l')_0^* \mr\\
&=a(n)_0^*\mr
\end{align*}
and similarly
\begin{align*}
\sum_{k, l \geq 0}\binom{n}{k}
\binom{k}{l}z_i^{n-k}(z_s-z_i)^{k-l} \otimes a(l)_s\cdot \mr
&=\sum_{l \geq 0}\binom{n}{l}
z_s^{n-l} \otimes a(l)_s\cdot \mr.
\end{align*}

Thus, the second element is in $N_{\Mr}$.
By \eqref{eq_dual2}, we have
\begin{align*}
b(-1)_0
&=\sum_{l \geq 0} \frac{(-1)^\Delta}{l!}\Bigl(L(1)^l b\Bigr)(2\Delta-l-1)_0^*.
\end{align*}
Since $V$ is a VOA of CFT type, $l$ runs through $l=0,\dots,\Delta$.
Hence, $2\Delta -l-1 \geq 0$. Thus, the third element is obtained from the second element.
\end{proof}

Set 
$$
D_{\Mr} = (\Or^\alg \otimes \Mr) / N_{\Mr},
$$
which is a $\Dr$-module.

\begin{rem}
\label{remark_D_KZ}
Let $\g$ be a finite-dimensional simple Lie algebra,
and $V_{\g,k}$ be the affine vertex operator algebra at level $k \in \C \setminus \Q$,
$M_0,M_1,\dots,M_r$ be Weyl modules which are the induced modules from finite-dimensional $\g$-modules.
Then, $D_\Mr$ is the Knizhnik--Zamolodchikov equations.
For another example, if $V$ is a Virasoro minimal model, then the corresponding D-module is the BPZ equation \cite{BPZ}.
\end{rem}
The following lemma is clear:
\begin{lem}
\label{lem_D_functor}
The assignment of $\Mr$ to $D_{\Mr}$ determines the following $\C$-linear functor:
\begin{align*}
D_{\bullet}:{\Vmodu}^\op \times {\Vmodu}^r &\rightarrow \underline{\Dr \text{-mod}},\\
(M_0,M_1,\dots,M_r) \quad &\mapsto \quad D_{\Mr}.
\end{align*}
\end{lem}

Let $\Or^\an$ be the sheaf of holomorphic functions on $\Xr$.
Denote the solution sheaf \\
$\mathrm{Hom}_{\Dr}(D_{\Mr},\Or^\an)$ by $\CB_{\Mr}$,
which is called {\it a conformal block} associated with $\Mr$.
Let $U\subset \Xr$ be an open subset and $C\in \CB_\Mr(U)$.
Let $\Mr\rightarrow \Or^\alg \otimes \Mr$ be the embedding defined by
$\mr\mapsto 1\otimes \mr$ for $\mr\in \Mr$.
Then, by $\Mr\rightarrow \Or^\alg \otimes \Mr \rightarrow D_{\Mr}$,
$C$ can be regarded as a linear map $\Mr \rightarrow \Or^\an$.
The following lemma characterizes a section of $\CB_\Mr(U)$ as a linear map $\Mr \rightarrow \Or^\an$:
\begin{lem}
\label{lem_block_linear}
A linear map
$$
C:\Mr \rightarrow \Or^\an(U),
$$
defines a section of $\CB_\Mr(U)$ if and only if it satisfies the following conditions:
\begin{enumerate}
\item[CB1)]
For any $i \in \{1,2,\dots,r\}$, $\mr\in \Mr$,
$$
C(L(-1)_i \cdot \mr)= \frac{d}{dz_i} C(\mr).
$$
\item[CB2)]
For any $i \in \{1,2,\dots,r\}$, $\mr\in \Mr$,
and $a \in V$, $n \in \Z$,
\begin{align*}
C(a(n)_i  \mr)
&=\sum_{k \geq 0}\binom{n}{k} (-z_i)^k C(a(-k+n)_0 \mr)
- \sum_{s \in [r], s \neq i} \sum_{k \geq 0}
\binom{n}{k}(z_s-z_i)^{n-k} C(u,a(k)_s \mr).
\end{align*}
\end{enumerate}
\end{lem}
\begin{proof}
The only if part is clear. Let $C:\Mr\rightarrow \Or^\an$ be a linear map satisfying (CB1).
Let $\tilde{C}: \Or \otimes \Mr \rightarrow \Or^\an(U)$ be the linear extension of $C$ as an $\Or$-module.
Then, $\tilde{C}$ is a $\Dr$-module homomorphism since
for $f\in \Or$ and $\mr\in \Mr$
$\pa_i \tilde{C}(f\otimes \mr)=\pa_i (f {C}(\mr))
=(\pa_if) {C}(\mr)+ f \pa_i{C}(\mr)
=(\pa_if) {C}(\mr)+ f {C}(L(-1)_i \mr)
=\tilde{C}(\pa_i f\otimes \mr)+ \tilde{C}(f \otimes L(-1)_i \mr)
=\tilde{C}(\pa_i \cdot (f\otimes \mr))$.
Hence, the assertion holds.
\end{proof}

%
%


\subsection{Coherence of conformal block}
\label{sec_D_coherent}
The following proposition is proved in \cite{H2} under the assumption that $M_0^\vee$ is $C_1$-cofinite and $r=3$.
We refine the proof but the idea is essentially due to Huang:
\begin{prop}
\label{prop_coherence}
Let $M_0,M_1,\dots,M_r\in \Vmodu$.
Assume that $M_1,\dots,M_r$ are $C_1$-cofinite $V$-modules and $M_0^\vee$ is a finitely generated $V$-module.
Then, $D_{\Mr}$ is a finitely generated $\Or^\alg$-module.
\end{prop}
\begin{proof}
By (M4), any module in $\Vmodu$ is a finite direct sum of modules of the form:
\begin{align*}
N=\bigoplus_{k \geq 0}N_{\Delta+k}
\end{align*}
for some $\Delta \in \C$.
Since the functor $D:\Vmodu^\op\times \Vmodu^r\rightarrow \underline{\Dr \text{-mod}}$
commutes with a finite direct sum,
we may assume that $M_0,M_1,\dots,M_r$ satisfy
\begin{align*}
M_i=\bigoplus_{k \geq 0}(M_i)_{\Delta_i+k}\;\;\;\text{ and }\;\;\;
M_0=\bigoplus_{k \geq 0}(M_0)_{\Delta'+k}
\end{align*}
for some $\Delta_1,\dots,\Delta_r,\Delta' \in \C$.

Define a $\Z_{\geq 0}^2$-grading on $\Mr$ by
\begin{align}
(\Mr)_{S,T}=\bigoplus_{\substack{k_1,\dots,k_r\geq 0\\k_0+k_1+\dots+k_r=T}}
(M_0^\vee)_{\Delta' + S}\otimes \bigotimes_{s=1}^r (M_s)_{\Delta_s+k_s}.
\label{eq_Mr_T}
\end{align}

By the assumption, there exists $N_0,N_1,\dots,N_r \in \Z_{\geq 0}$ such that:
$M_0^\vee$ is generated by $\bigoplus_{N_0 \geq k_0 \geq 0} (M_0^\vee)_{\Delta'+k_0}$ as a $V$-module
and $M_s = C_1(M_s)+\bigoplus_{N_s \geq n \geq 0} (M_s)_{\Delta_s+n}$
for any $s \in [r]$. Set
\begin{align}
N=N_1+\dots+N_r.
\label{eq_N_sum}
\end{align}

Let $W$ be a $\Or^\alg$-module generated by 
$\bigoplus_{\substack{N_0 \geq s\geq 0\\ N \geq t \geq 0 }}
 (\Mr)_{s,t}$.
We will show that 
\begin{align}
\Or^\alg \otimes \Mr_{S,T} \subset  W + N_\Mr\label{eq_induction}
\end{align}
for any $S,T\geq 0$ by a double induction on $S$,
which implies the assertion since $\bigoplus_{\substack{N_0 \geq s\geq 0\\ N \geq t \geq 0 }}
 (\Mr)_{s,t}$ is finitely dimensional.

First, we show that if $S \leq N_0$, then \eqref{eq_induction} holds for all $T\geq 0$  by induction on $T$.
If $T \leq N$, then the claim is clear.
Suppose that $T >N$ and the assertion hold for any $T-1$ or less.
Let $v \in (\Mr)_{S,T}$. 
Then, $v$ is a linear sum of the following vectors:
\begin{itemize}
\item
$a(-1)_s \mr$
\end{itemize}
for some $\Delta \geq 1$, $a\in V_\Delta$, $s\in [r]$ and $\mr \in (\Mr)_{S,T-\Delta}$.

Since $a(-k-1)_0^*\mr \in (\Mr)_{S-\Delta-k,T-\Delta}$
and $a(k)_s \mr \in (\Mr)_{S,T-k-1}$ for any $k \geq 0$
and $s\in [r]$,
by Lemma \ref{N_contain} and the induction hypothesis,
$a(-1)_s \mr \in W + N_\Mr$.
Hence, \eqref{eq_induction} holds for any $T\geq 0$ and $S\leq N_0$.

Let $S >N_0$ and assume that \eqref{eq_induction} holds for any $S-1$ or less and any $T\geq 0$.
Let $v \in (\Mr)_{S,T}$. 
Then, $v$ is a linear sum of the following vectors:
\begin{itemize}
\item
$a(\Delta-1+k)_0^* \mr$
\end{itemize}
for some $\Delta \geq 1$, $a\in V_\Delta$, $k\geq 1$ and $\mr \in (\Mr)_{S-k,T}$.
Since $a(k)_s \mr \in (\Mr)_{S-k,T+\Delta_a-k-1}$,
by Lemma \ref{N_contain}, 
$a(\Delta-1+k)_0^* \mr \in W + N_\Mr$ by the induction hypothesis.
Hence, the assertion follows.
\end{proof}

If $D_\Mr$ is a finitely generated $\Or^\alg$-module, then, by \cite[Theorem 1.4.10]{HTT}, 
$D_{\Mr}$ is a locally free $\Or^\alg$-module of finite rank as a $\Dr$-module.
Hence, we have:
\begin{cor}
\label{equi_dim}
For any $M_0,M_1,\dots,M_r \in \Vmodf$,
$D_{\Mr}$ is a finitely generated $\Or^\alg$-module.
In particular, the holomorphic solution sheaf $\mathrm{Hom}(D_{\Mr},\Or^\an)$ is a locally constant sheaf of finite rank on 
$\Xr$ and the rank is equal to the dimension of the vector space $\mathrm{Hom}(D_{\Mr},\Or^\an(U))$
for any connected simply-connected open subset $U \subset X_r$.
\end{cor}

\begin{rem}
Let $M(\la)$ be a (parabolic) Verma module of an affine VOA which has infinitely many singular vectors. Then, $M(\la)$ is $C_1$-cofinite, but $M(\la)^\vee$ is not finitely generated, and the D-module is not holonomic.
\end{rem}
%
%
%
%
%

For a locally constant sheaf, we can define a representation of the fundamental groupoid $\Pi_1(\Xr)$,
which plays an essential role in our construction of a lax 2-action of $\CPaB$.
We will show that this representation is natural with respect to $M_0,M_1,\dots,M_r \in \Vmodf$.
Let $U,V \subset \Xr$ be connected simply-connected open subsets, $q_0 \in U$, $q_1\in V$
and $\ga:[0,1]\rightarrow \Xr$ be a continuous map such that $\ga(0)=q_0$ and $\ga(1)=q_1$.
For any $C\in \CB_\Mr(U)$, 
let $A_\ga C$ be the analytic continuation of $C$ along the path $\ga$.
Then, $A_\ga$ depends on the homotopy classes of $\ga$ and thus gives a map
\begin{align*}
A:\mathrm{Hom}_{\Pi_1(\Xr)}(q_0,q_1) \rightarrow \mathrm{Hom}_{\Vect}(\CB_\Mr(U),\CB_\Mr(V)),
\quad \ga\mapsto A_\ga.
\end{align*}
Then, we have:
\begin{prop}
\label{prop_functor_monodromy}
Let $N_0,N_1,\dots,N_r \in \Vmodf$, and $f_i: N_i \rightarrow M_i$, $g:M_0\rightarrow N_0$ be $V$-module homomorphisms.
Then, the following diagram commutes:
\begin{align*}
\begin{array}{ccc}
\CB_\Mr(U) & \overset{A_\ga}{\rightarrow}&\CB_\Mr(V)\\
{}_{(g_*,f_1^*,\dots,f_r^*)}\downarrow &&\downarrow_{(g_*,f_1^*,\dots,f_r^*)}\\
\CB_{N_{[0;r]}}(U) & \overset{A_\ga}{\rightarrow}&\CB_{N_{[0;r]}}(V),
\end{array}
\end{align*}
where $(g_*,f_1^*,\dots,f_r^*)$ is defined by
 functoriality in Lemma \ref{lem_D_functor}.
\end{prop}
\begin{proof}
Let $n_{[0;r]}=u\otimes n_1\otimes n_2\otimes\cdots\otimes n_r \in N_{[0;r]}$ and $C \in \CB_\Mr(U)$.
Then, 
\begin{align*}
\left((g_*,f_1^*,\dots,f_r^*)C\right)(n_{[0;r]})=C(g^* u,f_1(n_1),\dots,f_r(n_r)).
\end{align*}
Hence, it is clear that $A_\ga \left((g_*,f_1^*,\dots,f_r^*)C\right)$ coincides
with $(g_*,f_1^*,\dots,f_r^*)\left(A_\ga C\right)$.
\end{proof}

\subsection{Fixing translation symmetry}
\label{sec_D_translation}
Let $M_0,M_1,\dots,M_r \in \Vmodu$.
In this section, we identify sections of the conformal block $\CB_\Mr$ and linear maps 
$C: \Mr\rightarrow \Or^\an(U)$ by Lemma \ref{lem_block_linear} for $U\subset \Xr$.
We note that $L(-1)_0^*=L(1)_0$ is locally nilpotent on $\Mr$.
Hence, for any $t\in \{1,\dots,r\}$ and $C\in \CB_\Mr(U)$,
\begin{align*}
C(\exp(-L(-1)_0^* z_t)\mr)=\sum_{k\geq 0} \frac{(-1)^k z_t^k}{k!}C(L(1)_0^k\mr)
\end{align*}
is a finite sum and defines an element in $\Or^\an(U)$.
We start from the following lemma:
\begin{lem}
\label{lem_translation_exp}
Let $t\in \{1,\dots,r\}$ and $C\in \CB_\Mr(U)$.
Then, 
\begin{align*}
(\pa_1+\dots+\pa_r) C(\exp(-L(-1)_0^* z_t)\mr) =0
\end{align*}
for any $\mr \in \Mr$.
\end{lem}
\begin{proof}
By \eqref{eq_ker1}, $C(\om(0)_t \mr)=
C(\om(0)_0^* \mr)-\sum_{1\leq s\leq r, s\neq t}C(\om(0)_s \mr)$.
Hence,
\begin{align}
C(L(-1)_0^* \mr)&= \sum_{1\leq s\leq r}C(L(-1)_s \mr),\label{eq_cyclic_translation}
\end{align}
which implies, by (CB1),
\begin{align*}
(\pa_1+&\dots+\pa_r) C(\exp(-L(-1)_0^* z_t)\mr) \\
&= -C(\exp(-L(-1)_0^* z_t)L(-1)_0^* \mr) 
+  \sum_{1\leq s\leq r} C(\exp(-L(-1)_0^* z_t) L(-1)_s \mr) =0.
\end{align*}
\end{proof}
Set 
\begin{align*}
\Ort^\an(U)=\{f\in \Or^\an\mid (\pa_1+\dots+\pa_r)f=0\},
\end{align*}
the sheaf of translation invariant holomorphic functions.
Then, Lemma \ref{lem_translation_exp}
implies that $C(\exp(-L(-1)_0^* z_t)\mr)$ is in $\Ort^\an(U)$ for any $\mr\in \Mr$.


Using this fact, we can rewrite the characterization of conformal blocks (CB1) and (CB2) to be translation invariant.
However, as can be seen from Lemma \ref{lem_translation_exp}, this requires that $t\in \{1,\dots,r\}$ be fixed and special.

Let $\CB_{\Mrt}(U)$ be a vector space consisting of linear maps
$$
C_t:\Mr \rightarrow \Ort^\an(U),
$$
such that:
\begin{enumerate}
\item[CBT1)]
For any $i \in \{1,2,\dots,r\}\setminus \{t\}$, $\mr\in \Mr$,
$$
C_t(L(-1)_i \cdot \mr)= \frac{d}{dz_i} C_t(\mr).
$$
\item[CBT2)]
For any $i \in \{1,2,\dots,r\}$, $\mr\in \Mr$,
and $a \in V$, $n \in \Z$,
\begin{align*}
C_t(a(n)_i  \mr)
&=\sum_{k \geq 0}\binom{n}{k} (z_t-z_i)^k C_t(a(-k+n)_0 \mr)
- \sum_{s \in [r], s \neq i} \sum_{k \geq 0}
\binom{n}{k}(z_s-z_i)^{n-k} C_t(u,a(k)_s \mr).
\end{align*}
\end{enumerate}

Define a linear map
$Q_t:\CB_\Mr(U) \rightarrow \CB_\Mrt(U)$ by
$Q_t(C)(\bullet)=C(\exp(-L(-1)_0^*z_t)\bullet)$ for $C\in \CB_\Mr(U)$.
Then, we have:
\begin{lem}
\label{lem_translation_isomorphism}
The linear map $Q_t$ is well-defined and gives an isomorphism of sheaves.
\end{lem}
\begin{proof}
Let $C\in \CB_\Mr(U)$ and $\mr\in \Mr$, $a\in V$ and $n\in \Z$.
Set $C_t = Q_t(C)$. By Lemma \ref{lem_translation_exp}, 
$C_t(\mr) \in \Ort^\an(U)$. (CBT1) clearly follows.
We note that by (CB2)
\begin{align*}
C_t(a(n)_i \mr)
&=C(\exp(-L(-1)_0^*z_t)a(n)_i \mr)\\
&=C(a(n)_i \exp(-L(-1)_0^*z_t)\mr)\\
&=\sum_{k \geq 0}\binom{n}{k} (-z_i)^k C(a(-k+n)_0^* \exp(-L(-1)_0^*z_t)\mr)\\
&- \sum_{s \in [r], s \neq i} \sum_{k \geq 0}
\binom{n}{k}(z_s-z_i)^{n-k} C(a(k)_s \exp(-L(-1)_0^* z_t)\mr).
\end{align*}
Since $a(k)_s$ and $\exp(-L(-1)_0^*z_t)$ commute with each other,
it suffices to consider the first term.
Assume $n \geq 0$.
Then by Lemma \ref{lem_P} and \eqref{eq_rem_dual} in Remark \ref{rem_dual}
\begin{align*}
\sum_{k \geq 0}\binom{n}{k} (-z_i)^k \left(\exp(-L(-1)z_t) a(-k+n)\right)_0^*
&= \left(\exp(-L(-1)z_t) P_{n}(a,z_i)\right)_0^* \\
&=\exp(-L(-1)_0^*z_t)P_{n}(a,z_i-z_t)_0^*,
\end{align*}
which implies (CBT2).
Assume $n \leq -1$ and set $m=-n-1$.
Then, by Lemma \ref{lem_binom}
\begin{align*}
\sum_{k \geq 0}&\binom{-m-1}{k} (-z_i)^k a(-k-m-1)_0^* \exp(-L(-1)_0^*z_t)\\
&=\sum_{k \geq 0}\binom{m+k}{k} z_i^k a(-k-m-1)_0^* \exp(-L(-1)_0^* z_t)\\
&=\left(\exp(-L(-1)z_t) \frac{1}{m!} \pa_i^m Y^-(a,z_i)\right)_0^*\\
&=  \exp(-L(-1)_0^*z_t) \left(\frac{1}{k!} \pa_i^k Y^-(a,z_i-z_t)_0^*\right).
\end{align*}
Hence, $Q_t$ is well-defined.

We will show that the linear map
$R_t: \CB_\Mrt \rightarrow \CB_\Mr, C_t \mapsto C_t(\exp( L(-1)_0^* z_t)\bullet)$ gives an inverse map.
Let $C_t \in \CB_\Mr$.
By (CBT2), similar to \eqref{eq_cyclic_translation}, we have
\begin{align}
C_t(L(-1)_0^* \mr)&= \sum_{1\leq s \leq r} C_t(L(-1)_s \mr).
\label{eq_Ct_0}
\end{align}
Hence, we have
\begin{align*}
R_t(C_t)(L(-1)_t \mr) &= C_t(\exp( L(-1)_0^* z_t) L(-1)_t \mr)\\
&=C_t( L(-1)_t\exp( L(-1)_0^* z_t) \mr)\\
&=C_t( L(-1)_0^* \exp( L(-1)_0^* z_t) \mr)-\sum_{1\leq s \leq r,s\neq t} C_t( L(-1)_s \exp( L(-1)_0^* z_t) \mr)\\
&=C_t( \left(\pa_t  \exp( L(-1)_0^* z_t)\right) \mr)-((\sum_{1\leq s \leq r,s\neq t} \pa_s) C_t)(\exp( L(-1)_0^* z_t) \mr)\\
&=\pa_t \left(R_t(C_t)(\mr)\right),
\end{align*}
where in the last equation we used $C_t((L(-1)_0^*)^k \mr) \in \Ort^\an$ for all $k \geq 0$.
Thus, $R_t(C_t)$ satisfies (CB1).
Similar to the above calculations, we can also show that (CB2) holds.
\end{proof}

The definition of a conformal block is symmetric with respect to the variables $z_1,\dots,z_r$.
The action of the fundamental groupoid $\Pi_1(X_r)$ is more obvious in this form (see Section \ref{sec_lax_def}).
However, we sacrifice symmetry to make a conformal block a translation invariant. 
This invariance is naturally required to define an operadic composition.

In the rest of this section, we give the definition of $\CB_\Mrt$ using a $\Drt$-module $D_\Mrt$.
It is the same as in Section \ref{sec_D_def}.

Let $\Ort^\alg \otimes \Mr$ be an $\Ort^\alg$-module
where the $\Ort^\alg$-module structure is defined by the multiplication on the left component.
Recall that
\begin{align*}
\Drt=\C[(z_i-z_j)^\pm,\pa_k\mid 1\leq i<j\leq r,k \in [r]],
\end{align*}
which is a subalgebra of $\Dr$.
Define a $\Drt$-module structure on $\Ort^\alg \otimes \Mr$
by
\begin{align}
\pa_i \cdot (f\otimes \mr) =
\begin{cases}
 (\pa_i f) \otimes \mr + f\otimes L(-1)_i \mr,\\
 (\pa_t f) \otimes \mr + f\otimes L(-1)_t \mr - f\otimes L(-1)_0^* \mr
\end{cases}
\label{eq_Dt_def}
\end{align}
for $i=1,\dots,r$ with $i\neq t$, $f \in \Ort^\alg$, $\mr \in \Mr$.

Let $N_{\Mrt}$ be the $\Drt$-submodule of 
$\Ort^\alg \otimes \Mr$ generated by the following elements:
\begin{align}
1 \otimes a(n)_i \mr - &\sum_{k \geq 0}\binom{n}{k} (z_t-z_i)^k \otimes a(-k+n)_0^* \mr
+ \sum_{s\in [r], s \neq i} \sum_{k \geq 0} \binom{n}{k}(z_s-z_i)^{n-k}\otimes a(k)_s \mr,
\label{eq_ker_t}
\end{align}
for all $\mr \in \Mr$, $a \in V$ and $i \in \{1,\dots,r\}$ and $n \in \Z$.
Set
\begin{align*}
D_\Mrt=(\Ort^\alg \otimes \Mr) /N_\Mrt,
\end{align*}
which is a $\Drt$-module.

\begin{lem}
For any $\mr \in \Mr$,
\begin{align*}
(\sum_{i\in [r]} \pa_i )\cdot 1\otimes \mr =0
\end{align*}
as an element in $D_\Mrt$.
\end{lem}
\begin{proof}
By \eqref{eq_ker_t},
\begin{align*}
\left(-\om(0)_0^* + \sum_{i\in [r]} \om(0)_i \right) \mr \in N_\Mrt.
\end{align*}
Hence, the assertion follows from \eqref{eq_Dt_def}.
\end{proof}

Let $C_t\in \CB_\Mrt(U)$ for an open subset $U\subset \Xr$.
Then, $C_t$ can be uniquely extended to a $\Ort$-module homomorphism
\begin{align*}
\tilde{C_t}: \Ort\otimes \Mr\rightarrow \Ort^\an(U), f\otimes \mr\mapsto f C_t(\mr).
\end{align*}
Then, $\tilde{C_t}$ is a $\Drt$-module homomorphism,
since 
for any $i \in [r]$ with $i \neq t$
\begin{align*}
\pa_i \tilde{C_t}(f\otimes \mr) &=
\pa_i (f {C_t}(\mr))= (\pa_i f)C_t(\mr)+ f (\pa_i C_t(\mr))\\
&=\tilde{C_t}(\pa_i \cdot f\otimes \mr)
\end{align*}
and by \eqref{eq_Ct_0}
\begin{align*}
\pa_t \tilde{C_t}(f\otimes \mr)&=(-\sum_{i \in [r],i\neq t}\pa_i)\tilde{C_t}(f\otimes \mr)\\
&=\tilde{C_t}\left(\left((-\sum_{i \in [r],i\neq t}\pa_i)f\right)\otimes \mr\right)
+\tilde{C_t}\left(f \otimes \left((-\sum_{i \in [r],i\neq t}L(-1)_i)\mr\right)\right)\\
&=\tilde{C_t}\left((\pa_t f)\otimes \mr\right)
+f\otimes {C_t}\left(\left((-\sum_{i \in [r],i\neq t}L(-1)_i)\mr\right)\right)\\
&=\tilde{C_t}\left((\pa_t f)\otimes \mr\right)
+f\otimes {C_t}\left(L(-1)_t\mr\right)
-f\otimes {C_t}\left(L(-1)_0^*\mr\right)\\
&=\tilde{C_t}\left(\pa_t \cdot (f\otimes \mr)\right).
\end{align*}
Hence, by (CBT2),
the $\Drt$-homomorphism $\tilde{C_t}$ factors through the quotient $D_\Mrt$.
Hence, we may regard it as an element in $\mathrm{Hom}_{\Drt}(D_\Mrt,\Ort^\an(U))$.

Conversely, for any map $\tilde{C_t}\in \mathrm{Hom}_{\Drt}(D_\Mrt,\Ort^\an(U))$,
by restriction, we have $\tilde{C_t}|_{\Mr}:\Mr\rightarrow \Ort^\an(U)$
and it is clear that $\tilde{C_t}|_{\Mr}$ satisfies (CBT1) and (CBT2).
Hence, we have:
\begin{prop}
The above map
\begin{align*}
\CB_\Mrt(U)\rightarrow \mathrm{Hom}_{\Drt}(D_\Mrt,\Ort^\an(U)),
\quad C_t \mapsto \tilde{C_t}
\end{align*}
gives an isomorphism of vector spaces.
\end{prop}

\begin{rem}
\label{rem_range_n}
From Lemma \ref{N_contain}, $N_\Mr$ contains the following elements:
\begin{align*}
&1\otimes a(n)_0^* \mr - \sum_{s \in [r]} \sum_{k \geq 0} \binom{n}{k}z_s^{n-k}\otimes a(k)_s \cdot \mr,
\end{align*}
where $n$ is a non-negative integer.
The proof of Lemma \ref{N_contain} does not run well if $n$ is a negative integer, and $z_s^{n}$ does not
contained in $\Or^\alg=\C[(z_i-z_j)^\pm,z_k\mid k,i\neq j]$.
\end{rem}
However, for $N_\Mrt$, by substituting $i=t$ into \eqref{eq_ker_t}, we have:
\begin{lem}
\label{N_contain_t}
Let $\mr \in \Mr$, $i \in \{1,\dots,r\}$ $a \in V$ and $\Delta \in \Z_{>0}$,
$b \in V_\Delta$, $n\in \Z$.
The following elements are contained in $N_{\Mrt}$:
\begin{align*}
&1 \otimes a(-1)_i \mr - \sum_{k \geq 0}(z_i-z_t)^k \otimes a(-k-1)_0^* \mr
 - \sum_{1\leq s\leq r, s \neq i} \sum_{k \geq 0} (z_i-z_s)^{-k-1}\otimes a(k)_s \mr,\\
&1\otimes a(n)_0^* \mr - 1\otimes a(n)_t \mr -
 \sum_{s\in [r], s\neq t} \sum_{k \geq 0} \binom{n}{k}(z_s-z_t)^{n-k}\otimes a(k)_s \cdot \mr,\\
&1\otimes b(-1)_0 \mr - \sum_{1\leq s  \leq r} \sum_{k,l \geq 0} 
\frac{(-1)^\Delta}{l!}\binom{2\Delta-l-1}{k}
(z_s-z_t)^{2\Delta-l-1-k} \otimes \Bigl(L(1)^l b\Bigr)(k)_s \mr.
\end{align*}
\end{lem}

\subsection{Local coherence of conformal block}
\label{sec_D_local_coherent}
Let $A\in \Tr_r$ and $\ee\in E(A)$.
Recall $r_A \in [r]$ is the label of the rightmost leaf of $A$
and $\deg_A^\ee:\Ort \rightarrow \Z$ is the degree map introduced in Section \ref{sec_config_deg}.
Let $\Ort^{\ee\reg}$ is a subalgebra of $\Ort^\alg$ generated by
\begin{align*}
(z_i-z_j),\quad \frac{\ze_{\ee}^{\deg_A^\ee(z_i-z_j)}}{z_i-z_j},\quad \frac{z_i-z_j}{\ze_{\ee}^{\deg_A^\ee(z_i-z_j)}} \quad \text{ for } 1\leq i <j \leq r,
\end{align*}
where, by Lemma \ref{lem_deg_formula},
\begin{align*}
\frac{\ze_{\ee}^{\deg_A^\ee(z_i-z_j)}}{z_i-z_j}=
\begin{cases}
\frac{z_{L(d(\ee))}-z_{R(d(\ee))}}{(z_i-z_j)(z_{L(u(\ee))}-z_{R(u(\ee)))}} \quad &\text{if }i,j \in \Leaf(\ee),\\
\frac{1}{z_i-z_j} \quad &\text{otherwise}.
\end{cases}
\end{align*}

\begin{lem}
\label{lem_noether}
The following properties hold:
\begin{enumerate}
\item
$\Ort^{\ee\reg}$ is a Noetherian ring;
\item
$\ze_\ee \in \Ort^{\ee\reg}$;
\item
$\deg_A^\ee(f)\geq 0$ for any $f\in \Ort^{\ee\reg}$.
\end{enumerate}
\end{lem}
\begin{proof}
Since $\Ort^{\ee\reg}$ is a finitely generated $\C$-algebra, (1) holds.
Since either $L(u(\ee))$ or $R(u(\ee)$ is not in $\Leaf(\ee)$,
$\frac{1}{z_{L(u(\ee))}-z_{R(u(\ee))}} \in \Ort^{\ee\reg}$, which implies (2).
(3) follows from the definition.
\end{proof}

\begin{rem}
\label{rem_subtle}
A naive way to define a ring satisfying Property (3) is 
\begin{align}
\{f\in \Ort^\alg \mid \deg_A^\ee(f)\geq 0\} \subset \Ort^\alg,
\label{eq_naive}
\end{align}
but this definition doesn't make it clear (to me) whether the ring is Noether or not.
\end{rem}

In Section \ref{sec_D_coherent}, we show that $D_{\Mr}$ is a finitely generated $\Or^\alg$-module.
In this section, by refining the arguments in Section \ref{sec_D_coherent},
we show that $D_{\Mra}$ is finitely generated as $\Ort^{\ee\reg}$-module in some sense.
As a consequence, we show that conformal blocks satisfy nice differential equations.

Assume that $M_1,\dots,M_r$ are $C_1$-cofinite
and let $N$ be the integer defined in \eqref{eq_N_sum}
and $\Mr=\bigoplus_{S,T\geq 0}(\Mr)_{S,T}$ the $\Z_{\geq 0}^2$-grading in \eqref{eq_Mr_T}.

\begin{lem}
\label{lem_recursion_T}
Let $S,T\geq 0$ and $\mr \in (\Mr)_{S,T}$.
Then, there exists $L \in \Z_{>0}$ and $N\geq k_i \geq 0$, $m_i \in \bigoplus_{S\geq s\geq 0} (\Mr)_{s, k_i}$, $g_i\in \Ort^{\ee\reg}$ for $i=1,\dots,L$ such that:
\begin{itemize}
\item
$\mr=\sum_{i=1}^L \frac{g_i}{\ze_\ee^{T-k_i}} m_i$ in $D_\Mra$.
\end{itemize}
\end{lem}
\begin{proof}
If $T\leq N$, then $m=m_1$ and $g_1=1$ satisfy the condition.
Assume that $T>N$. Then, similarly to the proof of Proposition \ref{prop_coherence},
we may assume that $\mr=a(-1)_p v$ for some $a\in V_+$ and $v \in (\Mr)_{T-\Delta_a}$.
By definition, in $D_\Mra$
\begin{align*}
a(-1)_p v &=\sum_{k \geq 0} (z_i-z_{r_A})^k a(-k-1)_0^* v
+ \sum_{s \in [r], s \neq p} \sum_{k \geq 0}
(-1)^n (z_p-z_s)^{-k-1} a(k)_s v\\
&=\sum_{k \geq 0} (z_i-z_{r_A})^k a(-k-1)_0^* v
+ \sum_{s \in [r], s \neq p} \sum_{k \geq 0}
(-1)^n \ze_\ee^{-(k+1)\deg_A^\ee(z_p-z_s)} \Bigl(\frac{\ze_\ee^{\deg_A^\ee(z_p-z_s)}}{z_p-z_s}\Bigr)^{k+1} a(k)_s v.
\end{align*}
Since $(z_i-z_{r_A})^k,(z_s-z_{r_A})^{2\Delta-l-1-k},\Bigl(\frac{\ze_\ee^{\deg_A^\ee(z_p-z_s)}}{z_p-z_s}\Bigr)^{k+1}$
are in $\Ort^{\ee\reg}$,
the assertion follows from $a(k)_s v\in (\Mr)_{S,T-k-1}, a(-k-1)_0^* v \in (\Mr)_{S-\Delta-k, T-\Delta_a}$ and the induction used in the proof of Proposition \ref{prop_coherence}.
\end{proof}


Set 
\begin{align*}
E_{A,\ee} = \sum_{l \in \Leaf(\ee)} (z_{l}-z_{R(d(\ee))})L(-1)_l,
\end{align*}
which is a linear map from $\Ort \otimes \Mr$ to itself.
The operator $L(-1)_l$ sends $(\Mr)_{S,T}$ to $(\Mr)_{S,T+1}$.
From Lemma \ref{lem_deg_formula}, $\deg_\ee (z_{l}-z_{R(d(\ee))})=1$ for any $l \in \Leaf(\ee)$.
Since
\begin{align}
E_{A,\ee} = \sum_{l \in \Leaf(\ee)} \ze_\ee \frac{(z_{l}-z_{R(d(\ee))})}{\ze_\ee}L(-1)_l
\label{eq_E_A_ee}
\end{align}
with $\frac{(z_{l}-z_{R(d(\ee))})}{\ze_\ee}\in \Ort^{\ee\reg}$, by Lemma \ref{lem_recursion_T} and \eqref{eq_E_A_ee},
we have:
\begin{lem}
\label{lem_E_filtration}
Let $S,T \geq 0$ and $\mr\in (\Mr)_{S,T}$.
Then, for any $p\geq 0$, there exists $L^p \in \Z_{>0}$ and $k_i^p \in \Z_{\geq 0}$, $m_i^p \in \bigoplus_{S\geq s\geq 0} (\Mr)_{s,k_i^p}$, $g_i^p \in \Ort^{\ee\reg}$ for $i=1,\dots,L^p$ such that:
\begin{enumerate}
\item
$k_i^p \leq N$ for $i=1,\dots,L$;
\item
$E_{A,\ee}^p \mr=\sum_{i=1}^{L^p} \frac{g_i^p}{\ze_\ee^{T-k_i^p}} m_i^p$ in $D_\Mra$.
\end{enumerate}
\end{lem}

The following proposition is important:
\begin{prop}
\label{prop_differential_equation}
Let $M_0,M_1,\dots,M_r \in \Vmodf$ and $\mr \in \Mr$.
Let $A\in \Tr$ and $\ee \in E(A)$.
Then, there exists $K \in \Z_{> 0}$
and $f_0,\dots,f_{K-1} \in \Ort^{\ee\reg}$ such that
\begin{align*}
\left(E_{A,\ee}^K+\sum_{n=0}^{K-1} f_n E_{A,\ee}^n\right)\mr=0
\end{align*}
in $D_\Mra$.
Furthermore, for any open subset $U\subset \Xr$ and $C_{r_A} \in \CB_\Mra(U)$,
$C_{r_A}$ satisfies the following differential equation:
\begin{align*}
\left(\left(\zeta_\ee \frac{d}{d\zeta_\ee}\right)^K + \sum_{n=0}^{K-1}f_n \left(\zeta_\ee \frac{d}{d\zeta_\ee}\right)^n\right) C_{r_A}(\mr)=0.
\end{align*}
\end{prop}

\begin{proof}
Let $L^p \in \Z_{>0}$ and $k_i^p \in \Z_{\geq 0}$, $m_i^p \in \bigoplus_{S\geq s\geq 0} (\Mr)_{s,k_i^p}$, $g_i^p \in \Ort^{\ee\reg}$
be those of Lemma \ref{lem_E_filtration}.
Set $v_p=\sum_{i=1}^{L^p} \frac{g_i^p}{\ze_\ee^{T-k_i^p}} m_i^p$.
Then, for any $p\geq 0$,
$v_p$ is an element of 
\begin{align}
\bigoplus_{\substack{S\geq s\geq 0 \\ N \geq  k\geq 0}} \ze_{\ee}^{-T+k}\Ort^{\ee\reg}\otimes (\Mr)_{s,k}.
\label{eq_ort_mr_k}
\end{align}
Since \eqref{eq_ort_mr_k} is a finitely generated $\Ort^{\ee\reg}$-module and
$\Ort^{\ee\reg}$ is Noether by Lemma \ref{lem_noether},
there exists $K \in \Z_{> 0}$
and $f_0,\dots,f_{K-1} \in \Ort^{\ee\reg}$ such that
\begin{align*}
v_K= \sum_{n=0}^{K-1} f_n v_n
\end{align*}
in \eqref{eq_ort_mr_k}.
Hence the assertion follows from Lemma \ref{lem_change_of_derivation}.
\end{proof}

\begin{prop}
\label{prop_L0}
Let $h_i \in \C$ and $\mr \in \Mra$
with $m_i\in (M_i)_{h_i}$ and $m_0 \in (M_0)_{h_0}^\vee$.
Then, there exists $N \in \Z_{>0}$ such that
for any $C_{r_A} \in \CB_\Mra$
\begin{align*}
(x_s \pa_s-h)^N C_{r_A}(\mr)=0,
\end{align*}
where $h =h_0- \sum_{1\leq s\leq r}h_s$.
\end{prop}
\begin{proof}
Let $\mr \in \Mra$.
By Lemma \ref{N_contain_t},
\begin{align*}
L(0)_0 \mr = \om(1)_0^* \mr =\sum_{1\leq s\leq r} \left((x_s-x_{r_A}) \om(0)_s+\om(1)_s\right) \mr.
\end{align*}
Thus, for any $C_{r_A} \in \CB_\Mra$,

\begin{align}
\begin{split}
C_{r_A}((L(0)_0-L(0)_1-\dots-L(0)_r )\mr) &= \sum_{1\leq s\leq r} (x_s-x_{r_A}) \pa_s C_{r_A} (\mr)\\
&= \sum_{1\leq s\leq r} x_s \pa_s C_{r_A} (\mr),\label{eq_cyclic_L0}
\end{split}
\end{align}
where in the last equality we used $\sum_{1\leq s\leq r} \pa_s C_{r_A} (\mr)=0$.
Since the action of the linear operator $(L(0)_0-L(0)_1-\dots-L(0)_r)-h$ on $\mr$ is nilpotent,
the assertion holds.
\end{proof}

The following Lemma is easy to show:
\begin{lem}
\label{lem_log_finite}
Let $h \in \C$ and $N\in \Z_{>0}$.
Assume that $f\in \C[[z]][z^\C,\log z]$ satisfies
\begin{align*}
(z\frac{d}{dz}-h)^N f =0.
\end{align*}
Then, there exists $a_0,\dots,a_{N-1}\in \C$ such that
 $f(z)=z^h (a_0+a_1 \log z+\dots +a_{N-1} (\log z)^{N-1})$.
\end{lem}

\subsection{Formal and global conformal blocks}
\label{sec_D_formal}
Let $A\in \Tr_r$.
In Section \ref{sec_config_D},
we introduce $\Drt$-modules $T_A^\conv$.
Let us consider the formal solutions of $D_{\Mra}$,
\begin{align*}
\mathrm{Hom}_{\Drt}(D_{\Mra}, T_A^\conv),
\end{align*}
which we call a {\it formal conformal block}.
Let $C_A \in \mathrm{Hom}_{\Drt}(D_{\Mra}, T_A^\conv)$.
Since formal power series in $T_A^\conv$ converge absolutely and uniformly in $U_A \subset \Xr$, 
$C_A$ defines a section of the (global) conformal block $\CB_\Mra(U_A)$.
This gives a linear map
\begin{align*}
s_A:\mathrm{Hom}_{\Drt}(D_{\Mra}, T_A^\conv) \rightarrow \mathrm{Hom}_{\Drt}(D_{\Mra}, \Ort^\an(U_A)).
\end{align*}

Conversely, we will show that any global conformal block can be expanded.
Let $C\in \CB_\Mra(U_A)$ and $\mr \in \Mra$.
In the $A$-coordinate, $C(\mr)$ is a function on $\{x_A,(\zeta_e)_{e\in E(A)}\}$.
By Proposition \ref{prop_differential_equation}
and Proposition \ref{prop_L0}, $C(\mr)$ satisfies the differential equations
and by Proposition \ref{thm_appendix} in Appendix A,
$C(\mr)$ has a convergent expansion of the form in
$$
\C[[x_A,\zeta_e\mid e\in E(A)]]^\conv[x_A^\C,\log x_A,\zeta_e^\C,\log\zeta_e \mid e\in E(A)]
$$
(see Appendix A for more detail).
We denote it by $e_A(C(\mr))$.
Furthermore, by Lemma \ref{lem_log_finite},
$e_A(C(\mr))$ is in 
\begin{align*}
T_A^\conv = \C[[\zeta_e\mid e\in E(A)]]^\conv[x_A^\C,\log x_A,\zeta_e^\C,\log\zeta_e \mid e\in E(A)]
\end{align*}
or, more exactly, under the assumptions of Proposition \ref{prop_L0},
\begin{align*}
x_A^h \bigoplus_{i=0}^N (\log x_A)^i \C[[\zeta_e\mid e\in E(A)]][\zeta_e^\C,\log\zeta_e \mid e\in E(A)].
\end{align*}
Thus, we have a linear map
\begin{align*}
e_A:\mathrm{Hom}_{\Drt}(D_{\Mra}, \Ort^\an(U_A))\rightarrow \mathrm{Hom}_{\Drt}(D_{\Mra}, T_A^\conv).
\end{align*}


It is clear that $e_A$ and $s_A$ are inverse.
Hence, we have:
\begin{thm}
\label{thm_expansion}
Let $M_0,M_1,\dots,M_r \in \Vmodf$.
For $A \in \Tr_r$, the expansion map
$$
e_A:\CB_{\Mra}(U_A) \rightarrow \mathrm{Hom}_{\Drt}(D_{\Mra}, T_A^\conv)
$$
is a natural isomorphism.
\end{thm}

Similar to Lemma \ref{lem_block_linear},
we can characterize an element in a formal conformal block $\mathrm{Hom}_{\Drt}(D_{\Mra}, T_A^\conv)$
by a linear map $C_A: \Mr \rightarrow T_A^\conv$.
\begin{lem}
\label{lem_formal_block}
A linear map
$$
C_A:\Mr \rightarrow T_A^\conv,
$$
defines an element in $\mathrm{Hom}_{\Drt}(D_\Mra, T_A^\conv)$ if and only if it satisfies the following conditions:
\begin{enumerate}
\item[fCB1)]
For any $i \in \{1,2,\dots,r\} \setminus \{r_A\}$, $\mr\in \Mr$,
$$
C_A(L(-1)_i \cdot \mr)= \frac{d}{dz_i} C_A(\mr).
$$
\item[fCB2)]
For any $i \in \{1,2,\dots,r\}$, $\mr\in \Mr$,
and $a \in V$, $n \in \Z$,
\begin{align*}
C_A(a(n)_i  \mr)
&=\sum_{k \geq 0}\binom{n}{k} e_A\left((z_{r_A}-z_i)^k\right) C_A(a(-k+n)_0 \mr)\\
&- \sum_{s \in [r], s \neq i} \sum_{k \geq 0}
\binom{n}{k}e_A\left((z_s-z_i)^{n-k}\right) C_A(u,a(k)_s \mr).
\end{align*}
\end{enumerate}
\end{lem}


\subsection{Fixing scale and rotational symmetry}
\label{sec_D_rotation}
In Section \ref{sec_D_translation}, we used the translation invariance of conformal field theory to eliminate one degree of freedom. This section shows that another degree of freedom can be eliminated using scale and rotational symmetry. This result will be used when defining the operadic composition.

\begin{rem}
\label{rem_translation}
The two degrees of freedom fixed above are derived from
\begin{align*}
\C \rtimes \C^\times=\left\{
\begin{pmatrix}
a & b \\
0 & a^{-1} \\
\end{pmatrix}
\Biggl| \; a\in \C^\times,b\in \C
\right\},
\end{align*}
a subgroup of the global conformal group $\mathrm{SL}_2\C$. The reason why the third degree of freedom cannot be fixed can be explained as follows:

Physically, when considering $D_\Mr$, 
we consider the situation where $r+1$ particles labeled by $M_0^\vee,M_1,\dots,M_r$
are placed at $r+1$ points, $\infty,z_1,\dots,z_r$, on the Riemann sphere.
Since we have fixed $M_0$ to $\infty$, the only symmetry we can use is the stabilizer subgroup of $\infty$
in $\mathrm{SL}_2\C$, $\C \rtimes \C^\times$.
\end{rem}

Let $M\in \Vmodu$. 
Decompose the linear operator $L(0)\in \End M$ as $L(0)=L(0)_s+L(0)_n$,
where $L(0)_s$ is semisimple and $L(0)_n$ is nilpotent.
For a formal variable $x$,
define
$x^{L(0)}:M\rightarrow M[x^\C,\log x]$
by
\begin{align*}
x^{L(0)}m = \exp(L_n(0)\log x)x^{L_s(0)}m
\end{align*}
for any $m\in M$.
Then, $x\frac{d}{dx}(x^{L(0)}m)=L(0)x^{L(0)}m$.
For $\mr\in \Mr$, set
\begin{align*}
x_{A}^{L(0)}\mr=x_A^{-L(0)}m_0 \otimes \bigotimes_{s=1}^r x_A^{L(0)}m_s.
\end{align*}
Then, we have:
\begin{prop}
\label{prop_rotation_del}
Let $C_A \in \mathrm{Hom}_{\Drt}(D_{\Mrt}, T_A^\conv)$
and $\mr\in \Mr$.
Then,
\begin{align*}
C_A(x_{A}^{L(0)}\mr) \in \C[[\zeta_e\mid e\in E(A)]][\zeta_e^\C,\log\zeta_e\mid e\in E(A)],
\end{align*}
that is, formal power series $C_A(x_{A}^{L(0)}\mr)$ does not contain $x_A$ and $\log x_A$.
\end{prop}
\begin{proof}
By \eqref{eq_cyclic_L0},
$\sum_{1\leq s\leq r} x_s \pa_s C_A(x_{A}^{L(0)}\mr)=0$.
In $A$-coordinate, this implies that\\
 $C_A(x_{A}^{L(0)}\mr)$ is in
$\{f \in T_A\mid \frac{d}{dx_A} f=0\}.$
Thus, the assertion holds.
\end{proof}

%
%

\section{Parenthesized intertwining operators and their composition}
\label{sec_parenthesized}
Let $A \in \Tr_r$.  
In the previous section, we introduced the notion of a formal conformal block, namely a linear map
$C_A : M_0^\vee \otimes \bigotimes_{s=1}^r M_s \;\longrightarrow\; T_A^\conv$.
This map is, roughly speaking, equivalent to a linear map of the form $
F_A : \bigotimes_{s=1}^r M_s \;\longrightarrow\; M_0 \,\hat{\otimes}\, T_A^\conv$,
where $\hat{\otimes}$ denotes a completed tensor product. 
The advantage of considering such maps is that they allow us to define the composition of a pseudo-braided category in a natural way.

\subsection{Definition of parenthesized intertwining operators}
\label{sec_parenthesized_def}
In this section, we introduce a notion of a {\it parenthesized intertwining operator}
and show that a space of parenthesized intertwining operators is isomorphic to the space of formal conformal blocks.

Let $r \in \Z_{>0}$ and $A \in \Tr_r$.
Let $V$ a vertex operator algebra
and $M_{0},M_1,\dots,M_r$ $V$-modules.
Set $\Mrr=\bigotimes_{s=1}^r M_s$.
\begin{dfn}
A {\it parenthesized intertwining operator} of type $(A,M_0,\Mrr)$
is a linear map
\begin{align*}
F_A:\Mrr
&\rightarrow M_{0}[[x_{v}^\C,\log x_{v} \mid v \in V(A)]],\\
\mrr = m_1 \otimes \cdots \otimes m_r &\mapsto 
F^A(\mrr)
\end{align*}
such that:
\begin{enumerate}
\item[PI0)]
For any $u \in M_0^\vee$ and $\mrr \in \Mrr$,
$$
\langle u, F_A(\mrr) \rangle \in T_A^\conv.
$$
\item[PI1)]
For any $p \in [r]\setminus r_A$, and $\mrr \in \Mrr$,
\begin{align*}
F_A(L(-1)_p \mrr)&=\frac{d}{dz_{p}} F_A(\mrr).
\end{align*}
\item[PI2)]
For any $p \in [r]$, $\mrr\in \Mrr$,
and $a \in V$, $n \in \Z$,
\begin{align*}
&F_A(a(n)_p  \mrr)\\
&=\sum_{k \geq 0}\binom{n}{k} e_A\left((z_{r_A}-z_p)^k\right) a(-k+n)F_A(\mrr)
- \sum_{s \in [r], s \neq p} \sum_{k \geq 0}
\binom{n}{k}e_A\left((z_s-z_p)^{n-k}\right) F_A(a(k)_s \mrr)
\end{align*}
\end{enumerate}
\end{dfn}

By setting $p=r_A$ in (PI2), we obtain the following proposition:
\begin{prop}
\label{prop_top_recursion}
Let $F_A \in \PI\binom{M_0}{\Mrr}$.
Then, for any $a \in V$ and $n \in \Z$,
\begin{align*}
a(n) F_A(\mrr) - F_A(a(n)_{r_A}\mrr) &=\sum_{s \in [r],s\neq r_A}\sum_{k \geq 0} \binom{n}{k} 
e_A\left((z_s-z_{r_A})^{n-k}\right) F_A(a(k)_s \mrr).
\end{align*}
\end{prop}

We end this section by showing that 
the space of parenthesized intertwining operators $\PI_A\binom{M_0}{\Mrr}$
and the space of formal conformal blocks $\mathrm{Hom}_{\Drt}(D_{\Mrt},T_A^\conv)$
are naturally isomorphic as vector spaces.

Let $F_A \in \PI_A\binom{M_0}{\Mrr}$.
Then, a linear map $C(F_A):M_0^\vee\otimes \Mrr \rightarrow T_A^\conv$
can be defined by
$C(F_A)(m_0 \otimes \mrr)=\langle m_0, F^A(\mrr) \rangle$ for $m_0 \in M_0^\vee$ and
$\mrr \in \Mrr$.
Then, $C(F_A)$ clearly satisfies Conditions (fCB1) and (fCB2) in Lemma \ref{lem_formal_block}
and thus defines an element in $\mathrm{Hom}_{\Drt}(D_{\Mrt},T_A^\conv)$.

Conversely, let $C_A \in \mathrm{Hom}_{\Drt}(D_{\Mrt},T_A^\conv)$
and define a linear map
$F(C_A):\Mrr \rightarrow \overline{M_0}[[x_v^\C,\log x_v\mid v\in V(A)]]$ by
$\langle u,F(C_A)(\mrr)\rangle = C_A(u\otimes \mrr)$ for any $u \in M_0^\vee$ and
$\mrr \in \Mrr$,
where 
$\overline{M_0}=\Pi_{h \in \C}(M_0)_h = (M_0^\vee)^*$, the completion of the vector space.
Here, we used $\dim (M_0)_h < \infty$.

We will show that each coefficients of $F(C_A)(\mrr)$ is in $M_0 \subset \overline{M_0}$.
By \eqref{eq_cyclic_L0}
\begin{align}
C_A((L(0)_0-\sum_{s\in [r]}L(0)_s) u, \mrr)=\left(\sum_{s \in [r]}z_s\pa_s\right)C_A(u,\mrr)
\label{eq_arg_L0}
\end{align}
for any $u \in (M_0)^\vee$ and $\mrr \in \Mrr$.
Let $m_s\in (M_s)_{h_s}$ for some $h_s\in \C$ for $s=0,1,\dots,r$.
Then, by Lemma \ref{lem_log_finite},
for any $p_v \in \C$ ($v \in V(A)$) the coefficients of $\Pi_{v\in V(A)}x_v^{p_v}$
in $F(C_A)(\mrr)$ are in $(M_0)_{r} \subset \bar{M_{0}}$ with $r=\sum_{s \in [r]}h_s + \sum_{v\in V(A)}p_v$.
Hence, $F(C_A):\Mrr \rightarrow M_0[[x_v^\C,\log x_v\mid v\in V(A)]]$
is well-defined and satisfies (PI0)-(PI2).
Hence, we have:
\begin{prop}
\label{fCB_PIO}
The functors
\begin{align*}
\PI_A\binom{\bullet_0}{\bullet_{[r]}}&:\Vmodu \times (\Vmodo)^r \rightarrow \Vect,\\
\Hom(D_{[\bullet_0;\bullet_{[r]}]^{r_A}},T_A^\conv)&:\Vmodu \times (\Vmodo)^r \rightarrow \Vect
\end{align*}
are naturally isomorphic to each other by the above maps.
\end{prop}

\subsection{Examples}
\label{sec_parenthesized_example}
In this section, we will consider trees with one or two leaves.
First we consider the case with one leaf, i.e., $\Tr_1=\{(1)\}$.
Let $M_0,M_1$ be $V$-modules.
In this case, there are zero vertices; thus, $T_{(1)}=\C$
and a parenthesized intertwining operator is a linear map
$f:M_1 \rightarrow M_0$.
\begin{lem}
\label{lem_PI_one}
A linear map $f:M_1 \rightarrow M_0$ is a parenthesized intertwining operator of type $((1),M_0,M_1)$
if and only if 
$f$ is a $V$-module homomorphism.
\end{lem}
\begin{proof}
In the one-leaf case, (PI0) and (PI1) are obvious and (PI2) is equivalent to the following condition:
\begin{align*}
F_{(1)}(a(n)m_1)=a(n)F_{(1)}(m_1)
\end{align*}
for any $m_1\in M_1$ and $a\in V$, $n\in \Z$.
Thus, the assertion holds.
\end{proof}
%
By the above lemma, we have:
\begin{prop}
\label{two_inter}
The functor $\PI_{(1)}\binom{\bullet_0}{\bullet_1}:\Vmodu \times \Vmodu^\op \rightarrow \Vect$
is naturally isomorphic to the hom-functor 
\begin{align*}
\Hom(\bullet_1,\bullet_0)&:\Vmodu \times \Vmodu^\op \rightarrow \Vect,
(M_0,M_1) \mapsto \Hom(M_1,M_0)
\end{align*}
as $\C$-linear functors.
\end{prop}

Next, consider a tree with two leaves, $\Tr_2=\{12,21\}$. Let us consider $12$.
In this case, there is one vertex, which we call $12$.
Then, $T_{12}=T_{12}^\conv = \C[x_{12}^\C, \log x_{12}]$.
\begin{lem}
\label{lem_log_int_pio}
Let $M_0,M_1,M_2$ be $V$-modules
and $\Y_1 \in I_{\log}\binom{M_0}{M_1M_2}$,
a logarithmic intertwining operator of type $\binom{M_0}{M_1M_2}$.
Then, a linear map defined by
$$
F(\Y_1):M_1\otimes M_2 \rightarrow M_0[[x_{12}^\C]][\log x_{12}],\;\; (m_1,m_2)\mapsto 
\Y_1(m_1,x_{12})m_2
$$
is a parenthesized intertwining operator of type $(12,M_0,(M_1,M_2))$.
\end{lem}
\begin{proof}
Similarly to the argument in \eqref{eq_arg_L0}, (PI0) holds.
(PI1) holds since  
\begin{align*}
F(\Y_1)(L(-1)m_1,m_2)&=\Y_1(L(-1)m_1,x_{12})m_2\\
&=\frac{d}{dx_{12}}\Y_1(m_1,x_{12})m_2=\pa_{1}F(\Y_1)(m_1,m_2).
\end{align*}
 Let $a \in V$ and $n\in \Z$.
 By (I1),
\begin{align*}
F(\Y_1)(m_1,a(n)m_2)&=\Y_1(m_1,x_{12})a(n)m_2\\
&=
[\Y_1(m_1,x_{12}),a(n)]m_2+
a(n)\Y_1(m_1,x_{12})m_2\\
&=a(n)\Y_1(m_1,x_{12})m_2
-\sum_{k\geq 0}\binom{n}{k}x_{12}^{n-k} \Y_1(a(k)m_1,x_{12})m_2\\
&=a(n)F(\Y_1)(m_1,m_2) - \sum_{k\geq 0}\binom{n}{k}e_{(12)}\left((z_{1}-z_2)^{n-k}\right) F(\Y_1)(a(n)m_1,m_2)
\end{align*}
and
\begin{align*}
&F(\Y_1)(a(n)m_1,m_2)\\
&=\Y_1(a(n)m_1,x_{12})m_2\\
&=\sum_{k\geq 0} \binom{n}{k}\left(
a(n-k)\Y_1(m_1,x_{12})(-x_{12})^k - \Y_1(m_1,x_{12})a(k)(-x_{12})^{n-k}
\right)\\
&=\sum_{k\geq 0} \binom{n}{k}e_{(12)}\left((z_2-z_1)^k\right)
a(n-k)F(\Y_1)(m_1,m_2)
-\sum_{k\geq 0} \binom{n}{k}e_{(12)}\left((z_2-z_1)^{n-k}\right)
F(\Y_1)(m_1,a(k)m_2).
\end{align*}
 
Hence, (PI2) holds.
\end{proof}

Thus, we have a linear map $I_{\log}\binom{M_0}{M_1M_2}\rightarrow \PI_{12}\binom{M_0}{M_1M_2}$,
which is clearly an isomorphism of vector spaces.
%
Hence, we have:
\begin{prop}
\label{three_inter}
The functor $\PI_{12}:\Vmodu \times (\Vmodo)^2  \rightarrow \Vect$
is naturally isomorphic to the functor of logarithmic intertwining operators
\begin{align*}
I_{\log}\binom{\bullet_0}{\bullet_1\bullet_2}&:\Vmodu \times (\Vmodo)^2 \rightarrow \Vect,
(M_0,M_1,M_2) \mapsto I_{\log}\binom{M_0}{M_1M_2}
\end{align*}
as $\C$-linear functors.
\end{prop}


\subsection{Composition of parenthesized intertwining operators}
\label{sec_parenthesized_comp}
Recall that in Section \ref{sec_operad_def} the operad structure was defined on $\{\Tr_r\}_{r\geq 0}$.
Let $r,r'\geq 1$,
$A\in \Tr_r$ and $B \in \Tr_{r'}$ and $p \in [r]$.
Then, $A\circ_p B \in \Tr_{r+r'-1}$.
Let $M_0,M_1^A,\dots,M_r^A$ and $M_1^B,\dots,M_{r'}^B$ be $V$-modules in $\Vmodf$
and $F_{A} \in \PI_A\binom{M_0}{M_{[r]}^A}$ and 
$F_B \in \PI_B \binom{M_p^A}{M_{[r']}^B}$.
Set
\begin{align*}
M_{[r\circ_p r']}^{A,B}=
M_1^A \otimes M_2^A\otimes \cdots \otimes M_{p-1}^A \otimes 
M_1^B\otimes \cdots \otimes M_{r'}^B \otimes M_{p+1}^A \otimes \cdots \otimes M_r^A.
\end{align*}

The purpose of this section is to show that
we can define an operadic composition $F_A\circ_p F_B$ which is a parenthesized intertwining operator of
type $(A\circ_p B, M_0, M_{[r\circ_p {r'}]}^{A,B})$.
Roughly speaking, $F_A\circ_p F_B$ is defined as:
\begin{align}
\begin{split}
M_1^A \otimes M_2^A\otimes \cdots &\otimes M_{p-1}^A \otimes 
\Bigl( M_1^B\otimes \cdots \otimes M_{r'}^B\Bigr) \otimes M_{p+1}^A \otimes \cdots \otimes M_r^A\\
\overset{\id\otimes F_B\otimes \id}{\longrightarrow}
M_1^A \otimes M_2^A\otimes \cdots \otimes M_{p-1}^A \otimes 
&\Bigl(M_p^A\hat{\otimes}T_B^\conv\Bigr) \otimes M_{p+1}^A \otimes \cdots \otimes M_r^A
\overset{F_A}{\longrightarrow} M_0 \hat{\otimes} T_A^\conv\hat{\otimes}T_B^\conv.
\label{eq_comp_rough}
\end{split}
\end{align}

Importantly, the labels of the leaves are shifted when trees are composited (see Fig. \ref{fig_partial_comp} in Section \ref{sec_operad_def}).
Hereafter, with this in mind, the variables in $F_B$ and $F_A$ are given in accordance with the labels of
$A\circ_p B$.

For $m_{[r\circ_p {r'}]}=m_1^A \otimes\cdots \otimes m_{p-1}^A \otimes m_{[{r'}]}^B\otimes  m_{p+1}^A \otimes \cdots \otimes m_r^{A}
\in M_{[r\circ_p {r'}]}^{A,B}$,
set
\begin{align}
(F_A \circ_p F_B)(m_{[r\circ_p{r'}]})
=F_A(m_1^A,\dots,m_{p-1}^A, F_B(m_{[{r'}]}^B),m_{p+1}^A,\dots,m_r^{A}).
\label{eq_FAB_def}
\end{align}
For simplicity, we will write $F_A(m_1^A,\dots,m_{p-1}^A, F_B(m_{[{r'}]}^B),m_{p+1}^A,\dots,m_r^{A})$ as 
\begin{align*}
F_A(m_{[r]\setminus p}^A,F_B(m_{[{r'}]}^B)).
\end{align*}
Soon after, we will see that \eqref{eq_FAB_def} is well-defined as a formal series of $M_0$-coefficients.
Note that if ${r'}=1$, then $F_B$ is just a $V$-module homomorphism,
and $F_A\circ_p F_B$ is clearly a parenthesized intertwining operator. Hence, we assume that ${r'}\geq 2$.

First, we describe vertices and edges of $A\circ_p B$ with those of $A$ and $B$.
There are $r+{r'}-2=(r-1)+({r'}-1)$ vertices in $A\circ_p B$, which can be written as the union set of the vertices of $A$ and $B$.
Also the edges of $A\circ_p B$ are $r+{r'}-3=(r-2)+({r'}-2)+1$ and can be written using the union set of the edges of $A$ and $B$ and the new edge joining the two trees $A$ and $B$, which we name $e_{AB}$.
Hence, we have:
\begin{align*}
V(A\circ_p B) &= V(A) \sqcup V(B),\\
E(A \circ_p B)&=E(A) \sqcup E(B) \sqcup \{e_{AB}\}\\
T_{A\circ_p B}&=\C[[\zeta_{e} \mid e \in E(A\circ_p B)]][\zeta_e^\C, \log \zeta_e, x_{A\circ_pB}^\C,\log x_{A\circ_pB} \mid e\in E(A\circ_p B)].
\end{align*}

The most difficult part of this section is to show that $\langle u,F_A(m_{[r]\setminus p}^A,F_B(m_{[{r'}]}^B))\rangle$
is in $T_{A\circ_p B}^\conv$ for any $u\in M_0^\vee$.
Let $\mfp \in \R_{>0}^{E(A\circ_p B)}$ be $A\circ_p B$-admissible.
For $X=A,B$, set $\mfp|_X =(p_e)_{e\in E(X)} \in \R_{>0}^{E(X)}$.
Then, it is clear that $\mfp|_X$ is $X$-admissible.
Set 
\begin{align*}
T_{A\sqcup B}^{\mfp}=\C[[\ze_e \mid e \in E(A)\sqcup E(B)]]_{\mfp|_A,\mfp|_B}^\conv
[\zeta_e^\C, \log \zeta_e \mid e\in E(A) \sqcup E(B)]
\end{align*}
and
\begin{align*}
\ze_{AB}=\ze_{e_{AB}},\quad x_A=x_{A\circ_p B},\quad r_A=r_{A\circ_pB}.
\end{align*}

For $h \in \C$, let
\begin{align*}
\prk: M_p^A = \bigoplus_{k\in \C}(M_p^A)_k \rightarrow (M_p^A)_h
\end{align*}
be the projection.
Let $L(0)=L_s(0)+L_n(0)$ be the decomposition of the locally finite operator $L(0)$
into the semisimple part, $L_s(0)$, and the nilpotent part, $L_n(0)$.
By Assumption (M3), there exists $M\in \Z_{\geq 0}$ such that 
$L_n(0)^{M+1}|_{M_p^A}=0$.
Then, we have:
\begin{lem}
\label{lem_comp_form_part}
For any $h\in \C$ and $u \in (M_0)^\vee$ and $m_{[r\circ_p {r'}]} \in M_{[r\circ_p {r'}]}^{A,B}$,
\begin{align*}
\langle x_A^{-L(0)} u, F_A(x_A^{L(0)}m_1^A, \dots, x_A^{L(0)}m_{p-1}^A, \prk F_B(x_B^{L(0)}m_{[{r'}]}), x_A^{L(0)}m_{p+1}^A,
\dots, x_A^{L(0)}m_{r}^A) \rangle
\end{align*}
is in 
\begin{align*}
\bigoplus_{i=0}^M \ze_{AB}^h (\log \ze_{AB})^i \C[[\ze_e \mid e \in E(A)\sqcup E(B)]]_{\mfp|_{A},\mfp|_{B}}^\conv[\ze_e^\C,\log \ze_e\mid
e\in E(A)\sqcup E(B)].
\end{align*}
Moreover, if $u\in (M_0)_{h_0}^\vee$, $m_i^A\in (M_i^A)_{h_i^A}$ and $m_j^B\in (M_j^B)_{h_j^B}$
for some $h_0,h_i^A,h_j^B \in \C$
and $M'\geq M$ satisfies $L_n(0)^{M'+1}$ on all modules $M_j^B$, $M_i^A$ and $M_0$,
then
\begin{align*}
\langle u, F_A(m_1^A, \dots, m_{p-1}^A, \prk F_B(m_{[{r'}]}),m_{p+1}^A,
\dots,m_{r}^A) \rangle
\end{align*}
is in
\begin{align*}
x_A^{\Delta_A}\ze_{AB}^{\Delta_{AB}}\bigoplus_{a,b=0}^{M'} (\log \ze_{AB})^a (\log z_A)^b \C[[\ze_e \mid e \in E(A)\sqcup E(B)]]_{\mfp|_{A},\mfp|_{B}}^\conv[\ze_e,\log \ze_e\mid
e\in E(A)\sqcup E(B)]
\end{align*}
with $\Delta_A= h_0-\sum_{i \in [r]\setminus p}h_i^A
-\sum_{j\in [r']}h_j^B$ and
$\Delta_{AB}=h- \sum_{j\in [r']} h_j^B$.
\end{lem}
\begin{proof}
Let $\{e_k\}_{k\in I}$ be a basis of $(M_p^A)_h$ and $\{e_k^*\}_{k\in I}$ be the dual basis.
By Proposition \ref{prop_rotation_del},
\begin{align}
\begin{split}
\langle & e_k^*, x_B^{-L(0)}F_B(x_B^{L(0)}m_{[{r'}]}^B) \rangle\\
&= x_B^{-h} \langle \exp(-L(0)_n\log x_B) e_k^*, F_B(x_B^{L(0)}m_{[{r'}]}^B) \rangle
\in \C[[\zeta_e\mid e\in E(B)]][\zeta_e^\C,\log \zeta_e\mid e\in E(B)]
\end{split}
\label{eq_A_part_comp}
\end{align}
and
\begin{align}
\begin{split}
\langle x_A^{-L(0)} m_0^A, F_A(x_A^{L(0)}m_1^A, & \dots, x_A^{L(0)}m_{p-1}^A,x_A^{L(0)}e_k,x_A^{L(0)}m_{p+1}^A, \dots,
x_A^{L(0)}m_{r}^A) \rangle\\
&\in \C[[\zeta_e\mid e\in E(A)]][\zeta_e^\C,\log \zeta_e\mid e\in E(A)].
\end{split}
\label{eq_B_part_comp}
\end{align}

Note that 
\begin{align}
\frac{x_B}{x_A}=\ze_{{AB}}\Pi_{e_i}\ze_{e_i}
\label{eq_ze_AB_run}
\end{align}
where $e_i$ run through all edges in $E(A)$ which connects the topmost vertex of $A$ to $p$.
 and 
\begin{align}
\left(\frac{x_B}{x_A}\right)^{L(0)}e_k=\left(\frac{x_B}{x_A}\right)^h \sum_{i=0}^M \frac{1}{i!}\left(\log\left(\frac{x_B}{x_A}\right)\right)^i L_n(0)^ie_k.
\label{eq_AB_nil}
\end{align}
Hence, we have
\begin{align*}
&\langle x_A^{-L(0)} u, F_A(x_A^{L(0)}m_1^A, \dots, x_A^{L(0)}m_{p-1}^A, \prk F_B(x_B^{L(0)}m_{{r'}}), x_A^{L(0)}m_{p+1}^A,
\dots, x_A^{L(0)}m_{r}^A) \rangle\\
&=\sum_{k\in I}
\langle x_A^{-L(0)} u, F_A(x_A^{L(0)}m_1^A, \dots, x_A^{L(0)}m_{p-1}^A, e_k , x_A^{L(0)}m_{p+1}^A,
\dots, x_A^{L(0)}m_{r}^A) \rangle
\langle e_k^*, F_B(x_B^{L(0)}m_{[{r'}]}^B)\rangle\\
&=\sum_{k\in I}
\langle x_A^{-L(0)} u, F_A(x_A^{L(0)}m_1^A, \dots, x_A^{L(0)}m_{p-1}^A, x_A^{L(0)}\left(\frac{x_B}{x_A}\right)^{L(0)}e_k , x_A^{L(0)}m_{p+1}^A,
\dots, x_A^{L(0)}m_{r}^A) \rangle\\
&\quad \times \langle e_k^*, x_B^{-L(0)} F_B(x_B^{L(0)}m_{[{r'}]}^B)\rangle.
\end{align*}
Thus, the assertion follows from \eqref{eq_AB_nil}, \eqref{eq_ze_AB_run}, \eqref{eq_A_part_comp} and \eqref{eq_B_part_comp}.
\end{proof}

Motivated by the above Lemma,
set
\begin{align*}
T_{A,B,p}^\mfp &=\C[[\ze_e \mid e \in E(A)\sqcup E(B)]]_{\mfp|_{A},\mfp_{B}}^\conv[\ze_e,\log \ze_e\mid
e\in E(A)\sqcup E(B)][[\ze_{AB}]][\zeta_{AB}^\C,\log \zeta_{AB}, x_A^\C,\log x_A],
\end{align*}
which is a formal power series in $\ze_{AB}$ with coefficients in $\C[[\ze_e \mid e \in E(A)\sqcup E(B)]]_{\mfp|_{A},\mfp|_{B}}^\conv[\ze_e^\C,\log \ze_e\mid e\in E(A)\sqcup E(B)]$.
That is $T_{A,B,p}^\conv$ is spanned by
\begin{align*}
(\log x)^l x^r (\log \zeta_{AB})^{l_{AB}} \zeta_{AB}^{r_{AB}} \sum_{k=0}^\infty f_k \zeta_{AB}^k
\end{align*}
with $l,l_{AB}\in \Z_{\geq 0}$ and $r,r_{AB}\in \C$,
$f_k \in \C[[\ze_e \mid e \in E(A)\sqcup E(B)]]_{\mfp|_{A},\mfp_{B}}^\conv$.

It is clear that $T_{A,B,p}^\mfp$ naturally inherits a $\C$-algebra structure.
By expanding with respect to $\zeta_{AB}$, we also obtain a $\C$-algebra  homomorphism
\begin{align*}
\Ort^\alg \rightarrow T_{A,B,p}^\mfp.
\end{align*}
Moreover, we can introduce a structure of a $\Drt$-module on $T_{A,B,p}^\mfp$,
and this embedding map is a homomorphism of $\Drt$-modules.

%
By Lemma \ref{lem_comp_form_part} and Assumption (M4),
we have:
\begin{lem}
\label{lem_comp_formalsp}
For any $u \in (M_0)^\vee$ and $m_{[r\circ_p {r'}]} \in M_{[r\circ_p {r'}]}^{A,B}$,
\begin{align*}
\langle u, (F_A\circ_p F_B)(m_{[r\circ_p {r'}]}) \rangle
\end{align*}
is in $T_{A,B,p}^\mfp$.
\end{lem}

By Lemma \ref{lem_comp_formalsp},
we can define a linear map
\begin{align*}
C_{A\circ_p B}: M_0^\vee \otimes M_{[n\circ_p m]}^{A,B} &\rightarrow \quad T_{A,B,p}^\mfp,\\
(u, m_{[r\circ_p {r'}]})\;\;\; &\mapsto \langle u, (F_A \circ_p F_B)(m_{[r\circ_p {r'}]})\rangle.
\end{align*}
We will show that
$C_{A \circ_i B}$ defines a formal solution of the D-module, $\mathrm{Hom}_{\Drt}
(D_{M_{[r\circ_p r']}^{A,B}},T_{A,B,p}^\mfp)$.

%

\begin{lem}
\label{lem_comp_recursion}
The following properties hold:
\begin{enumerate}
\item
If $i \in \Leaf(A\circ_p B)$ is not the rightmost leaf of $A\circ_p B$, i.e., $i\neq r_{A}$,
then
\begin{align*}
(F_{A} \circ_p F_B)(L(-1)_i m_{[r\circ_p {r'}]})=\frac{d}{dz_i}(F_{A} \circ_p F_B)(m_{[r\circ_p {r'}]}).
\end{align*}
\item
For any $n\in \Z$ and $a\in V$ and $i\in \Leaf(A\circ_p B)$,
\begin{align*}
(F_A\circ_p F_B)\left(a(n)_i m_{[r\circ_p {r'}]}\right)&=\sum_{k\geq 0}\binom{n}{k}e_{A\circ_p B}\left((z_{r_{A}}-z_i)^k\right)a(n-k)(F_A\circ_p F_B)\left(m_{[{r}\circ_p {r'}]}\right)\\
&-\sum_{\substack{s\in \Leaf(A\circ_p B),\\ s\neq i}}\sum_{k\geq 0} \binom{n}{k}e_{A\circ_p B}\left((z_s-z_i)^{n-k}\right)(F_A\circ_p F_B)\left(a(k)_s m_{[r\circ_p {r'}]}\right).
\end{align*}
\end{enumerate}
\end{lem}
\begin{proof}
Let $i \in \Leaf(A\circ_p B)$ which is not the rightmost leaf of $A\circ_i B$.
If $i$ is not the rightmost leaf of $B$, then
\begin{align*}
&(F_A\circ_p F_B)\left(L(-1)_i m_{[r\circ_p {r'}]}\right)\\
&=\sum_{h}F_A(m_1^A,\dots,m_{p-1}^A,\prk F_B(L(-1)_i m_{[{r'}]}^B),m_{p+1}^A,\dots,m_r^{A})\\
&=\sum_{h}F_A(m_1^A,\dots,m_{p-1}^A,\prk \frac{d}{dz_i}F_B(m_{[{r'}]}^B),m_{p+1}^A,\dots,m_r^{A})\\
&=\frac{d}{dz_i} \sum_{h}F_A(m_1^A,\dots,m_{p-1}^A,\prk F_B(m_{[{r'}]}^B),m_{p+1}^A,\dots,m_r^{A}).
\end{align*}
Let $i$ be the rightmost leaf of $B$. From the assumption, we may assume that $p$ is not the rightmost leaf of $A$.
Because if it is, then $i$ is the rightmost leaf of $A\circ_p B$.
By \eqref{eq_cyclic_translation},
\begin{align*}
&(F_A\circ_p F_B)\left(L(-1)_i m_{[r\circ_p {r'}]}\right)\\
&=\sum_{h}F_A(m_1^A,\dots,m_{p-1}^A,\prk F_B(L(-1)_i m_{[{r'}]}^B),m_{p+1}^A,\dots,m_r^{A})\\
&=\sum_{h}F_A(m_1^A,\dots,m_{p-1}^A,\prk L(-1)F_B(m_{[{r'}]}^B),m_{p+1}^A,\dots,m_r^{A})\\
&-\sum_{h}\sum_{s\in [{r'}],s\neq i}F_A(m_1^A,\dots,m_{p-1}^A,\prk F_B(L(-1)_s m_{[{r'}]}^B),m_{p+1}^A,\dots,m_r^{A})\\
&=\sum_{h}F_A(m_1^A,\dots,m_{p-1}^A,L(-1)\prk F_B(m_{[{r'}]}^B),m_{p+1}^A,\dots,m_r^{A})\\
&-\sum_{h}F_A(m_1^A,\dots,m_{p-1}^A,\prk \left(\sum_{s\in [r'],s\neq i}\frac{d}{dz_s} F_B(m_{[{r'}]}^B)\right),m_{p+1}^A,\dots,m_r^{A})\\
&=\sum_{h}\left(\frac{d}{dz_i}F_A\right)(m_1^A,\dots,m_{p-1}^A,\prk F_B(m_{[{r'}]}^B),m_{p+1}^A,\dots,m_r^{A})\\
&+\sum_{h}F_A(m_1^A,\dots,m_{p-1}^A,\prk\left( \frac{d}{dz_i}F_B(m_{[{r'}]}^B)\right),m_{p+1}^A,\dots,m_r^{A})\\
&=\frac{d}{dz_i}\left(\sum_{h}F_A(m_1^A,\dots,m_{p-1}^A,\prk F_B(m_{[{r'}]}^B),m_{p+1}^A,\dots,m_r^{A})\right).
\end{align*}
Here, we used $\sum_{s\in [{r'}]}\frac{d}{dz_s} F_B(m_{[{r'}]}^B)=0$.
Hence, (1) holds. 

We will show (2).
First assume that $i$ is a leaf of $A$, which is not $p$.
Then, for any $h\in \C$,
\begin{align}
\begin{split}
&F_A(a(n)_i m_{[r]\setminus p}^A,\prk F_B(m_{[{r'}]}^B))\\
&=\sum_{k\geq 0}\binom{n}{k}e_{A\circ_p B}\left((z_{r_A}-z_i)^k\right)a(n-k)(F_A\circ_p F_B)\left(m_{[r \circ_p {r'}]}\right)\\
&-\sum_{\substack{s\in \Leaf(A),\\ s\neq i,p}}\sum_{k\geq 0} \binom{n}{k}e_{A\circ_p B}\left((z_s-z_i)^{n-k}\right)F_A(a(k)_s m_{[r]\setminus p}^A,\prk F_B(m_{[{r'}]}^B))\\
&-\sum_{k\geq 0} \binom{n}{k}e_{A\circ_p B}\left((z_{r_B}-z_i)^{n-k}\right)F_A(m_{[r]\setminus p}^A, a(k)\prk F_B(m_{[{r'}]}^B)).
\label{eq_rec_fin}
\end{split}
\end{align}
Note that evaluated at $u\in M_0^\vee$,
the sum \eqref{eq_rec_fin} is finite 
since $a(k)\prk=0$ for sufficiently large $k$.

Summing \eqref{eq_rec_fin} for all $h \in \C$ yields an infinite sum. However, focusing on the coefficients of $\ze_{AB}^H$ for each $H\in \C$ shows that the sum is finite.
From Lemma \ref{lem_deg_formula},
$\deg_{A\circ_pB}^{AB}(z_{r_A}-z_i)=\deg_{A\circ_pB}^{AB}(z_{s}-z_i)=\deg_{A\circ_pB}^{AB}(z_{t_B}-z_i )=0$.
That is, these terms do not contain $\ze_{AB}$.
If $m_{[r\circ_p r']}$ satisfies the conditions of the second half of Lemma \ref{lem_comp_form_part}, then, similarly to the proof of Lemma \ref{lem_comp_form_part}, the order of $\ze_{AB}$ in the three terms of $\eqref{eq_rec_fin}$ is $h-\sum_{j \in [r']}h_j^B$.

Thus the following equality is a finite sum and well-defined for each $\ze_{AB}^H$-coefficients ($H\in \C$):
\begin{align*}
&(F_A\circ_p F_B)\left(a(n)_i m_{[r\circ_p {r'}]}\right)\\
&=\sum_{h}F_A(a(n)_i m_{[r]\setminus p}^A,\prk F_B(m_{[{r'}]}^B))\\
&=\sum_{h}\sum_{k\geq 0}\binom{n}{k}e_{A\circ_p B}\left((z_{t_A}-z_i)^k\right)a(n-k)(F_A\circ_p F_B)\left(m_{[r \circ_p {r'}]}\right)\\
&-\sum_{h}\sum_{\substack{s\in \Leaf(A),\\ s\neq i,p}}\sum_{k\geq 0} \binom{n}{k}e_{A\circ_p B}\left((z_s-z_i)^{n-k}\right)F_A(a(k)_s m_{[r]\setminus p}^A,\prk F_B(m_{[{r'}]}^B))\\
&-\sum_h\sum_{k\geq 0} \binom{n}{k}e_{A\circ_p B}\left((z_{t_B}-z_i)^{n-k}\right)F_A(m_{[r]\setminus p}^A, a(k)_p\prk F_B(m_{[{r'}]}^B)).
\end{align*}
Set
\begin{align*}
\Delta_{B}=\sum_{j\in [r']}h_j^B.
\end{align*}
By Proposition \ref{prop_top_recursion},
we have
\begin{align*}
&\sum_h\sum_{k\geq 0} \binom{n}{k}e_{A\circ_p B}\left((z_{r_B}-z_i)^{n-k}\right)F_A(m_{[r]\setminus p}^A, a(k)_p\prk F_B(m_{[{r'}]}^B))\\
&=\sum_h \sum_{k\geq 0}\sum_{s\in \Leaf(B)} \sum_{l\geq 0} \binom{n}{k}\binom{k}{l} e_{A\circ_p B}
\left((z_{r_B}-z_i)^{n-k}(z_s - z_{r_B})^{k-l}\right)
F_A(m_{[r]\setminus p}^A,\prk F_B(a(l)_s m_{[{r'}]}^B)).
\end{align*}
Since 
\begin{align*}
\deg_{A\circ_p B}^{AB}&\left(
e_{A\circ_p B}
\left((z_{r_B}-z_i)^{n-k}(z_s - z_{r_B})^{k-l}\right)
F_A(m_{[r]\setminus p}^A,\prk F_B(a(l)_s m_{[{r'}]}^B))\right)\\
&=(k-l)+ (h-\Delta_{B}-\Delta_a+l+1)\\
&=k+h-\Delta_{B}-\Delta_a+1,
\end{align*}
the coefficients of $\ze_{AB}^H$ are always finite sums for any $H\in \C$.
By Lemma \ref{lem_well_inverse},
\begin{align*}
\sum_{k\geq 0} \binom{n}{k}\binom{k}{l} e_{A\circ_p B}
\left(\frac{z_s - z_{r_B}}{z_{r_B}-z_i}\right)^k
&=\binom{n}{l}e_{A\circ_p B}
\left(\frac{z_s - z_{r_B}}{z_{r_B}-z_i}\right)^l
e_{A\circ_p B}\left(\frac{z_s - z_{i}}{z_{r_B}-z_i}\right)^{n-l}
\end{align*}
as formal power series.
Since $e_{A\circ_p B}$ is a ring homomorphism,
we have
\begin{align*}
\sum_{k\geq 0} \binom{n}{k}\binom{k}{l} e_{A\circ_p B}\left((z_{r_B}-z_i)^{n-k}(z_s - z_{r_B})^{k-l}\right)
=\binom{n}{l}e_{A\circ_p B}\left((z_{s}-z_i)^{n-l}\right).
\end{align*}

Hence, we have
\begin{align*}
\sum_h \sum_{k\geq 0}&\sum_{s\in \Leaf(B)} \sum_{l\geq 0} \binom{n}{k}\binom{k}{l} e_{A\circ_p B}
\left((z_{r_B}-z_i)^{n-k}(z_s - z_{r_B})^{k-l}\right)
F_A(m_{[r]\setminus p}^A,\prk F_B(a(l)_s m_{[{r'}]}^B))\\
&=\sum_h \sum_{s\in \Leaf(B)} \sum_{l\geq 0} \binom{n}{l}e_{A\circ_p B}
\left((z_s-z_{i})^{n-l}\right)
F_A(m_{[r]\setminus p}^A,\prk F_B(a(l)_s m_{[{r'}]}^B)).
\end{align*}
Hence, (2) holds.
Note that these calculations work because the functions being expanded are good functions in $U_{A\circ_p B}$ in the sense of Remark \ref{rem_ill_exp}.

Next, we assume that $i$ is a leaf of $B$. Then,
\begin{align*}
&(F_A\circ_p F_B)\left(a(n)_i m_{[r\circ_p {r'}]}\right)\\
&=\sum_{h}F_A(m_{[r]\setminus p}^A,\prk F_B(a(n)_i m_{[{r'}]}^B))\\
&=\sum_{h}\sum_{k\geq 0}\binom{n}{k}e_{A\circ_p B}\left((z_{r_B}-z_i)^k\right)
F_A(m_{[r]\setminus p}^A,\prk a(n-k) F_B( m_{[{r'}]}^B))\\
&-\sum_{h}\sum_{\substack{s\in \Leaf(B),\\ s\neq i}}\sum_{k\geq 0} \binom{n}{k}e_{A\circ_p B}\left((z_s-z_i)^{n-k}\right)F_A(m_{[r]\setminus p}^A,\prk F_B(a(k)_s m_{[{r'}]}^B)),
\end{align*}
which is well-defined as above
since
\begin{align*}
\deg_{A\circ_pB}^{AB}&\left(e_{A\circ_p B}\left((z_{r_B}-z_i)^k\right)
F_A(m_{[r]\setminus p}^A,\prk a(n-k) F_B( m_{[{r'}]}^B))\right)\\
&=h+k-\Delta_B
\end{align*}
and
\begin{align*}
\deg_{A\circ_pB}^{AB}&\left( e_{A\circ_p B}(z_s-z_i)^{n-k}F_A(m_{[r]\setminus p}^A,\prk F_B(a(k)_s m_{[{r'}]}^B))\right)\\
&=(n-k)+(h-\Delta_B-(\Delta_a-k-1))\\
&=n+h-\Delta_B -\Delta_a-1.
\end{align*}
Furthermore,
\begin{align*}
&\sum_{h}\sum_{k\geq 0}\binom{n}{k}e_{A\circ_p B}\left((z_{r_B}-z_i)^k\right)
F_A(m_{[r]\setminus p}^A,\prk a(n-k) F_B(a(k)_s m_{[{r'}]}^B))\\
&=\sum_{h}\sum_{k\geq 0}\sum_{l\geq 0}
\binom{n}{k}\binom{n-k}{l}
e_{A\circ_p B}\left((z_{r_B}-z_i)^k(z_{r_A}-z_{r_B})^{l}\right)
a(n-k-l)F_A(m_{[r]\setminus p}^A,\prk F_B(m_{[{r'}]}^B))\\
-&\sum_{h}\sum_{k\geq 0}\sum_{s\in \Leaf(A),s\neq p}
\sum_{l\geq 0}\binom{n}{k}\binom{n-k}{l}
e_{A\circ_p B}\left((z_{r_B}-z_i)^k(z_s-z_{r_B})^{n-k-l}\right)
F_A(a(l)_s m_{[r]\setminus p}^A,\prk F_B(m_{[{r'}]}^B)).
\end{align*}
These equations are well-defined as above.

If $n \geq 0$, then by Lemma \ref{identity}
\begin{align*}
\sum_{k\geq 0}\binom{n}{k}\binom{n-k}{l}
e_{A\circ_p B}\left((z_{r_B}-z_i)^k(z_s-z_{r_B})^{n-k-l}\right)&=\sum_{k'\geq 0}\binom{n}{n-k'}\binom{k'}{l}
e_{A\circ_p B}\left((z_{r_B}-z_i)^{n-k'}(z_s-z_{r_B})^{k'-l}\right)\\
&=\binom{n}{l}e_{A\circ_p B}\left((z_{s}-z_i)^{n-l}\right)
\end{align*}
and 
\begin{align*}
&\sum_{k,l\geq 0}\binom{n}{k}\binom{n-k}{l}
e_{A\circ_p B}\left((z_{r_B}-z_i)^k(z_{r_A}-z_{r_B})^{l}\right)a(n-k-l)\\
&=\sum_{k',l\geq 0}\binom{n}{k'}\binom{k'}{l}
e_{A\circ_p B}\left((z_{r_B}-z_i)^{n-k'}(z_{r_A}-z_{r_B})^{l}\right)a(k'-l)\\
&=\sum_{k',l'\geq 0}\binom{n}{k'}\binom{k'}{l'}
e_{A\circ_p B}\left((z_{r_B}-z_i)^{n-k'}(z_{r_A}-z_{r_B})^{k'-l'}\right)a(l')\\
&=\sum_{l'\geq 0}\binom{n}{l'}
e_{A\circ_p B}\left((z_{r_A}-z_{i})^{n-l'}\right)a(l').
\end{align*}
Thus, (2) holds if $n\geq 0$.
Assume  that $n <0$.
By Lemma \ref{lem_well_inverse},
we have
\begin{align*}
\sum_{k\geq 0}\binom{n}{k}\binom{n-k}{l}
e_{A\circ_p B}\left((z_{r_B}-z_i)^k(z_s-z_{r_B})^{n-k-l}\right)&=\binom{n}{l}e_{A\circ_p B}\left((z_{s}-z_i)^{n-l}\right).
\end{align*}
By
\begin{align*}
\sum_{k,l\geq 0}\binom{n}{k}\binom{n-k}{l} x^ky^l z^{k+l} &= 
\bigl(1+(x+y)z\bigr)^{n},
\end{align*}
we have
\begin{align*}
&\sum_{k,l\geq 0}\binom{n}{k}\binom{n-k}{l}
e_{A\circ_p B}\left((z_{r_B}-z_i)^k(z_{r_A}-z_{r_B})^{l}\right)a(n-k-l)\\
&=\sum_{l'\geq 0}
\binom{n}{l'}e_{A\circ_p B}\left((z_{r_A}-z_{i})^{l'}\right)a(n-l').
\end{align*}
Hence, the assertion holds.
\end{proof}

By Lemma \ref{lem_comp_recursion} and similarly to Lemma \ref{lem_formal_block},
$C_{A \circ_p B}$ defines an element in\\
 $\mathrm{Hom}_{\Drt}(D_{M_{[r\circ_p r']}^{A,B}},T_{A,B,p}^\conv)$.
Then, we have:
\begin{thm}
\label{thm_composition}
$F_A\circ_p F_B \in \PI_{A\circ_p B}\binom{M_0}{M_{[r\circ_p {r'}]}^{A,B}}$.
\end{thm}
\begin{proof}
(PI1) and (PI2) follow from Lemma \ref{lem_comp_recursion}.
It suffices to show (PI0).
Let $m_{[r \circ_p {r'}]} \in M_{[r\circ_p {r'}]}^{A,B}$ and $u\in M_0^\vee$.
By Proposition \ref{prop_differential_equation} and Lemma \ref{lem_comp_recursion},
there exists $K \in \Z_{>0}$ and $f_0,\dots,f_{K-1} \in \Ort^{AB\reg}$ such that:
\begin{align*}
\left(\left(\zeta_{AB}\frac{d}{d\zeta_{AB}}\right)^K+\sum_{n=0}^{K-1}e_{A\circ_p B}(f_n) \left(\zeta_{AB}\frac{d}{d\zeta_{AB}}\right)^n \right) C_{A\circ_p B}(u, m_{[r \circ_p r']})=0.
\end{align*}
Then, $\{f_i\}_{i=0,\dots,K-1}$ are holomorphic functions on $\D_{\mfp|_A}^\times \times \D_{\mfp|_B}^\times \times \D_{p_{AB}}$ in the $A\circ_p B$-coordinate,
where
\begin{align*}
\D_{\mfp|_A}^\times=\Pi_{e \in E(A)} \D_{p_e}^\times\\
\D_{\mfp|_B}^\times=\Pi_{e \in E(B)} \D_{p_e}^\times.
\end{align*}
Then, by Lemma \ref{lem_comp_formalsp} and Corollary \ref{cor_app} in Appendix A,
the sum
\begin{align*}
C_{A\circ_p B}(u, m_{[r\circ_p r']})=\sum_{h \in \C} \langle u,
F_A(m_{[r]\setminus p}^A,\prk F_B(m_{[{r'}]}^B))\rangle
\end{align*}
converges absolutely and uniformly to 
a holomorphic function on $U_{A \circ_p B}^{\mathfrak{p}}$ for any $A\circ_p B$-admissible $\mathfrak{p}\in \R_{>0}^{E(A\circ_p B)}$.
Hence, $F_A\circ_p F_B$ defines a section in $\CB_{M_0,M_{[r\circ_p r']}^{A,B}}(U_{A\circ_p B}^\mfp)$.
Since $\CB$ is a locally constant sheaf
and $U_{A\circ_p B}^\mfp \subset U_{A\circ_p B}$, 
by Lemma \ref{lem_simply_conn},
$F_A\circ_p F_B$ can be extended to a section in
$\CB_{M_0,M_{[r\circ_p r']}^{A,B}}(U_{A\circ_p B})$.
Then, by Theorem \ref{thm_expansion}, the series 
$C_{A\circ_p B}(u, m_{[r\circ_p r']})$ is in fact in $T_{A\circ_p B}^\conv$.
Thus, the assertion holds.
\end{proof}
By the above proof, we have:
\begin{cor}
\label{cor_composition}
The sum
$\sum_{h \in \C}\langle m_0^A,F_A(m_1^A,\dots,m_{p-1}^A,\prk F_B(m_{[r']}^B),m_{p+1}^A,\dots,m_r^{A})\rangle$
converges uniformly and absolutely in $U_{A\circ_p B}$.
\end{cor}

\section{Pseudo-braided category structure on conformal blocks}
\label{sec_lax}
In this section, we will show that $\Vmodf$ inherits a natural pseudo-braided category structure.
By Proposition \ref{prop_app_operad}, it suffices to construct a sequence of functors
\begin{align}
\{\CPaB(n) \rightarrow \Endp_\Vmodf(n)\}_{n \geq 0},\label{eq_seq_func_recall}
\end{align}
with the composition maps \eqref{eq_mu_diag_add},
satisfying (LM0) - (LM3).
The definition of \eqref{eq_seq_func_recall} is given in Section \ref{sec_lax_def}.
Difficulties are in showing that 
the composition maps are natural transformations,
which will be shown in Section \ref{sec_lax_proof} under the preparation of Section \ref{sec_lax_ass}, \ref{sec_lax_braid} and \ref{sec_lax_vac}.


\subsection{Definition of lax 2-action}
\label{sec_lax_def}


Recall that $\CPaB(n)$ is isomorphic, as a category, to the full subcategory of $\Pi_1(X_n(\C))$ spanned by the points $\{Q_A\}_{A\in \Tr_{n}}$ (see Section \ref{sec_def_cpab}).
Here
\begin{align*}
Q: \Tr_n \;\longrightarrow\; X_n(\R), 
\qquad 
A \;\longmapsto\; Q(A) = (q_1^A, \dots, q_n^A)
\end{align*}
is any map satisfying (Q1) and (Q2).
From the assumptions (Q1) and (Q2), for $g \in \mathrm{Hom}_\CPaB(r)(A,A')$ there is a naturally defined path (up to homotopy)
$\ga(g):[0,1]\rightarrow X_r$
whose endpoints are $Q(A),Q(A') \in X_r(\C)$. We write $[g]_Q$ for the homotopy class of this path.
To connect with vertex operator algebras, in addition to the homotopical conditions (Q1) and (Q2) we assume
\begin{enumerate}
\item[Q3)]
$Q(A) \in U_A$ for any $A\in \Tr_n$.
\end{enumerate}
In this section, using analytic continuation of conformal blocks along paths between the points $\{Q(A)\}_{A\in \Tr_n}$, we construct the sequence of functors \eqref{eq_seq_func_recall} and show that the construction is independent of the choice of $Q$.


%
%
%
%

Define a functor $\PI_Q:\CPaB(n) \rightarrow \Endp_\Vmodf(n)$ as follows:
For an object $A\in \Tr_n$,
\begin{align*}
\PI_Q(A)=\PI_{A}:\Vmodf \times (\Vmodfo)^n \rightarrow \Vect.\qquad \text{(for $n \geq 1$)}
\end{align*}
For $n=0$, define $\PI_Q(\emptyset) \in \Endp_\Vmodf(0)=\Func(\Vmodf,\Vect)$ by
\begin{align*}
\PI_Q(\emptyset)=\mathrm{Hom}(V,\bullet): \Vmodf \rightarrow \Vect,\quad
M\mapsto \mathrm{Hom}(V,M).
\end{align*}
\begin{rem}
\label{rem_rep_V}
For $\PI_Q(\emptyset)$ to be a representative functor, the dual module $V^\vee$ must be $C_1$-cofinite;
since $V$ itself is always $C_1$-cofinite, this assumption is true if $V$ is a self-dual vertex operator algebra.
Whenever we deal with $\PI_Q(\emptyset)$ below (Section \ref{sec_lax_vac} and Section \ref{sec_lax_proof}), we assume that $V^\vee$ is finitely generated.
\end{rem}

Note that by Theorem \ref{thm_expansion}
and
Lemma \ref{lem_translation_isomorphism},
we have linear isomorphisms
\begin{align*}
s_A: \PI_A\binom{M_0}{\Mrr} \rightarrow \CB_{M_{[0;r]}^{r_A}}(U_A),\\
e_A:\CB_{M_{[r]}^{r_A}}(U_A)\rightarrow \PI_A\binom{M_0}{\Mrr},\\
Q_{r_A}: \CB_{\Mr}(U_A) \rightarrow \CB_{M_{[0;r]}^{r_A}}(U_A),
\end{align*}
which are natural for $M_0,M_1,\dots,M_r\in \Vmodf$.
For a morphism $g:A\rightarrow A'$ in $\CPaB(r)$,
let $\ga \in [g]_Q$
and
define a linear map $A_\ga:\CB_{\Mr}(U_A)\rightarrow \CB_{\Mr}(U_{A'})$
by the analytic continuation along the path $\ga$,
which is a natural transformation by Proposition \ref{prop_functor_monodromy}.
Since $\CB$ is a locally constant sheaf, $A_\ga$ is well-defined and
independent of the choice of $\ga\in [g]_Q$.
Define a natural transformation $\rho_Q(g):\PI_Q(A) \rightarrow \PI_Q(A')$ by
the composition
\begin{align}
\begin{split}
&\PI_A\binom{M_0}{\Mrr}
\stackrel{s_A}{\rightarrow}
\CB_{M_{[r]}^{r_A}}(U_A)
\stackrel{Q_{r_A}^{-1}}{\rightarrow}
\CB_{\Mr}(U_A)\\
&\stackrel{A_{[g]_Q}}{\rightarrow} \CB_{\Mr}(U_{A'})
\stackrel{Q_{r_{A'}}}{\rightarrow}
\CB_{M_{[r]}^{r_{A'}}}(U_{A'})
\stackrel{e_{A'}}{\rightarrow}
\PI_{A'}\binom{M_0}{\Mrr}
\end{split}
\label{eq_def_rho_comp}
\end{align}

\begin{lem}
\label{lem_indep_Q}
Let $Q_0,Q_1:\Tr_r\rightarrow \Xr$ satisfy (Q1), (Q2) and (Q3).
Then, $\PI_{Q_0}=\PI_{Q_1}$ as functors.
\end{lem}
\begin{proof}
For any $A\in \Tr_r$, since $Q_0(A),Q_1(A) \in U_A$ and $U_A$ is connected by Lemma \ref{lem_simply_conn},
there exists a path $p_A:[0,1]\rightarrow U_A$ such that $p_A(0)=Q_0(A)$ and $p_A(1)=Q_1(A)$.
Let $g:A\rightarrow A'$ be a morphism in $\CPaB(r)$
and $F\in \PI_A\binom{M_0}{\Mrr}$.
For any $\ga \in [g]_{Q_0}$, it is clear that $p_{A'}\ga p_A^{-1} \in [g]_{Q_1}$
since $U_A$ is simply-connected by Lemma \ref{lem_simply_conn}.
Hence, the assertion holds.
\end{proof}
By Lemma \ref{lem_indep_Q},
the functor $\PI_Q:\CPaB(r)\rightarrow \Endp_\Vmodf(r)$ is independent of the choice of $Q:\Tr_r\rightarrow \Xr$.
Hence, we simply denote it by $\PI$.

Recall that the symmetric group $S_r$ acts on
$\Endp_{\Vmodf}(r)$ by permuting the inputs.  Namely, for
$F\in \Endp_{\Vmodf}(r)$ and $g\in S_r$,
\[
(g\cdot F)\binom{M_0}{M_1,\ldots,M_r}
 =
F\binom{M_0}{M_{g^{-1}(1)},\ldots,M_{g^{-1}(r)}} .
\]
The functors $I_{gA}$ and $g\cdot I_A$ are not literally identical,
since the space of formal power series $T_{gA}$ and $T_A$ are defined using
different coordinate systems.  To make the equivariance strict, we
replace $I_A$ as follows.

By the bijection \eqref{eq_Tr_PW_Sr}, write
$A=(w_A,g_A) \in \PW_r \times S_r$ and put $\overline A=(w_A,\id)$.
We replace $I_A$ by the naturally isomorphic functor
\[
\widehat I_A =g_A\cdot I_{\overline A}.
\]
The natural isomorphism $\widehat I_A\simeq I_A$ is induced by the
permutation
\[
X(g_A):X_r\longrightarrow X_r,\qquad
(z_1,\ldots,z_r)\longmapsto
(z_{g_A^{-1}(1)},\ldots,z_{g_A^{-1}(r)}),
\]
together with the corresponding relabelling of the formal variables.
After transporting the morphisms $\rho(\gamma)$ through these isomorphisms,
we again denote the resulting functor by $\PI$.
Then, we have:
\begin{lem}
\label{rem_add_symmetry}
With this convention, the functor
\[
\PI:\CPaB(r)\longrightarrow \Endp_\Vmodf(r)
\]
is strictly equivariant with respect to the symmetric group action.  In particular, condition {\rm (LM0)} holds.
\end{lem}

Now we have obtained a family of functors $\{\PI:\CPaB(r)\rightarrow \Endp_\Vmodf(r)\}_{r\geq 0}$ and 
the maps 
$\comp_{A,B_1,\dots,B_n}:\PI(A)\circ(\PI(B_1),\dots,\PI(B_n))\rightarrow \PI\left( A\circ (B_1,\dots,B_n)\right)$ defined in Theorem \ref{thm_composition}.
To show that $(\PI,\comp^\PI)$ is a lax 2-morphism, it remains to verify the
following:
\begin{enumerate}
\item
Extend the definition of $\comp^\PI$ to the case where one of $B_i$ is
$\emptyset$, {\rm  (LM2)}.
\item
Show the associativity of $\comp^\PI$, {\rm (LM1)}.
\item
Show that $\comp^\PI$ is a natural transformation between the functors $\CPaB(n) \times \CPaB(m) \rightarrow
\Endp_\Vmodf(n+m-1)$.
That is, the following diagram commutes 
\begin{align}
\begin{split}
\begin{array}{ccc}
\int_{N\in \Vmodf} \PI_{A}\binom{\bullet}{\bullet N \bullet} \otimes \PI_B\binom{N}{\bullet}
      &\overset{\comp_{A,B,p}^\PI}{\longrightarrow}&
\PI_{A\circ_p B}\binom{\bullet}{\bullet}
    \\
    {}_{{\rho(\ga_A)\otimes \rho(\ga_B)}}\downarrow 
      && 
    \downarrow_{{\rho(\ga_A\circ_p \ga_B)}}
    \\
\int_{N\in \Vmodf} \PI_{A'}\binom{\bullet}{\bullet N \bullet} \otimes \PI_{B'}\binom{N}{\bullet}
      &\overset{\comp_{A',B',p}^\PI}{\longrightarrow}&
\PI_{A'\circ_p B'}\binom{\bullet}{\bullet}
\label{eq_lax_diagram}
\end{array}
\end{split}
\end{align}
for any $A,A'\in \Tr_n$, $B,B'\in \Tr_m$
and $\CPaB$-morphisms $\ga_A:A\rightarrow A'$, $\ga_B:B\rightarrow B'$ and $p \in \{1,\dots,n\}$.
\item
Check the compatibility of $\comp^\PI$ with the symmetric group action, {\rm (LM3)}.
\end{enumerate}

Conditions (2) and (4) are immediate from the construction.  Hence it
remains to verify (1) and (3).

Condition (1) will be treated in Section~\ref{sec_lax_vac}.  For (3),
morphisms in $\CPaB$ are generated by alpha-type maps and crossings.  In
the following two sections, we verify the commutativity of
diagram~\eqref{eq_lax_diagram} for these two types of morphisms.
%
%

%
%
%
%
%

\subsection{Compatibility for associator}
\label{sec_lax_ass}

Let $r \geq 3$ and $A\in \Tr_r$.
Recall that an edge $e \in E(A)$ is called {\it alpha-type} if
the vertex on the right side of $d(e)$ is not a leaf. 
Recall also that we defined a $\CPaB(r)$-morphism
$\alpha_e:A\rightarrow \al_e A$ for an alpha-type edge $e\in E(A)$
(see Section \ref{sec_def_cpab} and Definition \ref{def_alpha}).
%
%
The following lemma is important:
\begin{lem}
\label{lem_alpha_intersection}
Let $A\in \Tr_r$ and $\ee\in E(A)$ be alpha-type.
Then, 
there exists $q \in U_A\cap U_{\alpha_\ee A}\cap X_r(\R)$ such that
the order of $q$ and $A$ are the same.
\end{lem}
\begin{proof}

Assume that the tree $A$ can be written locally as in Fig. \ref{fig_local_A1}.
In this case, $\al_{e_0} A$ can be written as in Fig. \ref{fig_local_A2}.
We will study the change of coordinates between $A$-coordinate $\{\zeta_e\}_{e\in E(A)}$ and $\al_\ee A$-coordinate $\{\zeta_{e'}\}_{e'\in E(\al_\ee A)}$.

Note that if we take an edge $e$ from a part of the figure written in abbreviated form as $\blacksquare$, 
$\zeta_e=\zeta_{e'}$ holds,
where $e'$ is the edge of $\al_\ee A$ corresponding to $e$ in an obvious way.

\begin{minipage}[c]{7cm}
\centering
\begin{forest}
for tree={
  l sep=20pt,
  parent anchor=south,
  align=center
}
[
[$np$
[$jn$,edge label={node[midway,left]{\bf e}}
[$ij$,edge label={node[midway,left]{a}}[$\blacksquare$]]
[$ln$,edge label={node[midway,right]{d}}
[$kl$,edge label={node[midway,left]{b}}[$\blacksquare$]]
[$mn$,edge label={node[midway,right]{c}}[$\blacksquare$]]
]
]
[$\blacksquare$]
]
]
\end{forest}
\captionof{figure}{}
\label{fig_local_A1}
\end{minipage}
\hfill
\begin{minipage}[c]{7cm}
\centering
\begin{forest}
for tree={
  l sep=20pt,
  parent anchor=south,
  align=center
}
[
[$np$
[$ln$,edge label={node[midway,left]{\bf e'}}
[$jl$,edge label={node[midway,left]{d'}}
[$ij$,edge label={node[midway,left]{a'}}[$\blacksquare$]]
[$kl$,edge label={node[midway,right]{b'}}[$\blacksquare$]
]
]
[$mn$,edge label={node[midway,right]{c'}}[$\blacksquare$]]
]
[$\blacksquare$]
]
]
]
\end{forest}
\captionof{figure}{}
\label{fig_local_A2}
\end{minipage}
Since
\begin{align*}
\ze_a&=\frac{z_{ij}}{z_{jn}}=\ze_{a'}\frac{\ze_{d'}}{1+\ze_{d'}}
,\quad 
&\ze_{a'}&=\frac{z_{ij}}{z_{jl}}=\frac{\ze_a}{1-\ze_d},\\
\ze_b&=\frac{z_{kl}}{z_{ln}}=\ze_{b'}\ze_{d'},\quad
&\ze_{b'}&=\frac{z_{kl}}{z_{jl}}=\ze_b\frac{\ze_d}{1-\ze_d}
\\
\ze_c&=\frac{z_{mn}}{z_{ln}}=\ze_{c'},\quad
&\ze_{c'}&=\frac{z_{mn}}{z_{ln}}=\ze_{c}\\
\ze_d&=\frac{z_{ln}}{z_{jn}}=\frac{1}{1+\ze_{d'}}
,\quad
&\ze_{d'}&=\frac{z_{jl}}{z_{ln}}=\frac{1-\ze_d}{\ze_d}\\
\ze_e&=\frac{z_{jn}}{z_{np}}=\ze_{e'}(1+\ze_{d'})
,\quad
&\ze_{e'}&=\frac{z_{ln}}{z_{np}}=\ze_d\ze_e,
\end{align*}
if all $\ze_a,\ze_b,\ze_c,\ze_e$ are sufficiently small
and $\ze_d=\frac{2}{3}$, for example,
then $\ze_{a'},\ze_{b'},\ze_{c'},\ze_{e'}$ are small
and $\ze_{d'}=\frac{1}{2}$.
Thus, by Lemma \ref{lem_one_fix}, the assertion holds for this tree.

If any of $ij, kl,mn$ is a leaf (see Fig. \ref{fig_local_A3}) or $p$ is on the other side (see Fig. \ref{fig_local_A4}), a similar calculation can be used to show the assertion.

\begin{minipage}[c]{7cm}
\centering
\begin{forest}
for tree={
  l sep=20pt,
  parent anchor=south,
  align=center
}
[
[$np$
[$jn$,edge label={node[midway,left]{\bf e}}
[$ij$,edge label={node[midway,left]{a}}[$\blacksquare$]]
[$ln$,edge label={node[midway,right]{d}}
[$l$,]
[$mn$,edge label={node[midway,right]{c}}[$\blacksquare$]]
]
]
[$\blacksquare$]
]
]
\end{forest}
\captionof{figure}{}
\label{fig_local_A3}
\end{minipage}
\hfill
\begin{minipage}[c]{7cm}
\centering
\begin{forest}
for tree={
  l sep=20pt,
  parent anchor=south,
  align=center
}
[
[$pn$
[$\blacksquare$]
[$jn$,edge label={node[midway,right]{\bf e'}}
[$ij$,edge label={node[midway,left]{a}}[$\blacksquare$]]
[$ln$,edge label={node[midway,right]{d}}
[$kl$,edge label={node[midway,left]{b}}[$\blacksquare$]]
[$mn$,edge label={node[midway,right]{c}}[$\blacksquare$]]
]
]
]
]
\end{forest}
\captionof{figure}{}
\label{fig_local_A4}
\end{minipage}
\end{proof}

\begin{rem}
Since
$U_{12}=\{(z_1,z_2)\in X_2\mid |\arg(z_1-z_2)|<\pi\}$,
$U_{12} \cap U_{21}\cap X_2(\R) =\emptyset$.
\end{rem}


\begin{lem}
\label{lem_alpha_int}
Let $A\in \Tr_r$ and $e_0\in E(A)$.
Assume that the edge $e_0$ is $\alpha$-type
and set $A'=\alpha_\ee A$.
For $F \in \PI_A\binom{M_0}{\Mrr}$ and $G \in \PI_{A'}\binom{M_0}{\Mrr}$,
$\rho(\alpha_\ee)F=G$ if and only if
$s_A(F)|_{U_A\cap U_{A'}}=s_{A'}(G)|_{U_A\cap U_{A'}}$ as holomorphic functions.
\end{lem}
\begin{proof}
By Lemma \ref{lem_alpha_intersection},
$Q:\Tr_r\rightarrow X_r$ can be taken to satisfy $Q_{A}=Q_{\al_\ee A}$ and (Q1), (Q2) and (Q3).
Then, $\rho(\alpha_\ee)F = e_{A'}(s_A(F))$. Thus, the assertion holds.
\end{proof}

We will use the following lemma:
\begin{prop}
\label{prop_alpha_compati}
Let $n,m\geq 1$ and
$A\in \Tr_n$, $B\in \Tr_m$ and $p \in \{1,\dots,n\}$.
\begin{enumerate}
\item
For any edge $\ee \in E(A)$ of $\alpha$-type,
the following diagram commutes:
\begin{align*}
\begin{split}
\begin{array}{ccc}
\int_{N\in \Vmodf} \PI_{A}\binom{\bullet}{\bullet N \bullet} \otimes \PI_B\binom{N}{\bullet}
      &\overset{\comp_{A,B,p}}{\longrightarrow}&
\PI_{A\circ_p B}\binom{\bullet}{\bullet}
    \\
    {}^{{\rho(\al_\ee)\otimes \rho(\id_B)}}\downarrow 
      && 
    \downarrow^{{\rho(\al_\ee \circ_p \id_B})}
    \\
\int_{N\in \Vmodf} \PI_{\al_\ee A}\binom{\bullet}{\bullet N \bullet} \otimes \PI_{B}\binom{N}{\bullet}
      &\overset{\comp_{\al_\ee A,B,p}}{\longrightarrow}&
\PI_{\al_\ee A\circ_p B}\binom{\bullet}{\bullet}.
\end{array}
\end{split}
\end{align*}
\item
For any edge $\ee \in E(B)$ of $\alpha$-type,
the following diagram commutes:
\begin{align*}
\begin{split}
\begin{array}{ccc}
\int_{N\in \Vmodf} \PI_{A}\binom{\bullet}{\bullet N \bullet} \otimes \PI_B\binom{N}{\bullet}
      &\overset{\comp_{A,B,p}}{\longrightarrow}&
\PI_{A\circ_p B}\binom{\bullet}{\bullet}
    \\
    {}^{{\rho(\id_A)\otimes \rho(\al_\ee )}}\downarrow 
      && 
    \downarrow^{{\rho(\id_A\circ_p \al_\ee)}}
    \\
\int_{N\in \Vmodf} \PI_{A}\binom{\bullet}{\bullet N \bullet} \otimes \PI_{\al_\ee B}\binom{N}{\bullet}
      &\overset{\comp_{A,\al_\ee B,p}}{\longrightarrow}&
\PI_{A \circ_p \al_\ee B}\binom{\bullet}{\bullet}.
\end{array}
\end{split}
\end{align*}
\end{enumerate}
\end{prop}
\begin{proof}
We only show the case of (1).
We will use the same symbols as those used in Lemma \ref{lem_comp_formalsp}.
Let $F_A \in \PI_A\binom{M_0^A}{M_{[n]}^A}$ and $F_B \in \binom{M_p^A}{M_[m]^B}$
and $\{e_{h,u}\}_{u \in I_h}$ be a basis of $(M_p^A)_h$ and $\{e_{h,u}^*\}_{u \in I_h}$ the dual basis
for $h\in \C$.
By Lemma \ref{lem_alpha_int}, for any $h \in \C$ and $u\in I_h$
\begin{align*}
\langle &m_0^A,F_A(m_1^A,\dots,m_{p-1}^A,e_{h,u},m_{p+1}^A,\dots,m_n^{A})\rangle|_{U_A \cap U_{\al_\ee A}}\\
&= \langle m_0^A,(\rho(\al_\ee)F_A)(m_1^A,\dots,m_{p-1}^A,e_{h,u},m_{p+1}^A,\dots,m_n^{A})\rangle
|_{U_A \cap U_{\al_\ee A}},
\end{align*}
which implies that
\begin{align*}
\langle &m_0^A,F_A(m_1^A,\dots,m_{p-1}^A,e_{h,u},m_{p+1}^A,\dots,m_n^{A})\rangle
\langle e_{h,u}^*,F_B(m_{[m]}^B)\rangle|_{U_{A\circ_p B} \cap U_{\al_\ee A\circ_p B}}\\
&= \langle m_0^A,(\rho(\al_\ee)F_A)(m_1^A,\dots,m_{p-1}^A,e_{h,u},m_{p+1}^A,\dots,m_n^{A})\rangle
\langle e_{h,u}^*,F_B(m_{[m]}^B)\rangle|_{U_{A\circ_p B} \cap U_{\al_\ee A\circ_p B}}.
\end{align*}
Since $\dim(M_p^A)_h$ is finite,
\begin{align*}
\langle m_0^A, F_A(m_1^A,&\dots,m_{p-1}^A,\prk F_B(m_{[m]}^B),m_{p+1}^A,\dots,m_n^{A})\rangle\\
&=\sum_{u\in I_h} 
\langle m_0^A,F_A(m_1^A,\dots,m_{p-1}^A,e_{h,u},m_{p+1}^A,\dots,m_n^{A})\rangle
\langle e_{h,u}^*,F_B(m_{[m]}^B)\rangle
\end{align*}
is equal to 
\begin{align*}
\langle m_0^A, (\rho(\al_\ee)F_A)(m_1^A,&\dots,m_{p-1}^A,\prk F_B(m_{[m]}^B),m_{p+1}^A,\dots,m_n^{A})\rangle
\end{align*}
on $U_{A\circ_p B} \cap U_{\al_\ee A\circ_p B}$ as holomorphic functions.
By Corollary \ref{cor_composition},
\begin{align*}
\sum_{k \in \C} \langle m_0^A, F_A(m_1^A,\dots,m_{p-1}^A,\prk F_B(m_{[m]}^B),m_{p+1}^A,\dots,m_n^{A})\rangle
\end{align*}
converges to absolutely and locally uniformly $\langle m_0^A, (F_A\circ_p F_B)(m_{[n\circ_p m]})\rangle$ in $U_{A\circ_p B}$
and similarly
$\langle m_0^A, (\rho(\al_\ee) F_A)(m_1^A,\dots,m_{p-1}^A,\prk F_B(m_{[m]}^B),m_{p+1}^A,\dots,m_n^{A})\rangle$
converges absolutely and locally uniformly to $\langle m_0^A, ((\rho(\al_\ee) F_A)\circ_p F_B)(m_{[n\circ_p m]})\rangle$
in $U_{\al_\ee A\circ_p B}$.
Hence, by Lemma \ref{lem_alpha_int},
\begin{align*}
\rho(\al_\ee)(F_A\circ_p F_B)= (\rho(\al_\ee) F_A)\circ_p F_B,
\end{align*}
which implies the assertion.
\end{proof}

\subsection{Compatibility for crossing}
\label{sec_lax_braid}
Let $v_0 \in V(A)$. 
Then, $A$ can be written as $A=B\circ_p (12) \circ (A_1,A_2)$ for some trees $A_1,A_2,B$ (Fig. \ref{fig_local_sigma1}).

\vspace{4mm}

\begin{minipage}[l]{7cm}
\centering
\begin{forest}
for tree={
  l sep=20pt,
  parent anchor=south,
  align=center
}
[$\fbox{B}$,edge label={node[midway,right]{p}}
[$v_0$
[\fbox{$A_1$}]
[\fbox{$A_2$}]
]
]
\end{forest}
\captionof{figure}{}
\label{fig_local_sigma1}
\end{minipage}
\hfill
\begin{minipage}[r]{7cm}
\centering
\begin{forest}
for tree={
  l sep=20pt,
  parent anchor=south,
  align=center
}
[$\fbox{B}$,edge label={node[midway,right]{p}}
[$v_0$
[\fbox{$A_2$}]
[\fbox{$A_1$}]
]
]
\end{forest}
\captionof{figure}{}
\label{fig_local_sigma2}
\end{minipage}
\vspace{4mm}


Set $\si_{v_0} A=B\circ_p (21) \circ (A_1,A_2)$ (Fig. \ref{fig_local_sigma2})
and define a morphism
$\si_{v_0}:A\rightarrow \si_{v_0} A$ by
\begin{align}
\si_{v_0}= \id_B \circ_p \si \circ (\id_{A_1},\id_{A_2}):
B\circ_p (12) \circ (A_1,A_2)\rightarrow 
B\circ_p (21) \circ (A_1,A_2).\label{eq_add_sigma_v}
\end{align}
Here \(\sigma:(12)\to(21)\) is the morphism in \(\CPaB(2)\) whose geometric
representative is fixed in Remark~\ref{rem_fix_braid}. Throughout this
section, the path representing \(\sigma_{v_0}\) is obtained from that
representative by operadic insertion.
The goal of this section is to explicitly represent the analytic continuation $A(\si_{v_0})$ and show the naturality of $\comp_{A,B,p}$ for $\si_{v_0}$.

We begin by explaining the basic idea.
Consider the particular case.
Suppose $A=1(2((34)(56)))$ and $v_0=46$ as in Fig. \ref{fig_local_sigma3}.

\begin{minipage}[l]{7cm}
\centering
\begin{forest}
for tree={
  l sep=20pt,
  parent anchor=south,
  align=center
}
[16
  [1
  ]
  [26,edge label={node[midway,right]{d}}
    [2
    ]
    [{\bf 46},edge label={node[midway,right]{c}}
      [34,edge label={node[midway,left]{a}}
        [3
        ]
        [4
        ]
      ]
      [56,edge label={node[midway,right]{b}}
        [5
        ]
        [6
        ]
      ]
    ]
  ]
]
\end{forest}
\captionof{figure}{}
\label{fig_local_sigma3}
\end{minipage}
\hfill
\begin{minipage}[r]{7cm}
\centering
\begin{forest}
for tree={
  l sep=20pt,
  parent anchor=south,
  align=center
}
[14
  [1
  ]
  [24,edge label={node[midway,right]{d}}
    [2
    ]
    [{\bf 64},edge label={node[midway,right]{c}}
      [56,edge label={node[midway,left]{b}}
        [5
        ]
        [6
        ]
      ]
      [34,edge label={node[midway,right]{a}}
        [3
        ]
        [4
        ]
      ]
    ]
  ]
]
\end{forest}
\captionof{figure}{}
\label{fig_local_sigma4}
\end{minipage}

Let $\{x_v\}_{v\in V(A)}$ denote the $x$-coordinate of $A$ and $\{y_w\}_{w\in V(\si_{v_0}A)}$ the $x$-coordinate of $\si_{v_0}A$.
Note that there is a natural one-to-one correspondence between $V(A)$ and $V(\si_{v_0}(A))$.
Under this identification, we notice the following relation:
\begin{align}
\begin{split}
x_v &=y_v - y_{v_0},\quad (\text{if }v\text{ is above }v_0\text{ and }R(v)=R(v_0))\\
x_{v_0} &= -y_{v_0}, \quad (v=v_0)\\
x_v & = y_v,\quad \text{ (otherwise)} .
\end{split}
\label{eq_xy_rel}
\end{align}
\begin{lem}
\label{lem_xy_coordinate}
The relation \eqref{eq_xy_rel} holds for any $A\in \Tr_r$ and $v_0 \in V(A)$.
\end{lem}
The analytic continuation along $\si_{v_0}$ can be roughly described using 
a $\Drt$-module homomorphism
\begin{align*}
\si_{v_0}^\alg: T_A^\conv\rightarrow T_{\si_{v_0}A}^\conv
\end{align*}
induced by the change of variables \eqref{eq_xy_rel} and 
\begin{align*}
(\si_{v_0}^\alg)_*:
\mathrm{Hom}(D_{M_{[0;r]}^{r_A}},T_A^\conv)\rightarrow \mathrm{Hom}(D_{M_{[0;r]}^{r_{\si_{v_0}A}}},T_{\si_{v_0} A}^\conv).
\end{align*}

Continue to consider the case of Fig.s \ref{fig_local_sigma3} and \ref{fig_local_sigma4}.
We will write $A$-coordinate using $\zeta$ and 
$\si_{v_0} A$-coordinate using $\xi$.
Then, we have
\begin{align}
\zeta_a &=\frac{z_{34}}{z_{46}}=-\xi_a, &\xi_a &= \frac{z_{34}}{z_{64}}=-\zeta_a\nonumber
\\
\zeta_b &=\frac{z_{56}}{z_{46}}=-\xi_b, &\xi_b &= \frac{z_{56}}{z_{64}}=-\zeta_b\nonumber\\
\zeta_c &=\frac{z_{46}}{z_{26}}=-\frac{\xi_c}{1-\xi_c}, &\xi_c &= \frac{z_{64}}{z_{24}}
=-\frac{\zeta_c}{1-\zeta_c}\label{eq_A_siA0}\\
\zeta_d &=\frac{z_{26}}{z_{16}}=\xi_d\frac{(1-\xi_c)}{(1-\xi_c\xi_d)}, &\xi_d &= \frac{z_{24}}{z_{14}}=\zeta_d \frac{1-\zeta_c}{1-\zeta_c\zeta_d}
\nonumber\\
x_A &=z_{16}=x_{\si_v A}(1-\xi_c \xi_d), &x_{\si_v A} &= z_{14}=x_A(1-\ze_c\ze_d).\nonumber
\end{align}

The reader will notice that if $\zeta_e$ are all small, then $\xi_e$ are all small too, and vice versa.
Thus $U_A \cap U_{\si_{v_0} A}$ is a non-empty set.
Let $\mathfrak{p}\in \R_{>0}^{E(A)}$ be $A$-admissible
and $Q=(q_1,\dots,q_6) \in U_A^\mathfrak{p}$.
Define a path $\ga:[0,1]\rightarrow X_6$ as follows according to Fig. \ref{fig_sigmav}:
\begin{figure}[t]
  \begin{minipage}[l]{7cm}
    \centering
    \includegraphics[width=5cm]{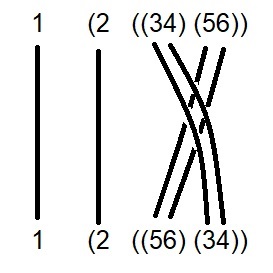}
    \captionof{figure}{}
\label{fig_sigmav}
  \end{minipage}
\end{figure}
\begin{align*}
\ga(t)&=(q_1(t),q_2(t),\dots,q_6(t))\\
&=\biggl(q_1,q_2,q_3+\frac{\exp(-\pi i t)-1}{2}(q_4-q_6),
q_4+\frac{\exp(-\pi i t)-1}{2}(q_4-q_6),\\
&q_5-\frac{\exp(-\pi i t)-1}{2}(q_4-q_6),
q_6-\frac{\exp(-\pi i t)-1}{2}(q_4-q_6)\biggr)
\end{align*}
By the convention fixed in Remark \ref{rem_fix_braid}, the path \(\gamma\) represents the homotopy class \([\sigma_{v_0}]_Q\) in \eqref{eq_add_sigma_v}. In fact, it is the unique path which satisfies:
\begin{itemize}
\item
$\ga(0)=Q$;
\item
$q_1(t), q_2(t),q_{3}(t)-q_4(t),q_{5}(t)-q_{6}(t),q_4(t)+q_6(t)$ are constant;
\item
$q_4(t)-q_6(t)=\exp(-\pi i t)(q_4-q_6)$.
\end{itemize}


In $A$-coordinate and $\si_{v_0} A$-coordinate, this path is expressed as follows:
\begin{align*}
\zeta_a(\ga(t)) &=\exp(\pi i t)\frac{q_{34}}{q_{46}}, &\xi_a(\ga(t)) 
&= -\exp(\pi i t) \frac{q_{34}}{q_{46}}\\
\zeta_b(\ga(t)) &=\exp(\pi i t)\frac{q_{56}}{q_{46}}, &\xi_b(\ga(t)) &= -\exp(\pi i t)\frac{q_{56}}{q_{46}}\\
\zeta_c(\ga(t)) &= \exp(-\pi i t) \frac{q_{46}}{q_{26}}\frac{1}{1+\frac{\exp(-\pi i t)-1}{2}\frac{q_{46}}{q_{26}}}, 
&\xi_c(\ga(t)) &= -\exp(-\pi i t) \frac{q_{46}}{q_{24}}\frac{1}{1-\frac{\exp(-\pi i t)-1}{2}\frac{q_{46}}{q_{24}}}
\\
\zeta_d(\ga(t)) &=\frac{q_{26}}{q_{16}}\frac{1+\frac{\exp(-\pi i t)-1}{2}\frac{q_{46}}{q_{26}}}{1+\frac{\exp(-\pi i t)-1}{2}\frac{q_{46}}{q_{16}}}, 
&\xi_d(\ga(t)) &= \frac{q_{24}}{q_{14}}\frac{1-\frac{\exp(-\pi i t)-1}{2}\frac{q_{46}}{q_{24}}}{1-\frac{\exp(-\pi i t)-1}{2}\frac{q_{46}}{q_{14}}}\\
x_A &=q_{16}\left(1+\frac{\exp(-\pi i t)-1}{2}\frac{q_{46}}{q_{16}}\right), &x_{\si_v A} &= q_{14}\left(1-\frac{\exp(-\pi i t)-1}{2}\frac{q_{46}}{q_{14}}\right),
\end{align*}
where $q_{ij}=q_i-q_j$.
Assume that $\mathfrak{p}$ is sufficiently small.
Then, by \eqref{eq_A_siA0},
$\xi_x(Q)$ is also small for $x=a,b,c,d$.
Hence, $\zeta_c(\ga(t))$ can be approximated by $\exp(-\pi i t)\frac{q_{46}}{q_{26}}$
and $\xi_c(\ga(t))$ can be approximated by $-\exp(-\pi i t) \frac{q_{24}}{q_{14}}$.
This is written as $\zeta_c(\ga(t))\sim \exp(-\pi i t)\frac{q_{46}}{q_{26}}$ and so on.

Thus, taking into account that $q_{ij}>0$ if $i<j$ and the branch cut
\begin{align*}
\D_p^\cut=\{z\in \C \mid 0<|z|<p, \;\;-\pi< \Arg(z)<\pi \},
\end{align*}
we find that there exists a small $\ep >0$ such that
\begin{align*}
\ga(t) &\in U_A\quad\quad \text{ if } 0\leq t<\ep,\\
\ga(t) &\in U_{\si_{v_0} A}\quad\; \text{ if } 1-\ep < t\leq 1.
\end{align*}
The same statement holds for a general $A\in \Tr_r$ and $v_0\in V(A)$.

Consider how the analytic continuation along the path $\ga$ of a function in $T_A^\conv$ 
can be written in the $\si_{v_0} A$-coordinate.
Note that the branch of $\Log$ is defined so that the value of $\Arg:\C^\cut \rightarrow \C$ is $(-\pi,\pi)$
(see \eqref{eq_log_def} in Section \ref{sec_config_conv}).
Then, along the path $\ga$, we have
\begin{align}
\begin{split}
\zeta_c(\ga(t))^r&=\exp(r \Log\ze_c(\ga(t)))=\exp(r \Log\left(-\frac{\xi_c(\ga(t))}{1-\xi_c(\ga(t))}\right))\\
&=\exp(r(-\pi i + \Log\left(\frac{\xi_c(\ga(t))}{1-\xi_c(\ga(t))}\right)))=\exp(-r\pi i)\left(\frac{\xi_c(\ga(t))}{1-\xi_c(\ga(t))}\right)^r\\
&=\exp(-\pi ir)\xi_c(\ga(t))^r \sum_{n\geq 0}(-1)^n \binom{-r}{n}\xi_c(\ga(t))^n,
\label{eq_zeta_cr}
\end{split}
\end{align}
where we used $\Arg (\zeta_c(\ga(t)))\sim -\pi i t$ (resp. $\Arg (\xi_c(\ga(t)))\sim \pi i -\pi it$) and
\begin{align}
\begin{split}
\Log \zeta_c(\ga(t)) &= \Log\left(-\frac{\xi_c(\ga(t))}{1-\xi_c(\ga(t))}\right) = -\pi i+\Log \left(\frac{\xi_c(\ga(t))}{1-\xi_c(\ga(t))}\right)\\
&=-\pi i+\Log \xi_c(\ga(t))+\sum_{n\geq 1}\frac{1}{n}\xi_c(\ga(t))^n.
\label{eq_zeta_crlog}
\end{split}
\end{align}
Similarly, by $\Arg (\ze_a(\ga(t)))\sim   \pi i t$ (resp. $\Arg (\xi_a(\ga(t))) \sim -\pi i +\pi i t$), for small $t>0$,
\begin{align}
\begin{split}
\zeta_a(\ga(t))^r&=\exp(\pi i r)\xi_a(\ga(t))^r,\\
\Log(\zeta_a(\ga(t)))&=\pi i +\Log(\xi_a(\ga(t))).
 \label{eq_zeta_a}
\end{split}
\end{align}

Based on this computation,
we consider the following $\C$-algebra homomorphism
\begin{align*}
\si_{v_0}^\alg:\C[\zeta_e^\C,\log\ze_e,x_A^\C,\log x_A \mid e\in E(A)] \rightarrow T_{\si_{v_0} A}:
\end{align*}

For any $r\in \C$,
\begin{align}
\si_{v_0}^\alg(\zeta_a^{r})&=\exp(\pi i r)\xi_a^r,
&\si_{v_0}^\alg(\log\ze_a)&=\pi i + \log\xi_a,\nonumber\\
\si_{v_0}^\alg(\zeta_b^r)&=\exp(\pi i r)\xi_b^r,
&\si_{v_0}^\alg(\log \zeta_b)&=\pi i + \log\xi_b,\nonumber\\
\si_{v_0}^\alg(\zeta_c^r)&=\exp(-\pi i r)\xi_c^r\sum_{n\geq 0}(-1)^n\binom{-r}{n}\xi_c^n,
&\si_{v_0}^\alg(\log \zeta_c)&=-\pi i + \log(\xi_c) + \sum_{n\geq 1}\frac{1}{n}\xi_c^n,\label{eq_si_alg}\\
\si_{v_0}^\alg(\zeta_d^r)&=\xi_d^r\sum_{m,n\geq 0}(-1)^{n+m}\binom{-r}{n}\binom{r}{m}
\xi_c^{m+n}\xi_d^n
&\si_{v_0}^\alg(\log \zeta_d)&=\log(\xi_d) + \sum_{n \geq 1}\frac{1}{n}\xi_c^n(\xi_d^n-1),\nonumber\\
\si_{v_0}^\alg(x_A^r)&=x_{\si_vA}^r\sum_{n\geq 0}\binom{r}{n}(-\xi_c\xi_d)^n
&\si_{v_0}^\alg(\log x_A)&=\log x_{\si_vA} - \sum_{n\geq 1} \frac{1}{n}(\xi_c\xi_d)^n.\nonumber
\end{align}
It is clear that this $\C$-algebra homomorphism can be extended to a $\C$-algebra homomorphism
from $T_A$ to $T_{\si_{v_0} A}$.
We will also denote it by $\si_{v_0}^\alg$.

\begin{rem}
\label{rem_change}
The map $\si_{v_0}^\alg$ is induced from the change of variables in Lemma \ref{lem_xy_coordinate} (being careful how to take the branch).
For example,
\begin{align*}
\ze_d&=\frac{z_{26}}{z_{16}}=\frac{z_{24}-z_{64}}{z_{14}-z_{64}}\\
&=\frac{z_{24}}{z_{14}}\frac{1-\frac{z_{64}}{z_{24}}}{1-\frac{z_{64}}{z_{14}}}
=\xi_d\frac{1-\xi_c}{1-\xi_c\xi_d},
\end{align*}
which is $\si_{v_0}^\alg(\ze_d)$.
\end{rem}
It is clear that we can define a similar map $\si_{v_0}^\alg:T_A\rightarrow T_{\si_{v_0}A}$ for general $A\in \Tr_r$ and ${v_0}\in V(A)$.
By the above argument, we have the following lemma:
\begin{lem}
\label{lem_D_sigma_isom}
For any $A\in \Tr_r$ and ${v_0}\in V(A)$,
$\si_{v_0}^\alg:T_A\rightarrow T_{\si_{v_0}A}$ is a $\Dr$-module isomorphism.
Moreover, there exists an $\si_{v_0}A$-admissible $\mathfrak{q} \in \R_{>0}^{E(\si_{v_0}A)}$ such that the image of $T_A^\conv$ is in $T_{\si_{v_0}A}^\mathfrak{q}$.
\end{lem}
\begin{proof}
Since what is being done is the change of variables (see Remark \ref{rem_change}), the first claim is obvious.
We will show that $T_A^\conv$ maps to $T_{\si_vA}^\conv$.
Since all series in \eqref{eq_si_alg}
are convergent, it suffices to show that the image of 
any convergent formal power series in $\C[[\ze_e\mid e\in E(A)]]$
converges absolutely in $U_{\si_{v_0}A}^\mathfrak{q}$
for some $\mathfrak{q}$.
This is evident from the following calculations:
For any convergent series $\sum_{l\geq 0}a_l\ze_c^l$,
\begin{align*}
\si_v^\alg(\sum_{l\geq 0}a_l\ze_c^l)
&=\sum_{l\geq 0}a_l (-\xi_c)^l \left(\sum_{n\geq 0}(-1)^n\binom{-l}{n}\xi_c^n\right)
\end{align*}
and
\begin{align*}
\sum_{l,n\geq 0}|(-1)^n\binom{-l}{n}||a_l| |\xi_c|^{l+n}
&=\sum_{l\geq 0} |a_l|\left(\frac{|\xi_c|}{1-|\xi_c|}\right)^l,
\end{align*}
which is convergent if $|\xi_c|$ is small enough.
We omit the details.
\end{proof}

Recall that $r_A$ is the rightmost leaf among all leaves of $A$.
$r_A\neq r_{\si_{v_0}A}$ if and only if $v_0$ is on the rightmost side of the tree.
Note that in this case $z_{r_A}- z_{r_{\si_{v_0}A}}$ can be written as $y_{v_0}$
in $x$-coordinate of $\si_{v_0}A$.
\begin{prop}
\label{prop_crossing_alg}
Let $A\in \Tr_r$ and $v_0 \in V(A)$ and $F_A \in \PI_A\binom{M_0}{\Mrr}$.
Then,
\begin{align*}
\rho(\si_{v_0})(F_A)=
\exp\left((z_{r_A}-z_{r_{\si_{v_0}A}})L(-1)\right) \si_{v_0}^\alg(F_A).
\end{align*}
\end{prop}
\begin{proof}
Set $A'=\si_{v_0}A$ and let $\mrr\in \Mrr$ and $u\in M_0^\vee$.
By \eqref{eq_def_rho_comp} and the above arguments,
\begin{align*}
\langle u, \rho(\si_{v_0})(F_A)(\mrr) \rangle&=
e_{A'}Q_{r_{A'}} A_{\ga} Q_{r_A}^{-1}s_A (\langle u,  F_A(\mrr)\rangle)\\
&=
e_{A'}Q_{r_{A'}} A_{\ga} s_A (\langle \exp\left(z_{r_A}L(-1)^* \right) u,  F_A(\mrr)\rangle)\\
&=(\langle \exp\left((z_{r_A}-z_{r_{A'}})L(-1)^* \right) u,  \si_{v_0}^\alg F_A(\mrr)\rangle)\\
&=(\langle u, \exp\left((z_{r_A}-z_{r_{A'}})L(-1)\right) \si_{v_0}^\alg F_A(\mrr)\rangle).
\end{align*}
\end{proof}

Define a linear map $B_{v_0}:T_A^\conv \rightarrow T_A^\conv$
by 
\begin{align*}
x_v^r &\mapsto x_v^r\qquad \text{ if }v\neq v_0\\
\log x_v &\mapsto \log x_v\qquad \text{ if }v\neq v_0\\
x_{v_0}^r &\mapsto \exp(-2\pi i r) x_{v_0}^r \\
\log x_{v_0} &\mapsto -2\pi i +\log x_{v_0}
\end{align*}
for $r\in \C$. Then, by \eqref{eq_zeta_crlog}, we have:
\begin{cor}
\label{cor_twist}
Let $A\in \Tr_r$ and $v_0 \in V(A)$ and $F_A \in \PI_A\binom{M_0}{\Mrr}$.
Then,
\begin{align*}
\rho(\si_{v_0}^2)(F_A)=B_{v_0}F_A.
\end{align*}
\end{cor}


Hence, we have:
\begin{prop}
\label{prop_sigma_compati}
Let $n,m\geq 1$ and
$A\in \Tr_n$, $B\in \Tr_m$ and $p \in \{1,\dots,n\}$.
\begin{enumerate}
\item
For any vertex ${v_0} \in V(A)$,
the following diagram commutes:
\begin{align*}
\begin{split}
\begin{array}{ccc}
\int_{N\in \Vmodf} \PI_{A}\binom{\bullet}{\bullet N \bullet} \otimes \PI_B\binom{N}{\bullet}
      &\overset{\comp_{A,B,p}}{\longrightarrow}&
\PI_{A\circ_p B}\binom{\bullet}{\bullet}
    \\
    {}^{{\rho(\si_{v_0})\otimes \rho(\id_B)}}\downarrow 
      && 
    \downarrow^{{\rho(\si_{v_0} \circ_p \id_B})}
    \\
\int_{N\in \Vmodf} \PI_{\si_{v_0} A}\binom{\bullet}{\bullet N \bullet} \otimes \PI_{B}\binom{N}{\bullet}
      &\overset{\comp_{\si_{v_0} A,B,p}}{\longrightarrow}&
\PI_{\si_{v_0} A\circ_p B}\binom{\bullet}{\bullet}.
\end{array}
\end{split}
\end{align*}
\item
For any vertex ${v_0} \in V(B)$,
the following diagram commutes:
\begin{align*}
\begin{split}
\begin{array}{ccc}
\int_{N\in \Vmodf} \PI_{A}\binom{\bullet}{\bullet N \bullet} \otimes \PI_B\binom{N}{\bullet}
      &\overset{\comp_{A,B,p}}{\longrightarrow}&
\PI_{A\circ_p B}\binom{\bullet}{\bullet}
    \\
    {}^{{\rho(\id_A)\otimes \rho(\si_{v_0})}}\downarrow 
      && 
    \downarrow^{{\rho(\id_A\circ_p \si_{v_0})}}
    \\
\int_{N\in \Vmodf} \PI_{A}\binom{\bullet}{\bullet N \bullet} \otimes \PI_{\si_{v_0} B}\binom{N}{\bullet}
      &\overset{\comp_{A,\si_{v_0} B,p}}{\longrightarrow}&
\PI_{A \circ_p \si_{v_0} B}\binom{\bullet}{\bullet}.
\end{array}
\end{split}
\end{align*}
\end{enumerate}
\end{prop}
\begin{proof}
(1) is obvious, since $\rho(\si_{v_0})$ is essentially a change of variables in Lemma \ref{lem_xy_coordinate} and has no effect on the $x$-coordinates of vertices below $v_0$.

We will show (2).
Let $F_A\in \PI_A$ and $F_B\in \PI_B$.
By Proposition \ref{prop_crossing_alg},
\begin{align*}
\comp_{A,B,p}\left(F_A, \rho(\si_{v_0}B)F_B \right)
&=F_A(\bullet_{[n]\setminus p}, \exp(z_{r_B}-z_{r_{\si_{v_0}B}}L(-1))F_B(\bullet)).
\end{align*}
If $p$ is not the rightmost leaf, then by (PI1) and Lemma \ref{lem_xy_coordinate}
\begin{align*}
F_A(\bullet_{[n]\setminus p}, &\exp\left((z_{r_B}-z_{r_{\si_{v_0})B}}L(-1)\right)\si_{v_0}^\alg F_B(\bullet))\\
&=(\exp\left((z_{r_B}-z_{r_{\si_{v_0}B}})\frac{d}{dz_p}\right) F_A)(\bullet_{[n]\setminus p}, \si_{v_0}^\alg F_B(\bullet))\\
&=\si_{v_0}^\alg \comp_{A,B,p}\left(F_A, F_B \right)\\
&= \rho(\si_{v_0})\comp_{A,B,p}\left(F_A, F_B \right),
\end{align*}
where in the last equation we used $r_{A\circ_p B}=r_{A\circ_p \si_{v_0}B}$,
which follows from the assumption on $p$.
If $p$ is the rightmost leaf, then by \eqref{eq_Ct_0},
\begin{align*}
F_A(\bullet_{[n]\setminus p}, &\exp\left((z_{r_B}-z_{r_{\si_{v_0}B}})L(-1)\right)\si_{v_0}^\alg F_B(\bullet))\\
&=(\exp\left((z_{r_B}-z_{r_{\si_{v_0}B}})L(-1)\right)(\exp\left((z_{r_B}-z_{r_{\si_{v_0}B}})\frac{d}{dz_p}\right) F_A)(\bullet_{[n]\setminus p}, \si_{v_0}^\alg F_B(\bullet))\\
&=(\exp\left((z_{r_{A\circ_p B}}-z_{r_{A\circ_p \si_{v_0}B}})L(-1)\right) \si_{v_0}^\alg \comp_{A,B,p}\left(F_A, F_B \right)\\
&=\rho(\si_{v_0})\comp_{A,B,p}\left(F_A, F_B \right).
\end{align*}
\end{proof}

\subsection{Compatibility for vacuum insertion}
\label{sec_lax_vac}
In this and the next section, we assume that $V\in \Vmodf$, that is, $V^\vee$ is finitely generated (see Remark \ref{rem_rep_V}).
Let $A\in \Tr_r$ with $r\geq 1$ and $p \in \{1,2,\dots,r\}$.
Then, $\PI_\emptyset=\mathrm{Hom}(V,\bullet)$ is a representative functor
and, by the property of the coend,
\begin{align*}
\PI_A \circ_p \PI_\emptyset &=\int_{N \in \Vmodf}
\PI_A\binom{\bullet_0}{\bullet_{A_1}N\bullet_{A_2}}\otimes \mathrm{Hom}(V,N)\\
&\cong \PI_A\binom{\bullet_0}{\bullet_{A_1}V\bullet_{A_2}}.
\end{align*}
Therefore, the composition of operad is simply to insert $V$ into the $p$-th variable.

To define a lax 2-action of $\CPaB$ on $\Vmodf$,
we need to define a natural transformation
\begin{align*}
\comp_{A,\emptyset,p}:\PI_A\binom{\bullet_0}{\bullet_{A_1}V\bullet_{A_2}} \rightarrow
\PI_{A\circ_p \emptyset}\binom{\bullet_0}{\bullet_{A_1}\bullet_{A_2}}.
\end{align*}

Let $\Pi_p:\C^r\rightarrow \C^{r-1}$ be the projection defined by
$(z_1,\dots,z_r)\mapsto (z_1,\dots,z_{p-1},z_{p+1},\dots,z_r)$ for $(z_1,\dots,z_r)\in \C^r$.
We first show the following lemma:
\begin{lem}
\label{lem_empty_int}
Let $A\in \Tr_r$ with $r \geq 1$ and $p \in\{1,\dots,r\}$.
Then, for sufficiently small $\mathfrak{t}=\{t_e\}\in \R^{E(\Ae)}$,
\begin{align*}
U_{A\circ_p \emptyset}^{\mathfrak{t}} \subset \Pi_p(U_A).
\end{align*}
\end{lem}
\begin{proof}

\begin{minipage}[l]{7cm}
\begin{forest}
for tree={
  l sep=20pt,
  parent anchor=south,
  align=center
}
[
[$nq$
[$pn$,edge label={node[midway,left]{a}}
[$j_1p$,edge label={node[midway,left]{b}}
[$i_1j_1$,edge label={node[midway,left]{$f_1$}}[$\blacksquare$]]
[$j_2p$,edge label={node[midway,right]{$e_1$}}
[$i_2j_2$,edge label={node[midway,left]{$f_2$}}
[$\blacksquare$]]
[$j_3p$,edge label={node[midway,right]{\bf$e_2$}}
[$i_3j_3$,edge label={node[midway,left]{$f_3$}}
[$\blacksquare$]][$p:$vacuum]]
]]
[$nm$,,edge label={node[midway,right]{c}}[$\blacksquare$][$\blacksquare$]]
]
[$\blacksquare$]
]
]
\end{forest}
\end{minipage}
\hfill
\begin{minipage}[r]{7cm}
\begin{forest}
for tree={
  l sep=20pt,
  parent anchor=south,
  align=center
}
[
[$nq$
[$j_3n$,edge label={node[midway,left]{$a$}}
[$j_1j_3$,edge label={node[midway,left]{$b$}}
[$i_1j_1$,edge label={node[midway,left]{$f_1$}}[$\blacksquare$]]
[$j_2j_3$,edge label={node[midway,right]{$e_1$}}[$i_2j_2$,edge label={node[midway,left]{$f_2$}}[$\blacksquare$]]
[$i_3j_3$,edge label={node[midway,right]{$f_3$}}[$\blacksquare$]]
]]
[$nm$,,edge label={node[midway,right]{$c$}}[$\blacksquare$][$\blacksquare$]]
]
[$\blacksquare$]
]
]
\end{forest}
\end{minipage}
We will write $A$-coordinate using $\zeta$ and
$A\circ_p \emptyset$-coordinate using $\xi$.
Then,
\begin{align*}
\xi_a &=\ze_a(1+\ze_b\ze_{e_1}\ze_{e_2})\\
\xi_b&= \ze_b\frac{1-\ze_{e_1}\ze_{e_2}}{1+\ze_b\ze_{e_1}\ze_{e_2}}\\
\xi_c&=\ze_c\frac{1}{1+\ze_b\ze_{e_1}\ze_{e_2}}\\
\xi_{e_1}&=\ze_{e_1}\frac{1-\ze_{e_2}}{1-\ze_{e_1}\ze_{e_2}}\\
\xi_{f_1}&=\ze_{f_1}\frac{1}{1-\ze_{e_1}\ze_{e_2}}\\
\xi_{f_2}&=\ze_{f_2}\frac{1}{1-\ze_{e_2}}\\
\xi_{f_3}&=\ze_{f_3}\frac{\ze_{e_2}}{1-\ze_{e_2}}.
\end{align*}
Fixing $\ze_{e_2}$ to a value satisfying $0<\ze_{e_2}<1$,
the change of variables from $\ze$ to $\xi$ is bi-holomorphic around the origin.
Thus, for this tree, the assertion holds.
The same is obtained in the general case.
\end{proof}

Let $M_0,M_1,\dots,M_{p-1},M_{p+1},\dots,M_r\in \Vmodf$
and
set
\begin{align*}
\Mrh&=M_0^\vee \otimes M_1\otimes 
\cdots \otimes M_{p-1} \otimes M_{p+1}\otimes \cdots\otimes M_r,\\
\Mr&=M_0^\vee \otimes M_1\otimes 
\cdots \otimes M_{p-1}\otimes V \otimes M_{p+1}\otimes \cdots \otimes M_r.
\end{align*}

Let $C\in \CB_{\Mr}(U_A)$ and 
let $i_p:\Mrh \rightarrow \Mr$
be a linear map defined by
\begin{align*}
i_p(\mrh)=m_1\otimes m_2\otimes \cdots
m_{p-1}\otimes \1 \otimes m_{p+1}
\otimes \cdots \otimes m_r,
\end{align*}
where $\1$ is the vacuum vector of $V$.
For any $\mrh$, by (CB1),
\begin{align*}
0=C(L(-1)_p i_p(\mrh))=\frac{d}{dz_p}C(i_p(\mrh)),
\end{align*}
which implies that $C(i_p (\mrh))$ is independent of
$z_p$. Hence, $C(i_p (\mrh))$ can be regarded as a holomorphic function on $\Pi_p(U_A)$.
Here, $\Pi_p(U_A)$ is an open subset of $\C^{r-1}$ since the projection $\Pi_p:\C^r\rightarrow \C^{r-1}$
is an open mapping
and contains $U_\Ae^{\mathfrak{t}} \subset \Pi_p(U_A)$ by Lemma \ref{lem_empty_int}.
Thus, $C$ induces a linear map
\begin{align*}
i_p^*(C):
\Mrh \rightarrow \mO_{X_{r-1}}^\an(U_\Ae^{\mathfrak{t}}),\quad \mrh\mapsto C(i_p(\mrh)).
\end{align*}
Since $a(n)\1=0$ for any $a\in V$ and $n\geq 0$,
it is clear that $i_p^*(C)$ gives a section of the conformal block, i.e., $i_p^*(C) \in \CB_{\Mrh}(U_\Ae)$.
Define a linear map
\begin{align*}
\comp_{A,\emptyset,p}:\PI_{A}\binom{\bullet}{\bullet V\bullet}\rightarrow \PI_{A\circ_p \emptyset}\binom{\bullet}{\bullet\bullet}
\end{align*}
by the composition
\begin{align*}
\PI_{A}\binom{\bullet}{\bullet V\bullet}
\overset{e_A^{-1}}{\rightarrow}\CB_{\bullet}(U_A)
\overset{i_p^*}{\longrightarrow}
\CB_{\bullet}(U_\Ae)
\overset{e_{\Ae}}{\rightarrow}\PI_{\Ae}\binom{\bullet}{\bullet \bullet}.
\end{align*}

Then we have:
\begin{prop}
Let $r\geq 1$, $A\in \Tr_r$ and $p \in \{1,\dots,r\}$.
Then, there exists a natural transformation
\begin{align*}
\comp_{A,\emptyset,p}:\PI_{A}\circ_p \PI_{\emptyset} \rightarrow \PI_{A\circ_p \emptyset}
\end{align*}
between the functors in $\Func\left(\Vmodf \times (\Vmodfo)^{r-1}, \Vect\right)$.
\end{prop}

For any $a\in V$, since $a=a(-1)\1$, the following lemma follows from the definition:
\begin{lem}
\label{lem_empty_lift}
For any $\mrh=m_1\otimes m_2\otimes \cdots \otimes m_{p-1}\otimes m_{p+1}
\otimes \cdots \otimes m_r \in \Mrh$ and $a\in V$,
\begin{align*}
C(m_1&\otimes m_2\otimes \cdots \otimes m_{p-1}\otimes a\otimes m_{p+1}
\otimes \cdots \otimes m_r)\\
&=\sum_{k\geq 0} z_p^k i_p^*(C)\left(a(-k-1)_0^* \mrh\right)
+\sum_{s \in [r],s\neq p}\sum_{k \geq 0}
(z_p-z_s)^{-k-1}i_p^*(C)\left(a(k)_s \mrh\right).
\end{align*}
In particular, $\comp_{A,\emptyset,p}$ is injective for any $M_0,M_1,\dots,M_r$.
\end{lem}
Hence, $C \in \CB_{\Mr}(U_A)$ can be recovered from $i_p^*(C)$.
Since all $(z_p-z_s)^{-k-1}$ and $z_p^k$ for $k\in \Z_{\geq 0}$ and $s\neq p$
are single-valued functions, by Lemma \ref{lem_empty_lift},
we have:
\begin{prop}
\label{prop_empty_commute}
Let $r\geq 1$ and
$A,A'\in \Tr_r$
and $\CPaB$-morphisms $g:A\rightarrow A'$ and $p \in \{1,\dots,r\}$.
Then, the following diagram commutes:
\begin{align*}
\begin{array}{ccc}
\PI_{A}\binom{\bullet}{\bullet V\bullet}&\overset{\comp_{A,\emptyset,p}}{\longrightarrow}&
\PI_{A\circ_p \emptyset}\binom{\bullet}{\bullet}\\
    {}^{{\rho(g)}}\downarrow 
      && 
    \downarrow^{{\rho(g\circ_p \emptyset)}}\\
\PI_{A'}\binom{\bullet}{\bullet V \bullet}
      &\overset{\comp_{A',\emptyset,p}}{\longrightarrow}&
\PI_{A'\circ_p \emptyset}\binom{\bullet}{\bullet}.
\end{array}
\end{align*}
\end{prop}

\subsection{Proof of Main Theorem}
\label{sec_lax_proof}

We first prove the following proposition:
\begin{prop}
\label{prop_natural_comp}
Let $n,m\geq 1$ and
$A,A'\in \Tr_n$, $B\in \Tr_m$
and $\CPaB$-morphisms $g:A\rightarrow A'$, $g_B:B\rightarrow B'$ and $p \in \{1,\dots,n\}$
Then, the following diagram commutes:
\begin{align}
\begin{split}
\begin{array}{ccc}
\int_{N\in \Vmodf} \PI_{A}\binom{\bullet}{\bullet N \bullet} \otimes \PI_B\binom{N}{\bullet}
      &\overset{\comp_{A,B,p}}{\longrightarrow}&
\PI_{A\circ_p B}\binom{\bullet}{\bullet}
    \\
    {}^{{\rho(g_A)\otimes \rho(g_B)}}\downarrow 
      && 
    \downarrow^{{\rho(g_A\circ_p g_B)}}
    \\
\int_{N\in \Vmodf} \PI_{A'}\binom{\bullet}{\bullet N \bullet} \otimes \PI_{B'}\binom{N}{\bullet}
      &\overset{\comp_{A',B',p}}{\longrightarrow}&
\PI_{A'\circ_p B'}\binom{\bullet}{\bullet}
\label{eq_lax_diagram_lem}
\end{array}
\end{split}
\end{align}
\end{prop}
\begin{proof}
Assume that for $\CPaB$-morphisms $g_A:A\rightarrow A'$ and $g_B:B\rightarrow B'$,
$g_{A'}:A'\rightarrow A''$ and $g_{B'}:B'\rightarrow B''$,
the diagram \eqref{eq_lax_diagram_lem} commutes.
Since $\rho$ is a groupoid homomorphism,
\begin{align*}
\left(\rho(g_{A'})\otimes \rho(g_{B'})\right) \left(\rho(g_A)\otimes \rho(g_B)\right)
&=\rho(g_{A'}g_A)\otimes \rho(g_{B'}g_B).
\end{align*}
By the functoriality of the operadic composition,
\begin{align*}
\rho(g_{A'}\circ_p g_{B'})\rho(g_A\circ_p g_B)&=\rho((g_{A'}\circ_p g_{B'}) (g_A\circ_p g_B))=
\rho((g_{A'}g_A)\circ_p (g_{B'}g_B))
\end{align*}
holds.
By assumption, the following diagram commutes:
\begin{align*}
\int_{N\in \Vmodf} \PI_{A}\binom{\bullet}{\bullet N \bullet} \otimes \PI_B\binom{N}{\bullet}
      &\quad\quad\quad\overset{\comp_{A,B,p}}{\longrightarrow}&
\PI_{A\circ_p B}\binom{\bullet}{\bullet}
    \\
    {}^{{\rho(g_A)\otimes \rho(g_B)}}\downarrow \quad\quad\quad\quad
      &\quad\quad\quad& 
    \downarrow^{{\rho(g_A\circ_i g_B)}}
    \\
    \int_{N\in \Vmodf} \PI_{A'}\binom{\bullet}{\bullet N \bullet} \otimes \PI_{B'}\binom{N}{\bullet}
      &\quad\quad\quad\overset{\comp_{A',B',p}}{\longrightarrow}&
\PI_{A'\circ_p B'}\binom{\bullet}{\bullet}\\
    {}^{{\rho(g_{A'})\otimes \rho(g_{B'})}}\downarrow \quad\quad\quad\quad
      &\quad\quad\quad& 
    \downarrow^{{\rho(g_{A'}\circ_i g_{B'})}}
    \\
\int_{N\in \Vmodf} \PI_{A''}\binom{\bullet}{\bullet N \bullet} \otimes \PI_{B''}\binom{N}{\bullet}
      &\quad\quad\quad\overset{\comp_{A'',B'',p}}{\longrightarrow}&
\PI_{A''\circ_p B''}\binom{\bullet}{\bullet}.
\end{align*}
Thus, the diagram \eqref{eq_lax_diagram_lem} commutes for $g_{A'}g_A:A\rightarrow A''$
and $g_{B'}g_B:B\rightarrow B''$.
Hence, it suffices to show that 
the diagram commutes for a generating set of the morphisms of $\CPaB(n)\times \CPaB(m)$.
Since any morphism in $\CPaB$ can be written as compositions of $\si_v$ and $\al_e$,
we may assume that
one of $g_A$ or $g_B$ is the identity map and the other is $\si_v$ or $\alpha_e$.
This follows from Proposition \ref{prop_alpha_compati} and Proposition \ref{prop_sigma_compati}.
\end{proof}

\begin{thm}
\label{thm_action}
Let $V$ be a vertex operator algebra.
Then, the pair of the functors $\PI:\CPaB \rightarrow \Endp_\Vmodf$
and the natural transformations $\comp_{A,B,p}: \PI_{A}\circ_p \PI_{B}\rightarrow \PI_{A \circ_p B}$ ($A,B\in \Tr_{*}$) is a lax 2-morphism of 2-operads.
Moreover,
if $V^\vee$ is a finitely generated $V$-module, then $\PI(\emptyset)$ is represented by $V \in \Vmodf$.
\end{thm}
\begin{proof}
By Proposition \ref{prop_natural_comp} and Proposition \ref{prop_empty_commute},
$\comp_{A,B,p}$ is a natural transformation. The rest of the conditions follow from the
construction and Lemma \ref{rem_add_symmetry}.
\end{proof}

\section{Representativity of lax 2-action}
\label{sec_rep}
In this section, we show that the lax 2-morphism satisfies the assumptions of
Theorem \ref{thm_add_BTC} when $V$ is a rational $C_2$-cofinite vertex operator algebra.
Thus, in this case, \(V\text{-mod}_{f.g.}\) has a braided tensor category
structure. We will further show at the end of this section that this braided
tensor category is balanced.
%
This gives another proof of the result of Huang and Lepowsky \cite{HL4,HL1,HL2,HL3,H1,H2,H3}.
It is important to note that we do not assume that $V$ is a rational $C_2$-cofinite vertex operator algebra until Section \ref{sec_rep_rational}.

Section \ref{sec_rep_zhu} reviews Frenkel-Zhu's bimodules.
Section \ref{sec_rep_pole} is the most important part of this section.
Here, we show that poles of a parenthesized intertwining operator between $C_1$-cofinite modules are uniformly bounded below after certain modifications. 
We further show that by looking at the deepest part of poles, we can obtain a map between Frenkel-Zhu bimodules
 from a parenthesized intertwining operator.
In Section \ref{sec_rep_rational}, under the assumption that $V$ is simple rational $C_2$-cofinite, we use this result to show that 
the lax 2-morphism satisfies the assumptions in Theorem \ref{thm_add_BTC}.
Section \ref{sec_explicit_braiding} gives explicit formulas for the braiding and associator in terms
of conformal blocks and proves that the resulting braided tensor category is
balanced.

\subsection{Zhu algebra and Frenkel-Zhu bimodules}
\label{sec_rep_zhu}
In \cite{Zh}, Zhu constructs an associative algebra $A(V)$ from a vertex operator algebra $V$
and shows that there is a one-to-one correspondence between simple
$V$-modules and left simple $A(V)$-modules.
Furthermore, in \cite{FZ}, Frenkel and Zhu construct an $A(V)$-bimodule $A(M)$
from a $V$-module $M$
and describe the space of intertwining operators in terms of $A(V)$-bimodules,
which is studied in more detail in \cite{Li}.
In this section, we recall their results.

Let $V$ be a vertex operator algebra and $M$ be a $V$-module.
Frenkel and Zhu define a bilinear operation $a*m$, $m*a$ and $a\circ m$ for $a\in V$ and $m\in M$ as follows:
\begin{align}
a * m &=\mathrm{Res}_z\left(Y(a,z)m \frac{(z+1)^{\Delta_a}}{z}
\right)=\sum_{k\geq 0}\binom{\Delta_a}{k} a(k-1)m, \label{eq_zhu_l}\\
m* a &=\mathrm{Res}_z\left(Y(a,z)m \frac{(z+1)^{\Delta_a-1}}{z}
\right)=\sum_{k\geq 0}\binom{\Delta_a-1}{k} a(k-1)m, \label{eq_zhu_r}\\
a\circ m &= \mathrm{Res}_z\left(Y(a,z)m \frac{(z+1)^{\Delta_a}}{z^2}
\right)=\sum_{k\geq 0}\binom{\Delta_a}{k} a(k-2)m, \label{eq_zhu_n}
\end{align}
where $\Delta_a$ is the $L(0)$-grading of $a$.
Let $O(M) \subset M$ be a subspace spanned by the elements of $\{a \circ m\}$ with $a\in V$ and $m \in M$ and set $A(M)=M/O(M)$.
If $M=V$, Zhu showed that two products \eqref{eq_zhu_l} and \eqref{eq_zhu_r} coincide in $A(V)$
and $A(V)$ is an associative algebra with respect to the product $*$ \cite{Zh}.
Moreover, Frenkel and Zhu showed that $A(M)$ is an $A(V)$-bimodule \cite{FZ}.
Since the left multiplication $a*-$ and the right multiplication $-*a$ raise the degree,
the following lemma follows:
\begin{lem}
\label{lem_ideal_zhu}
The subspace $V_+=\bigoplus_{k \geq 1}V_k$ is a two-sided ideal of $A(V)$.
\end{lem}


Set
\begin{align*}
\Om(M)=\{m\in M\mid a(\Delta_a -1+k)m=0\text{ for any }a\in V \text{ and } k\geq 1\},
\end{align*}
which is the ``lowest weight space''.
%
Then, define an action of $V$ on $\Om(M)$ by 
$o:V \rightarrow \End\, \Om(M),\;\; a\mapsto o(a)=a(\Delta_a-1)$.
Then, it was shown that $o(a)|_{\Om(M)}=0$ for any $a \in O(V)$
and thus $o$ defines a bilinear map $A(V)\otimes \Om(M) \rightarrow \Om(M)$.
In fact, $\Om(M)$ is a left $A(V)$-module \cite{Zh}
and thus defines a $\C$-linear functor
\begin{align}
\Om:\Vmodu\rightarrow \AVm,\;\;M\mapsto \Om(M),
\label{eq_functor_H}
\end{align}
where $\AVm$ is a category of left $A(V)$-modules.

It is noteworthy that since $o(\om)=L(0)$, if the action of $L(0)$ on $M$ is semisimple,
then $\om \in A(V)$ acts semisimply on $\Om(M)$.

\begin{lem}
\label{lem_V0}
Let $M=\bigoplus_{k\geq 0}M_{\Delta+k}$ be a simple $V$-module with $M_\Delta\neq 0$.
Then, $\Om(M)=M_{\Delta}$.
\end{lem}
\begin{proof}
It is obvious that $M_\Delta \subset \Om(M)$.
Suppose that $M_\Delta \neq \Om(M)$.
Since $\Om(M)$ is homogeneous,
there exists $k \geq 1$ such that $M_{\Delta+k} \cap \Om(M) \neq 0$. Then, $M_{\Delta+k} \cap \Om(M)$ generates a proper submodule of $M$, a contradiction.
\end{proof}
Note that $V_0$ is the one-dimensional trivial $A(V)$-module
and is isomorphic to $A(V)/V_+$ as a left $A(V)$-module.

Set 
\begin{align*}
L(M)&=A(M)\otimes_{A(V)}V_0,
\end{align*}
where $A(M)\otimes_{A(V)}V_0$ is the tensor product of a right $A(V)$-module $A(M)$
and a left $A(V)$-module $V_0$.
It is clear that $L(M)$ is a left $A(V)$-module
and defines a $\C$-linear functor
\begin{align*}
L:\Vmodu\rightarrow \AVm,\;\;M\mapsto L(M).
\end{align*}

We will later show that 
if $V$ is a simple rational $C_2$-cofinite vertex operator algebra,
then two functors $L$ and $\Om$ coincide.


\subsection{Factorization for standard trees}
\label{sec_rep_pole}
Let $\Delta', \Delta_i \in \C$ and $M_0,M_1,\dots,M_r$ be $V$-modules
such that $M_0=\bigoplus_{k\geq 0}(M_0)_{\Delta'+k}$
and $M_i=\bigoplus_{k\geq 0}(M_i)_{\Delta_i+k}$ for $i=1,\dots,r$.

\begin{minipage}[l]{4cm}
\begin{forest}
for tree={
  l sep=20pt,
  parent anchor=south,
  align=center
}
[15
  [1
  ]
  [25,edge label={node[midway,right]{}}
    [2
    ]
    [35,edge label={node[midway,right]{}}
      [3
      ]
      [45,edge label={node[midway,right]{}}
        [4
        ]
        [5
        ]
      ]
    ]
  ]
]
\end{forest}
\captionof{figure}{}
\end{minipage}
\hfill
\begin{minipage}[r]{9cm}
Recall that we set $\std_r = 1(2(3\dots (r-1 r)\dots) \in \Tr_r$,
which is called {\it a standard tree} and appeared in Section \ref{app_pseudo}. Then,
\begin{align*}
x_i&=x_{ir}=z_i-z_r,\\
T_{\std_r}&=\C\Bigl[\Bigl[\frac{x_2}{x_1},\frac{x_3}{x_2},\dots,\frac{x_{r-1}}{x_{r-2}}\Bigr]\Bigr][x_1^\C,\log x_1,\dots,x_{r-1}^\C,\log x_{r-1}].
\end{align*}
\end{minipage}

By \eqref{eq_std_fac} and Theorem \ref{thm_composition}, we have a natural transformation
\begin{align}
\PI_{(12)} \circ_2 \left( \PI_{(12)} \circ_2 \left( \cdots \circ_2 \left( \PI_{(12)} \circ_2 \PI_{(12)} \right) \right)\right)
\rightarrow \PI_{\std_r}.
\label{eq_std_fac_pi}
\end{align}

In this section, what we want to do is to prepare a proof of Assumption (2) in Theorem \ref{thm_add_BTC},
that is, \eqref{eq_std_fac_pi} is a natural isomorphism.

A parenthesized intertwining operator of type $(\std_r,M_0,\Mrr)$ is an $M_0$-valued series expanded in the $x$-coordinate $x_i=z_i-z_r$ ($i=1,\dots,r-1$).
For our purposes we need to consider series expanded in $z$-coordinate $z_1,\dots,z_r$ in the following domain
\begin{align}
\{(z_1,\dots,z_r)\in X_r\mid |z_r|<|z_{r-1}|<|z_{r-2}|<\dots <|z_2|<|z_1|\}.
\label{eq_standard_open}
\end{align}



Set
\begin{align*}
T_{\tstd_r} = \C[[\tfrac{z_2}{z_1},\dots,\tfrac{z_r}{z_{r-1}}]][z_1^\C,\dots,z_r^\C,\log(z_1),\dots,\log(z_r)].
\end{align*}
Similarly to Section \ref{sec_config_conv}, by the Taylor expansion,
we can define a $\C$-algebra homomorphism
\begin{align*}
e_{\tstd_r}: \Or^\alg \rightarrow T_{\tstd_r},
\end{align*}
which is a $\Dr$-module.

Let $C_\std: T_{\std_r} \rightarrow T_{\tstd_r}$ be a linear map defined by
the formal Taylor expansion, i.e.,
for example,
\begin{align*}
\frac{x_2}{x_1}&\mapsto \frac{z_2-z_r}{z_1-z_r}=\frac{z_2}{z_1}\frac{1-\frac{z_r}{z_2}}{1-\frac{z_r}{z_1}}=\frac{z_2}{z_1}\Bigl(1-\frac{z_r}{z_2}\Bigr)\sum_{k \geq 0}\Bigl(\frac{z_r}{z_1}\Bigr)^k \in T_{\tstd_r}.
\end{align*}
Then, it is clear that $C_\std$ is a $\Drt$-module homomorphism and the following diagram commutes:
\begin{align}
\begin{split}
\begin{array}{ccc}
T_{\std_r}&\overset{C_\std}{\longrightarrow} & T_{\tstd_r}\\
{}_{e_{\std_r}}\nwarrow && \nearrow_{e_{\tstd_r}}\\
&\Ort^\alg&.
\end{array}
\label{eq_std_com}
\end{split}
\end{align}

Let $F \in \PI_{\std}\binom{M_0}{\Mrr}$ and set
\begin{align}
\tilde{F} = \exp(L(-1)z_r)C_{\std}(F).
\label{eq_tilde_F_def}
\end{align}
More precisely, $\tilde{F}$ is defined by 
\begin{align*}
\langle u, \tilde{F}(\mrr)\rangle=\sum_{k\geq 0}\frac{z_r^k}{k!}  C_\std(\langle L(1)^k u, F\rangle)
\in T_{\tstd_r}
\end{align*}
for $u\in M_0^\vee$ and $\mrr\in \Mrr$.

By definition, it is clear that the image of $C_\std$ is in
\begin{align*}
 \C[[\tfrac{z_2}{z_1},\dots,\tfrac{z_r}{z_{r-1}}]][z_1^\C,\dots,z_{r-1}^\C,\log(z_1),\dots,\log(z_{r-1})].
\end{align*}
That is, it is a series which does not contain $\log z_{r}$ and $z_r^\C$ terms.
Therefore, the assignment $z_r = 0$ is well-defined for the image of $C_\std$, and clearly 
\begin{align*}
f=\lim_{\substack{z_r\mapsto 0\\z_i\mapsto x_i\\ i=1,\dots,r-1 }} C_\std(f)
\end{align*}
holds for any $f\in T_{\std_r}$.
Hence, we have:
\begin{align}
F=\lim_{\substack{z_r\mapsto 0\\z_i\mapsto x_i\\ i=1,\dots,r-1 }} \tilde{F}.
\label{eq_zr_substitute}
\end{align}

\begin{prop}
\label{prop_std_F_tilde}
The linear map $\tilde{F}$ satisfies the following conditions:
\begin{enumerate}
\item
For any $p\in [r]$ and $\mrr\in \Mrr$,
\begin{align*}
\tilde{F}(L(-1)_p \mrr)=\pa_p \tilde{F}(\mrr).
\end{align*}
\item
For any $p\in [r]$, $\mrr\in \Mrr$ and $a\in V$, $n\in \Z$,
\begin{align*}
\tilde{F}(a(n)_p \mrr)=\sum_{k\geq 0}\binom{n}{k} (-z_p)^k a(n-k)\tilde{F}(\mrr)
-\sum_{s\in [r],s\neq p}\sum_{k\geq 0}\binom{n}{k} e_{\tstd_r}\Bigl((z_s-z_p)^{n-k}\Bigr)\tilde{F}(a(k)_s \mrr).
\end{align*}
\end{enumerate}
\end{prop}
\begin{proof}
We first show (2).
By (PI2),
\begin{align}
\begin{split}
\tilde{F}(a(n)_p \mrr)&=\exp(L(-1)z_r) C_\std\Biggl(
\sum_{k\geq 0}\binom{n}{k} e_{\std_r}\Bigl((z_r-z_p)^k \Bigr) a(n-k)F(\mrr)\\
&-\sum_{s\in [r],s\neq p}\sum_{k\geq 0}\binom{n}{k} e_{\std_r}\Bigl((z_s-z_p)^{n-k}\Bigr){F}(a(k)_s \mrr)
\Biggr)
\label{eq_std_rec_proof}
\end{split}
\end{align}
Evaluating \eqref{eq_std_rec_proof} with $u\in M_0^\vee$ yields a finite sum.
Hence, by \eqref{eq_std_com},
\begin{align*}
\tilde{F}(a(n)_p \mrr)&=\exp(L(-1)z_r) \Biggl(
\sum_{k\geq 0}\binom{n}{k} e_{\tstd_r}\Bigl((z_r-z_p)^k \Bigr) a(n-k) C_\std F(\mrr)\\
&-\sum_{s\in [r],s\neq p}\sum_{k\geq 0}\binom{n}{k} e_{\tstd_r}\Bigl((z_s-z_p)^{n-k}\Bigr)C_\std{F}(a(k)_s \mrr)
\Biggr).
\end{align*}
By the proof of Lemma \ref{lem_translation_isomorphism}, we have
\begin{align*}
\exp(L(-1)z_r) \sum_{k\geq 0}\binom{n}{k} e_{\tstd_r}\Bigl((z_r-z_p)^k \Bigr)a(n-k)
=\sum_{k\geq 0}\binom{n}{k} (-z_p)^k a(n-k) \exp(L(-1)z_r).
\end{align*}
Hence, (2) holds.
For $p \in [r]$ with $p\neq r$, by \eqref{eq_std_com},
\begin{align*}
\tilde{F}(L(-1)_p \mrr)&=\exp(L(-1)z_r)C_\std F(L(-1)_p \mrr)
=\exp(L(-1)z_r)C_\std \pa_p F(\mrr)\\
 &=\exp(L(-1)z_r)\pa_p C_\std F(\mrr) = \pa_p \tilde{F}(\mrr).
\end{align*}
By substituting $p=r$, $a=\om$ and $n=0$ into (2),
\begin{align*}
\tilde{F}(L(-1)_r \mrr)&=L(-1) \exp(L(-1)z_r)C_\std{F}(\mrr) -\sum_{s\in [r],s\neq r}  \exp(L(-1)z_r)C_\std{F}(\om(0)_s\mrr)\\
&=L(-1) \exp(L(-1)z_r)C_\std{F}(\mrr) + \exp(L(-1)z_r)\pa_r C_\std{F}\\
&=\pa_r \tilde{F}(\mrr).
\end{align*}
Hence, the assertion holds.
\end{proof}

Note that similarly to the proof of Lemma \ref{N_contain},
\begin{align}
a(n)\tilde{F}(\mrr)=\sum_{s\in [r]}\sum_{k\geq 0}\binom{n}{k} z_s^{n-k} \tilde{F}(a(k)_s \mrr)
\label{eq_top_std}
\end{align}
holds for any $n\geq 0$.
However, by Remark \ref{rem_range_n}, Proposition \ref{prop_top_recursion} does not hold.

Let $\pr:M_0=\bigoplus_{k\geq 0}(M_0)_{\Delta'+k} \rightarrow (M_0)_{\Delta'}$ be the projection.
For $\mrr\in \Mrr$,
set (for the definition of $z_i^{L(0)}$, see Section \ref{sec_D_rotation})
\begin{align*}
z^{L(0)}\mrr= z_1^{L(0)}m_1 \otimes z_2^{L(0)}m_2 \otimes \dots \otimes z_{r-1}^{L(0)}m_{r-1}\otimes z_{r-1}^{L(0)}m_r.
\end{align*}

Let $\bF:M_{[r]} \rightarrow (M_0)_{\Delta'} \otimes {T}_{\tstd_r}$ be a linear map defined by
\begin{align}
\bF(\mrr)&=\pr z_1^{-L(0)} \tilde{F}(z^{L(0)}\mrr),
\label{eq_bar_F}
\end{align}
for $\mrr \in \Mrr$. 
Then, we have
\begin{align*}
(\sum_{s \in [r]}z_s\pa_s) \bF(\mrr)
&=(\sum_{s \in [r]}z_s\pa_s)  \Bigl(\pr z_1^{-L(0)} \tilde{F}(z^{L(0)}\mrr)\Bigr)\\
&=-\pr z_1^{-L(0)} L(0)\tilde{F}(z^{L(0)}\mrr) + \sum_{s\in [r]} \pr z_1^{-L(0)}\tilde{F}((L(0)_s+z_s L(-1)_s)z^{L(0)}\mrr)\\
&=0,
\end{align*}
where in the last equality we used \eqref{eq_top_std} with $a=\om$ and $n=1$.
Then, we have:
\begin{lem}
\label{lem_std_delta_r}
For any $\mrr \in \Mrr$ with $m_r \in (M_r)_{\Delta_r+k}$, the formal power series $\bF(\mrr)$ is in
\begin{align*}
\Bigl(\tfrac{z_{r}}{z_{r-1}}\Bigr)^{\Delta_r+k}
(M_0)_{\Delta'}\otimes \C\Bigl{[}\Bigl{[}\tfrac{z_2}{z_1},\tfrac{z_3}{z_2},\dots,\tfrac{z_{r}}{z_{r-1}}\Bigr{]}\Bigr{]}
\Bigl[\Bigl(\tfrac{z_2}{z_1}\Bigr)^\C,\log \Bigl(\tfrac{z_2}{z_1}\Bigr),
,\dots,
\Bigl(\tfrac{z_{r-1}}{z_{r-2}}\Bigr)^\C,\log \Bigl(\tfrac{z_{r-1}}{z_{r-2}}\Bigr),\log\Bigl(\tfrac{z_{r}}{z_{r-1}}\Bigr)\Bigr].
\end{align*}
\end{lem}
\begin{proof}
Similarly to Proposition \ref{prop_rotation_del}, it is clear that
\begin{align*}
\bF(\mrr) \in (M_0)_{\Delta'}\otimes \C\Bigl{[}\Bigl{[}\tfrac{z_2}{z_1},\tfrac{z_3}{z_2},\dots,\tfrac{z_{r}}{z_{r-1}}\Bigr{]}\Bigr{]}
\Bigl[\Bigl(\tfrac{z_2}{z_1}\Bigr)^\C,\log \Bigl(\tfrac{z_2}{z_1}\Bigr),
,\dots,
\Bigl(\tfrac{z_{r}}{z_{r-1}}\Bigr)^\C,\log\Bigl(\tfrac{z_{r}}{z_{r-1}}\Bigr)\Bigr]
\end{align*}
and
all nontrivial powers of $z_r$ in $\bF$ come from $z_r^{L(0)}m_r$ by \eqref{eq_zr_substitute}.
Hence, the assertion holds.
\end{proof}

\begin{lem}
\label{simple_non_dege}
Suppose that $M_0^\vee$ is generated by $(M_0)_{\Delta'}^*$ as a $V$-module.
Then, $F=0$ if and only if $\bF=0$.
\end{lem}
\begin{proof}
By assumption, $M_0^\vee$ is linearly spanned by
$\{a(n)u\}_{a \in V_+,u\in (M_0)_{\Delta'},n\in \Z}$.
Then, for any $m> 0$, $\mrr\in \Mrr$ and $u\in (M_0)_{\Delta'}^\vee$, 
by \eqref{eq_dual2},
we have:
\begin{align*}
\langle a(\Delta_a-1-m)u, F(\mrr)\rangle
&=\sum_{k \geq 0}\frac{(-1)^{\Delta_a}}{k!}
\langle u,  (L(1)^k a)((\Delta_a-k)-1+m)F(\mrr)\rangle,\\
&=\sum_{k \geq 0}\frac{(-1)^{\Delta_a}}{k!}
\langle u, \pr (L(1)^k a)((\Delta_a-k)-1+m)F(\mrr)\rangle=0,
\end{align*}
where we used $L(1)^{\Delta_a+1}a=0$ and \eqref{eq_top_std}.
Thus, $F=0$.
\end{proof}

By Proposition \ref{prop_std_F_tilde}, we have:
\begin{lem}
\label{lem_std_recursion}
Let $\mrr \in \Mrr$. Then, the following properties hold for $\bF$:
\begin{enumerate}
\item
For any $p \in \{1,\dots,r\}$ and $a \in V_{\Delta_a}$ with $\Delta_a >0$ and $\Delta_a-1 > n$,
\begin{align*}
\bF(a(n)_p  \mrr)
&=- \sum_{s \in [r], s \neq p} \sum_{k \geq 0}\binom{n}{k}
e_{\tstd_r}\left((z_s-z_p)^{n-k}\right) z_p^{\Delta_a-1-n}z_s^{-\Delta_a+1+k} \bF(a(k)_s\mrr).
\end{align*}
\item
For any $a \in V_+$ and $p\in [r]$,
\begin{align*}
\bF(a(\Delta_a-1)_p \mrr)
&=a(\Delta_a-1)\bF(\mrr)\\
&-\sum_{s \in [r],s\neq p}\sum_{k \geq 0} \binom{\Delta_a-1}{k}e_{\tstd_r}\left((z_s-z_p)^{\Delta_a-1-k}\right)
z_s^{-\Delta_a+1+k} \bF(a(k)_s m_{[r]}).
\end{align*}
\end{enumerate}
\end{lem}
\begin{proof}
Since $\pr$ is the projection onto the lowest weight space,
$a(n)^* u=0$ if $\Delta_a-1>n$.
Let $n \in \Z$ with $\Delta_a-1>n$.
By (PI2), we have
\begin{align*}
\bF(a(n)_p  \mrr)&=
\pr z_1^{-L(0)} F(z^{L(0)} a(n)_p  \mrr)\\
&=\pr z_1^{-L(0)} F(a(n)_p z^{L(0)} \mrr) z_p^{\Delta_a-1-n}\\
&=\sum_{k \geq 0}\binom{n}{k}e_{\tstd_r}\left((z_r-z_p)^k\right)
\pr z_1^{-L(0)} a(-k+n) F(z^{L(0)} \mrr)z_p^{\Delta_a-1-n}\\
&- \sum_{s \in [r], s \neq p} \sum_{k \geq 0}\binom{n}{k}
e_{\tstd_r}\left((z_s-z_p)^{n-k}\right) 
\pr z_1^{-L(0)} F(a(k)_s z^{L(0)} \mrr)z_p^{\Delta_a-1-n}\\
&=
- \sum_{s \in [r], s \neq p} \sum_{k \geq 0}\binom{n}{k}
e_{\tstd_r}\left((z_s-z_p)^{n-k}\right) 
\pr z_1^{-L(0)} F(a(k)_s z^{L(0)} \mrr) z_p^{\Delta_a-1-n}\\
&=
- \sum_{s \in [r], s \neq p} \sum_{k \geq 0}\binom{n}{k}
e_{\tstd_r}\left((z_s-z_p)^{n-k}\right) z_p^{\Delta_a-1-n}z_s^{-\Delta_a+1+k}
\pr z_1^{-L(0)} F(z^{L(0)}a(k)_s\mrr)\\
&=
- \sum_{s \in [r], s \neq p} \sum_{k \geq 0}\binom{n}{k}
e_{\tstd_r}\left((z_s-z_p)^{n-k}\right) z_p^{\Delta_a-1-n}z_s^{-\Delta_a+1+k}
\bF(a(k)_s\mrr).
\end{align*}

Similarly, by (PI2), we have
\begin{align*}
\bF(a(\Delta_a-1)_p  \mrr)&=
\pr z_1^{-L(0)} F(z^{L(0)} a(\Delta_a-1)_p  \mrr)\\
&= \pr z_1^{-L(0)} F(a(\Delta_a-1)_p z^{L(0)} \mrr)\\
&=\sum_{k \geq 0}\binom{\Delta_a-1}{k}e_{\tstd_r}\left((z_r-z_p)^k\right)
\pr z_1^{-L(0)} a(-k+\Delta_a-1) F(z^{L(0)} \mrr)\\
&- \sum_{s \in [r], s \neq p} \sum_{k \geq 0}\binom{\Delta_a-1}{k}
e_{\tstd_r}\left((z_s-z_p)^{\Delta_a-1-k}\right) 
\pr z_1^{-L(0)} F(a(k)_s z^{L(0)} \mrr)\\
&=
\pr z_1^{-L(0)} a(\Delta_a-1) F(z^{L(0)} \mrr)\\
&- \sum_{s \in [r], s \neq p} \sum_{k \geq 0}\binom{\Delta_a-1}{k}
e_{\tstd_r}\left((z_s-z_p)^{\Delta_a-1-k}\right) 
\pr z_1^{-L(0)} F(a(k)_s z^{L(0)} \mrr).
\end{align*}
Thus, (2) holds.

%
\end{proof}

The following lemma will be used later:
\begin{lem}
\label{lem_om_F}
For any $p \in [r]$ with $p \neq 1$ and $\mrr \in \Mrr$, we have the following equalities:
\begin{align}
\bF(L(0)_p+L(-1)_p m_{[r]}) &= z_p\frac{d}{dz_p}\bF(m_{[r]}) \label{eq_L_xp}
\\
\bF(L(0)_1+L(-1)_{1} m_{[r]}) &= L(0)\bF(m_{[r]})+z_{1}\frac{d}{dz_{1}}\bF(m_{[r]}).
\label{eq_L_x1}
\end{align}
\end{lem}
\begin{proof}
For any $p \in [r]$ with $p\neq 1$, by Proposition \ref{prop_std_F_tilde},
\begin{align*}
\bF(L(-1)_p m_{[r]}) &=\pr z_1^{-L(0)} \tF(z^{L(0)}L(-1)_p \mrr)\\
&=z_p \pr z_1^{-L(0)} \tF(L(-1)_pz^{L(0)}\mrr)\\
&= \pr z_1^{-L(0)} z_p\frac{d}{dz_p}\tF(z^{L(0)}\mrr) - \pr z_1^{-L(0)} \tF(z^{L(0)}L(0)_p\mrr)
\\
&=z_p\frac{d}{dz_p} \bF(\mrr) - \bF(L(0)_p m_{[r]}).
\end{align*}
In the case of $p=1$, we have
\begin{align*}
\bF(L(-1)_{1} m_{[r]}) &= 
 \pr z_1^{-L(0)} z_{1}\frac{d}{dz_{1}}F(z^{L(0)}\mrr) 
- \pr z_1^{-L(0)} F(z^{L(0)}L(0)_{1}\mrr)\\
&=L(0)\bF(m_{[r]})+z_{1}\frac{d}{dz_{1}}\bF(m_{[r]})-\bF(L(0)_{1} m_{[r]}).
\end{align*}
\end{proof}

Set
\begin{align*}
\ze_i =\frac{z_{i+1}}{z_i}\quad \quad \text{for }i=1,\dots,r-1
\end{align*}
and
\begin{align*}
\mathfrak{m}_{\std_r}&=\bigoplus_{i=1}^{r-1} \ze_i\C\Bigl{[}\Bigl{[}\ze_1,\ze_2,\cdots,\ze_{r-1}\Bigr{]}\Bigr{]},
\end{align*}
which is the maximal ideal of the local ring $\C[[\ze_1,\dots,\ze_{r-1}]]$.

We will rewrite the recursive formula of Lemma \ref{lem_std_recursion} 
using the Frenkel-Zhu products \eqref{eq_zhu_l}, \eqref{eq_zhu_r} and \eqref{eq_zhu_n}
so that the coefficients are contained in $\mathfrak{m}_{\std_r}$.
 This is the most crucial step in this section.

Let $s,p\in \{1,2,\dots,r\}$ with $s\neq p$, $n \in \Z$ with $\Delta_a-1-n>0$
and $k\in \Z_{\geq 0}$.
Then, we have: 
\begin{align*}
e_{\tstd_r}\left((z_s-z_p)^{n-k}\right)z_p^{\Delta_a-n-1}z_s^{-\Delta_a+k+1}
&=\begin{cases}
(\frac{z_s}{z_p})^{k+1-\Delta_a}(\frac{z_s}{z_p}-1)^{n-k}|_{1>|z_s/z_p|}
& (s>p)\\
(\frac{z_p}{z_s})^{\Delta_a-n-1}(1-\frac{z_p}{z_s})^{n-k}|_{1>|z_p/z_s|}
& (p>s).
\end{cases}
\end{align*}
Thus, we have:
\begin{lem}
\label{lem_pole_sp}
If $p>s$, then for any $n,k \in \Z$ with $\Delta_a-1-n >0$
\begin{align*}
e_{\std_r}\left((z_s-z_p)^{n-k}\right)z_p^{\Delta_a-n-1}z_s^{-\Delta_a+k+1} \in 
\mathfrak{m}_{\std_r}.
\end{align*}
\end{lem}

The following lemma is fundamental:
\begin{lem}
\label{lem_zhu}
Let $\mrr \in \Mrr$.
Then, the following properties hold for $\bF$:
\begin{enumerate}
\item
For any $a \in V_+$ and $p\in [r]$, there exists $f_{s,k}(\ze) \in \mathfrak{m}_{\std_r}$
such that:
\begin{align*}
\bF(a\circ_p m_{[r]})=\sum_{s \in [r],s\neq p}\sum_{k\geq 0}
f_{s,k}(\ze)\bF(a(k)_s m_{[r]}).
\end{align*}
\item
For any $a \in V_+$ and $p\in [r]$, there exists $g_{s,k}(\ze) \in \mathfrak{m}_{\std_r}$
such that:
\begin{align*}
\bF(a {*}_p m_{[r]})=
a(\Delta_a-1)\bF(m_{[r]})-\sum_{s < p }\sum_{k\geq 0}\binom{\Delta_a-1}{k} \bF(a(k)_s m_{[r]})
+\sum_{s \in [r],s\neq p}\sum_{k\geq 0}
g_{s,k}(\ze)\bF(a(k)_s m_{[r]}).
\end{align*}
\item
For any $a \in V_+$ and $p\in [r]$ with $p>1$, there exists $h_{s,k}(\ze) \in \mathfrak{m}_{\std_r}$
such that:
\begin{align*}
&\bF(a *_p m_{[r]}) -\bF(m_{[r]}*_{p-1} a)=\sum_{s \in [r]}\sum_{k\geq 0}
h_{s,k}(\ze)\bF(a(k)_s m_{[r]}).
\end{align*}
\item
For any $a\in V$ and $\Delta_a-1>n$, there exists $H(s,k)(\ze) \in \mathfrak{m}_{\std_r}$ such that:
\begin{align*}
\bF(a(n)_r \mrr) = \sum_{s \in [r],s\neq r}\sum_{k\geq 0}
H_{s,k}(\ze)\bF(a(k)_s m_{[r]}).
\end{align*}
\end{enumerate}
\end{lem}
\begin{proof}
By Lemma \ref{lem_std_recursion}, we have
\begin{align*}
\bF(a\circ_p m_{[r]})&=
\sum_{n \geq 0}\binom{\Delta_a}{n} \bF(a(n-2)_p m_{[r]})\\
&=
-\sum_{n \geq 0}\sum_{s \in [r], s \neq p} \sum_{k \geq 0}\binom{\Delta_a}{n} \binom{n-2}{k}
e_{\tstd_r}\left((z_s-z_p)^{n-2-k}\right) z_p^{\Delta_a-n+1}z_s^{-\Delta_a+k+1}\bF(u,a(k)_s m_{[r]}).
\end{align*}
By Lemma \ref{lem_pole_sp},
if $p>s$, then
\begin{align*}
&\sum_{n \geq 0}\binom{\Delta_a}{n} \binom{n-2}{k}e_{\tstd_r}\left((z_s-z_p)^{n-2-k}\right)z_p^{\Delta_a-n+1}z_s^{-\Delta_a+k+1} \in \mathfrak{m}_{\tstd_r}.
\end{align*}
We note that by Lemma \ref{lem_binom},
\begin{align*}
\sum_{n \geq 0}\sum_{k \geq 0}\binom{\Delta_a}{n}\binom{n-2}{k}
(-1)^{n+k}x^k(1-x)^{n-2-k}y^k
&=\frac{x^{\Delta_a}(1+y)^{\Delta_a}}{(1-x-xy)^2}.
\end{align*}
In particular,
\begin{align}
\sum_{n \geq 0}\binom{\Delta_a}{n}\binom{n-2}{k}
(-1)^{n+k}x^k(1-x)^{n-2-k}
\in x^{\Delta_a}\C[[x]].\label{eq_zhu_c1}
\end{align}
Hence, if $s>p$, then by \eqref{eq_zhu_c1} for any $k\geq 0$
\begin{align*}
&\sum_{n \geq 0}\binom{\Delta_a}{n} \binom{n-2}{k}e_{\std_r}\left((z_s-z_p)^{n-2-k}\right)z_p^{\Delta_a-n+1}z_s^{-\Delta_a+k+1}\\
&=\left(\frac{z_s}{z_p}\right)^{1-\Delta_a}
\sum_{n \geq 0}\binom{\Delta_a}{n} \binom{n-2}{k}
(-1)^{n+k}\left(\frac{z_s}{z_p}\right)^{k}
\Bigr(1-\frac{z_s}{z_p}\Bigl)^{n-2-k}\Bigl|_{1>|z_s/z_p|}
\in \mathfrak{m}_{\std_r}.
\end{align*}
Thus, (1) holds.

We will consider (2).
\begin{align*}
\bF(a {*}_p &m_{[r]}) =\sum_{n \geq 0}\binom{\Delta_a}{n} \bF(a(n-1)_p m_{[r]})\\
&=o(a)\bF(m_{[r]})
-\sum_{n \geq 0}\sum_{s \in [r], s \neq p} \sum_{k \geq 0}\binom{\Delta}{n} \binom{n-1}{k}
e_{\std_r}\left((z_s-z_p)^{n-1-k}\right) \bF(u,a(k)_s m_{[r]})z_p^{\Delta_a-n}z_s^{-\Delta_a+k+1}.
\end{align*}

By Lemma \ref{lem_pole_sp},
If $p>s$ and $\Delta_a > n$, then
\begin{align*}
&e_{\std_r}\left((z_s-z_p)^{n-1-k}\right) z_p^{\Delta_a-n}z_s^{-\Delta_a+k+1}
\in \mathfrak{m}_{\std_r}.
\end{align*}
In the case of $n=\Delta_a$ and $r>p>s$, we have
\begin{align*}
&\sum_{k \geq 0} \binom{\Delta_a-1}{k}
\left(1-\frac{z_p}{z_s}\right)^{\Delta_a-1-k}\Bigl|_{1>|z_p/z_s|}
\bF(a(k)_s m_{[r]})\\
&=\sum_{k \geq 0} \binom{\Delta_a-1}{k}\bF(a(k)_s m_{[r]})
+\sum_{k \geq 0} \binom{\Delta_a-1}{k}
\left(\left(1-\frac{z_p}{z_s}\right)^{\Delta_a-1-k}\Bigl|_{1>|z_p/z_s|}-1\right)
\bF(a(k)_s m_{[r]}).
\end{align*}
Note that $\left(1-\frac{z_p}{z_s}\right)^{\Delta_a-1-k}\Bigl|_{1>|z_p/z_s|}-1
\in \mathfrak{m}_{\std_r}$.

We will next consider the case of $p <  s$.
We note that again by Lemma \ref{lem_binom},
\begin{align*}
\sum_{n \geq 0}\sum_{k \geq 0}\binom{\Delta_a}{n}\binom{n-1}{k}
(-1)^{n+k}x^k(1-x)^{n-1-k}y^k
&=\frac{x^{\Delta_a}(1+y)^{\Delta_a}}{(1-x-xy)}.
\end{align*}
In particular
\begin{align}
\sum_{n \geq 0}\binom{\Delta_a}{n}\binom{n-1}{k}
(-1)^{n+k}x^k(1-x)^{n-1-k}
\in x^{\Delta_a}\C[[x]].
\label{eq_zhu_c3}
\end{align}
Hence, if $s>p$, then by \eqref{eq_zhu_c3} for any $k \geq 0$
\begin{align*}
&\sum_{n \geq 0}\binom{\Delta}{n} \binom{n-1}{k}
e_{\std_r}\left((z_s-z_p)^{n-1-k}\right)z_p^{\Delta_a-n}z_s^{-\Delta_a+k+1}\\
&=\left(\frac{z_s}{z_p}\right)^{1-\Delta_a}
\sum_{n \geq 0}\binom{\Delta}{n} \binom{n-1}{k}(-1)^{n+k+1}
\left(\frac{z_s}{z_p}\right)^{k}
\left(1-\frac{z_s}{z_p}\right)^{n-1-k}|_{1>|z_s/z_p|}
\in \mathfrak{m}_{\std_r}.
\end{align*}
Thus, (2) holds.

Next, we consider (3). Assume $p>1$.

Since
\begin{align*}
&\sum_{k\geq 0}\binom{\Delta_a-1}{k} \bF(a(k)_{p-1} m_{[r]})+\bF(m_{[r]}*_{p-1} a)\\
&=\sum_{k \geq 0}\left(\binom{\Delta_a - 1}{k} \bF(a(k)_{p-1}m_{[r]})+\binom{\Delta_a-1}{k} \bF(a(k-1)_{p-1} m_{[r]})\right)\\
&=\sum_{k \geq 0} \binom{\Delta_a}{k} \bF(a(k-1)_{p-1}m_{[r]})=\bF(a*_{p-1} m_{[r]}),
\end{align*}
by repeatedly using (2), we obtain
\begin{align*}
&\bF(a *_p m_{[r]}) -\bF(m_{[r]}*_{p-1} a)\\
 &= a(\Delta_a-1)\bF(m_{[r]})
-\sum_{s < p}\sum_{k\geq 0}\binom{\Delta_a-1}{k} \bF(a(k)_s m_{[r]})
-\bF(m_{[r]} *_{p-1} a)
+\sum_{s \in [r],s\neq p}\sum_{k\geq 0}
g_{s,k}(\ze)\bF(a(k)_s m_{[r]})\\
&=a(\Delta_a-1)\bF(m_{[r]})
-\sum_{s < p-1}\sum_{k\geq 0}\binom{\Delta_a-1}{k} \bF(a(k)_s m_{[r]})
-\bF(a*_{p-1}\mrr)
+\sum_{s \in [r],s\neq p}\sum_{k\geq 0}
g_{s,k}(\ze)\bF(a(k)_s m_{[r]})\\
&=\sum_{s \in [r]}\sum_{k\geq 0}
h_{s,k}(\ze)\bF(a(k)_s m_{[r]}),
\end{align*}
for some $h_{s,k}(\ze) \in \mathfrak{m}_{\std_r}$.

Finally, for any $\Delta_a-1>n$,
 (4) follows from
\begin{align*}
\bF( a(n)_r m_{[r]}) &= - \sum_{s \in [r], s \neq r} \sum_{k \geq 0}\binom{n}{k}
e_{\tstd_r}\Bigl( (z_s-z_r)^{n-k}\Bigr)z_r^{\Delta_a-n-1}z_s^{-\Delta_a+1+k}
\bF(a(k)_s m_{[r]})\\
&= - \sum_{s \in [r], s \neq r} \sum_{k \geq 0}\binom{n}{k}\Bigl(\frac{z_r}{z_s}\Bigr)^{\Delta_a-n-1}
\Bigl(1-\frac{z_r}{z_s}\Bigr)^{n-k}\Biggl|_{|z_r/z_s|<1} \bF(a(k)_s m_{[r]}).
\end{align*}
\end{proof}

Note that we have the following decomposition of the vector space:
\begin{align}
&\ze_{r-1}^{\Delta_r}(M_0)_{\Delta'}\otimes \C\Bigl{[}\Bigl{[}\ze_1,\ze_2,\dots,\ze_{r-1}\Bigr{]}\Bigr{]}
\Bigl[\ze_1^\C,\log \ze_1,\dots,\ze_{r-2}^\C,\log \ze_{r-2},\log \ze_{r-1} \Bigr]\nonumber \\
&\cong \bigoplus_{\la_1,\dots,\la_{r-2}\in (\C/\Z)^{r-2}}
\ze_1^{\la_1}
\ze_2^{\la_2}
\cdots \ze_{r-2}^{\la_{r-2}}\ze_{r-1}^{\Delta_r}
(M_0)_{\Delta'}[[\ze_{r-1}]]((\ze_1,\ze_2,\dots,\ze_{r-2}))[
\log \ze_1
,\dots,\log \ze_{r-2},\log \ze_{r-1}]. \label{eq_decomposition_power}
\end{align}
For $\la_1,\dots,\la_{r-2}\in (\C/\Z)^{r-2}$,
let $p_{\la_1,\dots,\la_{r-2}}$ be the projection onto
\begin{align*}
\ze_1^{\la_1}
\ze_2^{\la_2}
\cdots \ze_{r-2}^{\la_{r-2}}\ze_{r-1}^{\Delta_r}
(M_0)_{\Delta'}[[\ze_{r-1}]]((\ze_1,\ze_2,\dots,\ze_{r-2}))[
\log \ze_1
,\dots,\log \ze_{r-1}]
\end{align*}
with respect to the direct product \eqref{eq_decomposition_power}.
Then, for any $F \in \PI_{\std_r} \binom{M_0}{\Mrr}$,
it is clear that $p_{\la_1,\dots,\la_{r-2}} \circ F$ is again a
parenthesized intertwining operator of type $(\std_r,M_0,\Mrr)$.

Let $F \in \PI_{\std_r} \binom{M_0}{\Mrr}$ be a non-zero element.
Hereafter, we assume that $M_0^\vee$ (resp. $M_r$) is generated by $(M_0)_{\Delta'}^*$ (resp. $(M_r)_{\Delta_r}$) as $V$-modules.
Then, by Lemma \ref{simple_non_dege} and Lemma \ref{lem_std_delta_r}
 there exists $\la_1,\dots,\la_{r-2}\in (\C/\Z)^{r-2}$
such that 
\begin{align*}
p_{\la_1,\dots,\la_{r-2}} \circ \bF:\Mrr \rightarrow \ze_1^{\la_1}\ze_2^{\la_2} \cdots \ze_{r-2}^{\la_{r-2}}\ze_{r-1}^{\Delta_r}
M_0[[\ze_{r-1}]]((\ze_1,\ze_2,\dots,\ze_{r-2}))[\log\ze_1,\dots,\log\ze_{r-1}]
\end{align*}
 is a non-zero linear map.
Then, we have:
\begin{prop}
\label{pole_bound}
%
Assume that all $M_i$ ($i=1,\dots,r$)
are $C_1$-cofinite.
Then, there exists $\mu_i \in -\la_i + \Z$ ($i=1,\dots,r-1$) such that the following condition holds:
\begin{enumerate}
\item
For any $\mrr \in \Mrr$,
\begin{align}
\ze_1^{\mu_1}\ze_2^{\mu_2}\cdots \ze_{r-2}^{\mu_{r-2}} \ze_{r-1}^{-\Delta_r} p_{\la_1,\dots,\la_{r-2}} \circ \bF(\mrr)
\label{eq_p_power_ze}
\end{align}
is in $(M_{0})_{\Delta'}[[\ze_1,\dots,\ze_{r-1}]][\log\ze_1,\dots,\log\ze_{r-1}]$,
i.e., has no pole.
\end{enumerate}
\end{prop}
\begin{proof}
Since $M_i/C_1(M_i)$ is finite-dimensional,
there exists $N_i \in \Z$ such that\\
$M_i=\bigoplus_{N_i \geq k \geq 0} (M_i)_{\Delta_i+k} + C_1(M_i)$ for any $i=1,\dots,r$.
Set $N=\sum_{i=1}^r N_i$ and for $L \in \Z_{\geq 0}$
let $M_{[r]}^L$ be a vector subspace of $M_{[r]}$ spanned by
$\bigotimes_{i=1}^r(M_{i})_{\Delta_i + k_i}$
with $L+N \geq \sum_{i=1}^r k_i$.

Since $M_{[r]}^0$ is finite-dimension,
there exists $\mu_i \in -\la_i+\Z$ ($i=1,\dots,r-2$) such that
\begin{align*}
\ze_1^{\mu_1}\ze_2^{h_2}\cdots \ze_{r-2}^{\mu_{r-2}}\ze_{r-1}^{-\Delta_r} p_{\la_1,\dots,\la_{r-2}} \circ \bF(\mrr) \in (M_{0})_{\Delta'}[[\ze_1,\dots,\ze_{r-1}]][\log\ze_1,\dots,\log\ze_{r-1}]
\end{align*}
for any $m_{[r]} \in M_{[r]}^0$.

We will show by induction on $L \geq 0$ that (1) holds
for any $m_{[r]} \in M_{[r]}^L$.
There is nothing to prove when $L=0$.

Let $L>0$ and assume that the claim holds for any $L'$ with $L>L'$.
Let $m_{[r]} \in M_{[r]}^L$.
By definition of $M_{[r]}^0$, we may assume that $m_{[r]}$ is a linear sum of 
the vectors of the form $a(-1)_p v_{[r]}$ with $p\in [r]$, $a \in V_+$ and 
$v_{[r]} \in M_{[r]}^{L-1}$.

Thus, we may assume that $m_{[r]}=a(-1)_p v_{[r]}$.
Then, there exists $v'_{[r]} \in M_{[r]}^{L-1}$
such that $m_{[{r}]} = a * _p v_{[r]}+v'_{[r]}$.
If $p=1$, then by Lemma \ref{lem_zhu} (2), the claim holds.
If $p>1$, using (2) and (3) in Lemma \ref{lem_zhu},
we can replace $a*_p v_{[r]}$ with $v_{[r]}*_{p-1}a$,
and by repeating this the assertion follows from the induction assumption.
%
\end{proof}

\begin{lem}
\label{lem_non_deg_r}
Suppose that $M_0^\vee$ (resp. $M_r$) is generated by $(M_0)_{\Delta'}^*$ (resp. $(M_r)_{\Delta_r}$) as $V$-module.
Then, $F=0$ if and only if the coefficient of $\ze_{r-1}^{\Delta_r}$ in $\bF$ is zero.
\end{lem}
\begin{proof}
Assume that the coefficient of $\ze_{r-1}^{\Delta_r}$ in $\bF$ is zero.
Then, for any $u\in (M_0^\vee)_{\Delta'}$ and $\mrr \in \Mrr$ with $m_r\in (M_r)_{\Delta_r}$,
\begin{align*}
0&=\text{the coefficient of }\ze_{r-1}^{\Delta_r} \text{ in } \langle u,z_1^{-L(0)} \bF(z^{L(0)}\mrr) \rangle\\
&=\lim_{\ze_{r-1} \mapsto 0} \langle u,z_1^{-L(0)} \bF(\otimes_{i=1}^{r-1}z_i^{L(0)_i}m_i \otimes \exp((L(0)-\Delta_r)\log z_r)m_r) \rangle.
\end{align*}
Hence, by \eqref{eq_zr_substitute} and Lemma \ref{lem_std_delta_r}, $\langle u, F(\mrr) \rangle=0$.
By assumption, $M_r$ is spanned by $a(n)v$ with $a\in V$ and $v\in (M_r)_{\Delta_r}$
and $n<\Delta_a-1$,
which implies that $\langle u, F(\mrr) \rangle=0$ for any $\mrr\in \Mrr$.
Thus, the assertion follows from Lemma \ref{simple_non_dege}.
\end{proof}

Let $(\mu_1,\dots,\mu_{r-2})\in \C^{r-2}$ be complex numbers obtained in the above proposition
and $S(F)$ the set of $(k_1,\dots,k_{r-2}) \in \C^{r-2}$ such that
the coefficient of $\ze_1^{k_1}\ze_2^{k_2}\cdots \ze_{r-2}^{k_{r-2}}\ze_{r-1}^{\Delta_r}$ in\\
$p_{\la_1,\dots,\la_{r-2}} \circ \bF(\mrr)$ is not zero.

Let $\Delta_{\mu_1,\dots,\mu_{r-2}}(F)$ be the set of $(h_1,\dots,h_{r-2}) \in S(F)$ such that:
\begin{enumerate}
\item[D1)]
If $(k_1,\dots,k_{r-2}) \in S_{\mu_1,\dots,\mu_{r-2}}(F)$,
then $k_i-h_i \geq 0$ for all $i=1,\dots,r-2$.
\end{enumerate}
Since $S_{\mu_1,\dots,\mu_{r-2}}(F)$ is non-empty subset of $-(\mu_1,\dots,\mu_{r-2})+\Z_{\geq 0}^{r-2}$ by Lemma \ref{lem_non_deg_r} and Proposition \ref{pole_bound},
$\Delta_{\mu_1,\dots,\mu_{r-2}}(F)$ is not empty.

For $(h_1,\dots,h_{r-2})\in \Delta_{\mu_1,\dots,\mu_{r-2}}(F)$,
let 
\begin{align*}
P_{h_1,\dots,h_{r-2}}(F):M_{[r]} \rightarrow (M_{0})_{\Delta'}[
\log \ze_1,\dots,\log \ze_{r-1}]
\end{align*}
be a linear map
defined by taking the coefficient of 
$\ze_1^{h_1}\ze_2^{h_2}\cdots\ze_{r-2}^{h_{r-2}}\ze_{r-1}^{\Delta_r}$
in 
\begin{align*}
\bF: M_{[r]}\rightarrow \ze_{r-1}^{\Delta_r}(M_{0})_{\Delta'}[[\ze_1,\dots,\ze_{r-1}]][\ze_1^\C,\log \ze_1,\dots,\ze_{r-2}^\C,\log\ze_{r-2},\log\ze_{r-1}].
\end{align*}
Then, $P_{h_1,\dots,h_{r-2}}(F)$ is a non-zero linear map.

By Lemma \ref{lem_zhu} (1) and (D1), this linear map factors through
\begin{align*}
P_{h_1,\dots,h_{r-2}}(F):\bigotimes_{i=1}^r M_i/O(M_i) \rightarrow (M_0)_{\Delta'}[\log \ze_1,\dots,\log \ze_{r-1}].
\end{align*}
Since $M_r$ is generated by $(M_r)_{\Delta_r}$ as a $V$-module,
\begin{align*}
\bigoplus_{k >0} (M_r)_{\Delta_r+k} = \text{Span}_\C\{a(\Delta_a-1-k)m\}_{a\in V_+,\,m\in M, k > 0}.
\end{align*}
Then, by Lemma \ref{lem_zhu} (4), $P_{h_1,\dots,h_{r-2}}(F)$ further factors through
\begin{align*}
P_{h_1,\dots,h_{r-2}}(F):\Bigl(\bigotimes_{i=1}^{r-1} A(M_i)\Bigr) \otimes (M_r)/\bigoplus_{k >0} (M_r)_{\Delta_r+k}
\rightarrow (M_0)_{\Delta'}[\log \ze_1,\dots,\log \ze_{r-1}].
\end{align*}
and by Lemma \ref{lem_zhu} (3), this map induces the following left $A(V)$-module homomorphism
\begin{align*}
A(F): A(M_1)\otimes_{A(V)} \cdots \otimes_{A(V)} A(M_{r-1})\otimes_{A(V)} (M_r)_{\Delta_r}
\rightarrow (M_0)_{\Delta'}[\log \ze_1,\dots,\log \ze_{r-1}],
\end{align*}
where $\bullet \otimes_{A(V)} \bullet$ is the tensor product of a left $A(V)$-module and a right $A(V)$-module.
Hence, we have:
\begin{prop}
\label{prop_factor}
The linear map $P(F):M_{[r]} \rightarrow (M_{0})_{\Delta'}[\log \ze_1,\dots,\log \ze_{r-1}]$.
induces a left $A(V)$-module homomorphism
$A(F): A(M_1)\otimes_{A(V)} A(M_2)\otimes_{A(V)} \cdots \otimes_{A(V)} A(M_{r-1})\otimes_{A(V)} (M_r)_{\Delta_r}
\rightarrow (M_0)_{\Delta'}[\log \ze_1,\dots,\log \ze_{r-1}]$.
\end{prop}

We end this section by considering the action of $\om$.
Note that for any $m\in M$
\begin{align*}
\om*m &= (L(0)+2L(-1)+L(-2))m\\
m*\om &= (L(-1)+L(-2))m.
\end{align*}
In particular, 
\begin{align}
\om*m - m*\om &= (L(0)+L(-1))m,
\end{align}
which makes us possible to apply Lemma \ref{lem_om_F}.

\begin{prop}
\label{prop_omega_action}
For $r-1 \geq p\geq 2$ and $\mrr \in \Mrr$, we have the following equalities:
\begin{align*}
A(F)(\om*_p m_{[r]})- A(F)(m_{[r]}*_p \om) &= \Bigl((h_p-h_{p+1})+z_p\frac{d}{dz_p} \Bigr) A(F)(m_{[r]}),\\
A(F)(\om*_1 m_{[r]})- A(F)(m_{[r]}*_1 \om) &=\Bigl(h_1+z_1\frac{d}{dz_1}+L(0)\Bigr)A(F)(m_{[r]}).
\end{align*}
\end{prop}
\begin{proof}
For $r-1 \geq p\geq 2$, by Lemma \ref{lem_om_F},
\begin{align*}
A(F)&(\om*_p m_{[r]})- A(F)(m_{[r]}*_p \om)\\
&=\text{the coefficient of }\ze_1^{h_1}\ze_2^{h_2}\cdots
\ze_{r-2}^{h_{r-2}}\ze_{r-1}^{\Delta_r} \text{ in } p_{\la_1,\dots,\la_{r-1}}\bF((L(0)_p+L(-1)_p)\mrr)\\
&=\text{the coefficient of }\ze_1^{h_1}\ze_2^{h_2}\cdots
\ze_{r-2}^{h_{r-2}}\ze_{r-1}^{\Delta_r} \text{ in } p_{\la_1,\dots,\la_{r-1}}
z_p\frac{d}{dz_p}\bF(\mrr)\\
&=\Bigr(z_p\frac{d}{dz_p}+h_p-h_{p-1}\Bigl)\text{ the coefficient of }\ze_1^{h_1}\ze_2^{h_2}\cdots
\ze_{r-2}^{h_{r-2}}\ze_{r-1}^{\Delta_r} \text{ in }  p_{\la_1,\dots,\la_{r-1}}\bF(\mrr)\\
&=\Bigr(z_p\frac{d}{dz_p}+h_p-h_{p-1}\Bigl)A(F)(\mrr)
\end{align*}
and similarly, we have
\begin{align*}
A(F)&(\om*_1 m_{[r]})- A(F)(m_{[r]}*_1 \om)\\
&=\text{the coefficient of }\ze_1^{h_1}\ze_2^{h_2}\cdots
\ze_{r-2}^{h_{r-2}}\ze_{r-1}^{\Delta_r} \text{ in }p_{\la_1,\dots,\la_{r-1}}\Bigl(z_1\frac{d}{dz_1}+L(0)\Bigr) \bF(\mrr)\\
&=\Bigl(h_1+z_1\frac{d}{dz_1}+L(0)\Bigr)A(F)(\mrr).
\end{align*}
\end{proof}

By the above proposition, we have:
\begin{cor}
\label{cor_factor}
There are only finitely many possibilities for $(\la_1,\dots,\la_{r-2})\in (\C/\Z)^{r-2}$ and 
$(h_1,\dots,h_{r-2}) \in \C^{r-2}$ above, and they are determined by the eigenvalues of $\om$ on the A(V)-modules,
\begin{align}
\begin{split}
A(M_{r-1})\otimes_{A(V)}(M_r)_{\Delta_r},\quad A(M_{r-2})\otimes_{A(V)} A(M_{r-1})\otimes_{A(V)}(M_r)_{\Delta_r},\\
\dots,\quad  A(M_1)\otimes_{A(V)} \cdots \otimes_{A(V)} A(M_{r-1})\otimes_{A(V)} (M_r)_{\Delta_r}.
\label{eq_LA_om}
\end{split}
\end{align}
Moreover, if $\om$ acts semisimply on all the modules \eqref{eq_LA_om},
then $A(F)$ does not contain the logarithmic term, i.e.,
the image of $A(F)$ is $(M_0)_{\Delta'}$
and $A(F)$ defines an $A(V)$-module homomorphism
$A(F): A(M_1)\otimes_{A(V)} A(M_2)\otimes_{A(V)} \cdots \otimes_{A(V)} A(M_{r-1})\otimes_{A(V)}(M_r)_{\Delta_r}\rightarrow (M_0)_{\Delta'}$.
\end{cor}

\subsection{Representativity for rational vertex operator algebras}
\label{sec_rep_rational}
Throughout this section,
we assume that $V$ is a simple rational $C_2$-cofinite vertex operator algebra.
Then, $\Vmodf$ is a semisimple $\C$-linear category.

For any $A(V)$-module $N$,
Zhu constructs a $V$-module $\mathrm{Ind}\, N$ such that $\Om\bigl(\mathrm{Ind}\, N\bigr)$ is isomorphic to $N$ as an $A(V)$-module \cite{Zh}.
If $N$ is finite-dimensional,
then $\mathrm{Ind}\, N$ is a direct sum
of finitely many simple $V$-modules.
Since $V$ is $C_2$-cofinite,
any simple module is $C_2$-cofinite
and in particular $C_1$-cofinite by Proposition \ref{prop_ABD}.
Hence, $\mathrm{Ind}\,N\in \Vmodf$.
Hence, we have the following functors:
\begin{align}
\begin{split}
\Om:\Vmodf \rightarrow \AVmodu,\\
\mathrm{Ind}:\AVmodu\rightarrow \Vmodf,
\label{eq_equiv_ind}
\end{split}
\end{align}
where $\AVmodu$ is the category of left $A(V)$-modules whose objects are finite-dimensional.

In the rational case, the following results are known:
\begin{thm}{\cite{Zh}}
\label{thm_equiv}
Two functors \eqref{eq_equiv_ind}
give an equivalence of categories
between $\Vmodf$ and $\AVmodu$.
\end{thm}
\begin{thm}{\cite{FZ,Li}}
\label{thm_FZ}
For any $M_0,M_1,M_2\in \Vmodu$,
$I_{\log}\binom{M_0}{M_1M_2}$
is naturally isomorphic to 
$\mathrm{Hom}_{A(V)}(A(M_1)\otimes_{A(V)}\Om(M_2),\Om(M_0))$ as a $\C$-linear functor.
\end{thm}


We will use the following lemma:
\begin{lem}
\label{lem_functor_HL}
Two functors $\Om,L:\Vmodf \rightarrow \AVmodu$
are naturally isomorphic.
\end{lem}
\begin{proof}
By Lemma \ref{lem_V0},
\begin{align*}
L(M)=A(M)\otimes_{A(V)} V_0=A(M)\otimes_{A(V)} \Om(V).
\end{align*}
By Theorem \ref{thm_equiv} and Theorem \ref{thm_FZ}, for any $N \in \Vmodu$
\begin{align*}
\mathrm{Hom}(A(M)\otimes_{A(V)} \Om(V), \Om(N)) &\cong I\binom{N}{M,V}\\
&\cong \mathrm{Hom}_\Vmodu(M,N)\\
&\cong \mathrm{Hom}_{A(V)}(\Om(M),\Om(N)).
\end{align*}
Hence, $\Om(M)$ is isomorphic to $L(M)$ as an $A(V)$-module.
\end{proof}
%
%

Set
\begin{align*}
\Irr=\text{the set of isomorphism classes of simple } V\text{ modules}.
\end{align*}
Define $\boxtimes: \Vmodf \times \Vmodf \rightarrow \Vmodf$ by
\begin{align}
M_1\boxtimes M_2 = \mathrm{Ind}\left(A(M_1)\otimes_{A(V)}\Om(M_2) \right),\qquad M_1,M_2\in \Vmodf.
\label{eq_tensor_def}
\end{align}
For any $N_\la \in \Irr$, 
 by Theorem \ref{thm_FZ}, we have
\begin{align}
\mathrm{Hom}_{V\text{-mod}} (M_1\boxtimes M_2,N_\la) &= \mathrm{Hom}_{V\text{-mod}} (\mathrm{Ind}\left(A(M_1)\otimes_{A(V)}\Om(M_2) \right),N_\la) \nonumber\\
&\cong \mathrm{Hom}_{A(V)} (A(M_1)\otimes_{A(V)}\Om(M_2), \Om(N_\la))\cong I_{\log} \binom{N_\la}{M_1,M_2}.\label{eq_add_int_iso_box}
\end{align}
Since \(V\) is rational, we have \(I_{\log}=I\). Hence $M_1\boxtimes M_2$ is naturally isomorphic to
\begin{align}
M_1\boxtimes M_2 \cong \bigoplus_{N_\la \in \Irr}N_\la \otimes_\C I\binom{N_\la}{M_1M_2}^*.
\label{eq_tensor_inter_add}
\end{align}
It is important to note that the space of intertwining operators is contravariant in the two lower variables and covariant in the upper variable:
$$I\binom{\bullet}{\bullet\bullet}:\Vmodf^\op \times \Vmodf^\op \times \Vmodf \rightarrow \Vect,\;\;\;(M_1,M_2,M_0)\mapsto
I\binom{M_0}{M_1M_2}.$$
Thus, the dual in \eqref{eq_tensor_inter_add} is essential.
Since \(V\text{-mod}_{f.g.}\) is semisimple, 
by \eqref{eq_add_int_iso_box}, we have:
%
\begin{lem}
\label{lem_representative_two}
The functor $I\binom{\bullet}{M_1M_2}:\Vmodu \rightarrow \Vect,\;\;\;M \mapsto I\binom{M}{M_1M_2}$ is represented by $M_1\boxtimes M_2$.
\end{lem}
%
Let $M_1,\dots,M_r \in \Vmodf$.
Set 
\begin{align*}
\boxtimes_{\std_r}M_{[r]}
= M_1\boxtimes (M_2\boxtimes (\cdots (M_{r-1}\boxtimes M_r)\cdots)).
\end{align*}
Then, for any $M_0\in \Vmodf$
\begin{align}
\begin{split}
&\mathrm{Hom}(M_1\boxtimes (M_2\boxtimes (\cdots (M_{r-1}\boxtimes M_r)\cdots)),M_0)\\
&=\bigoplus_{N_{\la_{r-1}}\in \Irr} 
\mathrm{Hom}(M_1\boxtimes (M_2\boxtimes (\cdots M_{r-2} \boxtimes N_{\la_{r-1}}\cdots)),M_0)\otimes 
I\binom{N_{\la_{r-1}}}{M_{r-1}M_r}\\
&=\bigoplus_{N_{\la_{r-2}},N_{\la_{r-1}}\in \Irr} 
\mathrm{Hom}(M_1\boxtimes (M_2\boxtimes (\cdots M_{r-3} \boxtimes N_{\la_{r-2}}\cdots)),M_0)
\otimes
I\binom{N_{\la_{r-2}}}{M_{r-2}N_{\la_{r-1}}}
\otimes I\binom{N_{\la_{r-1}}}{M_{r-1}M_r}\\
&=\cdots\\
&=\bigoplus_{N_{\la_{1}},\dots,N_{\la_{r-1}}\in \Irr} 
\mathrm{Hom}(N_{\la_1},M_0)
\otimes
I\binom{N_{\la_1}}{M_1N_{\la_2}}\otimes I\binom{N_{\la_2}}{M_2N_{\la_3}}
\otimes \cdots\otimes
I\binom{N_{\la_{r-2}}}{M_{r-2}N_{\la_{r-1}}}
\otimes I\binom{N_{\la_{r-1}}}{M_{r-1}M_r}\\
&=\bigoplus_{N_{\la_{2}},\dots,N_{\la_{r-1}}\in \Irr} 
I\binom{M_0}{M_1N_{\la_2}}\otimes I\binom{N_{\la_2}}{M_2N_{\la_3}}
\otimes \cdots\otimes
I\binom{N_{\la_{r-2}}}{M_{r-2}N_{\la_{r-1}}}
\otimes I\binom{N_{\la_{r-1}}}{M_{r-1}M_r}\\
&=\int_{\bullet_2,\dots,\bullet_{r-1} \in \Vmodu}
I\binom{M_0}{M_1\bullet_2}\otimes
I\binom{\bullet_2}{M_2\bullet_3}\otimes
\cdots \otimes
I\binom{\bullet_{r-1}}{M_{r-1}M_r}\\
&=
I\binom{M_0}{M_1\bullet}\circ_2 \Biggl( 
I\binom{\bullet}{M_2\bullet}\circ_2
\Biggl(
\cdots
\circ_2 
\Biggl(
I\binom{\bullet}{M_{r-2}\bullet}\circ_2
I\binom{\bullet}{M_{r-1}M_r}\Biggr)\cdots\Biggr).
\label{eq_std_int}
\end{split}
\end{align}
Hence, by Theorem \ref{thm_composition},
we have a natural transformation
\begin{align*}
\Psi_r:
\mathrm{Hom}(M_1\boxtimes (M_2\boxtimes (\cdots (M_{r-1}\boxtimes M_r)\cdots)),M_0)
\rightarrow
\PI_{\std_r}\binom{M_0}{M_1M_2\dots M_r}.
\end{align*}

We will show the following Proposition:
\begin{prop}
\label{thm_factor}
The map $\Psi_r$ is an isomorphism
of vector spaces.
\end{prop}

By \eqref{eq_std_int} and Theorem \ref{thm_FZ}, we have
\begin{align}
\begin{split}
\mathrm{Hom}&(M_1\boxtimes (M_2\boxtimes (\cdots (M_{r-1}\boxtimes M_r)\cdots)),M_0)\\
&\cong \bigoplus_{N_{\la_{2}},\dots,N_{\la_{r-1}}\in \Irr} 
I\binom{M_0}{M_1N_{\la_2}}\otimes I\binom{N_{\la_2}}{M_2N_{\la_3}}
\otimes \cdots\otimes
I\binom{N_{\la_{r-2}}}{M_{r-2}N_{\la_{r-1}}}
\otimes I\binom{N_{\la_{r-1}}}{M_{r-1}M_r}\\
&\cong \bigoplus_{N_{\la_{2}},\dots,N_{\la_{r-1}}\in \Irr} 
\mathrm{Hom}(A(M_{1})\otimes_{A(V)}\Om(N_{\la_{2}}),\Om(M_0)) 
\otimes
\mathrm{Hom}(A(M_{2})\otimes_{A(V)}\Om(N_{\la_{3}}),\Om(N_{\la_{2}}))\\
&\otimes \cdots\otimes
\mathrm{Hom}(A(M_{r-1})\otimes_{A(V)}\Om(N_{\la_{r-1}}),\Om(N_{\la_{r-2}})) \otimes 
\mathrm{Hom}(A(M_{r-1})\otimes_{A(V)}\Om(M_{r}),\Om(N_{\la_{r-1}}))\\
&\cong
\mathrm{Hom}\Bigr(A(M_{1})\otimes_{A(V)}
A(M_2)\otimes_{A(V)}
\cdots \otimes_{A(V)}
A(M_{r-1})\otimes_{A(V)}\Om(M_r),
\Om(M_0)\Bigr).
\label{eq_AV_decomp}
\end{split}
\end{align}
Let $N_{\la_i}\in \Irr$ for $i=2,3,\dots,r-1$.
Assume that there exists $\Delta',\Delta_i, h_{i} \in \C$ 
such that:
\begin{align}
\begin{split}
M_0&=\bigoplus_{k\geq 0}(M_0)_{\Delta'+k}\\
M_i&=\bigoplus_{k\geq 0}(M_i)_{\Delta_i+k}\\
N_{\la_i}&=\bigoplus_{k\geq 0}(N_{\la_i})_{h_i+k}
\label{eq_N_la}
\end{split}
\end{align}
for $i=1,\dots,r$.
Let $f_1 \in \mathrm{Hom}(A(M_{1})\otimes_{A(V)}\Om(N_{\la_{2}}),\Om(M_0))$,
$f_i \in \mathrm{Hom}(A(M_{i})\otimes_{A(V)}\Om(N_{\la_{i+1}}),\Om(N_{\la_{i}}))$ ($i=2,3,\dots,r-2$)
and
$f_{r-1}\in \mathrm{Hom}(A(M_{r-1})\otimes_{A(V)}\Om(M_{r}),\Om(N_{\la_{r-1}}))$.
Define 
\begin{align*}
\mathrm{comp}(f_1,\dots,f_{r-1}):
A(M_{1})\otimes_{A(V)}A(M_2)\otimes_{A(V)}\cdots \otimes_{A(V)}A(M_{r-1})\otimes_{A(V)} \Om(M_r)
\rightarrow
\Om(M_0)
\end{align*}
by the composition of $f_1,\dots,f_{r-1}$.
By Theorem \ref{thm_FZ},
$\{f_i\}_{i\in 1,\dots,r-1}$ correspond
to the intertwining operator 
\begin{align*}
\Y_{f_1}(\bullet,z) &\in I\binom{M_0}{M_1N_{\la_2}}\\
\Y_{f_i}(\bullet,z) &\in I\binom{N_{\la_{i}}}{M_i N_{\la_{i+1}}}\quad \text{ for }i=2,\dots,r-2\\
\Y_{f_{r-1}}(\bullet,z) & \in I\binom{N_{\la_{r-1}}}{M_{r-1}M_r}.
\end{align*}
Define
$F_{f_1,\dots,f_{r-1}}:
\Mrr \rightarrow M_0[[x_1^\C,\log x_1,\dots,x_{r}^\C,\log x_r]]$ by
\begin{align}
F_{f_1,\dots,f_{r-1}}(\mrr)
= \Y_1(m_1,x_1)\Y_2(m_2,x_2)\cdots \Y_{r-2}(m_{r-2},x_{r-2})\Y_{r-1}(m_{r-1},x_{r-1})m_r
\end{align}
for $\mrr \in \Mrr$.
Then, by Theorem \ref{thm_composition},
$F_{f_1,\dots,f_{r-1}} \in \PI_{\std_r}\binom{M_0}{\Mrr}$
and by Lemma \ref{lem_P} and \eqref{eq_tilde_F_def}
\begin{align}
\begin{split}
&\tilde{F}_{f_1,\dots,f_{r-1}}(\mrr)\\
&= \exp(L(-1)z_r) C_{\std} \Y_1(m_1,x_1)\Y_2(m_2,x_2)\cdots \Y_{r-2}(m_{r-2},x_{r-2})\Y_{r-1}(m_{r-1},x_{r-1})m_r\\
&= \exp(L(-1)z_r) \Y_1(m_1,z_1-z_r)\Bigl|_{|z_1|>|z_r|}\Y_2(m_2,z_2-z_r)\Bigl|_{|z_2|>|z_r|}\cdots \Y_{r-1}(m_{r-1},z_{r-1}-z_r)\Bigl|_{|z_{r-1}|>|z_r|}m_r\\
&= \Y_1(m_1,z_1)\Y_2(m_2,z_2)\cdots \Y_{r-1}(m_{r-1},z_{r-1})\exp(L(-1)z_r) m_r.
\label{eq_tilde_F_z}
\end{split}
\end{align}

Then, by using the same notation in \eqref{eq_bar_F},
we have:
\begin{lem}
\label{comp_formal}
For any $m_{[r]} \in M_{[r]}$,
\begin{align*}
\ze_1^{h_1}\cdots \ze_{r-2}^{h_{r-2}}\ze_{r-1}^{\Delta_r}
\bF_{f_1,\dots,f_{r-1}}(\mrr)
\in M_0[[\ze_1,\dots,\ze_{r-1}]]
\end{align*}
and the constant term is 
equal to the linear map $\comp(f_1,\dots,f_{r-1})$.
\end{lem}
\begin{proof}
Let $t_i \in \C$ for $i=1,\dots,r$.
The coefficient of $z_r^{t_r}$ in $\exp(L(-1)z_r)z_r^{L(0)}m_r$ is in $(M_r)_{t_r}$,
which is zero unless $t_r \in \Delta_r +\Z_{\geq 0}$.
By Lemma \ref{lem_P} (6),
the coefficient of $z_{r-1}^{t_{r-1}}z_r^{t_r}$ in $\Y_{f_{r-1}}(z_{r-1}^{L(0)}m_{r-1},z_{r-1})\exp(L(-1)z_r)z_{r-1}^{L(0)}m_r$ is in $(N_{\la_{r-1}})_{t_{r-1}+t_r}$, 
by \eqref{eq_N_la},
which is zero unless 
$$
t_{r-1}+t_r \in h_{r-1}+\Z_{\geq 0}.
$$
Similarly,
the coefficient of $z_{r-2}^{t_{r-2}}z_{r-1}^{t_{r-1}}z_r^{t_r}$ in 
$\Y_{f_{r-2}}(z_{r-2}^{L(0)}m_{r-2})\Y_{f_{r-1}}(z_{r-1}^{L(0)}m_{r-1},z_{r-1})\exp(L(-1)z_r)z_{r}^{L(0)}m_r$ is in 
$(N_{\la_{r-2}})_{t_{r-2}+t_{r-1}+t_r}$, 
by \eqref{eq_N_la},
which is zero unless 
$$
t_{r-2}+t_{r-1}+t_r \in h_{r-2}+\Z_{\geq 0}.
$$
Set $h_r=\Delta_r$.
By repeating this argument, 
we see that if the coefficients of $z_1^{t_1}\cdots z_{r-1}^{t_{r-1}}z_r^{t_r}$ 
in 
\begin{align*}
\Y_1(z_1^{L(0)}m_1,z_1)\Y_2(z_2^{L(0)}m_2,z_2)\cdots \Y_{r-2}(z_{r-2}^{L(0)}m_{r-2},z_{r-2})\Y_{r-1}(z_{r-1}^{L(0)}m_{r-1},z_{r-1})\exp(L(-1)z_r)z_{r}^{L(0)}m_r
\end{align*}
are non-zero, 
then $\{t_i\}_{i=1,\dots,r-1}$ satisfies 
\begin{align*}
\sum_{r \geq i \geq p}t_i
&\in h_{i} +\Z_{\geq 0}\quad \text{ for any } r \geq p \geq 1\\
\Delta' &= \sum_{i =1}^r t_i
\end{align*}
or equivalently
\begin{align*}
\begin{cases}
t_{r} = \Delta_{r} +k_{r}\\
t_{i} = h_i-h_{i+1} +k_i-k_{i+1} & (i=2,\dots,r-1)\\
t_1 = \Delta'-h_2-k_2
\end{cases}
\end{align*}
for some $k_i \in \Z_{\geq 0}$ ($i=1,\dots,r$).
Then, 
$
z_1^{t_1}\cdots z_r^{t_r}=z_1^{\Delta'}
(\Pi_{i=2}^{r-1} \left(\frac{z_{i}}{z_{i-1}}\right)^{h_i+k_i})$.
Hence, the assertion holds.
\end{proof}
Now, Proposition \ref{thm_factor} follows from Lemma \ref{comp_formal}, Proposition \ref{prop_factor},
Proposition \ref{prop_omega_action}
and Corollary \ref{cor_factor}.

\begin{cor}
The natural transformation 
\begin{align*}
\comp_{\std_r,\emptyset,p}
:\PI_{\std_r}\circ_p \PI_{\emptyset} \rightarrow
\PI_{\std_r \circ_p \emptyset}
\end{align*}
is an isomorphism for any $r\geq 1$ and $p \in [r]$.
\end{cor}
\begin{proof}
For $r=1$, we have $\PI_{1}\circ_1 \PI_{\emptyset}=\PI\binom{\bullet}{V} \cong \mathrm{Hom}(V,\bullet)$
by Proposition \ref{two_inter}.
Assume that $r\geq 2$.
By Proposition \ref{thm_factor}, if $p <r$, then
\begin{align*}
&\PI_{\std_r}\binom{M_0}{M_1\dots M_{p-1}VM_{p+1}\dots M_r}\\
&\cong
\mathrm{Hom}\left(A(M_1)\otimes_{A(V)} \cdots \otimes_{A(V)} A(M_{p-1})\otimes_{A(V)} A(V)
\otimes_{A(V)} A(M_{p+1})\otimes_{A(V)} \cdots \otimes_{A(V)} \Om(M_r), \Om(M_0)\right)\\
&\cong
\mathrm{Hom}\left(A(M_1)\otimes_{A(V)} \cdots \otimes_{A(V)} A(M_{p-1})\otimes_{A(V)}A(M_{p+1})\otimes_{A(V)} \cdots \otimes_{A(V)} \Om(M_r), \Om(M_0)\right)\\
&\cong\PI_{\std_{r-1}}\binom{M_0}{M_1\dots M_{p-1}M_{p+1}\dots M_r}.
\end{align*}
Hence, the dimension of the vector spaces
$\PI_{\std_r}\binom{M_0}{M_1\dots M_{p-1}VM_{p+1}\dots M_r}$
and $\PI_{\std_{r-1}}\binom{M_0}{M_1\dots M_{p-1}M_{p+1}\dots M_r}$ are the same.
By Lemma \ref{lem_empty_lift},
$\comp_{\std_r,\emptyset,p}$ is injective
and thus $\comp_{\std_r,\emptyset,p}$ is a natural isomorphism.
In the case of $p=r$, by Lemma \ref{lem_V0} and Lemma \ref{lem_functor_HL},
\begin{align*}
&\PI_{\std_r}\binom{M_0}{M_1\dots M_{r-1}V}\\
&\cong
\mathrm{Hom}\left(A(M_1)\otimes_{A(V)} \cdots \otimes_{A(V)} A(M_{r-1}) \otimes_{A(V)} \Om(V), \Om(M_0)\right)\\
&\cong
\mathrm{Hom}\left(A(M_1)\otimes_{A(V)} \cdots \otimes_{A(V)} L(M_{r-1}), \Om(M_0)\right)\\
&\cong
\mathrm{Hom}\left(A(M_1)\otimes_{A(V)} \cdots \otimes_{A(V)} \Om(M_{r-1}), \Om(M_0)\right)\\
&\cong\PI_{\std_{r-1}}\binom{M_0}{M_1\dots M_{r-1}}.
\end{align*}
Hence, the assertion holds.
\end{proof}

Now, all the assumptions in Theorem \ref{thm_add_BTC} are shown for a rational $C_2$-cofinite vertex operator algebra. Hence, $(\Vmodf,\boxtimes,\rho)$ is a braided tensor category.

As a consequence, we obtain the following description of \(r\)-point conformal blocks:
\begin{thm}
\label{cor_factorize}
Assume that $V$ is a simple rational $C_2$-cofinite vertex operator algebra.
Then, for any $M_0,M_1,\dots,M_r\in \Vmodf$ and $A\in \Tr_r$,
\begin{align*}
\CB_{\Mr}(U_A)\cong \mathrm{Hom}_\Vmodf(\boxtimes_A \Mrr, M_0)
\end{align*}
as functors.
\end{thm}

\subsection{Explicit formula for braiding and associator}
\label{sec_explicit_braiding}
In this section, we describe explicitly the braiding and the associator
constructed in Theorem \ref{thm_add_BTC} in terms of conformal blocks.
Recall that, for \(M_1,M_2\in \Vmodf\), we have
\begin{align}
M_1\boxtimes M_2
  \cong
  \bigoplus_{\lambda\in\Irr}
  N_\lambda \otimes_\C
  I\binom{N_\lambda}{M_1\,M_2}^{*}.
\label{eq_tensor_last}
\end{align}
Let
\[
  R_{\lambda;M_1,M_2}:
  I\binom{N_\lambda}{M_1\,M_2}
  \rightarrow
  I\binom{N_\lambda}{M_2\,M_1}
\]
be the linear isomorphism induced from the analytic continuation along the clockwise representative of
\(\sigma:(12)\to(21)\)
fixed in Remark \ref{rem_fix_braid}.
More explicitly, \(R_{\lambda;M_1,M_2}\) is the composition
\begin{align*}
I\binom{N_\lambda}{M_1\,M_2} \xrightarrow{s_{12}} \CB_{N_\lambda,M_1,M_2}(U_{12}) \xrightarrow{A(\si)} \CB_{N_\lambda,M_1,M_2}(U_{21}) \xrightarrow{e_{21}}I\binom{N_\lambda}{M_2\,M_1}.
\end{align*}
Here \(s_{12}\) and \(e_{21}\) denote the standard identifications between
intertwining operators and conformal blocks developed in Section \ref{sec_D_formal}.
By the definition of the braiding in
Theorem \ref{thm_add_BTC}, $B_{M_1,M_2}:M_1\boxtimes M_2\longrightarrow M_2\boxtimes M_1$ is the Yoneda pullback of \(\rho(\sigma)^{-1}\), which is given by
\begin{align}
  B_{M_1,M_2}
  =
  \bigoplus_{\lambda\in\Irr}
  \id_{N_\lambda}\otimes \left((R_{\lambda;M_1,M_2})^{-1}\right)^*
  \label{eq_explicit_braiding}
\end{align}
where $\left((R_{\lambda;M_1,M_2})^{-1}\right)^*: I\binom{N_\lambda}{M_1\,M_2}^* \rightarrow I\binom{N_\lambda}{M_2\,M_1}^*$ 
denotes the transpose of its inverse.

We next describe the associator. Define
\[
  \mathcal H^{(12)3}_\lambda
  =
  \bigoplus_{\mu\in\Irr}
  I\binom{N_\lambda}{N_\mu\,M_3}
  \otimes_\C
  I\binom{N_\mu}{M_1\,M_2},
\]
and
\[
  \mathcal H^{1(23)}_\lambda
  =
  \bigoplus_{\nu\in\Irr}
  I\binom{N_\lambda}{M_1\,N_\nu}
  \otimes_\C
  I\binom{N_\nu}{M_2\,M_3}.
\]
These spaces are identified with the spaces of conformal blocks associated
with the trees \((12)3\) and \(1(23)\), respectively, with output
\(N_\lambda\).
Let
\[
  F_{\lambda;M_1,M_2,M_3}:
  \mathcal H^{(12)3}_\lambda
  \longrightarrow
  \mathcal H^{1(23)}_\lambda
\]
be the linear isomorphism induced from the analytic continuation along the path \(\al:(12)3\to1(23)\) in $X_3(\C)$.
Equivalently, \(F_{\lambda;M_1,M_2,M_3}\) is the the composition
\begin{align*}
&\bigoplus_{\mu\in\Irr}
I\binom{N_\lambda}{N_\mu\,M_3}
  \otimes_\C
  I\binom{N_\mu}{M_1\,M_2} \xrightarrow{\comp_1} I_{(12)3}\binom{N_\la}{M_1\,M_2\,M_3}  
  \xrightarrow{s_{(12)3}} \CB_{N_\lambda,M_1,M_2,M_3}(U_{(12)3})\\
   &\xrightarrow{A(\al)} \CB_{N_\lambda,M_1,M_2,M_3}(U_{1(23)})\xrightarrow{e_{1(23)}} I_{1(23)}\binom{N_\lambda}{M_1\,M_2\,M_3} \xrightarrow{\comp_2^{-1}} \bigoplus_{\nu\in\Irr}
  I\binom{N_\lambda}{M_1\,N_\nu}
  \otimes_\C
  I\binom{N_\nu}{M_2\,M_3}.
\end{align*}
Using
\[
  (M_1\boxtimes M_2)\boxtimes M_3
  \cong
  \bigoplus_{\lambda\in\Irr}
  N_\lambda\otimes_\C
  \left(\mathcal H^{(12)3}_\lambda\right)^*
\]
and
\[
  M_1\boxtimes(M_2\boxtimes M_3)
  \cong
  \bigoplus_{\lambda\in\Irr}
  N_\lambda\otimes_\C
  \left(\mathcal H^{1(23)}_\lambda\right)^*,
\]
the associator
\[
  \alpha_{M_1,M_2,M_3}:
  (M_1\boxtimes M_2)\boxtimes M_3
  \longrightarrow
  M_1\boxtimes(M_2\boxtimes M_3)
\]
is given by
\begin{align}
  \alpha_{M_1,M_2,M_3}
  =
  \bigoplus_{\lambda\in\Irr}
  \left(F_{\lambda;M_1,M_2,M_3}^{-1}\right)^*
  \otimes \id_{N_\lambda}.
  \label{eq_explicit_associator}
\end{align}

Finally, we will show that $\Vmodf$ is balanced (see Definition \ref{def_balanced}).
For any $M\in \Vmodf$,
let $\theta$ be a linear map defined by
\begin{align}
\theta_M: M\rightarrow M,\quad m\mapsto \exp(2\pi i L(0))m.
\label{eq_balancing_def}
\end{align}
Since 
\begin{align*}
\theta_M(a(n)m)&= \exp(2\pi i L(0))a(n)m\\
&=a(n) \exp(2\pi i (\Delta_a -n-1)) \exp(2\pi i L(0))m\\
&=a(n) \theta_M(m),
\end{align*}
$\theta_M$ is a $V$-module homomorphism
and, in fact, a natural isomorphism from the identity functor on $\Vmodf$ to itself.
Assume now that \(M_1,M_2,N_\la\) are simple and write
\[
  M_i=\bigoplus_{k\geq0}(M_i)_{\Delta_i+k},
  \qquad
  N_\lambda=\bigoplus_{k\geq0}(N_\lambda)_{\Delta_\lambda+k}.
\]
It suffices to show that
\begin{align*}
\theta_{M_1\boxtimes M_2}\circ (\theta_{M_1}^{-1}\boxtimes \theta_{M_2}^{-1})=B_{M_2,M_1}\circ B_{M_1,M_2}
\end{align*}
holds.
Since any $\Y(\bullet,z) \in I\binom{N_\la}{M_1M_2}$ is of the form
\begin{align}
\Y(\bullet,z)=\sum_{n \in \Z}a(n+\Delta_1+\Delta_2-\Delta_{\la}) z^{\Delta_\la -\Delta_1-\Delta_2-n-1},
\label{eq_exp_last}
\end{align}
by Corollary \ref{cor_twist} and \eqref{eq_exp_last}, the monodromy map with respect to the clockwise rotation
\begin{align*}
I\binom{N_\la}{M_1M_2} \overset{R_{\la;M_1,M_2}}{\longrightarrow}I\binom{N_\la}{M_2M_1} \overset{R_{\la;M_2,M_1}}{\longrightarrow}I\binom{N_\la}{M_1M_2}
\end{align*}
is given by $\exp(-2\pi i (\Delta_\la-\Delta_1-\Delta_2))$.
Hence, the action of $B_{M_2,M_1}\circ B_{M_1,M_2}$ on the $\la$-component of \eqref{eq_tensor_last}
is 
\begin{align}
\exp(2\pi i (\Delta_\la-\Delta_1-\Delta_2))
\end{align}
by \eqref{eq_explicit_braiding},
which obviously coincides with $\theta_{M_1\boxtimes M_2}\circ (\theta_{M_1}^{-1}\boxtimes \theta_{M_2}^{-1})$.
Hence, we have:
\begin{thm}
\label{thm_BTC}
Assume that $V$ is a simple rational $C_2$-cofinite vertex operator algebra.
Then, $(\Vmodf,\boxtimes,B,\al,l,r,\theta)$ is a balanced braided tensor category.
\end{thm}

\appendix
\def\thesection{\Alph{section}}
\section{Expansion and convergence}
\label{sec_appendix_A}

For $p >0$,
set
\begin{align*}
\D_p&=\{\ze\in \C\mid |\ze|<p\},\\
\D_p^\times&=\{\ze\in \C\mid 0<|\ze|<p\},\\
\D_p^\cut&=\{\ze\in \C^\cut\mid |\ze|<p \}.
\end{align*}
Let $p_1,\dots,p_r \in \R_{>0}$
and set $\mathfrak{p}=(p_1,\dots,p_r)$
\begin{align*}
\D_\mfp^a = \D_{p_1}^a \times \cdots \times \D_{p_r}^a, \quad \text{ for }a =\emptyset, \times, \cut.
\end{align*}
Let \(D^{\an}(V)\) denote the sheaf of analytic differential operators on an
open subset \(V\subset D_p\).

For $i=1,\dots,r$, let $N_i \geq 1$ and $\{f_k^i\}_{k=0,\dots,N_i-1}$ be holomorphic functions on $\D_\mfp^\times$ satisfying the following condition:
\begin{enumerate}
\item[PC)]\label{condition_PC}
$f_k^i(\ze_1,\dots,\ze_r)$ has an expansion of the form
in $\C[[\ze_i]]((\ze_1,\ze_2,\dots,\ze_{i-1},\ze_{i+1},\dots,\ze_r))$.
\end{enumerate}
That is, $f_k^i(\ze_1,\dots,\ze_r)$ is holomorphic at $\ze_i=0$ and has possible poles at $\ze_p=0$ ($p\neq i$).
For $i=1,\dots,r$, set
\begin{align}
P_i =(\ze_i \pa_i)^{N_i} + \sum_{k=0}^{N_i-1} f_k^{i} (\ze_i\pa_i)^k, \label{eq_differential_appendix}
\end{align}
which are elements of $D^\an(\D_\mfp^\times)$.
Let $D^\an(\D_\mfp^\times)\langle P_1,\dots,P_r \rangle$ be the left ideal of $D^\an(\D_\mfp^\times)$ generated by $P_1,\dots,P_r$
and set $M_P =D^\an(\D_\mfp^\times) / D^\an(\D_\mfp^\times)\langle P_1,\dots,P_r \rangle$,
which is a left $D^\an(\D_\mfp^\times)$-module.
It is clear that $M_P$ is a finitely generated $\mO_{\D_\mfp^\times}^\an$-module.
In particular, the sheaf of holomorphic solutions
\begin{align*}
V \mapsto \Hom_{D^\an(V)}(M_P, \mO^\an(V))\quad \quad\text{ for }V \subset \D_\mfp^\times
\end{align*}
 is a locally constant sheaf of finite rank.

We recall that 
$\C[[\ze_1,\dots,\ze_r]]_\mfp^\conv$ is the space of formal power series which converges
absolutely in $\D_\mfp$.
We will show the following proposition:
\begin{prop}
\label{thm_appendix}
Any holomorphic solution in $\Hom_{D^\an(\D_\mfp^\cut)}(M_P, \mO^\an(\D_\mfp^\cut))$
has a series expansion of the form 
$\C[[\ze_1,\dots,\ze_r]]_\mfp^\conv [\ze_1^\C,\log \ze_1,\dots,\ze_r^\C,\log \ze_r]$.
\end{prop}

To prove this, we will use the following lemma \cite[Lemma 2.4.13]{SST},
which follows from Baire's category theorem:
\begin{lem}{\cite{SST}}
\label{lem_SST}
Let $\psi(\ze)$ be a holomorphic function on $\D_\mfp^\times$.
Assume that for all $i \in [r]$ and all fixed values of
\begin{align*}
(\ze_1,\dots,\ze_{i-1},\ze_{i+1},\dots,\ze_r) \in \D_{p_1}^\times \times \D_{p_2}^\times \cdots \times \D_{p_{i-1}}^\times \times \D_{p_{i+1}}^\times  \times \cdots \times D_{p_r}^\times,
\end{align*}
$\psi(\ze)$ has a pole of finite order as a function of $\ze_i$.
Then, $\psi(\ze)$ is meromorphic at $\ze=0$.
\end{lem}

We will use the following lemma:
\begin{lem}
\label{lem_fixed_pole}
Let $i \in [r]$ and fix values of 
\begin{align*}
(\ze_1,\dots,\ze_{i-1},\ze_{i+1},\dots,\ze_r) \in \D_{p_1}^\times \times \D_{p_2}^\times \cdots \times \D_{p_{i-1}}^\times \times \D_{p_{i+1}}^\times  \times \cdots \times D_{p_r}^\times.
\end{align*}
There exists $N \in \Z$ such that
$\lim_{z_i \to 0} \ze_i^N \psi=0$ for any $\psi \in \Hom_{D^\an(\D_\mfp^\cut)}(M_P, \mO^\an(\D_\mfp^\cut))$.
\end{lem}
\begin{proof}
Let $\psi \in \Hom_{D^\an(\D_\mfp^\cut)}(M_P, \mO^\an(\D_\mfp^\cut))$.
Then, $\psi$ satisfies the differential equation
\begin{align*}
P_i \psi =\left( (\ze_i \pa_i)^{N_i} + \sum_{k=0}^{N_i-1} f_k^{i} (\ze_i\pa_i)^k \right)\psi=0
\end{align*}
and $f_k^i(\ze_1,\dots,\ze_r)$ is holomorphic at $\ze_i=0$,
$P_i$ is a differential equation with a regular singular point at $\ze_i=0$.
Hence, for each fixed values $\ze_1,\dots,\ze_{i-1},\ze_{i+1},\dots,\ze_r$,
there exists $N \in \Z$ such that $\lim_{\ze_i \to 0} \ze_i^N \psi=0$.
Furthermore, $N$ can be taken independent of $\psi$.
\end{proof}

\begin{proof}[proof of Proposition \ref{thm_appendix}]
Let $\psi_1,\dots,\psi_N$ be a basis of $\Hom_{D^\an(\D_\mfp^\cut)}(M_P, \mO^\an(\D_\mfp^\cut))$.
and set $\Psi=(\psi_1,\dots,\psi_N)$, which is a vector-valued holomorphic function.
The fundamental group $\pi_1(\D_\mfp^\times)=\Z^r$ is a free abelian group generated by $r$ elements $\ga_i$ which encircle $\ze_i=0$.
Consider the analytic continuation $\ga_i \Psi$ of $\Psi$ along the path $\ga_i$.
Then, there exists a matrix $M_i$ such that $\ga_i \Psi =\Psi M_i$.
Since $\pi_1(\D_\mfp^\times)$ is Abelian, the matrices $M_i$ commute with each other and 
are invertible. Then, there exists a commuting family of matrices $L_i$ such that
$\exp(2\pi i L_i) =M_i$. Set
\begin{align*}
s(\ze) = \Psi(\ze) \ze_1^{-L_1}\cdots \ze_r^{-L_r}.
\end{align*}
Then, each component of $s(\ze)$ is a single valued meromorphic function on $\D_\mfp^\times$.
By Lemma \ref{lem_fixed_pole}, for each $i \in \{1,\dots,r\}$ and fixed values $\ze_1,\dots,\ze_{i-1},\ze_{i+1},\dots,\ze_r$, each component of $s(\ze)$ has a pole of finite order at $\ze_i=0$.
Hence, by Lemma \ref{lem_SST}, $s(\ze)$ is a meromorphic function at $\ze=0$
and thus $\{\psi_k\}_{k=1,\dots,N}$ has a convergent expansion of the form $\C[[\ze_1,\dots,\ze_r]]_\mfp^\conv[\ze_1^\C,\log \ze_1,\dots,\ze_r^\C,\log \ze_r]$.
Hence, Proposition \ref{thm_appendix} is obtained.
\end{proof}

Next, we will show a result of a convergence of formal power series.


%
%
%
\begin{lem}
\label{lem_Frob}
Let $r \geq 0$, $p_1,\dots,p_r,q \in \R_{>0}$ and $N \in \Z_{>0}$,
$\{f_k(\ze_1,\dots,\ze_r,\xi)\}_{k=0,1,\dots,N-1},g(\ze_1,\dots,\ze_r,\xi)$
be continuous functions on $\D_\mfp^\cut \times \D_{q}$
which is holomorphic with respect to $\xi \in \D_q$
and $\{c_l(\ze_1,\dots,\ze_r)\}_{l=0,1,2,\dots}$ be continuous functions on $\D_\mfp^\cut$.
Set
\begin{align}
\psi=\sum_{l=0}^\infty c_l(\ze_1,\dots,\ze_r)\xi^l,
\label{eq_app_psi}
\end{align}
which is a formal power series of continuous function coefficients.
Assume that $\psi$ formally satisfies
\begin{align}
\left(\left(\xi \frac{d}{d\xi}\right)^N + \sum_{k=0}^{N-1} f_k(\ze_1,\dots,\ze_r,\xi)\left(\xi\frac{d}{d\xi}\right)^k\right)\psi=g(\ze_1,\dots,\ze_r,\xi).
\label{eq_diff_xi}
\end{align}
Then, the sum
\begin{align*}
\sum_{l=0}^\infty c_l(\ze_1,\dots,\ze_r)\xi^l
\end{align*}
converges absolutely and locally uniformly on $\D_\mfp^\cut \times \D_q$.
\end{lem}

\begin{proof}
Set $f_N(\ze_1,\dots,\ze_r,\xi)=1$ and let $a_{k,n}(\ze_1,\dots,\ze_r)$ and $b_{n}(\ze_1,\dots,\ze_r)$ be the coefficient of $\xi^n$ ($n=0,1,2,\dots$) in $f_k(\ze_1,\dots,\ze_r,\xi)$ and $g(\ze_1,\dots,\ze_r,\xi)$, respectively, for $k=0,\dots,N$.
Then,
\begin{align}
\begin{split}
f_k(\ze_1,\dots,\ze_r,\xi) &= \sum_{n \geq 0}a_{k,n}(\ze_1,\dots,\ze_r) \xi^n,\\
g(\ze_1,\dots,\ze_r,\xi) &= \sum_{n \geq 0}b_n(\ze_1,\dots,\ze_r)\xi^n.
\end{split}
\label{eq_app_fg}
\end{align}
Taking the coefficient of $\xi^p$ in the differential equation \eqref{eq_diff_xi},
by \eqref{eq_app_fg} and \eqref{eq_app_psi},
we have
\begin{align*}
\sum_{k=0}^{N} \sum_{p \geq l \geq 0} l^k a_{k,p-l}(\ze_1,\dots,\ze_r)c_l(\ze_1,\dots,\ze_r)
= b_p(\ze_1,\dots,\ze_r).
\end{align*}
For $t\geq 0$, set $P_{t}(l)= \sum_{k=0}^N l^k a_{k,t}(\ze_1,\dots,\ze_r)$.
Since $a_{N,p}(\ze_1,\dots,\ze_r)=0$ for any $p\geq 1$,
\begin{align}
P_{0}(p) c_p(\ze_1,\dots,\ze_r) = b_p(\ze_1,\dots,\ze_r) -\sum_{p> l \geq 0}P_{p-l}(l)c_l(\ze_1,\dots,\ze_r).
\label{eq_recursion_diff}
\end{align}
Let $K$ be a compact subset of $\D_\mfp^\cut$ and $R \in \R$ satisfy $q>R>0$.
Set
\begin{align*}
C_k = \max_{\substack{(\ze_1,\dots,\ze_r) \in K\\ \theta \in [0,2\pi]}} |f_k(\ze,R\exp(i\theta))|
\;\; \text{ and }\;\; C'=  \max_{\substack{(\ze_1,\dots,\ze_r) \in K\\ \theta \in [0,2\pi]}} |g(\ze,R\exp(i\theta))|.
\end{align*}
Since $f_k$ and $g$ are continuous on $\D_\mfp^\cut \times \D_q$, these numbers are finite.
Set $C= \max \{C_0,\dots,C_{N-1},C'\}$.
Then, by the Cauchy inequality,
\begin{align*}
|a_{k,t}(\ze_1,\dots,\ze_r)| \leq R^{-t}C \;\;\text{ and }\;\;
|b_t(\ze_1,\dots,\ze_r)| \leq R^{-t}C
\end{align*}
for any $t\geq 0$, $k=0,\dots,N-1$ and $(\ze_1,\dots,\ze_r)\in K$.
Hence, if $t \geq 1$,
\begin{align*}
|P_t(l)|\leq N C R^{-t}l^{N-1}.
\end{align*}
Since $P_0(l)=l^N+\sum_{k=0}^{N-1}a_{k,0}(\ze_1,\dots,\ze_r) l^k$,
$P_0(l)$ is a polynomial of degree $N$ in $l$
and its coefficients are bounded for any $(\ze_1,\dots,\ze_r)\in K$.
Hence, for any real number $B>0$,
there exists $M \geq 0$ such that $|P_0(l)| \geq B C l^{N-1}$
for any $l\geq M$.

Then, by \eqref{eq_recursion_diff} and combining the above inequalities,
we have
\begin{align*}
|c_p(\ze_1,\dots,\ze_r)| &\leq (BC)^{-1}p^{-(N-1)}
\left(CR^{-p} +N C R^{-p}\sum_{l=0}^{p-1}R^l|c_l(\ze_1,\dots,\ze_r)|l^{N-1}\right)
\end{align*}
for any $l \geq M$.
Hence, we have
\begin{align*}
R^p|c_p(\ze_1,\dots,\ze_r)| \leq  B^{-1}\left(1+N \sum_{l=0}^{p-1}R^l |c_l(\ze_1,\dots,\ze_r)|\right)
\text{ for any }p\geq M \text{ and }(\ze_1,\dots,\ze_r) \in K.
\end{align*}
To show that
$\sum_{l\geq 0}c_l(\ze_1,\dots,\ze_r)\xi^l$ converges uniformly and absolutely in $K \times \{|z|<R\}$, we will use the comparison test.
Set 
\begin{align*}
d_p =\begin{cases}
R^p|c_p(\ze_1,\dots,\ze_r)| & p \leq M\\
B^{-1}\left(1+N\sum_{l=0}^{p-1}d_l \right) & p>M
\end{cases}
\end{align*}
and $F(\xi)=\sum_{p=0}^\infty d_p \xi^p \in \C[[\xi]]$.
Then, $F(\xi)$ satisfies
\begin{align}
F(\xi) &= \sum_{p\leq M}\xi^p \left(d_p-B^{-1}-B^{-1}N\sum_{l=0}^{p-1}d_l \right)
+\sum_{p=0}^\infty \xi^p \left(B^{-1}+N\sum_{l=0}^{p-1}d_l\right)\\
&= Q(\xi)+B^{-1}\left(\frac{1}{1-\xi}+N\frac{\xi F(\xi)}{1-\xi}\right), \label{eq_W}
\end{align}
where $Q(\xi)=\sum_{p\leq M}\xi^p \left(d_p-B^{-1}-N\sum_{l=0}^{p-1}d_l \right)$,
a polynomial of degree $M$.
Hence, 
\begin{align*}
F(\xi)=\frac{B^{-1}+(1-\xi)Q(\xi)}{1-(1+NB^{-1})\xi},
\end{align*}
which is absolutely convergent if $|\xi|<\frac{1}{1+NB^{-1}}$.
Hence, $\sum_{l\geq 0}c_l(\ze_1,\dots,\ze_r)\xi^l$ converges uniformly and absolutely in 
$K \times \{|\xi|<\frac{R}{1+NB^{-1}} \}.$
Since $B>0$ and $0<R<q$ could be chosen arbitrarily,
$\psi$ converges uniformly and absolutely in $K \times \D_q$.
\end{proof}

\begin{prop}
\label{prop_one_convergence}
Let $p_1,\dots,p_r,q \in \R_{>0}$ and $N \in \Z_{>0}$,
$\{f_k(\ze_1,\dots,\ze_r,\xi)\}_{k=0,1,\dots,N-1}$
be continuous functions on $\D_\mfp^\cut \times \D_{q}$
which is holomorphic with respect to $\xi \in \D_q$.
Let $M \in \Z_{\geq 0}$ and $\{c_{l,i}(\ze_1,\dots,\ze_r)\}_{l =0,1,2,\dots,i=0,1,\dots,M}$ be continuous functions on $\D_\mfp^\cut$ 
and set
\begin{align*}
\psi=\sum_{i=0}^M \sum_{l \geq 0}c_{l,i}(\ze_1,\dots,\ze_r)\xi^l (\log\xi)^i.
\end{align*}
Assume that $\psi$ formally satisfies 
\begin{align}
\left(\left(\xi \frac{d}{d\xi}\right)^N + \sum_{k=0}^{N-1} f_k(\ze_1,\dots,\ze_r,\xi)\left(\xi \frac{d}{d\xi}\right)^k\right)\psi=0.
\label{eq_app_diff}
\end{align}
Then, for any $i=0,\dots,M$ the series
\begin{align*}
\sum_{l \geq 0}c_{l,i}(\ze_1,\dots,\ze_r)\xi^l 
\end{align*}
converges absolutely and locally uniformly on $\D_\mfp^\cut \times \D_q$.
\end{prop}
\begin{proof}
Set
\begin{align*}
\psi_i = \sum_{l \geq 0}c_{l,i}(\ze_1,\dots,\ze_r)\xi^l
\end{align*}
for $i=0,\dots,M$.
By taking the coefficients of $(\log \xi)^M$ in the differential equation \eqref{eq_app_diff},
we have
\begin{align*}
\left(\left(\xi \frac{d}{d\xi}\right)^N + \sum_{k=0}^{N-1} f_k(\ze_1,\dots,\ze_r,\xi)\left(\xi \frac{d}{d\xi}\right)^k\right)\psi_M=0.
\end{align*}
Hence, by Lemma \ref{lem_Frob}, the series $\psi_M$ converges absolutely and uniformly,
and thus $\psi_M$ is a continuous function on $\D_\mfp^\times \times \D_q$ and holomorphic with respect to $\xi$.

Similarly, by taking the coefficients of $(\log \xi)^{M-1}$ in \eqref{eq_app_diff}
we have
\begin{align*}
&\left(\left(\xi \frac{d}{d\xi}\right)^N + \sum_{k=0}^{N-1} f_k(\ze_1,\dots,\ze_r,\xi)\left(\xi \frac{d}{d\xi}\right)^k\right)
\psi_{M-1}(\ze_1,\dots,\ze_r,\xi)\\
&=\sum_{k=1}^{M}  f_k(\ze_1,\dots,\ze_r,\xi) Mk \left(\xi\frac{d}{d\xi}\right)^{k-1} \psi_{M}.
\end{align*}
Since the right hand side is a continuous function on $\D_\mfp^\times\times \D_q$ and holomorphic with respect to $\xi$.
By Lemma \ref{lem_Frob}, $\psi_{M-1}$ also satisfies this property.
Hence, the assertion holds by the induction.
\end{proof}

\begin{cor}
\label{cor_app}
Let $r \geq 0$, $p_1,\dots,p_r,q \in \R_{>0}$ and $N,M \in \Z_{>0}$, $h\in \C$,
$\{f_k(\ze_1,\dots,\ze_r,\xi)\}_{k=0,1,\dots,N-1}$
be continuous functions on $\D_\mfp^\cut \times \D_{q}$,
which is holomorphic with respect to $\xi \in \D_q$
and 
\begin{align*}
c_{l,i}(\ze_1,\dots,\ze_r) \in \C[[\ze_1,\dots,\ze_r]]_\mfp^\conv[\ze_1^\C,\log \ze,\dots,\ze_r^\C,\log\ze_r]
\end{align*}
for $l=0,1,2,\dots$, $i=0,1,\dots,M$ be convergent formal power series.
Set 
\begin{align}
\psi=\sum_{i=0}^M \sum_{l=0}^\infty c_{l,i}(\ze_1,\dots,\ze_r)\xi^{h+l} (\log\xi)^i,
\label{eq_app_psi2}
\end{align}
which is a formal power series of continuous function coefficients.
Assume that $\psi$ formally satisfies
\begin{align}
\left(\left(\xi \frac{d}{d\xi}\right)^N + \sum_{k=0}^{N-1} f_k(\ze_1,\dots,\ze_r,\xi)\left(\xi\frac{d}{d\xi}\right)^k\right)\psi=0.
\label{eq_diff_xi}
\end{align}
Then, the series $\psi$ converges absolutely and locally uniformly on $\D_\mfp^\cut \times \D_q$.
In particular, $\psi$ is a holomorphic function on $\D_{\mfp}^\cut \times \D_q^\cut$.
\end{cor}

\begin{proof}
By definition, $\{c_{l,i}(\ze_1,\dots,\ze_r)\}_{l=0,1,\dots,i=0,1,\dots,M}$ are continuous function
on $\D_\mfp^\cut$.
Hence, by Proposition \ref{prop_one_convergence}, the series $\psi$ converges absolutely and uniformly.
Since each term of $\psi$ is a holomorphic function $\D_\mfp^\cut \times \D_q^\cut$,
$\psi$ is also holomorphic. Hence, the assertion holds.
\end{proof}

\noindent
\begin{center}
{\bf Acknowledgements}
\end{center}

I wish to express my gratitude to Tomoyuki Arakawa for valuable discussions and comments.
I would also like to thank Shigenori Nakatsuka
and Thomas Creutzig for valuable comments.
This work is partially supported by the Research Institute for Mathematical Sciences,
an International Joint Usage/Research Center located in Kyoto University,
and RIKEN iTHEMS Program.

\bibliographystyle{alpha}
\bibliography{references}

\end{document}